\documentclass[a4paper, 12pt, oneside, notitlepage]{amsart}
\usepackage[margin=3cm]{geometry}
\usepackage{amsmath,amssymb,amsthm,graphicx,mathrsfs,bbm,url,stmaryrd,amscd,cancel}
\usepackage[normalem]{ulem}
\usepackage{marginnote}
\usepackage{amsthm}
\usepackage{wrapfig}
\usepackage{enumitem}
\usepackage{mathtools}
\usepackage[utf8]{inputenc}
 \usepackage[T1]{fontenc}
\usepackage[usenames,dvipsnames]{color}
\usepackage[colorlinks=true,linkcolor=Red,citecolor=Green]{hyperref}
\usepackage[super]{nth}
\usepackage[open, openlevel=2, depth=3, atend]{bookmark}
\hypersetup{pdfstartview=XYZ}
\usepackage[font=footnotesize]{caption}
\usepackage{a4wide}

\theoremstyle{plain}
\newtheorem{theorem}{Theorem}[section]
\newtheorem*{theorem*}{Theorem}

\newtheorem{lemma}[theorem]{Lemma}
\newtheorem{proposition}[theorem]{Proposition}
\newtheorem{corollary}[theorem]{Corollary}

\theoremstyle{definition}
\newtheorem{definition}[theorem]{Definition}
\newtheorem{example}[theorem]{Example}

\theoremstyle{remark}
\newtheorem{remark}[theorem]{Remark}

\numberwithin{equation}{section}

\newcommand{\C}{\mathbb{C}}
\newcommand{\R}{\mathbb{R}}

\newcommand{\CC}{\mathbb{C}}
\newcommand{\RR}{\mathbb{R}}
\newcommand{\HH}{\mathbb{H}}

\newcommand{\Z}{\mathbb{Z}}

\newcommand{\V}{\mathbb{V}}

\newcommand{\s}{\mathbf{s}}
\newcommand{\z}{\mathbf{z}}
\newcommand{\e}{\mathbf{e}}
\newcommand{\X}{\mathbf{X}}

\newcommand{\mc}{\mathcal}

\newcommand{\E}{\mathcal{E}}

\newcommand{\M}{\mathcal{M}}
\newcommand{\N}{\mathcal{N}}

\newcommand{\fo}{\mathfrak{o}}
\newcommand{\fh}{\mathfrak{h}}

\newcommand{\fg}{\mathfrak{g}}
\newcommand{\fz}{\mathfrak{z}}
\newcommand{\fk}{\mathfrak{k}}

\newcommand{\fn}{\mathfrak{n}}
\newcommand{\fs}{\mathfrak{s}}
\newcommand{\fp}{\mathfrak{p}}

\newcommand{\fm}{\mathfrak{m}}

\newcommand{\fa}{\mathfrak{a}}
\newcommand{\fap}{{\fa_{+}}}
\newcommand{\fapp}{{\fa_{++}}}

\newcommand{\roots}{\Delta}

\newcommand{\odms}{\Omega_{\mathrm{DMS}}^\scrC}

\newcommand{\scrC}{\mathscr{C}}
\newcommand{\scrL}{\mathscr{L}}

\newcommand{\scrK}{\mathscr{K}}
\newcommand{\scrV}{\mathscr{V}}
\newcommand{\scrU}{\mathscr{U}}
\newcommand{\scrJ}{\mathscr{J}}

\newcommand{\eps}{\varepsilon}

\newcommand{\ms}{\mathscr}

\newcommand{\dd}{\mathrm{d}}

\DeclareMathOperator{\vol}{vol}

\DeclareMathOperator{\Tr}{Tr}

\DeclareMathOperator{\WF}{WF}

\DeclareMathOperator{\ran}{ran}

\DeclareMathOperator{\supp}{supp}

\DeclareMathOperator{\comp}{comp}

\let\Re\undefined
\DeclareMathOperator{\Re}{Re}
\let\Im\undefined
\DeclareMathOperator{\Im}{Im}

\newcommand{\be}{\begin{equation}}
\newcommand{\ee}{\end{equation}}

\newcommand{\ProofDel}[1]{{\color{red}\sout{#1}}}

\newcommand{\ProofMathDel}[1]{{\color{red}\cancel{#1}}}
\newcommand{\ProofIssue}[2]{%
  \textsuperscript{{\color{Purple}\tiny$\blacktriangle$}}%
  \marginnote{\raggedright\tiny\color{Purple}\textbf{#1}: #2}%
}
\newcommand{\ProofIssueV}[3]{%
  \textsuperscript{{\color{Purple}\tiny$\blacktriangle$}}%
  \marginnote{\raggedright\tiny\color{Purple}\textbf{#1}: #2}[#3]%
}
\title[The spectrum of Anosov representations]{The spectrum of Anosov representations}

\author[Guedes Bonthonneau]{Yannick Guedes Bonthonneau}
\address{LAGA, CNRS, 99 avenue Jean Baptiste Clément, Villetaneuse.}
\email{bonthonneau@math.univ-paris13.fr}

\author[Lefeuvre]{Thibault Lefeuvre}
\address{Université Paris-Saclay, CNRS, Laboratoire de mathématiques d’Orsay, 91405, Orsay, France.}
\email{thibault.lefeuvre1@universite-paris-saclay.fr}

\author[Weich]{Tobias Weich}
\address{Institut für Mathematik, Universit\"at Paderborn}
\email{weich@math.upb.de}

\begin{document}

\begin{abstract}
Given a $P$-Anosov representation into a noncompact semisimple real algebraic group $G$, where $P < G$ is a parabolic subgroup, we construct a natural resonance spectrum for the classical dynamics associated with the representation. This spectrum is a complex analytic hypersurface in $(\fa_P^*)_{\C}$, the complexified dual of the Lie algebra of the split component of the associated Levi group $L < P$. We reinterpret several objects from the theory of Anosov representations within this spectral framework and investigate, in higher rank, questions that are classically related to Pollicott-Ruelle theory in the rank-one setting. In particular, the “leading resonance”—which is now a hypersurface—is identified with the critical hypersurface of the representation and, as in the rank-one case, the resonant states lead to continuous families of invariant measures. We prove that the multivariate zeta functions and Poincaré series associated with Anosov representations admit a meromorphic extension to $(\fa_P^*)_{\C}$. We also establish an asymptotic expansion in inverse powers of time for the correlation function of the diagonal flow under a Diophantine condition on the representation. Most of our results concerning Anosov representations are obtained as a byproduct of a general theory of \emph{Axiom A actions of type $(k,1)$}, where $k+1:=\dim \fa_P$, that we introduce in the article. 
\end{abstract}

\maketitle
\thispagestyle{empty}

\section{Introduction}

\subsection{Motivation} 

Let $S$ be a closed oriented surface of genus $\geq 2$ and let $\Gamma:=\pi_1(S)$. The Teichmüller space $\mc{T}$ is the space of hyperbolic metrics on $S$, up to isometries that are isotopic to the identity. It admits a natural representation-theoretic description as follows. The holonomy map identifies $\mathcal T$ with one of the two Fuchsian components of the character variety
\[
 \operatorname{Hom}(\Gamma,\mathrm{PSL}_2(\R))/\mathrm{PSL}_2(\R),
\]
that is, with a connected component consisting entirely of discrete and faithful representations; see \cite{Goldman-88}. This point of view led Hitchin to consider, for a split real semisimple Lie group $G$, the character variety $\mathrm{Hom}(\Gamma,G)/G$ of $\Gamma$ in $G$, where $G$ acts on $\mathrm{Hom}(\Gamma,G)$ by conjugation, and to single out distinguished components now known as \emph{Hitchin components}; see \cite{Hitchin-92}. This is generally regarded as the first example of a \emph{higher Teichmüller} space. It was subsequently incorporated into the broader theory of \emph{Anosov representations}, initiated by Labourie \cite{Labourie-06}: for a word-hyperbolic group $\Gamma$ and a proper parabolic subgroup $P<G$, the \emph{$P$-Anosov representations} (see Definition \ref{definition:theta-anosov}) form an open subset of the character variety of $\Gamma$ in $G$; see, for instance, \cite{Guichard-Wienhard-12,Gueritaud-Guichard-Kassel-Wienhard-17,Kapovich-Leeb-Porti-17}.

An important feature of these deformation spaces is that an Anosov representation $\rho:\Gamma\to G$ naturally gives rise to much more than a point in a character variety. Indeed, with each representation, one can typically associate geometric structures, dynamical systems, and spectral invariants, and much of the richness of the subject stems from the interplay between these different perspectives. The main purpose of the present paper is to construct a natural spectral invariant for general Anosov representations. To motivate our construction, we first review the rank-one picture and explain what has been missing so far in higher rank.

\subsubsection{Compact rank one setting}

By the uniformization theorem, a point $\rho\in\mathcal T$ determines a marked closed hyperbolic surface
\[
 X_\rho:=\rho(\Gamma)\backslash\mathbb H^2
 \simeq \rho(\Gamma)\backslash\mathrm{PSL}_2(\R)/\mathrm{PSO}(2).
\]
The positive Laplace--Beltrami operator on $X_\rho$ has purely discrete $L^2$-spectrum in $[0,\infty)$. It is a fundamental spectral-geometric invariant of the hyperbolic structure and is regarded as the \emph{quantum} spectrum of $X_\rho$ in mathematical physics. Its study remains a very active subject; see, for instance, the recent developments around the $1/4$-conjecture
 \cite{Wu-Xue-22,Lipnowski-Wright-24,Hide-Magee-23,Anantharaman-Monk-25,Hide-Macera-Thomas-25}.
 
The surface $X_\rho$ or, rather, its unit tangent bundle $SX_\rho$, also carries the geodesic flow, a paradigmatic model of chaotic dynamics. Under the identification $SX_\rho\simeq\rho(\Gamma)\backslash\mathrm{PSL}_2(\R)$, this flow is the right action of the diagonal subgroup
\[
A=\left\{\begin{pmatrix}e^{t/2}&0\\0&e^{-t/2}\end{pmatrix}:t\in\R\right\}.
\]
The corresponding \emph{classical} spectral invariant is called the \emph{Pollicott--Ruelle resonance spectrum} of the generator of the flow. It is a discrete subset of $\C$ and was originally studied to understand invariant measures and mixing properties of the geodesic flow; see \cite{Pollicott-85, Ruelle-86, Ruelle-87, Liverani-04,Faure-Sjostrand-11,Dyatlov-Zworski-16} among other references. For closed hyperbolic manifolds in all dimensions, there is an identification of the first band of these resonances, up to explicit shifts, with spectral parameters of the Laplace--Beltrami operator \cite{Dyatlov-Faure-Guillarmou-15,Guillarmou-Hilgert-Weich-18}. This is known as a \emph{quantum--classical correspondence}.

\subsubsection{Noncompact rank one setting}
The spectral picture becomes more delicate in the noncompact setting. In real rank one, (images of) Anosov representations are precisely the convex co-compact subgroups. If $X=\rho(\Gamma)\backslash\mathbb H^2$ is convex co-compact but noncompact, then the $L^2$-spectrum of its Laplacian is no longer discrete: it consists of the absolutely continuous spectrum $[1/4,\infty)$ together with at most finitely many eigenvalues in $(0,1/4)$ which are very often absent altogether, see \cite{Borthwick-07}. A discrete \emph{quantum} spectral invariant is recovered by meromorphically continuing the cut-off resolvent of the Laplacian. Its poles are the \emph{Laplace} or \emph{quantum resonances} \cite{Mazzeo-Melrose-87,Guillope-Zworski-97,Borthwick-07}. On the classical side, the geodesic flow has a discrete Pollicott--Ruelle resonance spectrum associated with its trapped set \cite{Dyatlov-Guillarmou-16}. The quantum--classical correspondence continues to relate these two resonance spectra \cite{Guillarmou-Hilgert-Weich-18}.

\subsubsection{Compact higher rank setting} Let us describe the natural analogues of the Laplace--Beltrami operator and the geodesic flow in higher rank. The Laplacian is replaced by the commutative algebra $\mathbb D(G/K)$ of invariant differential operators on the Riemannian symmetric space $G/K$ (where $K$ is a maximal compact subgroup), while the geodesic flow is replaced by the Weyl chamber action of a higher-dimensional split torus $A \simeq \R^{\mathrm{rank}(G)}$ on $G/M$, where $M$ is the centralizer of $A$ in $K$.

If $\Gamma<G$ is a uniform lattice, $\Gamma\backslash G/K$ is compact and the joint $L^2$-spectrum of $\mathbb D(G/K)$ is discrete. As $\mathbb{D}(G/K)$ is generated by $\dim(\fa)$ generators, the spectrum is a discrete subset of $\CC^{\dim(\fa)}$ and, via the Harish--Chandra isomorphism, it can be intrinsically expressed as a discrete subset of $\fa_\CC^*$.
It is a well-studied spectral invariant: the higher-dimensional spectral problem for $\mathbb D(G/K)$ was already proposed by Selberg \cite{Selberg-56}; see also the seminal contributions of Gel'fand \cite{Gelfand-62}, Gangolli \cite{Gangolli-68}, and Duistermaat--Kolk--Varadarajan \cite{Duistermaat-Kolk-Varadarajan-79}.
On the classical side, a discrete joint resonance spectrum in $\fa_\CC^*$ for the Weyl chamber action on $\Gamma\backslash G/M$ was constructed more recently in \cite{Guedes-Bonthonneau-Guillarmou-Hilgert-Weich-20,Guedes-Bonthonneau-Guillarmou-Weich-24}. Again, a quantum--classical correspondence showed that, in the cocompact higher-rank case, classical and quantum spectra are still two sides of the same coin; see \cite{Hilgert-Weich-Wolf-23}.

\subsubsection{Noncompact higher rank setting} Once one replaces the uniform lattice $\Gamma$ by the image of an Anosov representation $\rho : \Gamma \to G$, a natural ``good'' joint spectrum is no longer available. One can, of course, still consider the joint $L^2$-spectrum of $\mathbb D(G/K)$ on $\rho(\Gamma)\backslash G/K$. However, as in the case of noncompact convex co-compact surfaces, this does not appear to be a satisfactory spectral invariant. Indeed, the results of \cite{Weich-Wolf-23, Lutsko-Weich-Wolf-26} show that in most cases the spectrum is the entire imaginary plane without embedded eigenvalues and carries practically no information on $\rho(\Gamma)$. To the best of our knowledge, a theory of joint \emph{quantum} resonances for $\mathbb D(G/K)$ on noncompact higher-rank locally symmetric spaces is entirely absent. Likewise, there is currently no spectral theory for higher-rank Weyl chamber flows on such spaces.

\subsubsection{Content of the present paper}

The present paper intends to develop the spectral theory of higher-rank Weyl chamber flows on noncompact higher-rank spaces. Given a $P$-Anosov subgroup, we use a smooth domain of discontinuity around its trapped set to construct a joint resonance spectrum for the $A_{P}$-action. The approach of using this domain of discontinuity for the spectral theory is inspired by the recent work of Delarue-Monclair-Sanders \cite{Delarue-Monclair-Sanders-24,Delarue-Monclair-Sanders-25}. In contrast to the cocompact case, unless $P$ is a maximal proper parabolic subgroup, this spectrum is no longer discrete but a complex subvariety of codimension one in $\fa^*_\C$ (or in $(\fa_P^*)_\C$ if $P$ is not minimally parabolic).

We present many strong indications that this is a natural spectral invariant for Anosov subgroups:
\begin{itemize}
  \item Its leading real resonance hypersurface recovers the critical hypersurface of the representation which ubiquitously appears in the study of Anosov representations, see Theorem~\ref{theorem:intro-leading};
  \item For each point on the leading real resonance hypersurface, the corresponding resonant states lead to natural measures. Working on different types of tensor bundles, we recover Patterson--Sullivan, Burger--Roblin, and Bowen--Margulis--Sullivan measures that have been studied before in the context of homogeneous dynamics (cf. \cite{Quint-2002,Sambarino-14,ELO23,Lee-Oh-20}) and provide a unifying spectral framework for the analysis of these measures; see Theorems~\ref{theorem:leading} and~\ref{theorem:measures-max-entropy};
  \item In the rank one case, or when $P$ is maximally parabolic, it holds that $\dim(\fa_P)= 1$, thus a complex subvariety of codimension one is discrete. In this case, our resonance spectrum reduces to the usual Pollicott-Ruelle spectrum of the geodesic flow on the convex cocompact locally symmetric space introduced in \cite{Dyatlov-Guillarmou-16}, respectively to the Pollicott-Ruelle spectrum of the Axiom A flow constructed in \cite{Delarue-Monclair-Sanders-24,Delarue-Monclair-Sanders-25}.
\end{itemize}

Furthermore, as in the rank one case, the spectral framework allows us to prove new results that, \emph{a priori}, have nothing to do with spectral theory, but are of interest in their own right. In particular, we prove:
\begin{itemize}
\item A meromorphic continuation of multivariate zeta functions, see Theorem~\ref{theorem:intro-zeta-anosov};
\item An asymptotic expansion of the correlation function of the Weyl chamber action under a Diophantine condition on the representation, see Theorem~\ref{theorem:mixing-intro};
\item A meromorphic continuation of the Poincaré series associated with the Anosov representation, see Theorem~\ref{theorem:intro-poincare}.
\end{itemize}

\subsection{Higher-rank actions associated with Anosov subgroups}

We now describe the higher-rank action appearing in our results and introduce relevant notation. We refer to \S\ref{section:structure} for a more detailed exposition.

\subsubsection{Higher rank setting}

Throughout the article, $G$ denotes a noncompact connected semisimple real algebraic group.

We fix the identity component of a maximal split torus $A\cong\RR^{n}< G$, where $n$ is the rank of $G$. Let $P< G$ denote a proper parabolic subgroup containing $A$. The parabolic group $P$ does not need to be minimal; after conjugation, every parabolic may be assumed to contain $A$. We then choose a Langlands decomposition
\[
P=M_P A_P N_P,
\]
where $M_P$ is reductive, $A_P\cong \RR^{k+1}< A$ for some $k\geq 0$, and $N_P$ is unipotent. Let $\fa_P:=\operatorname{Lie}(A_P)$. As $A_P$ and $M_P$ commute, there is a natural homogeneous dynamical system via the right action of $A_P$ on $G/M_P$:
\[
A_P \times G/M_P \to G/M_P, \qquad (a,gM_P) \mapsto g a M_P.
\]
We will study a locally homogeneous version of this action. That is, we will quotient on the left by a discrete subgroup $\Gamma < G$.

In a higher-rank dynamical system, it is natural to single out an open cone of ``positive'' time directions. In the present setting, such a cone is given by a \emph{generalized Weyl chamber} $\scrC\subset \fa_P$ consisting of those elements $a \in \fa_P$ such that $\operatorname{ad}(a)|_{\fn_P}$ has only positive eigenvalues, where $\fn_P$ is the Lie algebra of $N_P$ (see \S\ref{sssec:weyl-chambers}). Since $\fa_P \subset \fa$, the cone $\scrC$ can also be viewed as a cone of $\fa$. Note that it is open in $\fa$ if and only if $\fa_P = \fa$; equivalently, $P$ is a minimal parabolic subgroup of $G$. Such open cones in $\fa$ are called \emph{open Weyl chambers} and will usually be denoted by $\fapp \subset \fa$. In what follows, we fix an open Weyl chamber $\fapp\subset\fa$ such that $\scrC\subset\overline{\fapp} =: \fap$. 

Conversely, parabolic subgroups of $G$ containing $A$ can be parametrized by cones $\scrC\subset \fa$ of this type (see \S\ref{sssec:levi}). We shall therefore use $\scrC$ to label the corresponding parabolic subgroup and its associated data, writing $P_\scrC=P, A_\scrC=A_P, M_\scrC=M_P, N_\scrC^+=N_P$.

\begin{example}[$\mathrm{PSL}(2,\R)$] The toy model in this setting is $G=\operatorname{PSL}(2,\RR)$, with $\fa = \{\mathrm{diag}(t,-t) ~:~ t\in\RR\}$. In this case, there are two proper parabolic subgroups containing $A$: the upper and lower triangular subgroups, corresponding to the cones  $\scrC_\pm = \{\mathrm{diag}(t,-t) : t\in\RR_\pm\}$. In both cases, $M_\scrC$ is trivial, and $A_{\pm\scrC} \cong A\cong \RR$ acts from the right on $\Gamma\backslash G$. It is exactly the geodesic flow on $\Gamma\backslash \HH^2$ and the choice of $\pm\scrC$ is simply the choice of positive time direction in $A$. 
\end{example}

\begin{example}[$\mathrm{SL}(3,\R)$] For $G = \mathrm{SL}(3,\R)$, we may take $\fa = \{\mathrm{diag}(x_1,x_2,x_3) ~:~ x_1+x_2+x_3=0\}$ with simple roots $\alpha_1 = x_1-x_2$ and $\alpha_2 = x_2-x_3$. We fix the open Weyl chamber $\fapp = \{x_1 > x_2 > x_3\} = \{\alpha_1 > 0, \alpha_2 > 0\}$, $\fap = \overline{\fapp}$. Relative to these choices, there are three natural choices of cones $\scrC \subset \fap$ leading to the standard proper parabolics containing $A$:
\begin{itemize}
\item The minimal parabolic (upper triangular matrices) corresponds to $\fa_{\scrC}=\fa$, $\scrC = \fapp$;
\item The maximal parabolic of type $(1,2)$, stabilizing the line $\R e_1$, corresponds to $\fa_{\scrC} = \ker \alpha_2 = \{\mathrm{diag}(2t,-t,t) ~:~ t \in \R\}$, and the positive cone is the ray $\scrC=\{\mathrm{diag}(2t,-t,t) ~:~ t > 0\}$. Its nilpotent radical corresponds to the roots $\alpha_1,\alpha_1+\alpha_2$;
\item The maximal parabolic of type $(2,1)$, stabilizing the plane $\R e_1 \oplus \R e_2$, corresponds to $\fa_{\scrC} = \ker \alpha_1 = \{\mathrm{diag}(t,t,-2t) ~:~ t \in \R\}$ with $\scrC = \{\mathrm{diag}(t,t,-2t) ~:~ t > 0\}$. Its nilpotent radical corresponds to $\alpha_2, \alpha_1 + \alpha_2$.
\end{itemize}
\end{example}

\begin{example}[Minimal parabolic subgroup] In higher rank, for a minimal parabolic $P_\scrC$, $M_\scrC$ is compact, and the right $A_\scrC$-action on $\Gamma\backslash G/M_\scrC$ is the \emph{Weyl chamber flow}. This flow has been widely studied: for uniform lattices, it is a standard example of a higher-rank Anosov action \cite{Katok-Spatzier-94}; for nonuniform lattices, it is studied in a number-theoretic context \cite{Einsiedler-Lindenstrauss-Michel-Venkatesh-11}.
\end{example}

\subsubsection{Anosov representations and subgroups}

We will study the case of Anosov subgroups that are associated with general proper parabolic subgroups $P_\scrC < G$.

Let $\mu_\fapp,\lambda_\fapp:G\to\overline{\fapp}$ be the associated Cartan and Jordan projections, respectively. 
If $\pi_\scrC:\fa\to\fa_\scrC$ denotes the orthogonal projection (with respect to the inner product coming from the restriction of the Killing form to $\fa$), we set $\mu_\scrC:=\pi_\scrC\circ\mu_\fapp$ and $\lambda_\scrC:=\pi_\scrC\circ\lambda_\fapp$, see \S\ref{section:structure} for details. Let $\roots\subset\fa^*$ be the roots of $(G,A)$. Then the parabolic subgroup $P_{\scrC}$ determines the subset $\roots_\scrC^+:=\{\alpha\in\roots:\fg_\alpha\subset\fn_\scrC^+\}$ and, conversely, this set of roots uniquely determines $P_\mathscr C$, which is the description of parabolic subgroups used in \cite[Section 2.2]{Gueritaud-Guichard-Kassel-Wienhard-17} for instance. The roots in $\Delta$ give rise to the following definition of Anosov representations based on Cartan projections:

\begin{definition}[$P_\scrC$-Anosov representation/subgroup]
\label{definition:theta-anosov}
Let $\Gamma$ be a finitely generated group, $|\bullet|$ be a word length on $\Gamma$ chosen with respect to some finite generating set, and $\rho : \Gamma \to G$ be a discrete and faithful representation. The representation $\rho$ is called \emph{$P_\scrC$-Anosov} if there exist positive constants $C_1,C_2 > 0$ such that for all $\gamma \in \Gamma$ and $\alpha \in \roots_\scrC^+$,
\[
\alpha\left( \mu_\fapp(\rho(\gamma))\right) \geq C_1|\gamma|-C_2.
\]
If $P_\scrC$ is a minimal parabolic, we say that the representation is \emph{Borel Anosov}\footnote{We use this slightly imprecise terminology because it is widely used in the literature, even though, when $G$ is not quasi-split, minimal parabolic subgroups are not Borel subgroups in the sense of being maximal connected solvable algebraic subgroups.}.
A \emph{$\scrC$-Anosov subgroup} is the image $\rho(\Gamma)$ in $G$ of a $\scrC$-Anosov representation. To simplify the notation, we also use the expressions \emph{$\scrC$-Anosov representation} and \emph{$\scrC$-Anosov subgroup}.
\end{definition}

This definition is not the original one from Labourie \cite{Labourie-06}, but it was shown to be equivalent; see \cite{Guichard-Wienhard-12,Gueritaud-Guichard-Kassel-Wienhard-17, Kapovich-Leeb-Porti-17,Bochi-Potrie-Sambarino-19} for instance. The Anosov property implies that $\Gamma$ is word-hyperbolic (see \cite[Theorem 1.4]{Kapovich-Leeb-Porti-18} or \cite[Theorem 3.2]{Bochi-Potrie-Sambarino-19}). Throughout this article, we will mostly use the notion of Anosov subgroups rather than Anosov representations. In particular, in the rest of the introduction, $\Gamma < G$ now denotes a $\scrC$-Anosov subgroup in the sense of Definition \ref{definition:theta-anosov}.

\subsubsection{Axiom A actions of type $(k,1)$}

We now further assume that $\Gamma < G$ is a torsion-free $\scrC$-Anosov subgroup.

In general, $M_\scrC$ is noncompact for a nonminimal parabolic $P_\scrC$, so $\Gamma$ need not act properly discontinuously on the whole of $G/M_\scrC$. A central insight of Delarue--Monclair--Sanders \cite{Delarue-Monclair-Sanders-24,Delarue-Monclair-Sanders-25} is that, for $\scrC$-Anosov subgroups, there exist non-empty $A_\scrC$-invariant open subsets $\odms\subset G/M_\scrC$ such that $\Gamma$ acts properly discontinuously on them. We can thus consider the (non-empty) quotient manifold $\M := \Gamma\backslash\odms$; by construction, the right $A_\scrC$-action descends to the quotient. In addition, $\M$ contains a closed $A_\scrC$-invariant subset $\scrJ \subset \M$ (called the \emph{trapped set}) where the right $A_\scrC$-action has interesting dynamical features, which we now introduce.

Define the \emph{incoming} ($-$) and \emph{outgoing} ($+$) \emph{tails} as
\[
 \mc T_\pm := \{gM_\scrC \in \M ~:~ \overline{\scrC}\ni a \mapsto g\exp(\mp a)M_\scrC \text{ is not a proper map}\},
\]
and the trapped set as
\[
 \scrJ := \mc T_+\cap\mc T_-.
\]
The tails are also called forward ($-$) and backward ($+$) tails. The trapped set $\scrJ$ should be thought of as a higher rank analogue of the trapped set in rank $1$ (the set of unit vectors generating a geodesic trapped in a compact region of the manifold) but the main difference is that it is now noncompact. It is the relevant set for the dynamics of the $A_\scrC$-action.

We will establish that the $A_\scrC$-action on $\scrJ$ is hyperbolic. That is
there is a continuous splitting of $T\M$ over $\scrJ$
\[
T_{\scrJ}\M = E_0\oplus E_s\oplus E_u,
\]
such that $E_0$ is tangent to the $A_\scrC$-orbits, and there is a Riemannian metric on $\M$ such that for each $a\in\scrC$ there exist constants $C_a,\beta_a>0$ such that, for all $t\geq0$,
\begin{align*}
 \|d(\exp(ta)) v\| \leq C_a e^{-\beta_a t}\|v\|, & \qquad \forall v\in E_s|_{\scrJ},\\
 \|d(\exp(ta)) v\| \geq C_a^{-1} e^{\beta_a t}\|v\|, & \qquad  \forall v\in E_u|_{\scrJ}.
\end{align*}
See \eqref{eq:stable-unstable} and Proposition~\ref{prop:hyperbolicity} for the relevant technical lemmas establishing hyperbolicity, and Theorem~\ref{thm:DMS-main} which summarizes the main properties of the $A_\scrC$-action on $\M$. Additionally, it is known that the higher-rank $A_\scrC$-action on $\scrJ$ fibers as an Abelian cocycle extension over a \emph{metric Anosov flow} on $\scrJ$ (see \cite{Sambarino-14,Sambarino-15}). However, the trapped set $\mathscr{J}$ typically has a highly fractal structure—even in rank $1$—and is therefore poorly suited for analysis. It is merely a metric space, on which one cannot define distributions of arbitrary order. 

It turns out that the action on $\M$ also fibers as a cocycle over a \emph{smooth} Axiom A flow. See \cite[Corollary B]{Delarue-Monclair-Sanders-25} and Theorem~\ref{thm:DMS-main} below for a slightly refined version adapted to our spectral-theoretic purposes and the application to Poincaré series. We will refer to such an object as an \emph{Axiom A action of type $(k,1)$} (see Definition~\ref{definition:axiom-a-action} for details). Let us emphasize that the ability to work on a smooth manifold $\M$ (as opposed to the fractal metric space $\scrJ$) is crucial for microlocal analysis. Just as Pollicott--Ruelle resonances extend beyond geodesic flows on locally symmetric spaces, our resonance analysis applies in the broader setting of Axiom A actions of type $(k,1)$ . These are higher-rank dynamical systems satisfying certain assumptions on the dynamics on $\scrJ$ and $\mc T_\pm$ and carrying the free Abelian cocycle structure described above. Theorem~\ref{thm:DMS-main} shows that actions arising from Anosov representations belong to this broader class. For the sake of simplicity, we only state the results in the introduction for Anosov representations and refer the reader to \S\ref{sec:introduction-geometric-setting}--\S\ref{section:decay} for more general statements.

\subsection{Resonance spectrum of Anosov representations}
\label{ssection:spectral-theory}

We can now introduce the resonance spectrum as a joint spectrum of the commuting vector fields generating the $A_\scrC$-action.

\subsubsection{Definition and existence}

As $A_\scrC$ acts on $\M$ by diffeomorphisms, there is a corresponding infinitesimal representation of $\fa_\scrC$ by vector fields on $\M$, 
denoted by
\[
\X : \fa_\scrC \to C^\infty(\M,T\M), \qquad a \mapsto \X(a).
\]
We define a joint spectrum for this commuting family through distributional joint eigenstates with geometric support and wavefront restrictions.
For this, we rely on the standard notion of wavefront set $\WF$ (see \cite{Hormander-90}), which records the locations and cotangent directions of singularities of a distribution. We also use the co-unstable bundle $E_u^*:=(E_0\oplus E_u)^\perp \subset T^*\M$, defined in \eqref{equation:t*m}.

We first introduce the notion of \emph{dynamical resonances}:

\begin{definition}[Dynamical resonances]
\label{definition:rt}
Let $\scrC\subset \fa$ be a generalized Weyl chamber, and $\Gamma < G$ be a torsion-free $\scrC$-Anosov subgroup. The point $\mathbf{s} \in (\fa_\scrC^*)_{\C}$ is
a dynamical resonance if there exists a distribution $u \in \mc{D}'(\M), u \neq 0,$ such that
\begin{equation}\label{eq:res_cond}
\WF(u) \subset E_u^*, \qquad \supp(u) \subset \mc{T}_+, \qquad (-\X(a)-\mathbf{s}(a)) u = 0, \quad\forall a\in\fa_\scrC.
\end{equation}
The resonance spectrum is the set of all dynamical resonances and is denoted by $\sigma_{\mathrm{RS}}^{(0)}$.
\end{definition}
The dynamical resonances are the joint eigenvalues of the infinitesimal $\fa_\scrC$-action on the space of distributions on $\M$ with wavefront set in $E_u^*$ and support in the tail $\mc{T}_+$. The latter two conditions are natural as they already appear in the context of Axiom A flows; see \cite{Dyatlov-Guillarmou-16}. In addition, the set of dynamical resonances can be shown to coincide with a joint \emph{Taylor spectrum} on anisotropic Hilbert spaces (see an earlier version of the present article available online \cite[\S4.7]{arxiv_v1}). However, as this statement involves substantial additional notation which is not used in the rest of the article, we refrain from proving it here.

More generally, the resonance spectrum can be defined for any \emph{admissible bundle} $\E \to \M$ (see \S\ref{sssection:admissible} for the definition of admissible bundles). In the introduction, we restrict to one of the most important examples. Let $\Sigma_\M$ denote the (smooth) vector bundle of $1$-forms in the joint kernel of the contractions $\iota_{\X(a)}$, $a\in\fa_\scrC$. Note that $\mathrm{rk}(\Sigma_\M) = \dim(\M)-(k+1)$.
In other words, $\Sigma_\M$ is the bundle of covectors transverse to the $A_\scrC$-orbits.
Using the Lie derivative $\X(a)u := \mathcal L_{\X(a)}u$ for $u \in C^\infty(\M,\Lambda^m \Sigma_{\M})$, we can extend the operator to
\[
\X : C^\infty(\M,\Lambda^m \Sigma_{\M}) \otimes \fa_\scrC \to C^\infty(\M,\Lambda^m \Sigma_{\M}).
\]
Similarly to \eqref{eq:res_cond}, we then define the dynamical resonances $\sigma_{\mathrm{RS}}^{(m)} \subset (\fa_\scrC^*)_{\C}$ as the set of all $\mathbf{s} \in (\fa_\scrC^*)_{\C}$ such that there exists $u \in \mc{D}'(\M, \Lambda^m \Sigma_{\M}), u \neq 0$ such that \eqref{eq:res_cond} holds.

We shall prove the following:

\begin{theorem}[Existence of a resonance spectrum]
\label{theorem:intro-rt}
Let $\scrC\subset \fa$ be a generalized Weyl chamber, and $\Gamma < G$ be a torsion-free $\scrC$-Anosov subgroup. For $0\leq m\leq\operatorname{rk}(\Sigma_\M)$, the resonance spectrum $\sigma_{\mathrm{RS}}^{(m)} \subset (\fa_\scrC^*)_{\C}$ is a complex analytic variety of codimension $1$.
\end{theorem}

That is, $\sigma_{\mathrm{RS}}^{(m)}$ is globally described as the zero set of a holomorphic function. This holds more generally for any admissible bundle $\E \to \M$. The (non-)connectedness of $\sigma_{\mathrm{RS}}^{(m)} \subset (\fa_\scrC^*)_{\C}$ is still poorly understood by the authors. In rank $1$, one finds that $\sigma_{\mathrm{RS}} \subset \C$ is a collection of isolated points (the poles of a resolvent).

\subsubsection{Leading resonant hypersurface} As in rank $1$, the first resonant hypersurface plays an important role as it controls the relevant statistics of the underlying dynamical system (this will be further discussed in \S\ref{sssection:correlation}). Let us recall that the \emph{Benoist cone} of a discrete subgroup $\Gamma < G$ is defined as the asymptotic cone of its Cartan projection
 	\[
 	\mathscr L :=
 \mathrm{AC}\bigl(\mu_\scrC(\Gamma)\bigr) =
 \left\{
 \lim_{\ell\to\infty}t_\ell \mu_\scrC(\gamma_\ell)
 ~:~
 t_\ell\to0^+,\ \gamma_\ell\in \Gamma
 \right\} \subset \fa_\scrC.
 \] 
 Given $\varphi \in \bigl(\scrL \bigr)^{*,\circ}\subset\fa_\scrC^*$, the interior of the dual Benoist cone (i.e. $\varphi>0$ on $\scrL$),  one defines the following critical exponent
\begin{equation}
\label{equation:delta-intro}
\delta(\varphi) := \lim_{t \to \infty} t^{-1} \log \sharp \{\gamma \in \Gamma ~:~ \varphi(\mu_\scrC(\gamma)) \leq t\} \in (0,\infty).
\end{equation}
The following holds:

\begin{theorem}[Leading resonant hypersurface]
\label{theorem:intro-leading}
There exists a real-analytic codimension $1$ submanifold $\mathbf{C}^{(d_s)} \subset \sigma_{\mathrm{RS}}^{(d_s)} \cap \fa_\scrC^*$ such that:
\begin{enumerate}[label=\emph{(\roman*)}]
\item $\mathbf{C}^{(d_s)}$ coincides with $\{\delta(\varphi)=1\}$;
\item $\fa_\scrC^* \setminus \mathbf{C}^{(d_s)} = \mathbf{C}^{(d_s),+} \sqcup \mathbf{C}^{(d_s),-}$ is the union of two open connected components;
\item $\mathbf{C}^{(d_s),+}$ is convex and is strictly convex if $\Gamma < G$ is Zariski dense;
\item If $0\leq m < d_s$, and $\mathbf{s} \in \sigma^{(m)}_{\mathrm{RS}}$, then
 \[
\Re(\mathbf{s})  \in \mathbf{C}^{(d_s),- }.
\]
\end{enumerate}
\end{theorem}

The real codimension $1$ submanifold $\mathbf{C}^{(d_s)} + i \fa_\scrC^*$ plays the role of the ``critical axis'' $h_{\mathrm{top}}(\phi_1) + i\R$ in rank $1$. We shall call $\mathbf{C}^{(d_s)}$ the \emph{critical hypersurface}, or the \emph{leading resonant hypersurface}. Theorem~\ref{theorem:leading}, items~\emph{\ref{it:leading-strict-convexity}} and~\emph{\ref{it:leading-concavity}}, describes the convexity of $\mathbf{C}^{(d_s),+}$ and shows that its complexification (the local holomorphic codimension $1$ submanifold passing through $\mathbf{C}^{(d_s)}$) $\mathbf{C}^{(d_s)}_{\C} \subset (\fa_\scrC^*)_{\C}$ is concave in the $i  \fa_\scrC^*$ direction at $\mathbf{C}^{(d_s)}$. See Figure \ref{figure:spectrum2}.

In addition, we show that the leading resonant states associated with a point $\varphi\in\mathbf{C}^{(d_s)}$ correspond to the Burger--Roblin measure of the Anosov representation (introduced for Borel Anosov groups in \cite{ELO23}); see \S\ref{ssection:mme} and, in particular, Theorem~\ref{theorem:measures-max-entropy}. The pairing of the resonant state with the corresponding coresonant state associated with $\varphi$ leads to the \emph{Bowen--Margulis--Sullivan (BMS) measure} $\mu_\varphi$ on the trapped set $\scrJ$. This measure has infinite mass when $k>0$ (and is finite when $k=0$) and depends on a choice of $\varphi \in \mathbf{C}^{(d_s)}$. See \S\ref{sssection:correlation}, where it is further described.

\begin{figure}[h!]
\begin{center}
\includegraphics[scale=1.4]{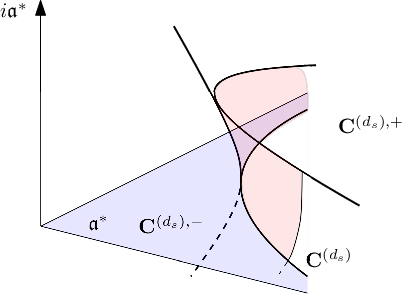}
\caption{In blue: $\fa_\scrC^*$, the real part of $(\fa_\scrC^*)_{\C}$. In red: the complex leading resonant hypersurface $\mathbf{C}^{(d_s)}_{\C}$. The intersection of $\mathbf{C}^{(d_s)}_{\C}$, the red variety, with $\fa_\scrC^*$, the blue plane, is $\mathbf{C}^{(d_s)}$, the critical hypersurface.}
\label{figure:spectrum2}

\end{center}

\end{figure}

Let us emphasize that $\mathbf{C}^{(d_s)}$ was already identified as an important object in the literature (see \cite[Theorem A]{Sambarino-24} for instance). It follows from Proposition \ref{prop:cartan==lyapunov} below that the spectrally defined critical hypersurface $\mathbf{C}^{(d_s)}$ \emph{coincides} with the critical hypersurface identified in the literature.

\subsection{Local mixing/decay of correlations for the diagonal flow} \label{sssection:correlation}

Let $\varphi \in \mathbf{C}^{(d_s)}$ be a point on the critical hypersurface. The tangent space $T_\varphi \mathbf{C}^{(d_s)} \subset \fa_\scrC^*$ induces a line
\[
L_\varphi := \{a \in \fa_\scrC ~:~ \alpha(a) = 0, \forall \alpha \in T_\varphi \mathbf{C}^{(d_s)}\} \subset \fa_\scrC.
\]
Since $T_\varphi \mathbf{C}^{(d_s)} \oplus \R \varphi = \fa_\scrC^*$ (see Lemma \ref{lemma:transverse-intersection}), one can choose $u_\varphi \in L_\varphi$ as the unique vector such that $\varphi(u_\varphi)=1$. We are interested in the ergodic properties of the flow $(e^{tu_\varphi})_{t \in \R}$ generated by $u_\varphi$ on the trapped set $\mathscr{J}$. This flow is called the \emph{diagonal flow} or \emph{directional flow}.\footnote{Sambarino's refraction flow is the flow on the quotient $\mathscr K_\varphi$ introduced below.}

The BMS measure $\mu_\varphi$ introduced in \S\ref{ssection:spectral-theory} has infinite mass, is preserved by $(e^{tu_\varphi})_{t \in \R}$ and, more generally, by the full $A_\scrC$-action, and has support equal to $\mathscr{J}$. It plays the role of a measure of maximal entropy in this problem. More precisely, set $\fh_\varphi:=\ker(\varphi)\subset\fa_\scrC$ and $H_\varphi:=\exp(\fh_\varphi)$. Then $H_\varphi$ acts properly on $\mathscr{J}$ (Theorem \ref{thm:DMS-main}) and
\[
\pi_\varphi : \mathscr{J} \to \mathscr{K}_\varphi, \qquad \mathscr{K}_\varphi := \mathscr{J}/H_\varphi,
\]
is a principal $H_\varphi$-bundle over a compact metric space. The flow $(e^{tu_\varphi})_{t \in \R}$ descends to $\mathscr{K}_\varphi$ and yields a (metric) Anosov flow $(\phi_t^\varphi)_{t \in \R}$. Up to a positive normalization constant, the measure $\mu_\varphi$ is then given by $\mu_\varphi = \vol_{\fh_\varphi} \wedge \pi_\varphi^* \nu_\varphi$, where $\nu_\varphi$ is the (probability) measure of maximal entropy of $(\phi_t^\varphi)_{t \in \R}$ on $\mathscr{K}_\varphi$ and $\vol_{\fh_\varphi}$ is the Lebesgue measure on $\fh_\varphi$.

We are interested in the \emph{local mixing} (or \emph{decay of correlations}) for the diagonal flow with respect to $\mu_\varphi$. If $\Gamma < G$ is Zariski dense, it is known that there exists $\kappa>0$ such that for all $f,g \in C^0_{\comp}(\mathscr{J})$:
\[
\begin{split}
\int_\mathscr{J} f\circ e^{tu_\varphi} \cdot g~ \dd \mu_\varphi &= \kappa \cdot t^{-k/2} \int_\mathscr{J} f ~\dd \mu_\varphi \int_\mathscr{J} g ~\dd \mu_\varphi + o(t^{-k/2}),\qquad t\to+\infty.
\end{split}
\]
This was established by Thirion \cite{Thirion-09} for ping-pong groups in $\mathrm{SL}_n(\R)$, by Sambarino \cite{Sambarino-15} for Zariski-dense hyperconvex representations of surface groups, by Chow--Sarkar \cite{Chow-Sarkar-23} for Borel Anosov subgroups on $\Gamma\backslash G$, and by Sambarino \cite{Sambarino-24} for general Zariski-dense Anosov representations. See also \cite{Dang-24} for the topological mixing of diagonal flows on $\Gamma\backslash G$.

In certain cases, this expansion can be refined. The next statement involves a notion of \emph{Diophantine representation}, which we refrain from introducing here due to its technicality. Roughly speaking, we say that a representation is Diophantine if there exists a finite subset of its Lyapunov spectrum (see \S\ref{sssection:lyapunov}) which is Diophantine, that is, badly approximable by rational numbers (see Definition \ref{definition:diophantine} for a precise statement). In certain cases, this is easily seen to be a generic condition with respect to the representation (Proposition \ref{proposition:almost-diophantine}). For instance, we show that Lebesgue almost every Hitchin representation $\rho : \pi_1(\Sigma) \to \mathrm{SL}_n(\R)$ is Diophantine provided $2g \geq n^3$, where $g$ is the genus of $\Sigma$; see the discussion after Proposition \ref{proposition:almost-diophantine}.

\begin{theorem}[Local mixing/decay of correlations for the diagonal flow]
\label{theorem:mixing-intro}
If $\Gamma < G$ is Zariski dense and Diophantine, then for all $f,g \in C^\infty_{\comp}(\M)$ and $t \geq 1$,
\begin{equation}
\label{equation:stationary-phase-intro}
\begin{split}
\int_\mathscr{J} f\circ e^{tu_\varphi} \cdot g~ \dd \mu_\varphi &= \kappa \cdot  t^{-k/2}\left(\int_\mathscr{J} f ~\dd \mu_\varphi \int_\mathscr{J} g ~\dd \mu_\varphi + a(t,f,g)\right),
\end{split}
\end{equation}
where $a : [1,+\infty) \times C^\infty_{\comp}(\M) \times C^\infty_{\comp}(\M)$ is smooth in all variables, linear in the second and third variables, and
\[
a(t,f,g) \sim \sum_{j \geq 1} t^{-j} C_j(f,g),
\]
where $C_j : C^\infty_{\comp}(\M) \times C^\infty_{\comp}(\M) \to \C$ is a continuous bilinear form. The asymptotic expansion means that, for every $N\in\Z_{\geq0}$ and every compact subset $K\subset\M$, there exist $C:=C(N,K)>0$ and $\ell:=\ell(N,K)\in\Z_{\geq0}$ such that for all $f,g \in C^\ell_{\comp}(\M)$ with support in $K$, for all $t \geq 1$:
\[
\left|a(t,f,g)- \sum_{j =1}^{N} t^{-j} C_j(f,g)\right| \leq C t^{-(N+1)}\|f\|_{C^{\ell}(\M)}\|g\|_{C^{\ell}(\M)}.
\]

\end{theorem}

The constant $\kappa$ is explicit and given by the determinant of a covariance matrix, see \eqref{equation:hessian-phase}. The integers $\ell > 0$ could be computed explicitly as well from the stationary phase lemma. Finally, it can be shown that the bilinear forms $C_j$ are non-zero for $j \geq 1$; see \cite[Lemma 4.16]{Cekic-Lefeuvre-Munoz-26} where a similar result is established. We believe that it should be possible to extend Theorem \ref{theorem:mixing-intro} to $\Gamma\backslash G$ for Borel Anosov subgroups, as $M_\scrC = M$ is compact in this case. In addition, the Diophantine condition seems technical and could likely be removed, although we do not know how to do it; see \S\ref{sssection:counting} for further discussion.

\subsection{Zeta functions and Poincaré series} \label{ssection:consequences}

\subsubsection{Meromorphic extension of dynamical zeta function}

The \emph{dynamical zeta function} associated with the Anosov subgroup $\Gamma$ is defined by
\begin{equation}
\label{equation:zeta-intro}
\zeta_\Gamma(\mathbf{s}) := \prod_{[\gamma] \in [\Gamma]^\sharp} (1-e^{-\mathbf{s}(\lambda_\scrC(\gamma))})^{-1}, \qquad \Re(\mathbf{s})\in\mathbf{C}^{(d_s),+},
\end{equation}
where $[\Gamma]^\sharp$ denotes the set of primitive conjugacy classes in $\Gamma$. This complex function is the analogue of the dynamical Ruelle zeta function for Anosov or Axiom A flows (rank $1$ case). It was established in \cite{Sambarino-24} that $\zeta_\Gamma$ is holomorphic in $\{\mathbf{s} \in (\fa_\scrC^*)_{\C} ~:~ \Re(\mathbf{s}) \in \mathbf{C}^{(d_s),+}\}$ and has a singularity of order $1$ on $\mathbf{C}^{(d_s)}$ (see the discussion after Theorem \ref{theorem:leading} for a definition of an order one singularity).

In this article, we prove the following result:

\begin{theorem}
\label{theorem:intro-zeta-anosov}
For any $\scrC$-Anosov subgroup,  the zeta function $\zeta_\Gamma$ defined in \eqref{equation:zeta-intro} admits a meromorphic extension to $(\fa_\scrC^*)_{\C}$.
\end{theorem}

Partial results (for projective Anosov representations) were established in \cite{Pollicott-Sharp-24, Delarue-Monclair-Sanders-24}. It should be possible to deduce from Theorem \ref{theorem:intro-zeta-anosov} counting results for elements in $\Gamma$; see \S\ref{sssection:counting} for further discussion.

\subsubsection{Poincaré series} We introduce the following Poincaré series of the Anosov subgroup:
\begin{equation}
\label{equation:poincare}
\eta_\Gamma(\mathbf{s}) := \sum_{\gamma \in \Gamma} e^{-\mathbf{s}(\mu_\scrC(\gamma))}, \qquad \Re(\mathbf{s})\in\mathbf{C}^{(d_s),+}.
\end{equation}
 This is the higher-rank analogue of the series $\sum_\gamma e^{-sd_{\HH^n}(o,\gamma o)}$ in rank $1$, where $o \in \HH^n$ is an arbitrary origin. Meromorphic continuation of this series in rank $1$ goes back to Huber \cite{Huber-56,Huber-59}, who obtained it for compact hyperbolic surfaces. It was subsequently extended to all geometrically finite hyperbolic manifolds by Guillarmou--Mazzeo \cite{Guillarmou-Mazzeo-12}. The multivariate higher-rank version of the Poincar\'e series \eqref{equation:poincare} appears ubiquitously, in particular, in the context of higher-rank Patterson--Sullivan theory, and dates back at least to the definition of Manhattan curves by Burger \cite{Burger-93}. For Anosov representations, the complex function $\eta_\Gamma$ is holomorphic in $\{\mathbf{s} \in (\fa_\scrC^*)_{\C} ~:~ \Re(\mathbf{s}) \in \mathbf{C}^{(d_s),+}\}$, and, by \cite[Corollary 5.7.2]{Sambarino-24}, it diverges at every point of $\mathbf{C}^{(d_s)}$. It thus identifies this hypersurface as the boundary of its domain of absolute convergence. See also the recent book \cite{DKLZ26} which considers Poincaré series for Borel Anosov representations with values in distributions on the complete flag variety, and computes its renormalized value at the boundary of its domain of absolute convergence.
 
 In contrast to the rank-one setting, we are not aware of any previous result establishing a meromorphic extension of the multivariate Poincar\'e series in higher rank. In the present article, we prove the following:
\begin{theorem}
\label{theorem:intro-poincare}
For any torsion-free $\scrC$-Anosov subgroup, the Poincaré series $\eta_\Gamma$ defined in \eqref{equation:poincare} admits a meromorphic extension to $(\fa_\scrC^*)_{\C}$.
\end{theorem}

For certain classes of representations, it should be possible to compute explicitly the value at $0$ of the Poincaré series, that is, $\eta_\Gamma(0)$.
See \S\ref{sssection:poincare-zero} for further discussion.

\subsection{Perspectives} We detail a non-exhaustive list of questions arising from our work.

\subsubsection{Quantum-classical correspondence} 

\label{sssection:correspondence}

On the unit tangent bundle $SN$ of a convex cocompact (resp.\ closed) hyperbolic manifold
$N := \Gamma\backslash\HH^n$ (rank one), it is known that resonant states of the geodesic flow 
associated with resonances in a ``first band'' can be integrated over the spheres of the
fibration $SN \to N$ to produce resonant states (resp.\ eigenfunctions) for the Laplacian on $N$. 
See \cite{Dyatlov-Faure-Guillarmou-15} for closed hyperbolic manifolds (or 
\cite{Guillarmou-Hilgert-Weich-21} for compact locally symmetric spaces of rank one), and
\cite{Guillarmou-Hilgert-Weich-18} for compact and convex cocompact hyperbolic surfaces.

In higher rank, a quantum--classical correspondence was established in 
\cite{Hilgert-Weich-Wolf-23} when $\Gamma < G$ is a uniform lattice; the corresponding quantum
operators form a family of commuting invariant differential operators on the locally symmetric space of
higher rank. We expect that there should also exist a higher-rank analogue of the first band 
when $\Gamma < G$ is an Anosov subgroup, which would allow one to construct resonant states 
for the commuting family of invariant differential operators on the corresponding (noncompact) quotient of the symmetric space. Note that in \cite{Delarue-Guillarmou-Monclair-2025} a quantum--classical correspondence for special classes of Anosov representations in $\mathrm{SO}(2,2)$ is established. Although this is still dynamically a rank-one situation (which corresponds to a spacelike geodesic flow on $\mathrm{AdS}^3$), it indicates that the methods of quantum--classical correspondence are compatible with Anosov subgroups.

\subsubsection{Value at $0$ of the Poincaré series. Order of vanishing at $0$ of the Ruelle zeta function} \label{sssection:poincare-zero} The value at $0$ of the meromorphic extension of the Poincaré series \eqref{equation:poincare} should have a topological interpretation in terms of $\Gamma$ and/or $G$. For a representation $\rho:\Gamma\to G$, write $\eta_\rho:=\eta_{\rho(\Gamma)}$. More precisely, it is reasonable to expect that $\rho \mapsto \eta_\rho(0)$ is locally constant on the subset of Anosov representations inside the character variety. We believe that such a statement is approachable for Hitchin representations close to the Fuchsian locus. The expected result is $\eta_\rho(0)=1-1/\chi(\Sigma)$, where $\chi(\Sigma)$ is the Euler characteristic of the surface (of genus $\geq 2$). This should be thought of as a generalization of similar results obtained on negatively curved surfaces (rank-one case); see, for instance, \cite{Dang-Riviere-24}. One hope is to be able to detect the connected component of the space of Anosov representations using $\eta_\rho(0)$ as a topological invariant. Similarly, the order of vanishing of the Ruelle zeta function at zero could be a topological invariant; see \cite{Dyatlov-Zworski-17, Dang-Guillarmou-Riviere-Shen-20, Cekic-Delarue-Dyatlov-Paternain-22} among other references.

\subsubsection{Essential spectral gap} \label{sssection:counting} Assume in this subsection that $\Gamma$ is torsion-free and Zariski dense. Given $\varphi\in\bigl(\scrL\bigr)^{*,\circ}\subset\fa_\scrC^*$, one may look at
\[
N_\varphi(T) := \sharp \{[\gamma] \in [\Gamma]^\sharp ~:~ \varphi(\lambda_\scrC(\gamma)) \leq T\},
\]
where $[\Gamma]^\sharp$ denotes the nontrivial primitive conjugacy classes of $\Gamma$. It was shown in \mbox{\cite{Sambarino-14,Potrie-Sambarino-17}} that this number satisfies
\begin{equation}
\label{equation:asymptotic-nt}
N_\varphi(T) \sim_{T \to +\infty} \mathrm{Li}(e^{\delta(\varphi)T}),
\end{equation}
where $\mathrm{Li}(x) := \int_2^x \dd t/\log(t)$. This result also follows from the meromorphic extension of the zeta function (Theorem \mbox{\ref{theorem:intro-zeta-anosov}}), using a standard Tauberian type argument called the Wiener-Ikehara theorem (see \mbox{\cite{Parry-Pollicott-90}} for instance). It was further established in \mbox{\cite[Theorem E]{Delarue-Monclair-Sanders-24}} (projective Anosov case), \mbox{\cite[Theorem 5.5]{Pollicott-Sharp-24}} ($(1,2)$-Anosov representations of surface subgroups into $\mathrm{SL}_n(\R)$ which are $(1,1,2)$ hyperconvex) and \mbox{\cite[Theorem 1.9]{Chow-Sarkar-24}} (general case) that \mbox{\eqref{equation:asymptotic-nt}} holds with a power saving remainder term.

We believe that this indicates an \emph{essential spectral gap} for the affine subspace $\varphi + i\fa^*_\scrC$, namely that, outside of any small neighborhood of $\varphi$, the remainder of the spectrum $\sigma_{\mathrm{RS}}$ lies at a positive, uniform distance $d(\varphi) > 0$ away from $\varphi + i\fa^*_\scrC$. This essential spectral gap would also likely imply that the asymptotic expansion \eqref{equation:stationary-phase-intro} of the correlation function in inverse powers of time (Theorem \ref{theorem:mixing-intro}) holds independently of the Diophantine condition, which we believe to be true. It might even be plausible that an essential spectral gap holds uniformly for the whole of $\mathbf{C}^{(d_s)} + i\fa^*_\scrC$, that is, $\inf_{\varphi \in \mathbf{C}^{(d_s)}} d(\varphi) > 0$.

More generally, one may consider $\varphi_1, \ldots, \varphi_{d} \in \fa_\scrC^*$, $d$ linear forms in the interior of the dual of the Benoist cone ($\varphi_i > 0$ on $\scrL$). A natural quantity is to compute the asymptotic as $T \to +\infty$ of
\[
N_{\varphi_1,\ldots,\varphi_{d}}(T) := \sharp \{[\gamma] \in [\Gamma]^\sharp ~:~ \varphi_j(\lambda_\scrC(\gamma))\leq T\quad\text{for every }1\leq j\leq d\}.
\]
Likewise, this should behave as $N(T)\sim_{T \to +\infty} c \cdot e^{h T}/T^\alpha$, for some constants $c, h, \alpha > 0$, and one may also ask whether a power-saving error term can be obtained.

\subsubsection{Trace formula} For analytic (or Gevrey) Anosov flows, Jézéquel \mbox{\cite{Jezequel-21}} established a trace formula relating the resonance spectrum (Pollicott--Ruelle resonances) of the flow generator to periodic orbits. Namely, in the sense of distributions, he proved that
\begin{equation}
\label{equation:trace-formula}
\sum_{\lambda \in \sigma_{\mathrm{RS}}} m(\lambda)e^{\lambda t} = \sum_{\gamma \in \mathbf P} \dfrac{\ell_\gamma^\sharp \delta_{\ell_\gamma}(t)}{|\det(1-P_\gamma)|},
\end{equation}
where $\mathbf P$ denotes the set of periodic orbits, $m(\lambda)$ is the algebraic multiplicity of the resonance, $\ell_\gamma^\sharp$ the primitive length, $P_\gamma$ the Poincaré map along $\gamma$, and $\delta_{\ell_\gamma}(t)$ the Dirac mass at $\ell_\gamma$. Observe that the sum on the left-hand side is, indeed, over a discrete multiset because the resonance spectrum is discrete in rank one.

The $A_\scrC$-action on $\M$ associated with an Anosov subgroup is a Lie group action on a locally homogeneous space; it is thus analytic. It is therefore natural to ask whether a higher-rank analogue of the trace formula \eqref{equation:trace-formula} could hold. In such a setting, the formula should be replaced by a distribution of $\mathbf{t} \in \fa_\scrC$, and the right-hand side should be a weighted sum of Dirac masses $\delta_{\lambda_\scrC(\gamma)}(\mathbf{t})$ located at the Jordan projections.

However, the spectral side (left-hand side) is less clear. It would likely involve an integral of the form
\[
\int_{\sigma_{\mathrm{RS}}} e^{\varphi(\mathbf{t})} \, \dd\mu(\varphi),
\]
for a certain natural measure $\mu$ on $\sigma_{\mathrm{RS}}$, but we have not yet been able to identify this measure.

\subsubsection{Other perspectives} We now list several related questions:
\begin{itemize}

\item Relatively Anosov representations are a significant relaxation of the usual Anosov subgroup assumption (Definition \mbox{\ref{definition:theta-anosov}}) which can be thought of as a higher rank version of geometrically finite groups. In \S\ref{section:domain}, we construct a domain of discontinuity for such subgroups. However, it remains unclear whether the spectral approach developed in the present paper can be extended to this setting. To some extent, this situation is analogous to defining Pollicott--Ruelle resonances on negatively curved manifolds with cusps, as carried out in \cite{Bonthonneau-Weich-22}.

\item For Hitchin representations into $\mathrm{SL}_3(\R)$, it is known that they can be interpreted as the holonomy representations of convex real projective structures (see \cite{Choi-Goldman-97}). This interpretation provides additional geometric structure, which could allow one to reinterpret the resonance spectrum we define in terms of objects intrinsic to the surface itself.

\item We will see in Proposition \ref{proposition:almost-diophantine} that almost every Hitchin representation $\rho : \pi_1(\Sigma) \to \mathrm{SL}_n(\R)$ is Diophantine, provided $2g \geq n^3$, where $g \geq 2$ is the genus of the closed surface $\Sigma$. This is likely to be true without any restriction on the pair $(g,n)$. More generally, when $\Gamma$ is word-hyperbolic and $G$ is a noncompact semisimple Lie group, \emph{most representations} $\rho : \Gamma \to G$ should be Diophantine. This is left for future investigation.

\end{itemize}

\subsection{Comparison with earlier work} 

\subsubsection{Free Abelian cocycles} The study of ergodic properties of free Abelian cocycles over hyperbolic flows, including questions concerning multidimensional periodic data (called the \emph{Lyapunov spectrum} below), goes back to the pioneering work of Lalley \cite{Lalley-87,Lalley-89a,Lalley-89b}; see also the work of Babillot--Ledrappier \cite{Babillot-Ledrappier-98}. This topic is intimately connected to the counting of periodic orbits for hyperbolic flows under homological constraints; see \cite{Katsuda-Sunada-86, Phillips-Sarnak-87, Pollicott-91, Sharp-92, Sharp-93, Anantharaman-00, Pollicott-Sharp-01} among other references. However, all these articles make use of symbolic codings (Markov partitions) and exploit the fact that the flow can be coded by a suspension flow over a subshift of finite type. Our approach to free Abelian cocycles is completely different and builds on modern tools from microlocal analysis; see \S\ref{sssection:spectrum-biblio}. In particular, this allows us to recover known results and to establish new results that are out of reach of classical tools, such as the meromorphic extension of Ruelle zeta functions.

\subsubsection{(Quasi-)invariant measures}
Associating natural (quasi-)invariant measures with discrete subgroups $\Gamma$ is a classical topic dating back to the work of Patterson \cite{Patterson-76} and Sullivan \cite{Sullivan-79}. In their seminal works, both authors not only constructed these quasi-invariant \emph{Patterson--Sullivan measures} supported on the limit set but also provided a spectral interpretation in terms of boundary values of Laplace eigenfunctions; see \cite[Theorem 3.7]{Patterson-76} and \cite[Proposition 28]{Sullivan-79}. Sullivan also showed, for geometrically finite groups, how a product formula relates the Patterson--Sullivan measures on the boundary to the measure of maximal entropy for the geodesic flow. The latter is called the \emph{Bowen--Margulis--Sullivan (BMS) measure}; see \cite[\S4]{Sullivan-79} or \cite{Margulis-69,Bowen-74}. Another related class of measures on the unit tangent bundle consists of the \emph{Burger--Roblin measures}. These are quasi-invariant under the geodesic flow, invariant under the horocyclic flow, and were introduced by Burger \cite{Burger-90} and later generalized to the $\mathrm{CAT}(-1)$ setting by Roblin \cite{Roblin-03}.

For general Zariski-dense discrete subgroups in higher rank, Albuquerque \cite{Albuquerque-99} and Quint \cite{Quint-2002} constructed Patterson--Sullivan measures on the visual and Furstenberg boundaries of the symmetric space, respectively (see also related work of Link \cite{Link-04,Link-06}). Specializing to the case of Anosov subgroups, Lee--Oh \cite{Lee-Oh-20} and Sambarino \cite{Sambarino-24} showed that the space of Patterson--Sullivan measures is uniquely parametrized by the critical hypersurface. Moreover, Sambarino \cite{Sambarino-15,Sambarino-24} used the Patterson--Sullivan measures to construct the family of $A_\scrC$-invariant measures $\mu_\varphi$ on $\scrJ$ (which is the higher-rank version of the BMS measure) and showed how these measures relate to the measures of maximal entropy for the refraction flows. Lee--Oh \cite{Lee-Oh-20} proved ergodicity and classification results for horospherical invariant measures, while Edwards--Lee--Oh \cite{ELO23} introduced higher-rank Burger--Roblin measures and studied their role in local mixing and matrix-coefficient asymptotics. Generalizations beyond the Borel Anosov case appear in more recent work by Kim \cite{Kim-25} and Choi--Kim \cite{Choi-Kim-26}. In the present paper, all these measures appear naturally as resonant states (or pairings thereof) on different types of tensor bundles and are parametrized by the leading resonant hypersurface. Although we do not provide any new or more general constructions of invariant measures, our approach provides a spectral-theoretic perspective on these measures that was previously unavailable in higher rank.

\subsubsection{Dynamics on the trapped set} The $A_\scrC$-invariant BMS measures $\mu_\varphi$ (see \S\ref{sssection:correlation}) on $\scrJ$ allow one to study the statistical properties of the $A_\scrC$-action (ergodicity, mixing, etc.). When $k>0$, however, these BMS measures have infinite mass, so mixing must be understood locally; Oh--Pan introduced the term \emph{local mixing} for the resulting diffusive decay of correlations \cite{Oh-Pan-19}.

Local mixing (or \emph{decay of correlations}) on the 
trapped set with respect to the BMS measures was initially considered
by Thirion \cite{Thirion-09} for ping-pong groups in $\mathrm{SL}(n,\R)$ in a particular direction (the maximal growth direction), by Sambarino \cite{Sambarino-15} for Zariski-dense hyperconvex representations of surface groups, by Chow--Sarkar \cite
{Chow-Sarkar-23} for Borel Anosov groups, and by Sambarino \cite{Sambarino-24} for general Zariski-dense Anosov representations. We also refer to Oh's ICM plenary talk \cite[\S\S 7 and 8]{Oh-ICM-26} for a recent survey on these local mixing results. However, so far, none of these results is
quantitative and only provides the first term in the
expansion of the correlation function \eqref
{equation:stationary-phase-intro}.

Motivated by the results of the
present article, the second author, together with Cekić
and Muñoz-Thon, proved a related result on a full asymptotic
expansion of the correlation function on Abelian covers
of isometric extensions of Anosov flows
\cite{Cekic-Lefeuvre-Munoz-26}. We are also aware of an independent work in preparation by Kim-Oh-Pan \cite{KOP26} on establishing a full asymptotic expansion of the correlation function for cusped positive representations without assuming a Diophantine condition.

\subsubsection{Domain of discontinuity. Spectral theory of smooth Axiom A flows for Anosov representations}
Our work builds on essential ideas from \cite{Delarue-Monclair-Sanders-24,Delarue-Monclair-Sanders-25} concerning the construction of domains of discontinuity for Anosov subgroups. A key insight of these works is that these domains give rise to smooth hyperbolic dynamics, which can be viewed either as hyperbolic higher-rank actions—a viewpoint elaborated in the present paper—or as families of Axiom A flows, which is the perspective adopted in \cite{Delarue-Monclair-Sanders-24,Delarue-Monclair-Sanders-25}.
Most parts of Theorem~\ref{thm:DMS-main} are established in \cite{Delarue-Monclair-Sanders-25}.
However, since the primary focus of \cite{Delarue-Monclair-Sanders-24,Delarue-Monclair-Sanders-25} is on Axiom A flows rather than on higher-rank actions (although the latter are mentioned), we could not base our microlocal work entirely on their results.
In particular, Assumption $\hyperlink{AA2}{\rm(A2)}$ for the Axiom A action of type (k,1) is not addressed there (note, however, that a very similar assumption appears in \cite[Section 5]{Delarue-Guillarmou-Monclair-2025} in a related special case of spacelike geodesic flows on AdS$^3$).

We decided to give a self-contained proof of Theorem~\ref{thm:DMS-main} in \S\ref{section:domain} for the following reasons: First, while many of the results in \S\ref{section:domain} were already shown in \cite{Delarue-Monclair-Sanders-25}, our techniques and language differ, so this provides an alternative approach to their important theorems.
While \cite{Delarue-Monclair-Sanders-25} largely avoids Lie group decompositions that involve choices, such as Iwasawa or horospherical decompositions, our approach instead fully exploits those.
Second, the (sometimes quite advanced) structure theory of reductive groups introduced in \S\ref{section:structure} is required anyway for the meromorphic continuation of the Poincar\'e series in \S\ref{section:poincare}.
Furthermore, we also need to refer to the internal mechanisms of the proof of the domain of discontinuity for the analysis of non-transversal points of intersection in the proof of the meromorphic continuation of the Poincar\'e series.
While it should, in principle, be possible to rely instead on the corresponding arguments in \cite{Delarue-Monclair-Sanders-25}, doing so would make the exposition hard to follow due to the differences in language and techniques.

Finally, our approach naturally shows that most properties of domains of discontinuity continue to hold under weaker assumptions on $\Gamma$, such as $\Gamma$ being a divergence group or a transverse subgroup of a regular subgroup (we refer to  \S\ref{section:domain} for details) and affirmatively answers a question raised in \cite[\S 1.4]{Delarue-Monclair-Sanders-25}.
This result could prove particularly valuable for extending the spectral theory beyond Anosov subgroups, for instance to the class of relatively Anosov subgroups, for which we verify that all assumptions hold except the compactness of the trapped set.

Apart from merely constructing the domains of discontinuity and the hyperbolic dynamical systems, \cite{Delarue-Monclair-Sanders-24,Delarue-Monclair-Sanders-25} also derive important consequences for the spectral theory, dynamical properties of the flow, and related zeta functions. In \cite{Delarue-Monclair-Sanders-24}, in the setting of projective Anosov representations, dynamical resonances for the associated Axiom A flow are defined. Combined with results of Stoyanov, this provides the first proof of exponential mixing for Sambarino's refraction flow in this setting. In \cite[Theorem E]{Delarue-Monclair-Sanders-25}, a meromorphic extension of a family of zeta functions for general Anosov subgroups is shown; in \cite{Delarue-Guillarmou-Monclair-2025}, Delarue, Guillarmou, and Monclair prove a meromorphic continuation of Poincar\'e series associated with the spacelike geodesic flow on quotients of $\mathrm{AdS}^3$.
The main difference is that their approach reduces the higher-rank dynamics to a rank-one dynamical system (a flow), whereas our aim in the present paper is to study the full higher-rank action and to associate to it joint resonances, multivariate meromorphic zeta functions, and Poincaré series.

\subsubsection{Spectral theory of flows/actions using microlocal analysis} \label{sssection:spectrum-biblio} The theory developed in \S\ref{section:results-cocycles}--\S\ref{section:decay} builds on earlier work \cite{Faure-Roy-Sjostrand-08,Faure-Sjostrand-11,Faure-Tsuji-13,Dyatlov-Zworski-16,Dyatlov-Guillarmou-16} in rank~1, which establishes the meromorphic extension of the resolvent on anisotropic spaces for hyperbolic dynamical systems using microlocal analysis. See also \cite{Liverani-04,Gouezel-Liverani-06,Baladi-Tsuji-07} for a similar approach based on anisotropic Banach spaces. In turn, these techniques can be exploited to establish the meromorphic extension of the Ruelle zeta function; see \cite{Giulietti-Liverani-Pollicott-13,Dyatlov-Guillarmou-16,Dyatlov-Zworski-17,Dyatlov-Guillarmou-18}.

For Anosov actions on compact manifolds ---including Weyl chamber actions associated with uniform lattices in semisimple Lie groups---, a microlocal theory was recently developed in \cite{Guedes-Bonthonneau-Guillarmou-Hilgert-Weich-20,Guedes-Bonthonneau-Guillarmou-Weich-24} to study SRB measures and measures of maximal entropy \cite{Humbert-25}. The approach developed in \S\ref{section:results-cocycles}--\S\ref{section:decay} of the present paper is, to a certain extent, reminiscent of these works. However, in the setting of uniform lattices, the spectrum---also defined in $\C^{\mathrm{rank}(G)}$---consists of isolated points, whereas for Anosov representations one obtains complex codimension-one subvarieties. This is a notable feature of Anosov-subgroup quotients as opposed to lattices.

\subsection{Organization of the paper}

The paper is organized as follows:
\begin{itemize}
\item In \S\ref{sec:introduction-geometric-setting}, we introduce the notion of Abelian cocycles over hyperbolic flows and review some of their basic properties;
\item In \S\ref{section:results-cocycles}, we construct the resonance spectrum and prove the meromorphic extension of the product resolvent and dynamical zeta functions;
	\item In \S\ref{section:leading-resonance}, we study the leading resonant hypersurface and the associated invariant measures (of maximal entropy);
	\item In \S\ref{section:decay}, we discuss local mixing for the diagonal flows $(e^{tu_\varphi})_{t\in\R}$ with respect to the measures of maximal entropy;
\item In \S\ref{section:structure}, we introduce the necessary Lie-theoretic background. In particular, we discuss the properties of parabolic subgroups of $G$. We review some geometric aspects of the space $G/M_\scrC$, whose homogeneous dynamics we study. We also introduce the notion of divergent subgroups of $G$ and define Anosov representations and all related notions;
\item In \S\ref{section:domain}, we establish the existence of a domain of discontinuity and show that Anosov representations fit into the theory developed in \S\ref{section:results-cocycles}--\S\ref{section:decay}; in particular, we prove Theorems \ref{theorem:intro-rt}, \ref{theorem:intro-leading}, \ref{theorem:intro-zeta-anosov}, and \ref{theorem:mixing-intro};
\item In \S\ref{section:poincare}, we prove the meromorphic extension of the Poincaré series of Anosov representations (Theorem \ref{theorem:intro-poincare}).

\end{itemize}
\subsection{Notation and conventions}

\label{sssection:distributions}

If $H$ and $G$ are groups, $H<G$ means that $H$ is a subgroup of $G$, not necessarily a proper subgroup.

Let $\M$ be a smooth manifold without boundary. The space $\mc{D}'(\M)$ denotes the space of distributions on $\M$, that is, the dual of smooth compactly supported densities $C^\infty_{\comp}(\M,\Omega^1 \M)$, where $\Omega^1 \M \to \M$ denotes the density bundle. (In other words, $\mc{D}'(\M)$ is a space of ``generalized functions''.) The space $\mc{D}'_{\comp}(\M)$ denotes the space of compactly supported distributions on $\M$.

More generally, if $\mc{E} \to \M$ is a complex vector bundle, then $\mc{D}'(\M,\mc{E})$ is the dual of $C^\infty_{\comp}(\M, \mc{E}^* \otimes \Omega^1 \M)$. Given $u \in \mc{D}'(\M,\mc{E})$ and $\varphi \in C^\infty_{\comp}(\M,\mc{E}^*\otimes \Omega^1\M)$, we write $(u,\varphi)$ for the standard $\C$-bilinear pairing, and $\langle u,\varphi\rangle := (u,\overline{\varphi})$. If 
\[
P : C^\infty(\M,\mc{E}) \to C^\infty(\M,\mc{E})
\]
is a differential operator, its adjoint operator
\[
P^* : C^\infty(\M, \mc{E}^* \otimes \Omega^1 \M) \to C^\infty(\M, \mc{E}^* \otimes \Omega^1 \M)
\]
is defined by the equality $\langle P \varphi, \psi \rangle = \langle \varphi, P^* \psi \rangle$, for all $\varphi \in C^\infty_{\comp}(\M,\mc{E}), \psi \in C^\infty_{\comp}(\M,  \mc{E}^* \otimes \Omega^1 \M)$. If $\Gamma \subset T^*\M \setminus \{0\}$, then $\mc{D}'_\Gamma(\M)$ denotes the space of distributions with wavefront set in $\Gamma$.

Finally, given $\s = (s_0,\s') \in \C \oplus \C^k$, we will use the notation $|\Re(\s)| \gg 0$ to denote an open domain $U \subset \C^{k+1}$ such that there exists a constant $C > 0$ such that
\begin{equation}
\label{equation:bigres}
U \subset \{\s \in \C^{k+1} ~:~ \Re(s_0) > C(|\Re(\s')|+1)\},
\end{equation}
where $|\bullet|$ is the standard Euclidean norm on $\R^k$. \\

\noindent \textbf{AI disclosure.} AI was exclusively used during the final stage of the preparation of this manuscript to correct grammatical errors, polish language, verify mathematical notations and consistency, and conduct additional literature searches. All AI-assisted suggestions and references were independently verified by the authors. The authors take full responsibility for the content and correctness of the article. \\

\noindent \textbf{Acknowledgments.} We thank Andrés Sambarino for many fruitful discussions that led to this work. We are also grateful to Benjamin Delarue and Daniel Monclair for numerous helpful discussions concerning their domains of discontinuity. We thank Joachim Hilgert, Nicolas Tholozan, and Gaëtan Chenevier for motivating discussions and valuable comments. We also thank Hee Oh and Pratyush Sarkar for comments on an earlier version of this article. This project has received funding from the European Research Council (ERC) under the European Union's Horizon research and innovation programme (grant agreement No. 101162990), the Deutsche Forschungsgemeinschaft (DFG) Grant No. SFB-TRR 358/1 2023--491392403 (CRC ``Integral Structures in Geometry and Representation Theory''), and the Agence Nationale de la Recherche (ANR) through the grant ADYCT (ANR-20-CE40-0017).

\section{Axiom A actions of type $(k,1)$}
\label{sec:introduction-geometric-setting}
 
We introduce the notion of \emph{free Abelian cocycles over flows} and explain their relationship with locally free Abelian actions on $\R^k$-principal bundles; see \S\ref{ssection:definition-free-abelian-cocycles}. A key feature of free Abelian cocycles is that they may be interpreted as reparametrizations of a single flow, a point of view developed in \S\ref{sec:reparameterization}. We then restrict our attention to the case of interest in which the base flow is hyperbolic. In this case, we call these objects \emph{Axiom A actions of type $(k,1)$}; this is the subject of \S\ref{ssection:hyperbolic-base-dynamics}. The notion of the \emph{limit cone}, together with related concepts, is discussed in \S\ref{ssection:lyapunov-spectrum}. Finally, we conclude this section with a discussion of admissible bundles and associated distributions in \S\ref{ssection:bundles-distributions}.

\subsection{Free Abelian cocycles over flows} \label{ssection:definition-free-abelian-cocycles}

Fix a nonnegative integer $k$. Let $\M$ be a smooth connected $d_{\M}$-dimensional manifold without boundary. Let $A = \R^{k+1}$, and let
\[
\tau : A \times \M \to \M, 
\]
be a locally free Abelian action, mapping elements of $A$ to diffeomorphisms of $\M$. We denote by $\fa$ the Lie algebra of $A$, and identify it with an algebra of vector fields on $\M$. This infinitesimal representation is denoted by
\[
\fa \to C^\infty(\M,T\M), \qquad a \mapsto \X(a) := \partial_t \tau_{e^{ta}}|_{t=0}.
\]
We let $E_{0,\M} := \fa \subset T\M$. We will make the following assumption:
\begin{itemize}
\item[\textbf{\hypertarget{AA1}{(A1)}}] There exists a hyperplane $\fh \subset \fa$ such that the action of $H := \exp(\fh)$ on $\M$ is \emph{free and proper}.
\end{itemize}

This implies that $\N := \M/H$ is a smooth connected manifold without boundary (not necessarily compact) and $\pi : \M \to \N$ is an $H$-principal bundle, where $H\simeq\R^k$. Let
\[
d_{\mc{N}} := d_{\M}-k
\]
be the dimension of $\mc{N}$. Since $\R^k$ is contractible, the principal bundle $\M\to\mc{N}$ is trivial; that is, $\M \simeq \mc{N} \times H$ (see Lemma \ref{lemma:trivial}). After choosing an arbitrary smooth section $\sigma : \mc{N} \to \M$, a trivializing diffeomorphism is given by
\begin{equation}
\label{equation:trivialization}
\mc{N} \times H \ni (x, h) \mapsto \tau_h\left(\sigma(x)\right) \in \M.
\end{equation}

Let $A' = A/H \simeq \R$ be the quotient group; it acts on $\N$ and is generated by some element $X_\N \in C^\infty(\N,T\N)$ identified with a vector field on $\N$. This corresponds to choosing some $X_\M\in \fa \setminus \fh$ such that $\pi_\ast(X_\M)=X_\N$, with $X_\M$ and $X_\N$ viewed as vector fields on $\M$ and $\N$, respectively. Then in the trivialization $\M \simeq \N \times H$ given by \eqref{equation:trivialization}, we can write
\[
X_\M(x,h) \simeq_{\sigma} X_\N(x) + w(x)
\]
for some $w \in C^\infty(\N, \fh)$. Changing the trivialization via a global section $\sigma'$ we can always write $\sigma'(x):=\tau_{u(x)}(\sigma(x))$ for some smooth function $u\in C^\infty(\N,\fh)$. We then find
\begin{equation}
\label{equation:lala-cob}
X_\M(x,h) \simeq_{\sigma'} X_\N + w - X_\N u.
\end{equation}
Hence, once $\fh$ and $X_\M$ are fixed, the object that is well-defined independently of the choice of trivialization $\sigma$ is the class of $w$ modulo coboundaries $\{ X_{\N}u ~:~ u \in C^\infty(\N, \fh)\}$. We emphasize that we will mostly use an additive notation for cocycles and write $\fh$ instead of $H$.

Let $(\phi_t)_{t \in \R}$ and $(\Phi_t)_{t \in \R}$ be the flows generated respectively by $X_\N$ and $X_\M$. Note that
\begin{equation}\label{eq:abelian-extension-dynamics}
\Phi_t(x,h) = \left(\phi_t(x), h + \int_0^t w\circ\phi_s(x) \dd s\right).
\end{equation}
The choice of $X_\M$ also provides a decomposition of $\fa$ into $\fa = \R X_\M \oplus \fh$. The flow $(\Phi_t)_{t \in \R}$ is an \emph{extension} of $(\phi_t)_{t \in \R}$ in the sense that
\begin{equation}
\label{equation:extension}
\pi \circ \Phi_t = \phi_t \circ \pi, \qquad \forall t \in \R.
\end{equation}
In addition, it is \emph{equivariant} with respect to the principal bundle structure, namely,
\begin{equation}
\label{equation:equivariant}
\tau_h \circ \Phi_t = \Phi_t \circ \tau_h, \qquad \forall  t \in \R,\ h \in H.
\end{equation}
Note that, equivalently, one may start with a triple $(\mc{N},(\phi_t)_{t \in \R},w)$ and construct a flow $(\Phi_t)_{t \in \R}$ on the trivial principal $\R^k$-bundle $\N \times \R^k$ by \eqref{eq:abelian-extension-dynamics}. This leads to the following definition:

\begin{definition}[Free Abelian cocycles over flows]
\label{definition:free-abelian-cocycles}
A \emph{free Abelian cocycle over a flow} consists of a triple $(\mc{M},A,\tau)$ satisfying $\hyperlink{AA1}{\rm(A1)}$ or, equivalently, of a triple $(\mc{N}, (\phi_t)_{t \in \R}, w)$ as above.
\end{definition}

Later, we will focus on the case where $(\phi_t)_{t \in \R}$ is a hyperbolic flow. See \S\ref{ssection:hyperbolic-base-dynamics}.

\subsection{Properness of $\R^k$-actions}
\label{sec:reparameterization}

We now investigate the change of $\fh$ or $X_\M$ in the decomposition $\fa = \fh \oplus \R X_\M$. With the notation above, for $\alpha\in\fh^\ast$, we let
\[
\fh' := \{ \alpha(h)X_\M + h  ~:~ h\in\fh\}\subset \fa, \qquad H' := \exp(\fh').
\]
The following holds:

\begin{lemma}\label{lemma:properness-action-abstract}

\begin{enumerate}[label=\emph{(\roman*)}]
	\item The action of $H'$ on $\M$ is free if and only if for every periodic point $x\in\N$ of $(\phi_t)_{t \in \R}$ with period $T$, 
	\[
	\frac{1}{T}\int_0^T \alpha(w)\circ\phi_s(x) \dd s \neq - 1.
	\]
	
	\item The action of $H'$ on $\M$ is proper if and only if the map $F : \N \times \R \to \N \times \N \times \R$ defined by
\[
F(x,t) := \Big(x,\phi_t(x), \int_0^t \left(\alpha(w)\circ\phi_s(x) +1\right) \dd s \Big)
\]
is proper.
	\item If $\N$ is compact, $F$ is proper if and only if there exist $\eps,T>0$ so that
\begin{equation}
\label{equation:yannickerie}
\left|1 + \frac{1}{T}\int_0^T \alpha(w)\circ\phi_s(x) \dd s \right| > \eps, \qquad \forall x \in \N
\end{equation}
\end{enumerate}
\end{lemma}

In particular, if there exists $\varepsilon>0$ such that for all $x\in\N$, $\alpha(w)(x) > -1+\varepsilon$, then $H'$ acts freely and properly on $\M$.

\begin{proof}
If $\alpha=0$, then $H'=H$ and all three assertions are immediate; hence assume $\alpha\ne0$. 

(i) Consider a fixed point of the action of $H'$:
\[
\tau_{h + \alpha(h)X_\M}(x,v)= (x,v) \in \M.
\]
This implies that $\phi_{\alpha(h)}(x)=x$, so $x$ is a periodic point of $(\phi_t)_{t \in \R}$ with period $T=\alpha(h)$ and 
\[
h + \int_0^{\alpha(h)} w\circ\phi_s(x) \dd s = 0. 
\]
Applying $\alpha$ and dividing by $T$ yields
\[
\dfrac{1}{T}\int_0^T \alpha(w)\circ\phi_s(x) \dd s = -1. 
\]
Conversely, assume that, for some $x\in\N$ and some nonzero $T\in\R$,
\[
\phi_T(x)=x,\qquad \int_0^T \alpha(w)\circ\phi_s(x) \dd s = - T.
\]
We set
\[
h = - \int_0^T w\circ\phi_s(x) \dd s.
\]
Then $\alpha(h) = T$, and $(x,0)$ is a fixed point of $\tau_{h + \alpha(h)X_\M}$. \\

(ii) We assume that the function $F$ is proper, and prove the properness of the action. For this, consider a sequence of points contained in a compact subset $z_n=(x_n,v_n) \in \N \times \fh \simeq_\sigma \M$ and a sequence $h_n\in\fh$ that leaves every compact set, such that
\[
\tau_{h_n + \alpha(h_n)X_\M}(z_n) = \left(\phi_{\alpha(h_n)}x_n, v_n + h_n + \int_0^{\alpha(h_n)} w \circ \phi_s(x_n) \dd s\right)
\]
is contained in a compact subset. In particular, $(\phi_{\alpha(h_n)}(x_n))_{n \geq 0}$ is bounded. We want to prove that 
\[
\left\|h_n + \int_0^{\alpha(h_n)}w\circ\phi_s(x_n) \dd s\right\| \to +\infty. 
\]
If $(\alpha(h_n))_{n \geq 0}$ is bounded, the statement is immediate, so we may assume that $|\alpha(h_n)|\to +\infty$. However, $F$ is proper by assumption so
\[
\left|\alpha(h_n) + \int_0^{\alpha(h_n)}\alpha(w)\circ\phi_s(x_n) \dd s\right| \to +\infty,
\]
which proves that $H'$ is proper. Conversely, let us assume the map $F$ is not proper and consider sequences $(x_n, t_n)_{n \geq 0}$ such that $(x_n)_{n \geq 0}$ and $(\phi_{t_n}(x_n))_{n \geq 0}$ are each contained in a compact subset, $(t_n)_{n \geq 0}$ leaves every compact of $\R$, and
\[
t_n + \int_0^{t_n} \alpha(w)\circ\phi_s(x_n) \dd s
\]
is bounded. Take $u$ so that $\alpha(u)=1$; set $h_n^0 = t_n u$ and
\[
h_n = h_n^0 - \int_0^{t_n} (w-u\alpha(w))\circ\phi_s(x_n) \dd s. 
\]
We find that $\alpha(h_n) = t_n$ and note that $(h_n)_{n \geq 0}$ leaves every compact set. Then:
\[
\tau_{h_n + \alpha(h_n) X_\M}(x_n, 0) = \left(\phi_{t_n}(x_n), u\left(t_n + \int_0^{t_n} \alpha(w)\circ\phi_s(x_n) \dd s\right)\right)
\]
is contained in a compact subset while $(h_n)_{n \geq 0}$ is unbounded. This completes the proof. \\ 

(iii) Since $\N$ is compact, by item (ii), $F$ is proper if and only if 
\[
\N \times \R \ni (x,t)\mapsto \int_0^t (1+\alpha(w))\circ\phi_s(x) \dd s \in \R
\]
is proper. This means that for any $C_0>0$, there exists $C_1>0$ such that
\[
|t|> C_1 \Longrightarrow \left|\int_0^t (1+\alpha(w))\circ\phi_s(x) \dd s \right| > C_0, \qquad \forall x \in \N.
\]
Taking $T=C_1+ 1$ and $\eps = C_0/(1+C_1)$, we obtain \eqref{equation:yannickerie}. On the other hand, let us assume that \eqref{equation:yannickerie} holds. Without loss of generality, we may assume that
\[
\int_0^T(1+\alpha(w))\circ\phi_s(x)\dd s > T \eps, \qquad \forall x \in \N.
\]
This implies that
\[
\int_0^{nT} (1+\alpha(w))\circ\phi_s(x)\dd s > n T \eps, \qquad \forall x \in \N, n \geq 0,
\]
Set $f:=1+\alpha(w)$ and $M:=\|f\|_\infty$. Splitting an arbitrary orbit segment into subsegments of length $T$ (and shifting the initial point when $t<0$) gives
\[
\left|\int_0^t f\circ\phi_s(x)\dd s\right|\geq
\left\lfloor\frac{|t|}{T}\right\rfloor T\eps-TM.
\]
Hence, $F$ is proper. 
\end{proof}

It is easily verified that the conditions in items (i), (ii), and (iii) only depend on the class of $w$ modulo coboundaries $\{X_{\N} u ~:~ u \in C^\infty(\N,\fh)\}$. The choice of a specific representative $w$ in the class can be described geometrically:

\begin{lemma}\label{lemma:transverse-graph-section}
The following statements hold:
\begin{enumerate}[label=\emph{(\roman*)}]
\item The graph of the section $\sigma : \N \to \M$ is transverse to the orbits of the $H'$-action if and only if for every $x\in\N$, $\alpha(w)(x) \neq -1$.

\item If there exists $\eps > 0$ such that $\alpha(w) > -1+\varepsilon$ for all $x \in \N$, then $\N\simeq \M/H'$ and $\sigma : \N \to \M$ is a section for the action of $H'$ (that is, $\N \times \fh' \to \M, (x,h') \mapsto \tau_{h'}\sigma(x)$ is a diffeomorphism).
\end{enumerate}
\end{lemma}

\begin{proof}
If $\alpha=0$, both assertions are immediate; hence assume $\alpha\ne0$.

(i) In the identification $\M\simeq \N\times \fh$ given by $\sigma$, at a point $(x,0) \in \M$, the tangent space to the $\fh'$-orbit is
\[
\{ (\alpha(h)X_\N, h + \alpha(h)w(x)) ~:~ h\in\fh\}.
\]
It is transverse to $\{(u,0) ~|~ u\in T_x\N\}$ if and only if $h + \alpha(h)w(x) \neq 0$ for all $h \neq 0$.

Now, if $h + \alpha(h)w(x) \neq 0$, then, taking $h \in \fh$ such that $\alpha(h) \neq 0$, we find that $\alpha(w)(x) \neq -1$. Conversely, if $\alpha(w)(x)\neq -1$, then either $\alpha(h)=0$, and thus $h + \alpha(h)w(x) \neq 0$ for $h \neq 0$, or $\alpha(h) \neq 0$, and thus $h + \alpha(h)w(x) \neq 0$. \\

(ii) Assume that $\alpha(w)> -1 +\varepsilon$. Then for each $x\in\N$, the map
\[
\kappa_x:h \mapsto h + \int_0^{\alpha(h)}w\circ\phi_s(x)\dd s
\]
is a diffeomorphism from $\fh$ onto itself. Indeed, to show this, pick $v\in\fh$ so that $\alpha(v)=1$ and decompose any $h\in\fh$ as $h= u + t v$, where $\alpha(u)=0$ and $t\in\R$. For a given $h' \in \fh$, solving $\kappa_x(h)=h'$ gives
\[
t(h') = t(h) + \int_0^{t(h)}\alpha(w)\circ\phi_s(x)\dd s.
\]
The condition $\alpha(w) > -1 + \eps$ implies that the right-hand side is a function of $t(h)$, with derivative larger than $\eps>0$, so it is a diffeomorphism of the line. This determines a unique solution $t(h)$ to the equation. Then
\[
u(h) = u(h') +\int_0^{t(h)} \Big(\alpha(w) v - w\Big)\circ\phi_s(x) \dd s.
\]
This proves that $\kappa_x$ is a bijection, and it is elementary to show that it is a local diffeomorphism. Likewise, by the same argument, one shows that the map
\begin{equation}
\label{equation:kappatilde}
\tilde{\kappa}_x:h \mapsto h + \int_0^{\alpha(h)}w\circ\phi_{s-\alpha(h)}(x)\dd s
\end{equation}
is a diffeomorphism from $\fh$ onto itself.

We claim that this implies that the map
\[
\N\times \fh' \owns (x,h')\mapsto \tau_{h'}\sigma(x) \in\M
\]
is a bijection. Since we already know that it is a local diffeomorphism (because of the transversality established above), this would finish the proof. First, assume that
\[
\tau_{h'}\sigma(x)=\tau_{h''}\sigma(x'). 
\]
Without loss of generality, we may assume $h''=0$, and deduce directly that $\kappa_x(h) = 0$, so that $h=0$, and $x=x'$. On the other hand, set $y=\tau_{h}\sigma(x)$ for some $h\in\fh$. Solving $\tau_{h'}\sigma(x')=y$ gives $\tau_{-h'+h}\sigma(x)=\sigma(x')$. Writing $h' = h_0 + \alpha(h_0)X_\M$ for some $h_0 \in \fh$, we find that $\phi_{\alpha(h_0)}x'=x$ and $h = h_0 + \int_0^{\alpha(h_0)} w(\phi_s x') \dd s$. That is $h = \tilde{\kappa}_x(h_0)$, where the diffeomorphism $\tilde{\kappa}_x : \fh \to \fh$ was introduced in \eqref{equation:kappatilde}. Therefore, $h_0$ is uniquely determined and so is $x'$.
\end{proof}

\subsection{Hyperbolic base dynamics} 

\label{ssection:hyperbolic-base-dynamics}

We now further investigate the properties of the free Abelian cocycles when the flow $(\phi_t)_{t \in \R}$ on the base is hyperbolic.

\subsubsection{Axiom A actions of type $(k,1)$} \label{sssecton:assumptions}

We shall now make further assumptions on the flow $(\phi_t)_{t \in \R}$ on the base. Note that similar assumptions are used in \cite[\S5]{Delarue-Guillarmou-Monclair-2025}.

\begin{enumerate}
\item[\textbf{\hypertarget{AA2}{(A2)}}] \textbf{Uniform properness.} There exist two $\phi_t$-invariant closed subsets $\mc{T}_{\pm} \subset \mc{N}$ such that the following holds: for every compact subset $C_\mp \subset \mc{N} \setminus \mc{T}_{\mp}$, the map
\[
\phi : C_\mp \times \R_\pm \to \mc{N}, \qquad (x,t) \mapsto \phi_t(x)
\]
is proper. In addition, there exists a relatively compact open subset $\scrV \subset \mc{N}$ such that for all $x \in \mc{T}_\pm$, there exists $T \geq 0$ such that for all $t \geq T$, $\phi_{\mp t} x \in \scrV$.
\end{enumerate}
We shall refer to these properties as uniform properness in the future (on $\mc{N} \setminus \mc{T}_{-}$) or in the past (on $\mc{N} \setminus \mc{T}_{+}$). The sets $\mc{T}_{\pm}$ are called the incoming ($-$) and outgoing ($+$) tails. Assumption $\hyperlink{AA2}{\rm(A2)}$ guarantees that points either escape to infinity in the future, or they are trapped in a compact region of $\mc{N}$ (and the same holds in the past).

\begin{enumerate}
\item[\textbf{\hypertarget{AA3}{(A3)}}] \textbf{Hyperbolicity.} The \emph{trapped set} $\scrK := \mc{T}_- \cap \mc{T}_+$ is nonempty and compact, and there exists a continuous flow-invariant splitting of $T\N$ over $\mathscr{K}$ as
\[
T_\mathscr{K}\N = E_{0,\N} \oplus E_{s,\N} \oplus E_{u,\N},
\]
where $E_{0,\N}$ is spanned by the flow direction, $E_{s,\N}$ and $E_{u,\N}$ satisfy
\begin{equation}
\label{equation:hyperbolicity-condition}
\|d \phi_t|_{E_{s,\mc{N}}}\| \leq C e^{-\lambda_0 t}, \qquad \|d \phi_{-t}|_{E_{u,\mc{N}}}\| \leq C e^{-\lambda_0 t}, \qquad \forall t \geq 0,
\end{equation}
for some uniform constants $C, \lambda_0 > 0$ that depend on an arbitrary choice of an auxiliary norm $\|\bullet\|$ on $T \N$. 

\end{enumerate}
It follows from $\hyperlink{AA2}{\rm(A2)-(A3)}$ that, if $\scrK \neq \N$, every relatively compact open subset $\scrV\subset \N$ satisfying $\scrK \Subset \scrV$ satisfies $\ms{K} = \cap_{t \in \R} \phi_t(\mathscr{V})$. Such a set $\scrK$ is called \emph{locally maximal} in the literature. In addition, one verifies that, under the above assumptions, for all $x \in \mc{T}_\pm$, $d(\phi_{\mp t} x, \scrK) \to_{t \to +\infty} 0$. We will also assume throughout that $\scrK$ is not reduced to a finite collection of hyperbolic periodic orbits, and that the topological entropy of $(\phi_t)_{t \in \R}$ is strictly positive on $\scrK$.

When $\mathscr{K} = \mc{N}$ and $\mc{N}$ is compact, the flow is said to be \emph{Anosov}. When $\mathscr{K} \neq \mc{N}$, the subset $\mathscr{K}$ is usually very fractal. We let
\begin{equation}
\label{equation:j}
\mathscr{J} := \mathscr{K} \times \fh \subset \mc M.
\end{equation}
Notice that $\mathscr{J}$ is $A$-invariant. Although not compact, it will play the role of the ``trapped set'' for the whole $A$-action and we will loosely refer to $\scrJ$ as the trapped set. We also define

\begin{equation}
\label{equation:tails2}
 \mc{T}_{\pm,\M} := \mc{T}_{\pm} \times \fh.
\end{equation}
To avoid confusion, we may also denote by $\mc{T}_{\pm,\mc{N}} := \mc{T}_{\pm}$ the incoming and outgoing tails on the base $\mc{N}$. It is a standard fact that $E_{s,\mc{N}}$ (resp. $E_{u,\mc{N}}$) admits a natural flow-invariant extension to the incoming tail $\mc{T}_-$ (resp. the outgoing tail $\mc{T}_+$); see, for instance, \cite[Lemma 2.10]{Dyatlov-Guillarmou-16}. In addition, these bundles can be continuously extended to $\mc{N}$ (but they are \emph{a priori} not flow-invariant outside $\mc{T}_\pm$).

This leads to the following definition:
\begin{definition}[Axiom A actions of type $(k,1)$]
\label{definition:axiom-a-action}
Let $(\mc{M},A,\tau)$ be a free Abelian cocycle over a flow (Definition \ref{definition:free-abelian-cocycles}). If the assumptions $\hyperlink{AA1}{\rm(A1)}$, $\hyperlink{AA2}{\rm(A2)}$ and $\hyperlink{AA3}{\rm(A3)}$ are satisfied for some choice of hyperplane $\fh \subset \fa$, we say that the action of $A$ on $\M$ is an \emph{Axiom A action of type $(k,1)$}.
\end{definition}

\begin{remark}
More generally, actions of type $(k,\ell)$ could be defined by requiring that $\R^{k+\ell}$ acts on $\M$ in a locally free manner, that there exists a plane $\mathfrak{h} \simeq \R^k$ such that $H := \exp(\mathfrak{h})$ acts freely and properly on $\M$, and that there is an Anosov $\R^\ell$\-action on the quotient space $\mc{N} := \M/H$.
Most of the spectral analysis carried out in this first part should be transferable to such actions of type $(k,\ell)$ and one should obtain a resonance spectrum consisting of varieties of codimension $\ell$ in $\mathfrak a_\C^*\cong \C^{k+\ell}$. Note that $(0,\ell)$-actions with compact $\mc N$ are the usual rank-$\ell$ Anosov actions and it was shown in \cite{Guedes-Bonthonneau-Guillarmou-Hilgert-Weich-20} that the resonance spectrum is discrete (i.e., a subvariety of complex codimension $\ell$ in $\C^\ell$). It seems to be a very interesting question to explore whether there are classes of discrete subgroups that are structurally in between lattices and Anosov subgroups and that have sufficiently strong structural properties to lead to Axiom A actions of type $(k,\ell)$.
\end{remark}

When describing the leading resonant hypersurface of the $A$-action, it will be convenient to make the following additional assumption on the dynamics over the trapped set:

\begin{itemize}
\item[\textbf{\hypertarget{AA4}{(A4)}}] The flow $(\phi_t)_{t \in \R}$ is topologically transitive on $\scrK$.
\end{itemize}
In particular, Assumption $\hyperlink{AA4}{\rm(A4)}$ excludes pathological path-disconnected trapped sets.

\subsubsection{Partially hyperbolic flow} \label{sssection:phf}

The flow $(\Phi_t)_{t \in \R}$ on $\ms{J}$ is a partially hyperbolic flow with central direction given by $E_{0,\M} \simeq \fa$ (direction of the action) and there is an invariant splitting
\begin{equation}
\label{equation:tm}
T_{\ms{J}}\M = E_{0,\M} \oplus E_{s,\M} \oplus E_{u,\M},
\end{equation}
where $E_{s,\M}$ and $E_{u,\M}$ satisfy estimates similar to \eqref{equation:hyperbolicity-condition}. These bundles have an explicit expression in terms of the cocycle $w \in C^\infty(\N,\fh)$, see \cite[Example 4.2.4]{Cekic-Lefeuvre-24}. It can be checked that $\dd \pi : E_{s/u,\M} \to E_{s/u,\mc{N}}$ is an isomorphism, where $\pi : \M \to \N$ is the footpoint projection. Notice that the decomposition \eqref{equation:tm} is invariant under the full $A$-action.

We also introduce the dual splitting
\begin{equation}\label{equation:t*m}
T_{\ms{J}}^\ast \M = E_{0,\M}^* \oplus E_{s,\M}^* \oplus E_{u,\M}^*.
\end{equation}
Here,
\[
E^*_{0,\M}(E_{s,\M} \oplus E_{u,\M})  = E^*_{s,\M}(E_{0,\M} \oplus E_{s,\M}) = E^*_{u,\M}(E_{0,\M} \oplus E_{u,\M})= 0.
\]
By duality, it is immediate to verify that $\dd\pi^\top : E^*_{s/u,\N} \to E^*_{s/u,\M}$ is an isomorphism as well. In addition, in the trivialization $\M \simeq \N \times \fh$, writing $T_m^*\M \simeq T_{\pi(m)}^*\N \oplus \fh^*$, we have
\begin{equation}
\label{equation:esm}
E^*_{s/u,\M} = \{ (x,h ; \xi, 0) ~:~ x \in \N, h \in \fh, \xi \in E^*_{s/u,\N}\}.
\end{equation}

In what follows, when the context is clear, we shall drop the subscripts $\M$ and $\mc{N}$ decorating these bundles. Observe that, although neither $E_{s,\M}^*$ nor $E_{u,\M}^*$ is a smooth bundle in general, $E_{s,\M}^* \oplus E_{u,\M}^*$ is always a smooth vector bundle as it is the annihilator of the direction of the action
\begin{equation}
\label{equation:smooth-tralala}
E_{s,\M}^* \oplus E_{u,\M}^* = \{ \xi \in T^*\M ~:~ \forall a \in \fa, \xi(\X(a)) = 0\}.
\end{equation}

\subsection{Lyapunov spectrum and limit cone} \label{ssection:lyapunov-spectrum}

\subsubsection{Definition}\label{sssection:lyapunov}

We want to introduce the Lyapunov projection and the limit cone for Axiom A actions of type $(k,1)$. Both terms are borrowed from the language of discrete subgroups in higher rank. In \S\ref{sec:connection-axiom-a}, we will see that for Axiom A actions obtained from Anosov subgroups, this terminology is consistent.

Recall that, in the identification $\M \simeq \N \times \fh$, the flow $(\Phi_t)_{t\in \R}$ of $X_{\M}$ can be written in the form \eqref{eq:abelian-extension-dynamics} (free Abelian cocycle over $(\phi_t)_{t \in \R}$). Let $\mathbf{P}$ be the set of periodic orbits of $(\phi_t)_{t \in \R}$ on $\mathscr{K}$, and let $\mathbf{P}^\sharp$ denote the primitive orbits. For $\gamma \in \mathbf{P}$, we introduce its \emph{Lyapunov projection}
\begin{equation}
\label{equation:lyapunov-projection}
\lambda(\gamma) := \ell_\gamma X_{\M} - \int_0^{\ell_\gamma} w(\phi_s x) \dd s \in \fa,
\end{equation}
where $\ell_\gamma$ denotes the period of $\gamma$ and $x \in \gamma$ is an arbitrary point. 

It follows from \eqref{eq:abelian-extension-dynamics} and $\hyperlink{AA1}{\rm(A1)-(A3)}$ that the set $\mathrm{L}$ of nonzero elements $a \in \fa$ such that there exists $z \in \M$ such that $\tau_a z = z$ is discrete and splits into two disjoint subsets separated by the hyperplane $\mathfrak h \subset \mathfrak a$:
\begin{equation}\label{eq:lyapunov_are_A_periods}
\mathrm{L} = \{\lambda(\gamma) ~:~ \gamma \in \mathbf{P}\} \cup \{-\lambda(\gamma) ~:~\gamma \in \mathbf{P}\}.
\end{equation}
The hyperplane $\fh$ divides $\fa$ into two disjoint open half-spaces; in the above union, the elements $+\lambda(\gamma)$ correspond to the elements contained in the same half-space as $X_{\M}$. In particular, $\lambda(\gamma)$ does not depend on the choice of $\fh$ and $X_\M$. We call
\begin{equation}
\label{equation:lyapunov-spectrum}
\mathrm{L}^+ := \{\lambda(\gamma) ~:~ \gamma \in \mathbf{P}\}
\end{equation}
the \emph{Lyapunov spectrum} of the action. By construction, $\mathrm L$ is intrinsic to the $A$-action on $\M$ and $\mathrm{L}^+$ depends on the choice of sign for $X_{\M}$.

We can now introduce an analog of the limit cone defined by Benoist for discrete subgroups $\Gamma < G$ of semisimple Lie groups in the seminal paper \cite{Benoist-97}:

\begin{definition}[Limit cone of Axiom A actions of type $(k,1)$] \label{definition:limit-cone-cocycle}
The \emph{limit cone} $\mathscr{L} \subset \fa$ is the asymptotic cone of $\mathrm{L}^+$ defined as
\begin{equation}\label{eq:limit-cone}
\mathscr{L} := \overline{\{ t \lambda(\gamma) ~:~ t > 0, \gamma \in \mathbf{P}\}} = \overline{\RR_+\mathrm{L}^+}.
\end{equation}
\end{definition}

\subsubsection{Description of the limit cone} \label{sssection:convexity-limit-cone}We first establish its convexity:

\begin{lemma}[Convexity of the limit cone]
\label{lemma:convexity-limit-cone}
Assume $\hyperlink{AA1}{\rm(A1)-(A4)}$ hold. Then the limit cone  $\mathscr{L}$ is convex, that is for $v_1,v_2 \in \mathscr{L}$, $v_1 + v_2 \in \mathscr{L}$.
\end{lemma}

\begin{proof}
Let $\gamma_1, \gamma_2 \in \mathbf{P}$ and $x_1 \in \gamma_1, x_2 \in \gamma_2$. Let $S_+ \subset \mathscr{K}$ be an orbit segment of $(\phi_t)_{t \in \R}$ connecting an $\eps$-neighborhood of $x_1$ to an $\eps$-neighborhood of $x_2$, where $\eps > 0$ is chosen small enough (such a segment exists by transitivity of $(\phi_t)_{t \in \R}$ on $\mathscr{K}$ (Assumption $\hyperlink{AA4}{\rm(A4)}$)). Similarly, let $S_- \subset \mathscr{K}$ be an orbit segment connecting (a neighborhood of) $x_2$ to $x_1$. Let $n \geq 1$. The pseudo-orbit $\gamma_1^n \sqcup S_+ \sqcup \gamma_2^n \sqcup S_-$ starting at $x_1$ ends close to $x_1$ and can thus be closed by the shadowing lemma; that is, there exists a genuine periodic orbit $\gamma_n \in \mathbf{P}$ shadowing this pseudo-orbit. In addition, a quick computation shows that
\[
\lambda(\gamma_n) = n\lambda(\gamma_1) + n \lambda(\gamma_2) + \mc{O}(1).
\] 
Dividing by $n$ and passing to the limit $n \to +\infty$, we find that $\lambda(\gamma_1) + \lambda(\gamma_2) \in \mathscr{L}$ as $\mathscr{L}$ is closed. This easily implies the claimed result.
\end{proof}

In the next lemma, $\mathrm{M}_\phi$ denotes the set of probability measures supported on the trapped set $\mathscr{K}$ and invariant under the flow $(\phi_t)_{t \in \R}$. The following characterizations of the limit cone also hold:

\begin{lemma}
\label{lemma:formula-for-L}
Assume $\hyperlink{AA1}{\rm(A1)-(A3)}$ hold. Then the limit cone $\mathscr{L}$ coincides with
\begin{enumerate}[label=\emph{(\roman*)}]
\item The set of accumulation points, as $T \to +\infty$, of
\begin{equation}
\label{equation:suzanne}
\left\{\mu \left(X_\M - T^{-1} \int_0^T w(\phi_s x) \dd s\right) ~:~ \mu > 0, \phi_s(x) \in \mathscr{V}, \forall s \in [0,T]\right\},
\end{equation}
where $\scrK\Subset\scrV\Subset\N$ is a relatively compact open neighborhood of $\scrK$.
\item $\R_+ C$, where
\[
C := \left\{ X_\M - \int_{\mathscr{K}} w(x) \dd m(x) ~:~ m \in \mathrm{M}_{\phi}\right\}.
\]
\end{enumerate}
\end{lemma}

\begin{proof}
(i) We assume that there is a single basic set; the general case follows from an adaptation of the argument. Consider a Markov partition with rectangles $R_0, ..., R_N$ for the flow $(\phi_t)_{t \in \R}$ on $\mathscr{V}$ of small diameter. Let $x \in \mathscr{V}$ such that $\phi_t(x) \in \mathscr{V}$ for all $t \in [0,T]$. The trajectory of $x$ is encoded by $a_0 ... a_p$, where $a_i \in \{0,...,N\}$ and $R_{a_0}...R_{a_p}$ is the sequence of rectangles encountered by the flowline of $x$. Define $j_0 := a_0 \in \{0,...,N\}$, $r_0:= 0$, and let $0 \leq s_0 \leq p$ be the last occurrence of $j_0$ in the sequence $a_0...a_p$, that is
\[
\underbrace{a_0}_{=j_0} ... \underbrace{a_{s_0}}_{=j_0} a_{s_0+1} ... a_p. 
\]
By the shadowing lemma, the orbit of $x$ between $R_{a_0} = R_{a_{r_0}}$ and $R_{a_{s_0}}$ can be approximated by a periodic orbit $\gamma_0$. (If $s_0=0$, the orbit cannot be periodic and the first segment of the orbit of $x$ is simply discarded.) Then define $j_1 := a_{s_0+1} \neq j_0$, $r_1 := s_0+1$ and $r_1 \leq s_1 \leq p$ the last occurrence of $j_1$ in the sequence $a_0...a_p$. Similarly, by the shadowing lemma, the orbit of $x$ between $R_{a_{r_1}}$ and $R_{a_{s_1}}$ can be approximated by a periodic orbit $\gamma_1$. We then iterate this process a finite number of times until the sequence $a_0...a_p$ is exhausted and $\gamma$ is approximated by a concatenation $\gamma_0, \gamma_1, ..., \gamma_\ell$ of periodic orbits in $\mathscr{V}$. We then have:
\[
\mu \cdot T^{-1} \int_0^T w(\phi_s x) \dd s = \mu  \cdot T^{-1} \left(\int_{\gamma_0} w + ... +\int_{\gamma_\ell} w + \mc{O}(1)\right),
\]
where the remainder term $\mc{O}(1)$ is uniform with respect to $x$ and $T > 0$. Hence:
\[
\begin{split}
&\mu X_\M - \mu \cdot T^{-1} \int_0^T w(\phi_s x) \dd s \\
&= \mu(1-T^{-1}(\underbrace{\ell_{\gamma_0} + ... + \ell_{\gamma_\ell})}_{=T + \mc{O}(1)})X_\M + \mu T^{-1}(\lambda(\gamma_0)+...+\lambda(\gamma_\ell)) + \mc{O}(T^{-1}) \\
& = \underbrace{\mu T^{-1}(\lambda(\gamma_0)+...+\lambda(\gamma_\ell))}_{\in \mathscr{L}} + \mc{O}(T^{-1}).
\end{split}
\]
The claim then easily follows from Lemma \ref{lemma:convexity-limit-cone}. Conversely, for $\gamma\in\mathbf P$, $x\in\gamma$, and $r>0$, take $T=n\ell_\gamma$ and $\mu=r\ell_\gamma$. Then the expression in \eqref{equation:suzanne} equals $r\lambda(\gamma)$ for every $n$; varying $r$ and taking closures gives the reverse inclusion. \\

(ii) Every invariant probability measure admits an ergodic decomposition. In turn, invariant ergodic measures are weak-$*$ limits of periodic orbits on a hyperbolic basic set. Using Lemma \ref{lemma:convexity-limit-cone}, the result is immediate. 
\end{proof}

Finally, the \emph{dual limit cone} $\mathscr{L}^* \subset \fa^*$ is defined as
\begin{equation}\label{eq:dual-limit-cone}
\scrL^\ast = \{ \theta\in\fa^\ast \ |\ \theta\geq 0 \text{ on }\scrL\}.
\end{equation}
It will play an important role later on.

\subsubsection{Further properties of the limit cone}
We can now introduce a notion of rank for the cocycle:

\begin{definition}[Rank of the cocycle]
\label{definition:rank}
The \emph{rank} of $(\Phi_t)_{t \in \R}$ is defined as the minimal dimension $0 \leq d \leq k$
such that there exists a $d+1$-dimensional subspace $V \subset \fa$ with $\mathscr{L} \subset V$. We say that the cocycle has \emph{full rank} if $d=k$.
\end{definition}
\begin{remark}
 Note that in the setting of discrete subgroups, Benoist \cite{Benoist-97} showed that for a Zariski-dense $\Gamma < G$, the limit cone has nonempty interior. The full rank condition should thus be seen as an analog of the Zariski-density condition for discrete subgroups.
\end{remark}
As the limit cone is (up to the choice of a positive direction for $X_{\M}$) intrinsic to the $A$-action on $\M$, the rank does not depend on the choice of either $\fh$ or $X_\M$.

We now introduce a notion of non-arithmeticity for the Lyapunov spectrum.

\begin{definition}[Non-arithmetic condition]
\label{definition:non-arithmetic}
We say that the cocycle is \emph{non-arithmetic} if there are no nonzero linear forms $\theta \in \fa^*$ such that for all $\gamma \in \mathbf{P}$, $\theta(\lambda(\gamma)) \in 2\pi\Z$.
\end{definition}

Equivalently, after fixing any Euclidean inner product on $\fa$, there are no lines in $\fa$ such that the orthogonal projection of $\mathrm{L}$ onto this line is contained in a lattice.

We now discuss a strengthening of the non-arithmetic condition. We first introduce the following terminology for a subset of $\R^{k+1}$:

\begin{definition}[Diophantine subsets of $\R^{k+1}$]
\label{definition:diophantine}
Let $B \subset \R^{k+1}$ be a finite subset. We say that it is \emph{Diophantine} if there exist $C, \nu > 0$ such that, for every $\xi \in \R^{k+1}$ with $|\xi| \geq C$, there exists $\lambda \in B$ such that:
\[
|e^{i\xi\cdot\lambda}-1| \geq |\xi|^{-\nu}.
\]
\end{definition}

In Appendix \ref{appendix:diophantine}, it is established that Lebesgue-almost every subset of cardinality $k+2$ in $\R^{k+1}$ is Diophantine. This leads to the following condition:

\begin{definition}[Diophantine condition] \label{definition:diophantine-lyapunov} We say that the cocycle is \emph{Diophantine} if there exists a finite subset $B \subset \{\lambda(\gamma) ~:~ \gamma \in \mathbf{P}\} \subset \fa\simeq\R^{k+1}$ of the Lyapunov spectrum (see \eqref{equation:lyapunov-spectrum}) which is Diophantine in the sense of Definition \ref{definition:diophantine}.
\end{definition}

Note that the Diophantine condition (Definition \ref{definition:diophantine-lyapunov}) implies the non-arithmetic condition (Definition \ref{definition:non-arithmetic}).

\subsubsection{Growth functional}

The next construction should be seen as a generalization of the $\varphi$-critical exponent introduced by Quint \cite{Quint-2001}. For $\varphi \in \fa^*$ such that $\varphi > 0$ on the limit cone $\scrL\setminus\{0\}$, we introduce the exponent 
\begin{equation}
\label{equation:delta}
\delta(\varphi) := \lim_{t\to+\infty} t^{-1} \log \sharp \{\gamma \in \mathbf{P} ~:~ \varphi(\lambda(\gamma)) \leq t\}.
\end{equation}
This limit exists whenever $\varphi > 0$ on the limit cone, see \cite{Quint-2001}.

\subsubsection{Hyperplanes with proper action}

When $\tau : A\times\M \to \M$ is an Axiom A action of type $(k,1)$, one can give an explicit and tractable criterion characterizing the set of hyperplanes $\fh'\subset\fa$ having a free and proper action on $\M$, refining the general results of \S\ref{sec:reparameterization}.

We recall that $X_\M \in \fa$ is the chosen generator complementary to $\fh$; $\fh$ is a choice of reference hyperplane in $\fa$ whose action is free and proper and $\N = \M / \exp \fh$ is the quotient. In the following statement, $\omega$ stands for (real) analytic regularity.

\begin{proposition}
\label{proposition:proper-action-phi}
Let $(\M,A,\tau)$ be an Axiom A action of type $(k,1)$ (Definition \ref{definition:axiom-a-action}) of regularity $C^r$ ($r=\infty$ or $r=\omega$). Let $\fh'\subset \fa$ be a hyperplane transverse to $X_\M$. Then:
\begin{enumerate}[label=\emph{(\roman*)}]
\item $H':= \exp(\fh')$ has a proper action on $\M$ if and only if $\fh' \cap \mathscr{L} = \{0\}$.

\item In this case, $\M/H'$ is $C^r$-diffeomorphic to $\N$.

\item Writing $\fh' = \ker \varphi$ for some linear form $\varphi \in \mathrm{int}(\scrL^*)$, choose $u_\varphi \in \fa$ such that $\varphi(u_\varphi)=1$. Then $u_\varphi$ induces a flow $(\phi_t^\varphi)_{t \in \R}$ on the quotient space $\N$ which is a reparameterization of $(\phi_t)_{t \in \R}$. It satisfies Assumptions $\hyperlink{AA2}{\rm(A2)}$ and $\hyperlink{AA3}{\rm(A3)}$ (and likewise for $\hyperlink{AA4}{\rm(A4)}$ when applicable). Its set of periods is given by 
\[
\{\varphi(\lambda(\gamma)) ~:~ \gamma \in \mathbf{P}\}.
\]
\end{enumerate}
\end{proposition}

\begin{proof}
(i) By item (ii) of Lemma \ref{lemma:properness-action-abstract}, the action of $H'$ is proper if and only if for all sequences $(x_n, t_n)_{n \geq 0}$ such that $(x_n)_{n \geq 0}$ and $(\phi_{t_n}(x_n))_{n \geq 0}$ are each contained in a compact subset, and $t_n \to \infty$, 
\[
\left|\int_0^{t_n}(1+\alpha(w))\circ\phi_s(x_n) \dd s\right| \to +\infty. 
\]
Using the closing lemma, it can be established that there exist a constant $C>0$ and periodic points $y_n\in \scrK$ such that, for all $n \geq 0$,
\[
\left|\int_0^{t_n}(1+\alpha(w))\circ\phi_s(x_n) \dd s - \int_0^{t_n}(1+\alpha(w))\circ\phi_s(y_n) \dd s \right| \leq C. 
\]
This proves that the action of $H'$ is proper if and only if it is proper when restricted to $\scrJ$. Since $\scrJ \simeq \scrK \times \fh$ and $\scrK$ is compact, item (iii) of Lemma \ref{lemma:properness-action-abstract} implies that $H'$ acts properly if and only if there exists $\eps,T>0$ such that for all $x\in\scrK$
\begin{equation}\label{eq:good-average-on-K}
\left| \frac{1}{T} \int_0^T (1+\alpha(w))\circ\phi_s(x) \dd s\right| > \eps . 
\end{equation}
On each (path-connected) basic set of $\scrK$, \eqref{eq:good-average-on-K} has constant sign; without loss of generality, let us assume it is positive (if not, one applies the argument given below to each basic set separately). The lower bound \eqref{eq:good-average-on-K} then implies that for all $t > T$:
\[
\int_0^t (1+\alpha(w))\circ\phi_s(x) \dd s > (t-T)\eps.
\]
In turn, this implies that for all periodic orbits $\gamma$ of $(\phi_t)_{t \in \R}$, $\int_\gamma (1+\alpha(w)) > \eps$, and by density of periodic orbits, that for every flow-invariant probability measure $m$ on $\scrK$, $\int_{\scrK} (1+\alpha(w)) \dd m > \eps$.

We now relate this to the limit cone. Observe that it follows from Lemma \ref{lemma:formula-for-L} that $\fh'\cap \scrL = \{0\}$ if and only if there exists $\eps>0$ such that for every $h\in \fh$ such that $\alpha(h)=1$ and every invariant probability measure $m\in \mathrm{M}_\phi$,
\begin{equation}
\label{equation:ptit-dej}
\left| \int_{\scrK} (h+w)~ \dd m \right| > \eps. 
\end{equation}
In turn, this is equivalent to the existence of $\eps>0$ such that for every $m\in\mathrm{M}_\phi$, 
\begin{equation}
\label{equation:ptit-dej2}
\left| \int_{\scrK} (1+\alpha(w))\dd m\right| > \eps. 
\end{equation}
(Indeed, that \eqref{equation:ptit-dej2} implies \eqref{equation:ptit-dej} is immediate. For the converse, if \eqref{equation:ptit-dej2} does not hold, then there exists a measure $m \in \mathrm{M}_\phi$ such that $\int_{\scrK} (1+\alpha(w))\dd m=0$. Setting $h := -\int_{\scrK} w~\dd m$, we find that $\alpha(h)=1$ and $\int_{\scrK} (h+w)~ \dd m = 0$.) The condition \eqref{equation:ptit-dej2} is precisely the condition derived in the previous paragraph for the properness of the $H'$-action, which concludes the proof.  \\

(ii) We start by working in the $C^\infty$ setting and deal with the analytic case using an approximation argument at the end. Start by observing that according to \eqref{equation:lala-cob}, setting
\[
u_T = -\frac{1}{T} \int_0^T (T-t)w\circ\phi_t  \dd t \in C^\infty(\N,\fh),
\]
and replacing the trivializing section $\sigma$ by $\sigma_T := \tau_{u_T}\sigma$, we replace $w$ by the new cocycle
\[
w_{T} := w - X_\N u_T =\frac{1}{T} \int_{0}^T w\circ\phi_t \dd t \in C^\infty(\N,\fh).
\]
We now work in a sufficiently small neighbourhood $U_\mathscr{K}$ of $\mathscr{K}$. By \eqref{eq:good-average-on-K}, provided $U_\mathscr{K}$ is chosen small enough, for all $x\in U_\mathscr{K}$, for some $T,\eps>0$,
\[
\frac{1}{T} \left| \int_0^{T} (1+\alpha(w))\circ\phi_s(x)\dd s\right| > \eps. 
\]
This ensures that the new $w_{\scrK}=w_{T|U_\mathscr{K}} : U_\scrK \to \fh$ satisfies
\[
|\alpha(w_\scrK) +1| > \eps.
\]
Now, let $\chi$ be a smooth cutoff supported in $U_\scrK$, and equal to $1$ near $\scrK$. First assume that $\alpha(w_\scrK)+1>0$. Since $(1-\chi)w_T$ is supported outside of $\scrK$, there exists a smooth function $u$ so that 
\[
(1-\chi)w_T = X_\N u. 
\]
It follows that by an additional change of section, we can replace further $w_T$ by $\tilde{w}=\chi w_T$. Naturally, we have globally
\[
\alpha(\tilde{w}) + 1  > \eps. 
\]
In the other case, when $\alpha(w_\scrK)+1 < 0$, it is better to find $u$ smooth satisfying
\[
(1-\chi)(w_T + 2 h_0) = X_\N u,\text{ where }h_0\in \fh\text{ with } \alpha(h_0)= 1
\]
so that now $\tilde{w}= \chi w_T - 2h_0(1-\chi)$. In both cases, $|\alpha(\tilde{w}) + 1 | > \eps$, and Lemma \ref{lemma:transverse-graph-section} applies to prove that the corresponding map $\tilde{\sigma}:\N\to\M$, which is a section of the bundle $\M\to \M/H$ is actually also a section of the bundle $\M\to \M/H' \simeq \N$. 

Since we have used a cutoff function, we cannot directly upgrade this argument to the analytic case. However writing $\tilde{\sigma}= \tau_{v}\sigma$, for some $C^\infty$ map $v$, we can find $v'$ another map, analytic, that approximates $v$ arbitrarily well in global $C^1$ norm \cite[Proposition 8]{Grauert-58}. The corresponding $w' = \tilde{w} + X_\N(v'-v)$ is arbitrarily $L^\infty$-close to $\tilde{w}$, so that the condition of Lemma \ref{lemma:transverse-graph-section} is still satisfied, and this provides an analytic solution to the problem. \\

(iii) To see that $(\phi^\varphi_t)_{t \in \R}$ is a reparameterization of $(\phi_t)_{t \in \R}$, it suffices to observe that both flows have the same orbit foliation. The rest of the statement is elementary. 
\end{proof}

\subsection{Bundles and distributions}

\label{ssection:bundles-distributions}

\subsubsection{Admissible bundles}
\label{sssection:admissible} 

Let $\mc{E} \to \M$ be a smooth vector bundle. As $\M$ retracts onto $\mc{N}$, $\mc{E} = \pi^* \mc{E}_{\mc{N}}$ for some smooth vector bundle $\mc{E}_{\mc{N}} \to \mc{N}$,
where $\pi : \M \to \mc{N}= \M/H$ is the projection. A section $f \in C^\infty(\M,\E)$ is called a \emph{pullback} section if there exists $s \in C^\infty(\N,\E_\N)$ such that $f = \pi^*s$.

Consider a linear Lie algebra embedding into first-order differential operators:
\[
\mathbf{X}_{\mc{E}} : \fa \to \mathrm{Diff}^1(\M,\mc{E}).
\]
That is, $[\X_{\mc{E}}(a),\X_{\mc{E}}(b)] = 0$ for all $a,b \in \fa$.
We say that $\mathbf{X}_{\mc{E}}$ is an \emph{admissible lift} of the action to $\mc{E}$ if:
\begin{itemize}
\item $\mathbf{X}_{\mc{E}}(a)(\varphi  \otimes s) = \mathbf{X}(a)\varphi  \otimes s + \varphi  \otimes \mathbf{X}_{\mc{E}}(a)s$ for all $a \in \fa$, $\varphi \in C^\infty(\M)$ and $s \in C^\infty(\M,\mc{E})$;
\item $\mathbf{X}_{\mc{E}}(h)\left(\pi^* s\right) = 0$ for all $s \in C^\infty(\N, \E_\N)$, $h \in \fh$.
\end{itemize}
In short, we say that $\mc{E} \to \M$ is an \emph{admissible} bundle if it is equipped with such a lift of the action. Note that the following holds:

\begin{lemma}
\label{lemma:juilley}
A section $f \in C^\infty(\M,\E)$ is a pullback section if and only if $\X_{\E}(h) f = 0$ for all $h \in \fh$.
\end{lemma}

\begin{proof}
By assumption, any pullback section satisfies $\X_{\E}(h)( \pi^*s) = 0$, so it remains to prove the converse statement. Let $U \subset \N$ be an open subset. Suppose $\X_{\E}(h) f = 0$ for all $h \in \fh$ and write the section locally on $\pi^{-1}(U) \subset \M$ as $f = \sum_i f_i \cdot \pi^* s_i$ for some functions $f_i$ defined on $\M$ and a frame $s_i\in C^\infty(U, \E_{\N})$. Then $\X_{\E}(h) f = 0 = \sum_i \X(h) f_i \cdot \pi^* s_i$, so $\X(h) f_i = 0$ for all $h \in \fh$. That is, $f_i$ is constant on each $H$-fiber of $\pi^{-1}(U)$ and can thus be written as $f_i = \pi^* \widetilde{f}_i$ for some function defined on $U \subset \N$.
\end{proof}

Any admissible lift $\X_\E$ of the action allows us to define a differential operator $\X_{\E_\N}$ of order $1$ acting on smooth sections $C^\infty(\N,\E_\N)$ of $\E_\N \to \N$ which satisfies the Leibniz rule
\[
\X_{\E_\N}(f\otimes s) = X_\N f \otimes s + f \otimes \X_{\E_\N}\ s, \quad \forall f \in C^\infty(\mc{N}), s \in C^\infty(\mc{N},\mc{E}_{\N}).
\]
We call such a differential operator a \emph{lift} of $X_\N$ to $\E_\N \to \N$. Indeed, given $s \in C^\infty(\N,\mc{E}_\N)$, it is immediate to verify that $\X_{\E}(X_\M) (\pi^* s) = \pi^*f$ for some (unique) $f \in C^\infty(\N,\E_\N)$ by Lemma \ref{lemma:juilley}, and we define $\X_{\E_\N}$ by setting $\X_{\E_\N} s = f$. A quick computation shows that $\X_{\E_\N}$ is a differential operator of order $1$ satisfying the Leibniz rule.

Conversely, we claim that any lift $\X_{\E_\N}$ of the vector field $X_\N$ to $\E_\N$ induces a lift $\X_\E$ of the $A$-action to $\E$. Indeed, any smooth section of $\E = \pi^*\E_\N$ can be locally written as a linear combination of pullback sections $\pi^* s$, where $s \in C^\infty(\N,\E_\N)$ (with $C^\infty$ coefficients defined on an open subset of $\M$). Hence, as $\X_\E$ is required to satisfy the Leibniz rule, it suffices to define $\X_\E$ on pullback sections and we set:
\[
\X_\E(X_\M) (\pi^* s) := \pi^*(\X_{\E_\N} s), \qquad \X_\E(h)(\pi^*s) := 0,  \forall h \in \mathfrak{h}.
\]
It is immediate to verify that $[\mathbf{X}_{\mc{E}}(a), \mathbf{X}_{\mc{E}}(b)] = 0$ for all $a, b \in \fa$ in this case.

\begin{lemma}
The above constructions provide a one-to-one correspondence between lifts $\X_{\E_\N}$ of $X_\N$ on the vector bundle $\E_\N \to \N$ and admissible lifts $\X_\E$ of the action on the vector bundle $\E \to \M$, where $\E=\pi^*\E_\N$.
\end{lemma}

The proof is immediate. While any vector bundle $\E_\N \to \N$ admits an admissible lift $\X_\E$, this is most useful when we have a \emph{geometric description} of such a lift. Typical examples are provided by $\mc{E} = \Lambda^m T^*\M$.

\begin{example}[Differential forms] The bundles $\mc{E} = \Lambda^m (E_{s,\M}^*\oplus E^*_{u,\M})$, for $0 \leq m \leq d_{\mc{N}}-1$, are admissible. They correspond to the vector bundles of $m$-forms over $\M$ vanishing on $\fa$. There is a natural operator $\mathbf{X}_{\mc{E}}$ defined on these bundles by setting, for $a \in \fa$ and $u \in C^\infty(\M,\mc{E})$,
\begin{equation}
\label{equation:lie}
\X_{\mc{E}}(a) u := \mc{L}_{\X(a)} u,
\end{equation}
where $\mc{L}$ stands for the Lie derivative. We will mostly omit the index $\mc{E}$ and write $\mathbf{X}$ in this case. Notice that $\E = \pi^*\E_\N$ with
$\E_\N = \Lambda^m(E_{s,\N}^*\oplus E_{u,\N}^*)$ and $\pi : \M \to \N$ is the footpoint projection.
\end{example}

We also introduce the shorthand notations
\[
\X_\M := \X_\E(X_\M), \qquad \X_\N := \X_{\E_\N}.
\]
In the particular case where $\E$ is a (sub)bundle of forms on $T^*\M$, we may also write $\mc{L}_{X_\M}$ and $\mc{L}_{X_\N}$.

Finally, we say that a norm $\|\bullet\|$ on $\E$ is \emph{admissible} if, for all $z \in \M$ and $h \in H$, the linear map $E_h(z) : \E_z \to \E_{\tau_h z}$ is an isometry or, equivalently, $\|\bullet\|$ is the pullback of a norm from $\E_{\N}$. Similarly, a connection $\nabla^{\E}$ is \emph{admissible} if $\nabla^{\E} = \pi^* \nabla^{\E_{\N}}$, for some smooth connection on $\E_{\N} \to \N$.

\subsubsection{Propagator} \label{sssection:propagator}

Given $a \in \fa$, the first-order differential operator $\X_\E(a)$ defines a propagator $t \mapsto e^{-t\X_\E(a)}$ such that for any $f \in C^\infty(\M,\E)$, $f(t) := e^{-t\X_\E(a)} f \in C^\infty(\M,\E)$ satisfies the transport equation:
\[
f(0) = f, \qquad \partial_t f(t) = -\X_\E(a) f(t).
\]
It can be verified that
\begin{equation}
\label{equation:propagator-def}
[e^{-\X_\E(a)}f](\tau_a z) = E_a(z) (f(z))
\end{equation}
for some linear map $E_a(z) : \E_{z} \to \E_{\tau_az}$ which depends smoothly on all parameters; see, for instance, \cite[Lemma 9.1.1]{Lefeuvre-book}. (Recall that $\tau : A \times \M \to \M$ denotes the action.) \\

Let $\scrK \Subset \scrV_0 \Subset \scrV_0'$ be two relatively compact open neighborhoods of $\scrK$ and further assume that the following holds: whenever $x \in \scrV_0$ and $T \in \R$ satisfy $\phi_T(x) \in \scrV_0$, one has $\phi_t(x) \in \scrV_0'$ for all $t \in [\min\{0,T\},\max\{0,T\}]$. See Lemma \ref{lemma:deja} for the existence of such neighborhoods. We also let $\scrU_0 := \pi^{-1}(\scrV_0)$ where $\pi : \M \to \N$ is the projection. 

\begin{lemma}
\label{lemma:vilain}
Let $\|\bullet\|$ be an admissible norm on $\E$ (see \S\ref{sssection:admissible}). Then there exists $C > 0$ such that for all $t \geq 0$:
\[
\|e^{t \X_{\M}}\|_{L^\infty(\scrU_0,\E) \to L^\infty(\scrU_0,\E)} \leq C e^{C t}.
\]
\end{lemma}

More generally, the bound holds by replacing $L^\infty(\scrU_0,\E)$ by $C^k(\scrU_0,\E)$, where the $C^k$-norm is measured using an admissible connection $\nabla^{\E}$ (see \S\ref{sssection:admissible} for admissible connections); that is,
\[
\|f\|_{C^k(\scrU_0,\E)} := \sum_{j=0}^k \|(\nabla^{\E})^j f\|_{L^\infty(\scrU_0,\E\otimes(T^*\M)^{\otimes j})}.
\]

\begin{proof} 
This follows immediately from the compactness of $\scrV_0$ and the admissibility of the norm.
\end{proof}

\subsubsection{Distributions} We now equip $\fa$ with a Euclidean metric and an orientation. Recall that $E_{s,\M}^* \oplus E_{u,\M}^*$ is a smooth vector bundle over $\M$ (see \eqref{equation:smooth-tralala}). Let $E \subset T\M$ be an arbitrary smooth subbundle of rank $d_{\M}-(k+1)$, everywhere transverse to $E_{0,\M} \simeq \fa$; that is, $E \oplus E_{0,\M} = T\M$. We define
\[
\mathrm{vol}_{\fa} \in C^\infty(\M,\Lambda^{k+1} T^*\M),
\]
by 
\begin{equation}
\label{equation:def-vola}
\iota_v\mathrm{vol}_{\fa}=0,\quad \forall v\in E, \qquad \mathrm{vol}_{\fa}|_{E_{0,\M}} = \mathrm{Lebesgue}_{\fa},
\end{equation}
where $\mathrm{Lebesgue}_{\fa}$ stands for the Euclidean volume form of a fixed Euclidean metric on $\fa$. In what follows, we write $\dd a$ for Lebesgue measure on $\fa$ and $\dd\xi$ for the induced Lebesgue measure on $\fa^*$.

The smooth $(k+1)$-form $\mathrm{vol}_{\fa}$ is \emph{not} invariant under the $A$-action. However, for $0 \leq m \leq d_{\N}-1$, we may define a sesquilinear pairing
\[
C^\infty_{\comp}(\M,\Lambda^m (E_{s,\M}^* \oplus E_{u,\M}^*)) \times C^\infty(\M,\Lambda^{d_{\N}-1-m} (E_{s,\M}^* \oplus E_{u,\M}^*))) \to \C
\]
by
\begin{equation}
\label{equation:pairing}
\big\langle\!\langle \varphi, \psi \rangle\!\big\rangle_{\M} := \int_{\M}\vol_{\fa}\wedge \varphi \wedge \overline{\psi},
\end{equation}
and a straightforward computation shows that \eqref{equation:pairing} is \emph{independent} of $E$. In addition, this pairing can be extended to compactly supported distributions provided they satisfy the usual transverse wavefront set condition (see \cite[Lemma 4.3.1]{Lefeuvre-book} for instance).

\begin{lemma}
For all $\varphi, \psi$ as above and $a \in \fa$, the following identities hold:
\[
\big\langle\!\langle\mc{L}_{\mathbf{X}(a)}\varphi, \psi\rangle\!\big\rangle_{\M} = -\big\langle\!\langle \varphi, \mc{L}_{\mathbf{X}(a)} \psi \rangle\!\big\rangle_{\M}, \qquad \big\langle\!\langle (e^a)^*\varphi,\psi \rangle\!\big\rangle_{\M}= \big\langle\!\langle \varphi, (e^{-a})^*\psi \rangle\!\big\rangle_{\M}.
\]
\end{lemma}

\begin{proof}
Let $a \in \fa$. Since $e^a$ acts isometrically on $E_{0,\M}\simeq \fa$, we have $(e^a)^*\vol_\fa|_{E_{0,\M}}=\mathrm{Lebesgue}_{\fa}$. Hence $(e^a)^*\vol_{\fa}-\vol_{\fa} \in \Lambda^{k+1}(E_{u,\M}^*\oplus E_{s,\M}^*)$ and thus the needed wedge identity holds for any $\varphi, \psi$ as in the lemma:
\[
\vol_{\fa}\wedge \varphi \wedge \overline{\psi}  =  (e^a)^*\vol_{\fa} \wedge \varphi \wedge \overline{\psi} .
\]
Differentiating with respect to $a$, we also obtain
\begin{equation}
\label{equation:dodo-dodo}
 \mathcal{L}_{\X(a)}  \vol_{\fa}\wedge \varphi \wedge \overline{\psi}  = 0.
\end{equation}
The above two identities easily imply the claimed results after integration over $\M$.
\end{proof}

\begin{remark}
In the specific case in which $E_{s,\M}\oplus E_{u,\M}$ is smooth, one can choose $E=E_{s,\M}\oplus E_{u,\M}$ in \eqref{equation:def-vola}, and $\vol_{\fa}$ then becomes invariant under the $A$-action as a (smooth) $(k+1)$-form. In rank one, this condition holds, for instance, if the flow is contact. For Anosov representations, $E_{s,\M}\oplus E_{u,\M}$ is always smooth and such a choice can be made.
\end{remark}

\section{Spectral theory}
\label{section:results-cocycles}

\subsection{Resonance spectrum} \label{ssection:results1} We present our main results in a form that is intrinsic to the $A$-action on $\M$, without relying on any particular choice of decomposition $\M \simeq \N \times \R^k$ (i.e., a choice of global trivializing section $\sigma : \N \to \M$). Throughout, $\mc{E} \to \M$ is a fixed admissible vector bundle (see \S\ref{sssection:admissible}). 

\subsubsection{Existence of a resonance spectrum} We start by introducing the notion of a resonance spectrum. 

\begin{definition}[Resonance spectrum]
\label{definition:rt-intro}
A point $\mathbf{s} \in \fa^*_{\C}$ is called a \emph{dynamical resonance} (for $\mc{E}$-valued sections) if there exists a nonzero $u \in \mc{D}'(\M,\mc{E})$ such that
\[
\WF(u) \subset E_{u,\M}^*, \quad \supp(u) \subset \mc{T}_{+,\M}, \quad (-\X(a)-\s(a))u = 0, \forall a\in\mathfrak{a}.
\]
We denote the set of dynamical resonances by $\sigma_{\mathrm{RS},+} \subset \fa_\C^*$.
\end{definition}

We also use $(-\X-\mathbf{s})u=0$ as shorthand. In the Anosov case ($\mathscr{K}=\mc{N}$), the support condition on $u$ is vacuous. The space of \emph{resonant states} is defined as
\begin{equation}
\label{equation:ruelle-taylor-resonant}
\mathrm{Res}(\s) := \{u \in \mc{D}'_{E_u^*}(\M,\mc{E}) ~:~ \supp(u) \subset \mc{T}_{+,\M}, (-\X-\s) u = 0\}.
\end{equation}

A ``negative'' resonance spectrum $\sigma_{\mathrm{RS},-}$ can be defined as the set of $\mathbf{s} \in \fa^*_{\C}$ for which there are nontrivial solutions of $(+\X-\mathbf{s})u = 0$ satisfying $\WF(u) \subset E_{s,\M}^*$ and $\supp(u) \subset \mc{T}_{-,\M}$.
In certain contexts, if only the positive resonances appear, we will drop the index $+$ and simply write $\sigma_{\mathrm{RS}}$. If there is more than one vector bundle, we might also write $\sigma_{\mathrm{RS},\pm}(\mc{E})$ to distinguish their spectra.

Our first aim is to prove the following result:
\begin{theorem}[Existence of the resonance spectrum]
\label{theorem:rt-anosov1}
Under assumptions $\hyperlink{AA1}{\rm(A1)-(A3)}$, the following properties hold:
\begin{enumerate}[label=\emph{(\roman*)}]
\item \textbf{\emph{Resonance spectrum.}} The subset
\[
\sigma_{\mathrm{RS},+} \subset \mathfrak a_\C^*
\]
is a complex variety of codimension $1$. It is called the \emph{resonance spectrum} of the action.

\item \textbf{\emph{Resonant states.}} For $\mathbf{s} \in \sigma_{\mathrm{RS},+}$, $\mathrm{Res}(\mathbf s)$ is a finite-dimensional vector space. \end{enumerate}
\end{theorem}

When $\E$ is a real vector bundle and $\X$ is a real operator, $\sigma_{\mathrm{RS},+}$ is invariant under the complex conjugation operator $\mathrm{Conj} : \fa^*_{\C} \to \fa^*_{\C}, \s \mapsto \overline{\s}$. Indeed, \eqref{equation:ruelle-taylor-resonant} shows that $\s \in \sigma_{\mathrm{RS},+}$ if and only if $\overline{\s} \in \sigma_{\mathrm{RS},+}$; if $u$ is a resonant state at $\s$, then $\overline{u}$ is a resonant state at $\overline{\s}$.

By a complex variety of codimension $1$, we mean that $\sigma_{\mathrm{RS},+}$ can be locally described as the $0$-level set of a holomorphic function. We shall see in \S\ref{section:general-zeta-function-result} that it actually coincides with the $0$-level set of a holomorphic function defined globally on $\fa^*_\C$, called the \emph{dynamical determinant}. The same result holds for the negative resonance spectrum $\sigma_{\mathrm{RS},-}$ by replacing $-\mathbf{X}$ by $\mathbf{X}$, after interchanging the roles of $\mc{T}_-$ and $\mc{T}_+$ and of $E_s^*$ and $E_u^*$.

\begin{remark}
\label{remark:duality}
The positive and negative resonance spectra are related by
\[
\sigma_{\mathrm{RS},+}(\mc{E}) = \mathrm{Conj}(\sigma_{\mathrm{RS},-}(\mc{E}^* \otimes \Omega^1\M)),
\]
where $\mathrm{Conj} : \fa^*_{\C} \to \fa^*_{\C}$ denotes complex conjugation. When $\mc{E} = \Lambda^m (E_{s,\M}^*\oplus E_{u,\M}^*)$ and $\X$ is the Lie derivative operator, for $0 \leq m \leq d_{\mc{N}}-1$, the duality \eqref{equation:pairing} yields
\[
\begin{split}
\sigma_{\mathrm{RS},+}(\Lambda^m(E_{s,\M}^* \oplus E_{u,\M}^*)) & = \mathrm{Conj}(\sigma_{\mathrm{RS},-}(\Lambda^{d_{\mc{N}}-1-m} (E_{s,\M}^* \oplus E_{u,\M}^*))) \\
& = \sigma_{\mathrm{RS},-}(\Lambda^{d_{\mc{N}}-1-m} (E_{s,\M}^* \oplus E_{u,\M}^*)),
\end{split}
\]
 where the second equality follows because the vector bundle is real.
\end{remark}

\subsubsection{Technical results} \label{ssection:results2} The results described in the previous subsection are all intrinsic to the action of $A$ on $\M$. However, they will be derived using the explicit factorization of the action in terms of an abelian cocycle over the hyperbolic flow $(\phi_t)_{t \in \R}$. We will assume throughout this section that we have chosen a vector field $X_\M \in \fa$ as well as a $k$-dimensional subspace $\fh\subset \fa$ acting freely and properly on $\M$. We then obtain a vector field $X_{\mc{N}}$ on $\N$ which generates the flow $(\phi_t)_{t \in \R}$ on $\mc{N}$. Let $\alpha_{\mc{N}} \in C^0(\mc{N},T^*\mc{N})$ be the $1$-form such that
\begin{equation}
\label{equation:alpha}
\alpha_{\mc{N}}(X_{\N}) = 1, \qquad \alpha_{\mc{N}}(E_{s,\mc{N}} \oplus E_{u,\mc{N}})=0,
\end{equation}
where the stable and unstable bundles are continuously extended outside the incoming and outgoing tails $\mc{T}_{\pm,\mc{N}}$. Note that $X_{\mc{N}}$ does not vanish on a small neighborhood $\mathscr{V}$ of $\scrK$, so $\alpha_{\mc{N}}$ is well-defined near $\scrK$, and it is extended arbitrarily on the complement of an open neighborhood of $\scrK$.

There is a dual splitting
\[
T^*_{\mathscr{K}}\mc{N} = E^*_{0,\mc{N}} \oplus E^*_{s,\mc{N}} \oplus E^*_{u,\mc{N}},
\]
where the bundles are defined by the annihilation relations
\begin{equation}
\label{equation:pratique}
E^*_{0,\mc{N}}(E_{s,\mc{N}} \oplus E_{u,\mc{N}}) = E^*_{s,\mc{N}}(E_{0,\mc{N}} \oplus E_{s,\mc{N}}) = E^*_{u,\mc{N}}(E_{0,\mc{N}} \oplus E_{u,\mc{N}})= 0.
\end{equation}

If $\mc{E} \to \M$ is an admissible bundle in the sense of \S\ref{sssection:admissible} and $\X$ is an admissible lift of the action, recall the notation $\X_{\M}=\X_{\E}(X_\M)$. Furthermore, $\mc{E} = \pi^* \mc{E}_{\mc{N}}$ for some smooth bundle $\mc{E}_{\mc{N}} \to \mc{N}$. We let $\mathfrak{h}^*_{\C} := \mathfrak{h}^* \otimes_{\R} \C$ be the complexified dual.  For $\mathbf{s}' \in \mathfrak{h}^*_{\C}$, we then introduce the extension operator
\[
\mathbf{e}(\mathbf{s}') : C^\infty(\mc{N}, \mc{E}_{\mc{N}}) \to C^\infty(\M,\mc{E}), \qquad \mathbf{e}(\mathbf{s}') u(x,h) := e^{-\mathbf{s}'(h)} \pi^* u(x,h).
\]
 For every $h_0 \in \mathfrak{h}$ , it satisfies the relations
\begin{equation}
\label{equation:relations-e}
\mathbf{X}(h_0)\mathbf{e}(\mathbf{s}') = -\mathbf{s}'(h_0) \mathbf{e}(\mathbf{s}'), \qquad \X_{\M}\mathbf{e}(\mathbf{s}') = \mathbf{e}(\mathbf{s}')(\X_{\N} - \mathbf{s}'(w)),
\end{equation}
 where $w\in C^\infty(\N, \fh)$ is the cocycle given by the choice of $X_\M$ and $\fh$. The proof of these identities is a straightforward computation.

Let $\alpha_\M\in\fa^*$ be defined by $\alpha_\M(X_\M)=1$ and $\alpha_\M|_{\fh}=0$. Recall that $\fa^*_{\C} = \C \alpha_{\M} \oplus \mathfrak{h}_\C^*$. For
\[
\mathbf{s} = s_0 \alpha_{\M} + \mathbf{s}' \in \fa^*_{\C},
\]
 (which we often abbreviate as $\mathbf{s}=(s_0,\mathbf{s}')$), it is convenient to introduce the operator
\begin{equation}
\label{equation:ps}
P_\pm(\mathbf{s}) := \mp \X_{\N} -s_0 + \mathbf{s}'(w),
\end{equation}
acting on $C^\infty(\mc{N},\mc{E}_{\mc{N}})$.

 Equation \eqref{equation:relations-e} shows that, for $f \in C^\infty(\mc{N},\mc{E}_{\mc{N}})$,
\begin{equation}
 \label{equation:ps_extension}
(-\X_{\M}-s_0)\mathbf{e}(\mathbf{s}')f = \mathbf{e}(\mathbf{s}')(-\X_{\N} -s_0+ \mathbf{s}'(w)) f = \mathbf{e}(\mathbf{s}') P_+(\mathbf{s})f.
\end{equation}
Combining this identity with the first equation in \eqref{equation:relations-e}, this suggests a relation between the resonance spectrum $\sigma_{\mathrm{RS},+}$ of the cocycle and the Pollicott--Ruelle spectrum of the operator $P_+(\mathbf{s})$.

In the following, $K_A$ denotes the Schwartz kernel of an operator $A$. For a closed cone $\mathrm{C} \subset T^*\M$, recall that $\mc{D}'_{\mathrm{C}}(\M)$ denotes the subset of distributions with wavefront set in $\mathrm{C}$. Our aim is to prove the following result:

\begin{theorem}
\label{theorem:rt-anosov2}
The following properties hold:
\begin{enumerate}[label=\emph{(\roman*)}]
\item\label{it:mero} \textbf{\emph{Meromorphic extension.}} The family of operators 
\begin{equation}
\label{equation:meromorphic-extension-eu0}
\fa^*_{\C} \ni \mathbf{s} \mapsto P_+(\mathbf{s})^{-1}\in \mc{L}(C^\infty_{\comp}(\mc{N},\mc{E}_{\mc{N}}), \mc{D}'(\mc{N}, \mc{E}_{\mc{N}}))
\end{equation}
is holomorphic for $|\Re(\s)| \gg 0$ and admits a meromorphic extension to $\fa^*_{\C}$ with singularities again given by $\sigma_{\mathrm{RS},+}$. In addition,
\begin{equation}
\label{equation:meromorphic-extension-eu}
\fa^*_{\C} \ni \mathbf{s} \mapsto P_+(\mathbf{s})^{-1}\in \mc{L}(\mc{D}'_{\comp,E_{u,\N}^*}(\mc{N},\mc{E}_{\mc{N}}), \mc{D}'_{E_{u,\N}^*}(\mc{N}, \mc{E}_{\mc{N}}))
\end{equation}
also admits a meromorphic extension to $\fa^*_{\C}$ with singularities given by $\sigma_{\mathrm{RS},+}$.

\item\label{it:wavefront} \textbf{\emph{Wavefront set of the Schwartz kernel.}} For every $\mathbf{s}_\star \in \fa^*_{\C}$, there exist
\begin{enumerate}[label=\emph{(\alph*)}]
\item a neighborhood $U \subset \fa^*_{\C}$ of $\mathbf{s}_\star$;
\item holomorphic operators $U \ni \mathbf{s} \mapsto H(\mathbf{s}), R(\mathbf{s})$ bounded as maps $C^\infty_{\comp}(\mc{N},\mc{E}_{\mc{N}}) \to \mc{D}'(\mc{N},\mc{E}_{\mc{N}})$, where $R(\mathbf{s})$ has rank at most $N$;
\item\label{it:singularities} continuous functions $\fh_\C^*\ni \mathbf{s}' \mapsto \sigma_i(\mathbf{s}')\in \C$ for $i=1,...,N$, such that $\mathbf{s} \mapsto (s_0-\sigma_1(\mathbf{s}'))...(s_0-\sigma_N(\mathbf{s}'))$ is holomorphic on $U\subset \fa_\C^*$ and $\sigma_{\mathrm{RS}} \cap U = \{s_0 = \sigma_i(\mathbf{s}')$ for some $1 \leq i \leq N\}$,
\end{enumerate}
such that
\begin{equation}
\label{equation:p-}
P_+(\mathbf{s})^{-1} = H(\mathbf{s}) + \dfrac{R(\mathbf{s})}{(s_0-\sigma_1(\mathbf{s}'))...(s_0-\sigma_N(\mathbf{s}'))}, \qquad \forall \mathbf{s} \in U.
\end{equation}
In addition,
\begin{equation}
\begin{split}
\label{equation:unif}
 \WF'(K_{H(\mathbf{s})}), \WF'(K_{R(\mathbf{s})}) \subset_{\mathrm{unif.}} \mathrm{C_{sing}},
\end{split}
\end{equation}
where
\[
\mathrm{C_{sing}} := \Delta(T^*\mc{N}) \cup \Omega_+ \cup (E_u^*|_{\mc{T}_+} \times E_s^*|_{\mc{T}_-}) \subset T^*(\mc{N} \times \mc{N}),
\]
$\Delta(T^*\mc{N}) \subset T^*\mc{N} \times T^*\mc{N}$ denotes the diagonal, 
\[
\Omega_+ := \{(\widetilde{\phi}_t(x,\xi), (x,\xi)) ~:~ t \geq 0, (\xi,X_\mc{N}(x)) = 0, x,\phi_t(x) \in \mc{N}\},
\] 
and $\widetilde{\phi}_t(x,\xi):= (\phi_t(x),d\phi_t^{-\top}(x)\xi)$ is the symplectic lift of the flow.

\end{enumerate}
\end{theorem}

The same result holds for $P_-(\s)^{-1}$ up to reversing the $\pm$ signs. The kernel of $R(\mathbf{s})$ actually satisfies the bounds
\begin{equation}
\label{equation:precise}
 \supp(K_{R(\mathbf{s})}) \subset \mc{T}_+ \times \mc{T}_-, \qquad \WF'(K_{R(\mathbf{s})}) \subset E_u^* \times E_s^*,
 \end{equation}
as we shall see in the proof. That $\mathbf{s} \mapsto P_+(\mathbf{s})^{-1}$ is meromorphic as a family of operators
\[
\mc{L}(C^\infty_{\comp}(\mc{N},\mc{E}_{\mc{N}}), \mc{D}'(\mc{N},\mc{E}_{\mc{N}}))
\]
means that for all $u \in C^\infty_{\comp}(\mc{N},\mc{E}_{\mc{N}})$ and $v \in C^\infty_{\comp}(\mc{N},\mc{E}_{\mc{N}}^* \otimes \Omega^1\mc{N})$, the family
\[
\mathbf{s} \mapsto (P_+(\mathbf{s})^{-1}u, v)\in \C
\]
is meromorphic with singularities contained in $\sigma_{\mathrm{RS}}$. The meromorphic extension of \eqref{equation:meromorphic-extension-eu} is an immediate consequence of \eqref{equation:meromorphic-extension-eu0} and \eqref{equation:unif}.

In \emph{\ref{it:mero}}, recall that \eqref{equation:bigres} means that there exists a constant $C > 0$ such that $\s \mapsto P_+(\mathbf{s})^{-1}$ is holomorphic in $\{\Re(s_0) > C(|\Re(\s')|+1)\}$. In \emph{\ref{it:wavefront}}, the uniform inclusion of the wavefront set in \eqref{equation:unif} is understood with respect to $\s\in U$, i.e., for any closed cone $\mathrm{C}' \subset T^*(\mc{N}\times\mc{N}) \setminus \{0\}$ such that $\mathrm{C_{sing}} \cap \mathrm{C}' = \emptyset$, for any $\chi \in C^\infty(\mathcal{N} \times \mathcal{N})$ with sufficiently small support, for all $q > 0$, there exist a constant $C_q > 0$ and an open subset $U' \subset U$ with $\mathbf{s}_\star \in U'$ such that for all $\mathbf{s} \in U'$,
\[
|\widehat{\chi K_{H(\mathbf{s})}}(\mathbf{f}(\xi))|, |\widehat{\chi K_{R(\mathbf{s})}}(\mathbf{f}(\xi))| \leq C_q \langle \xi \rangle^{-q}, \qquad \forall \xi \in \mathrm{C}'.
\]
Here, $\widehat{\bullet}$ is the Fourier transform in $\R^{2d_{\mc N}}$ and $\chi$ is assumed to have sufficiently small support in a coordinate patch of $\mc{N} \times \mc{N}$, so that both the distributions $\chi K_{H(\mathbf{s}),R(\mathbf{s})}$ and the cotangent vectors in $T^*(\mc{N}\times\mc{N})$ can be identified with their Euclidean counterparts. Finally, writing $\xi = (\eta_1,\eta_2) \in \R^{2d_{\N}}$, $\mathbf{f}(\xi) := (\eta_1,-\eta_2)$ is the flip map in the last cotangent variable.

In item \emph{\ref{it:wavefront}\ref{it:singularities}}, if $\s_\star \notin \sigma_{\mathrm{RS}}$, the functions $\mathbf{s}' \mapsto \sigma_i(\mathbf{s}')$ are not defined and $R \equiv 0$; the statement reduces to a claim about the wavefront set of $P_+(\mathbf{s})^{-1}$. The case $N=1$ in \emph{\ref{it:wavefront}\ref{it:singularities}} will play an important role in describing the leading hypersurface, and a simplification occurs in this case. This is stated in the following remark.

\begin{remark}[Simple resonance]
\label{remark:simple}
When $N=1$, \eqref{equation:p-} simplifies using \eqref{equation:cas-cool} to
\begin{equation}
\label{equation:cas-cool2}
P_+(\mathbf{s})^{-1} = \widetilde{H}(\mathbf{s}) - \dfrac{\Pi_{\mathbf{s}'}}{s_0-\sigma(\mathbf{s}')},
\end{equation}
where $\sigma(\mathbf{s}') := \sigma_1(\mathbf{s}')$ is the unique root (depending holomorphically on $\mathbf{s}'$) and
\[
\Pi_{\mathbf{s}'} : C^\infty_{\comp}(\mc{N},\mc{E}_{\mc{N}}) \to \mc{D}'(\mc{N},\mc{E}_{\mc{N}})
\]
is a rank-one operator depending holomorphically on $\mathbf{s}'$, satisfying $\Pi_{\mathbf{s}'}^2 = \Pi_{\mathbf{s}'}$. Notice that $\Pi_{\mathbf{s}'}^2$ is well-defined because of the wavefront-set condition \eqref{equation:unif} and the support property \eqref{equation:precise}. 
\end{remark}

\subsubsection{Proofs}

\label{section:spectrum}

In order to prove Theorem \ref{theorem:rt-anosov2}, we will need the following lemma, which allows us to construct neighborhoods of the trapped set $\scrK$ with good dynamical properties. 

\begin{lemma}
\label{lemma:deja}
There exist two relatively compact neighborhoods $\scrK \Subset \scrV_0 \Subset \scrV'_0$ such that for every $x \in \scrV_0$ and every $T \geq 0$:
\[
\begin{split}
& \phi_T (x) \in \scrV_0 \implies \phi_t(x) \in \scrV'_0, \forall t \in [0,T], \\
& \phi_{-T} (x) \in \scrV_0 \implies \phi_t(x) \in \scrV'_0, \forall t \in [-T,0].
\end{split}
\]
These neighborhoods can be chosen arbitrarily small around $\scrK$.
\end{lemma}

\begin{proof}
We fix an arbitrarily small neighborhood $\scrV'_0$ of $\scrK$. We argue by contradiction and assume that there is no such neighborhood $\scrV_0$ of $\scrK$. We prove only the positive-time claim; the negative-time claim follows by a straightforward adaptation of the argument. 
 Then there is a nested sequence of relatively compact open neighborhoods $(\scrV_n)_{n\geq0}$ with $\overline{\scrV_{n+1}}\subset\scrV_n$ and $\bigcap_{n\geq0}\overline{\scrV_n}=\scrK$. For every $n$, there are $x_n\in\scrV_n$, $T_n\geq0$, and $t_n\in(0,T_n)$ such that $\phi_{T_n}x_n\in\scrV_n$ and $\phi_{t_n}x_n \notin \scrV'_0$. Notice that, by changing $t_n$ if necessary, we can always assume that $y_n := \phi_{t_n} x_n \in K_0 \setminus \scrV'_0$, where $K_0$ is a compact neighborhood of $\scrV'_0$.

First, observe that $t_n \to +\infty$ (hence $T_n \to +\infty$). Indeed, if $(t_n)_{n\geq0}$ is bounded, after passing to a subsequence, $t_n \to t_\infty$ and similarly $x_n \to x_\infty \in \scrK$. Then $y_n = \phi_{t_n} x_n \to \phi_{t_\infty} x_\infty =: y_\infty \in \scrK$, which is absurd since $y_\infty \in K_0 \setminus \scrV'_0$.

Hence $x_n \to x_\infty \in \scrK$ and $t_n \to +\infty$. After passing to a subsequence, we can further assume $y_n = \phi_{t_n} x_n \to y_\infty \in K_0 \setminus \scrV'_0$. Thus $\phi_{-t_n} y_n = x_n \to x_\infty \in \scrK$. Notice that, since the flow is uniformly proper in the past (Assumption \hyperlink{AA2}{\rm(A2)}), we must have $y_\infty \in \mc{T}_+$. Set $s_n:=T_n-t_n$. We then have $\phi_{s_n+t_n}x_n = \phi_{T_n}x_n \in \scrV_n$ and $\phi_{T_n} x_n \to x'_\infty \in \scrK$ after passing to a subsequence. Notice that $s_n > \eps > 0$ because there is a uniform positive lower bound on the time needed to flow from $\scrV_n$ to $K_0 \setminus \scrV'_0$. Let us show that $(s_n)_{n \geq 0}$ is bounded. If not, as the flow is uniformly proper in the future (Assumption \hyperlink{AA2}{\rm(A2)}), we must have $y_n \to y_\infty \in \mc{T}_-$ since $\phi_{s_n} y_n = \phi_{T_n} x_n \to x'_\infty \in \scrK$. Hence $y_\infty \in \mc{T}_- \cap \mc{T}_+ = \scrK$, which contradicts $y_\infty \in K_0 \setminus \scrV'_0$. This implies that $(s_n)_{n \geq 0}$ is bounded, so $s_n \to s_\infty > 0$ after passing to a subsequence. Then $\phi_{s_n}y_n = \phi_{T_n} x_n \to \phi_{s_\infty}(y_\infty) = x'_\infty \in \scrK$. This is again a contradiction because $y_\infty \in K_0 \setminus \scrV'_0$. This concludes the proof.
\end{proof}

\begin{proof}[Proof of Theorem \ref{theorem:rt-anosov2}]
For $|\Re(\s)| \gg 0$ (recall that this means that $\Re(s_0) > C(1+|\Re(\s')|)$ for some constant $C > 0$, see \eqref{equation:bigres}), the families $\s \mapsto P_\pm(\s)^{-1} \in \mc{L}(C^\infty_{\comp}(\mc{N},\mc{E}_{\mc{N}}), \mc{D}'(\mc{N}, \mc{E}_{\mc{N}}))$ are holomorphic with inverses given by the convergent integral
\begin{equation}
\label{equation:explicit-inverse}
P_\pm(\s)^{-1} = - \int_0^{+\infty} e^{t P_\pm(\s)} \dd t.
\end{equation}
Note that the convergence of this integral is straightforward when $\E_{\N} = \C \times \N$ is the trivial line bundle (hence $\E = \C \times \M$ is also trivial); for general admissible bundles, the convergence relies on Lemma \ref{lemma:vilain} , whose hypotheses follow from Lemma \ref{lemma:deja}. From now on, we assume that $\E_\N = \N \times \C$ is the trivial bundle, which entails only the evident notational changes.

Let $\scrK \Subset \scrV_0 \Subset \scrV_0' \Subset \N$ be as in Lemma \ref{lemma:deja}. Let $\scrV_1 \Supset \scrV_0'$ be a slightly larger neighborhood with smooth boundary. Using a standard trick due to Conley-Easton \cite{Conley-Easton-71}, one can modify the vector field $X_{\N}$ near the boundary of $\scrV_1$ so that it becomes strictly convex (see \cite{Dyatlov-Guillarmou-18} or \cite[Section 5.2]{Jin-Tao-25} where this trick is used). That is, there is a vector field $\widetilde{X}_{\N}$ agreeing with $X_{\N}$ except near $\partial\scrV_1$, whose flow $(\widetilde{\phi}_t)_{t\in\R}$ satisfies $x, \widetilde{\phi}_T x \in \scrV_1 \implies \widetilde{\phi}_t x \in \scrV_1$ for all $t \in [0,T]$. We let $\widetilde{P}_\pm(\s)$ be the corresponding operators with $X_{\N}$ replaced by $\widetilde{X}_{\N}$. It then follows from a straightforward adaptation of \cite[Theorem 1]{Dyatlov-Guillarmou-16} or \cite[Theorem 5.14]{Cekic-Guillarmou-Lefeuvre-24} (allowing the potential to depend analytically on an extra parameter $\s' \in \fh^*_{\C}$) that $\widetilde{P}_\pm(\s) : \mc{H}^r_{\comp}(\scrV_1) \to \mc{H}^r(\scrV_1)$ is Fredholm of index $0$ on some scale of anisotropic Sobolev spaces $\mc{H}^r(\scrV_1)$ on $\{\Re(s_0) > C_0 + C_1 |\Re(\s')| - C_2 r\} \subset \fa^*_{\C}$ for some constants $C_i > 0, i=0,1,2$. 
Using Theorem \ref{theorem:fredholm-several}, we deduce that $\widetilde{P}_\pm(\s)^{-1} : C^\infty_{\comp}(\scrV_1) \to \mc{D}'(\scrV_1)$ admit a meromorphic extension to $\fa^*_{\C}$.

The properties of the nested sets $\scrV_0 \Subset \scrV_0' \Subset \scrV_1$ (Lemma \ref{lemma:deja}) and \eqref{equation:explicit-inverse} show that $\widetilde{P}_\pm(\s)^{-1} = P_\pm(\s)^{-1}$ 
as operators in $\mc{L}(C^\infty_{\comp}(\scrV_0), \mc{D}'(\scrV_0))$ for $|\Re(\s)| \gg 0$ (this simply follows from the fact that $(\phi_t)_{t \in \R}$ and $(\widetilde{\phi}_t)_{t \in \R}$ coincide in $\scrV_0'$, as well as the integral expression \eqref{equation:explicit-inverse}). 
Therefore, the unmodified operator $P_\pm(\s)^{-1} \in \mc{L}(C^\infty_{\comp}(\scrV_0), \mc{D}'(\scrV_0))$ also admits a meromorphic extension to $\fa^*_{\C}$.

We now claim that this implies that the families $P_\pm(\s)^{-1} : C^\infty_{\comp}(\N) \to \mc{D}'(\N)$ admit a meromorphic extension to $\fa^*_{\C}$ as operators on the whole manifold $\N$. 
Consider first $f \in C^\infty_{\comp}(\scrV_0)$ and set $u(\s) := P_+^{-1}(\s)f$ (which is well-defined for $|\Re(\s)| \gg 0$). 
Let $\psi \in C^\infty_{\comp}(\N, \Omega^1\N)$. We need to show that $\s \mapsto (u(\s),\psi)$ can be meromorphically extended. Note that if $\psi$ is supported outside of the closure of the forward flowout of the support of $f$, we have $(u(\s),\psi)=0$ for all $\s \in \fa^*_{\C}$. Therefore, we can assume that $\psi$ is supported near a point in the forward flowout of the support of $f$. As a consequence, there exists $T > 0$ such that $e^{TP_-(\s)} \psi$ has compact support in $\scrV_0$. Using \eqref{equation:explicit-inverse}, we obtain, for $|\Re(\s)| \gg 0$:
\[
(u(\s),\psi) = (u(\s), e^{TP_-(\s)}\psi) - \int_0^T (e^{tP_+(\s)} f, \psi) \dd t.
\]
The second term is holomorphic with respect to $\s$, while the first term admits a meromorphic extension to $\fa^*_{\C}$ by the previous paragraph.

We now fix an arbitrary compact set $K \subset \mc{N}$. Let $f \in C^\infty(\mc{N})$ with compact support in $K$. Let $\scrV_K$ be a small open neighborhood of $K\cap\mc{T}_-$ and there exists a time $T > 0$ such that $\phi_T(\scrV_K) \subset \scrV_0$. Let $\chi \in C^\infty_{\comp}(\mc{N})$ be a cutoff function with support in $\scrV_K$ and equal to $1$ on a neighborhood of $K\cap\mc{T}_-$. We write
\[
P_+^{-1}(\s) f = P_+^{-1}(\s)(\chi f) + P_+^{-1}(\s)((1-\chi)f).
\]
Notice that $\s \mapsto P_+^{-1}(\s)((1-\chi)f) \in C^\infty(\N)$ is holomorphic since the flow is uniformly proper in the future (Assumption $\hyperlink{AA2}{\rm(A2)}$); the integral \eqref{equation:explicit-inverse} defining $P_+^{-1}(\s)$ therefore ranges over only a finite time interval. To treat the first term, we write
\[
P_+^{-1}(\s) = - \int_0^{+\infty} e^{tP_+(\s)} \dd t = - \int_0^T e^{tP_+(\s)} \dd t + P_+^{-1}(\s) e^{TP_+(\s)}.
\]
This yields:
\[
P_+^{-1}(\s)(\chi f) = - \int_0^T e^{tP_+(\s)} (\chi f)\dd t + P_+^{-1}(\s) (e^{TP_+(\s)} \chi f).
\]
The first term is clearly holomorphic with respect to $\s \in \fa^*_{\C}$. For the second term, notice that $e^{TP_+(\s)} \chi f$ has support in $\scrV_0$ by construction. By the previous paragraph, it admits a meromorphic extension to $\fa^*_{\C}$. This concludes the proof of \eqref{equation:meromorphic-extension-eu0}.

The claims on the wavefront set (Theorem \ref{theorem:rt-anosov2}, item \emph{\ref{it:wavefront}}) can be deduced by straightforward adaptations of, for instance, analogous statements in \cite[Theorem 2]{Dyatlov-Guillarmou-16}. Finally, \eqref{equation:meromorphic-extension-eu} follows from \eqref{equation:meromorphic-extension-eu0} and \eqref{equation:unif}.

\end{proof}

We denote by $\sigma \subset \fa^*_{\C}$ the set of singularities of $P_+(\s)^{-1} : C^\infty_{\comp}(\N) \to \mc{D}'(\N)$. Note that $\sigma$ is a complex variety of codimension $1$. It also follows from the proof of Theorem \ref{theorem:rt-anosov2} that $\sigma$ coincides with the poles of $P_+(\s)^{-1} : C^\infty_{\comp}(\scrV_0) \to \mc{D}'(\scrV_0)$.

\begin{lemma}\label{lemma:sigma_RT_equals_sigma}
 We have $\sigma= \sigma_{\mathrm{RS}}$. In particular, $\sigma_{\mathrm{RS}}\subset \mathfrak a_\C^*$ is a complex variety of codimension one which coincides with the poles of
\[
\mathfrak a_\C^*\ni \mathbf{s} \mapsto P_+(\mathbf{s})^{-1} \in \mc{L}(C^\infty_{\comp}(\N),\mc{D}'(\N)).
\]
\end{lemma}

\begin{proof}
Let $\mathbf{s}_\star = (s_{0\star},\mathbf s_\star') \in \sigma_{\mathrm{RS}}$; by definition, there exists $u \in \mc{D}'(\M), u \neq 0$ such that $\supp(u) \subset \mc{T}_{+,\M}$, $\WF(u) \subset E^*_{u,\M}$, and $(-\X-\mathbf{s}_\star)u = 0$. We claim that $u = \mathbf{e}(\mathbf{s}_\star')f$ for some $f \in \mc{D}'(\N)$, $f\neq 0$, $\supp(f) \subset \mc{T}_{+,\mathcal{N}}$, $\WF(f) \subset E^*_{u,\mathcal{N}}$. Indeed, $u$ can be restricted to the zero section $\mc{N} \times \{0\} \subset \M$; this follows from the fact that the conormal bundle $N^*(\mc{N} \times \{0\})$ is transverse to $E_{u,\M}^*$, so the restriction is well-defined (see \cite[Lemma 4.3.2]{Lefeuvre-book}).
 Since the footpoint projection $\pi : \mc{N} \times \{0\} \to \mc{N}$ is a diffeomorphism, we can write $u|_{\mc{N} \times \{0\}} = \pi^*f$ for some $f \in \mc{D}'(\N)$.
Using \eqref{equation:relations-e}, it is then immediate to check that $u = \mathbf{e}(\mathbf{s}_\star')f$ and $f$ satisfies the required properties. In addition, using the second equation of \eqref{equation:relations-e}, we find that $P_+(\mathbf{s}_\star)f = 0$. This implies that $f$ is a resonant state of $P_+(\s)$ (see, for instance, \cite[Theorem 2]{Dyatlov-Guillarmou-16}) and consequently $(P_+(\s))^{-1}$ has a pole at $\s_\star$.

Conversely, if $\mathbf{s}_\star \in \sigma$, there exists $f \in \mc{D}'(\N), f\neq 0$, with $\supp(f) \subset \mc{T}_{+,\mc{N}}$, $\WF(f) \subset E^*_{u,\mc{N}}$ such that $P_+(\mathbf{s}_\star)f = 0$. 
Setting $u := \mathbf{e}(\mathbf{s}_\star')f$, we have $(-\X-\mathbf{s}_\star)u = 0$ and $u$ satisfies the required support condition on $\M$ by construction. 
Also, note that $u$ is a tensor product distribution on $\M \simeq \N \times \fh$; its wavefront set is thus given in the trivialization $T^*\M \simeq T^*\N \oplus \fh^*$ by $\{(\xi,0) ~:~ \xi \in E_{u,\N}^*\}$ (see \cite[Lemma 4.2.7]{Lefeuvre-book}), which coincides with $E_{u,\M}^*$ by \eqref{equation:esm}. This concludes the proof.
\end{proof}

\begin{proof}[Proof of Theorem \ref{theorem:rt-anosov1}] This is an immediate consequence of Lemma \ref{lemma:sigma_RT_equals_sigma}.
\end{proof}

\subsection{Laplace transforms and product resolvents}
\label{section:laplace_transform}

The following result will be key to the proof of the meromorphic extension of Poincaré series (Theorem \ref{theorem:intro-poincare}).

\subsubsection{Meromorphic continuation of the Laplace transform of cones} We recall that $X_{\M} \in \fa$ is transverse to $\mathfrak h$. Let
\[
\fa_{\geq 0} := \R_+X_\M + \fh, \qquad \fa_{> 0} := \R_+^*X_\M + \fh.
\]
Let $\scrC\subset \fa_{> 0}$ be an open convex cone such that $\overline{\scrC} \subset \fa_{> 0}$ and $\overline{\scrL} \subset \scrC$. For $f\in C^\infty_{\comp}(\M, \mc{E})$, $\mathbf s\in\mathfrak a_\C^*$, we set
\begin{equation}\label{eq:Laplace-transform-dynamics}
T^\scrC_\pm(\mathbf{s})f = \int_{\scrC} e^{(\mp\X_{\mc{E}}-\s)(a)} f ~ \dd \vol_{\fa}(a).
\end{equation}
Using Lemma \ref{lemma:vilain}, this is readily seen to be well-defined for $|\Re(\s)| \gg 0$ (see \eqref{equation:bigres} for the definition of this notation) as $|e^{-\s(a)}| \leq e^{-\eps |\Re(\s)| |a|}$ for some $\eps > 0$ if $\Re(\s) \in \mathrm{int}(\scrC^{*})$ (the interior of the dual cone of $\scrC$). 
That is, there exists $C > 0$ such that for all $\s' \in \fh^*_{\C}$, \eqref{eq:Laplace-transform-dynamics} converges for $\Re(s_0) > C(1+|\Re(\s')|)$. 
As we shall see, the condition $\overline{\scrC} \subset \fa_{> 0}$ can actually be dropped when we consider compactly supported distributions. 

 For the next statement, we use the temporary notation $E_+^* := E_{u,\M}^*, E_-^*:=E_{s,\M}^*$. We also recall that $E_{s,\M}^* \subset T^*\M$ (resp. $E_{u,\M}^*$) is the annihilator of $E_{s,\M} \oplus E_{0,\M} \subset T\M$ (resp. $E_{u,\M} \oplus E_{0,\M}$), see \S\ref{sssection:phf}.

\begin{theorem}[Laplace transform of cones] \label{theorem:product-resolvents}
Let $\scrC \subset \fa_{\geq 0}$ be an open convex cone such that $\overline{\mathscr{L}}\subset \scrC$. Then:
\[
T^\scrC_\pm(\mathbf{s}) : \mc{D}'_{\comp,E_{\pm,\M}^*}(\M,\E) \to \mc{D}'_{E_{\pm,\M}^*}(\M, \mc{E})
\]
admits a meromorphic extension to $\fa^*_{\C}$. In addition,
\begin{equation}
\label{equation:wftc}
\begin{split}
\WF'(T^\scrC_\pm(\s))& \subset \{((m,\xi);e^{\mp a}(m,\xi)) ~:~ a\in \partial \scrC, (m,\xi) \in T^*\M \setminus \{0\}\} \\
& \qquad \cup\{ ((m,\xi); e^{\mp a}(m,\xi)) ~:~ a\in\scrC, m \in \M, \xi\in E_{-,\M}^* \oplus E_{+,\M}^*\} \\
							&\qquad \cup E^*_{\pm,\M}|_{\mc{T}_{\pm, \M}} \times E^*_{\mp,\M}|_{\mc{T}_{\mp, \M}}. 
\end{split}
\end{equation}
The Schwartz kernel of $T^\scrC_\pm(\s)$ has a meromorphic continuation to $\fa^*_{\C}$ as a distribution with wavefront set contained in the right-hand side of \eqref{equation:wftc}.
\end{theorem}

The most relevant application of Theorem \ref{theorem:product-resolvents} is when $\scrC$ is an orthant, namely
\[
\mathscr{C} = \left\{ \textstyle \sum_{i=0}^k x_i \e_i ~:~ x_i \geq 0 \right\},
\]
where $(\e_0, ..., \e_k)$ is a basis of $\fa$. In this case, writing $\mathbf{s}= \sum_{i=0}^k s_i \e_i^*$, where $(\e_0^*, ..., \e_k^*)$ denotes the dual basis, and $X_i := \X_{\mc{E}}(\e_i)$, one has $T^\scrC_-(\mathbf{s}) = (s_0 - X_0)^{-1} \dots (s_k - X_k)^{-1}$. As in the preceding sections, the proof carries over \emph{verbatim} to vector-valued distributions, so we only treat the case $\E = \M \times \C$.

\subsubsection{Non-trapped directions}
\label{sec:7-prelim}

We begin by studying Laplace transforms over cones disjoint from the limit cone $\scrL$. For simplicity, we only treat the case of $T_-^{\scrC}(\s)$, the case of $T_+^{\scrC}(\s)$ being completely analogous.	

\begin{proposition}\label{prop:Laplace-outside-L}
Let $\scrC' \subset \fa$ be an open cone such that
\[
\overline{\scrC'} \cap \left(\mathscr{L} \cup -\scrL\right) = \{0\}.
\]
Then:
\begin{enumerate}[label=\emph{(\roman*)}]
\item $T^{\scrC'}_-$ admits a \emph{holomorphic} continuation to $\fa^\ast_\C$ as an operator 
\[
T^{\scrC'}_-(\mathbf{s}) : C^\infty_{\comp}(\M) \to C^\infty(\M),\quad \mathcal{D}'_{\comp}(\M) \to \mathcal{D}'(\M).
\]
\item The wavefront set of the Schwartz kernel of $T_-^{\scrC'}(\s)$ is given by
\begin{equation}
\label{equation:gaga}
\begin{split}
\WF'(T^{\scrC'}_-(\s)) \subset & \{((m,\xi);e^a(m,\xi)) ~:~ a\in \partial \scrC', (m,\xi) \in T^*\M \setminus \{0\}\} \\
& \cup\{ ((m,\xi); e^{a}(m,\xi)) ~:~ a\in\scrC', m \in \M, \xi\in E_{s,\M}^* \oplus E_{u,\M}^*\setminus \{0\}\}.
\end{split}
\end{equation}
\end{enumerate}
\end{proposition}

The bound on the wavefront set at the boundary $\partial \scrC'$ in (ii) is not always sharp and can sometimes be improved under additional assumptions on $\partial \scrC'$ (for instance, if $\scrC'$ admits a stratification). 
In particular, it follows from (i) that the choice of cone $\scrC$ in the definition of $T_-^\scrC(\s)$ (see \eqref{eq:Laplace-transform-dynamics}) is irrelevant to the singularities of $T_-^{\scrC}(\s)$ provided it contains the closure of the limit cone $\overline{\scrL}$. That is, if $\scrC_1 \subset \fa_{\geq 0}$ is another cone containing $\overline{\scrL}$, then $T_-^\scrC(\s)-T_-^{\scrC_1}(\s)$ is holomorphic on $\fa^*_{\C}$. In addition, it follows from the second item of the proposition above that, if $\mathrm{C}$ is a closed cone in $T^\ast \M\setminus\{0\}$ not intersecting $E^\ast_{u,\M}$, and $U \subset \M$ is a relatively compact open subset of $\M$, then there exists $\mathrm{C}'\supset \mathrm{C}$ not intersecting $E^\ast_{u,\M}$ such that $\s \mapsto T_-^{\scrC'}(\mathbf{s})$ extends holomorphically as a continuous operator from $\mathcal{D}'_{\mathrm{comp},\mathrm{C}}(\M)$ to $\mathcal{D}'_{\mathrm{C}'}(U)$. If $\mathrm{C}$ is a narrow cone around $E^\ast_{s,\M}$, then $\mathrm{C}'$ can also be chosen to be a narrow cone around $E^\ast_{s,\M}$.

The proof of Proposition \ref{prop:Laplace-outside-L} relies on the following uniform nontrapping statement:

\begin{lemma}
\label{lemma:onsebarre}
Let $K\subset \M$ be a compact subset and $\scrC' \subset \fa$ as in Proposition \ref{prop:Laplace-outside-L}. Then there exists $T>0$ such that for all $a\in\scrC'$ with $|a|> T$, $e^a(K)\cap K = \emptyset$. 
\end{lemma}

\begin{proof}
 Suppose, for contradiction, that no such $T > 0$ exists. We can find sequences $(a_n)_{n \geq 0}$ with $a_n \in \scrC'$, $|a_n| \to +\infty$ and $(m_n)_{n \geq 0}$, $m_n \in K$ such that $e^{a_n}m_n \in K$. Write $m_n = (x_n,v_n)$ and $a_n=(t_n,h_n)$. Observe that $(t_n)_{n \geq 0}$ is unbounded because $\fh$ acts freely and properly on $\M$ by assumption. After passing to a subsequence, either $t_n\to+\infty$ or $t_n\to-\infty$; we treat the former case, while the latter is analogous and yields an accumulation point in $-\scrL$. After passing to a subsequence, we may also assume that $m_n \to m_\infty = (x_\infty,v_\infty)$ as well as $e^{a_n}m_n \to m_\infty' = (x_\infty',v_\infty')$. 
By uniform properness (Assumption $\hyperlink{AA2}{\rm(A2)}$), we have $x_\infty \in \mc{T}_-$ and $x_\infty' \in \mc{T}_+$. Let $\scrV \subset \mc{N}$ be the neighborhood of $\scrK$ from the nonuniform trapping assumption. Since $x_\infty \in \mc{T}_-$ and $x_\infty'\in \mc{T}_+$, we can decompose $t_n = t_n^- + \tau_n + t_n^+$, where $0\leq t_n^\pm\leq C$ are bounded by a constant $C > 0$ independent of $n \geq 0$, $\tau_n \to +\infty$, and $\phi_{t_n^- + \tau}(x_n) \in \scrV$ for all $\tau \in [0,\tau_n]$.

By assumption,
\[
e^{a_n}m_n = (\phi_{t_n}x_n, v_n + h_n + \int_0^{t_n} w(\phi_s x_n) \dd s) \in K
\]
is bounded. As $\scrC'$ is a cone and $(t_n,h_n) \in \scrC'$, $(1,h_n/t_n) \in \scrC'$. Setting $y_n := \phi_{t_n^-}x_n \in \scrV$, the previous relation gives
\[
h_n/t_n = -t_n^{-1} \int_0^{t_n} w(\phi_s x_n) \dd s + o(1) = - \tau_n^{-1} \int_0^{\tau_n} w (\phi_s(y_n)) \dd s + o(1),
\]
that is,
\begin{equation}
\label{equation:mouche}
\left(1,-\tau_n^{-1} \int_0^{\tau_n} w (\phi_s(y_n)) \dd s + o(1)\right) \in \scrC'.
\end{equation}
However, by \eqref{equation:suzanne} in Lemma \ref{lemma:formula-for-L}, item (i), the accumulation points (as $n \to +\infty$) of \eqref{equation:mouche} belong to the limit cone $\scrL$. This contradicts the assumption $\overline{\scrC'} \cap \left(\mathscr{L} \cup -\scrL\right) = \{0\}$.
\end{proof}

We can now prove Proposition \ref{prop:Laplace-outside-L}.
\begin{proof}
(i) To prove the first part of the statement, it suffices to fix arbitrary compact sets $K_1, K_2 \subset \M$ and prove that $T_-^{\scrC'}(\s)$ maps smooth functions compactly supported in $K_1$ to smooth functions on $K_2$ (not necessarily compactly supported), and that this restriction extends to a holomorphic family on $\fa^\ast_\C$.
By enlarging these sets, we may assume $K=K_1=K_2$. By duality, it suffices to show that the map $C^\infty_{\comp}(K)\to C^\infty(K)$ extends holomorphically.
By Lemma \ref{lemma:onsebarre}, for $\varphi\in C^\infty_{\comp}(K)$, the restriction $T_-^{\scrC'}(\s)\varphi_{| K}$ is defined by an integral over a bounded set in $\fa$.
Elementary techniques of differentiation under the integral sign then apply. \\

(ii) The proof follows from standard wavefront-set analysis and compactness of the support. Indeed, the wavefront set of the Schwartz kernel of $(e^{a})^*$ is given by $\WF'((e^{a})^*)=\{(m,\xi;e^a(m,\xi)) ~:~ m \in \M, \xi \in T^*\M \setminus \{0\}\}$. Applying \cite[Exercise 6.2.10]{Lefeuvre-book} with $\chi = \mathbf{1}_{\scrC'}$, the characteristic function of $\scrC'$, we find the second term of the wavefront set in \eqref{equation:gaga}; the first term comes from the wavefront set of the step function $\chi$ itself.
\end{proof}

\subsubsection{Laplace transform of the half space} We begin by considering the particular case in which $\scrC := \fa_{\geq 0} = \R_+X_\M + \fh$, i.e., we study the operator $T_{-}^{\fa_{\geq 0}}(\s)$. Recall that $\pi : \M \to \N$ denotes the footpoint projection $\pi(x,h)=x$. For $\s \in \fa^*_{\C}$, we introduce the (weighted) pushforward map $\pi_*(\s) : C^\infty_{\comp}(\M) \to C^\infty_{\comp}(\N)$:
\[
\pi_*(\s) f(x) := \int_{\fh} e^{-\s'(h)} f(x,h) \dd h.
\]
This operator also extends to compactly supported distributions.

\begin{lemma}
Let $f \in C^\infty_{\comp}(\M)$ and $\s = (s_0,\s') \in \fa^*_\C$ such that $|\Re(\s)| \gg 0$. Then
\begin{equation}
\label{equation:formula-sympa}
T_{-}^{\fa_{\geq 0}}(\s)f(x,h) = - e^{\s'(h)}(P_-^{-1}(\s) \pi_*(\s)f)(x).
\end{equation}
\end{lemma}

In particular, the Schwartz kernel of $T_{-}^{\fa_{\geq 0}}(\s)$ is given by $T_{-}^{\fa_{\geq 0}}(\s ; x,x',h,h') = - e^{\s'(h-h')} P_-(\s)^{-1}(x,x')$, and therefore it extends meromorphically to $\fa^*_{\C}$ by Theorem \ref{theorem:rt-anosov2} (see the comment after Theorem \ref{theorem:rt-anosov2} for the relevant statement on $P_-(\s)^{-1}$).

\begin{proof}
Observe that $T_{-}^{\fa_{\geq 0}}(\s)$ is invariant under the $H$-action. Since the $H$-action is a free and proper $\R^k$-action, $T_{-}^{\fa_{\geq 0}}(\s)$ is a convolution operator in the $h$-variable. That is, we may write
\[
T_{-}^{\fa_{\geq 0}}(\s)= \int_{\fh} K(\s)(h-h') \dd h',
\]
where $K(\s)(h) : C^\infty_{\comp}(\N) \to \mc{D}'(\N)$ is operator-valued and acts on functions of $\N$. We claim that
\begin{equation}
\label{equation:aspirateur}
K(\s)(h-h') = - e^{\s'(h-h')} P_-^{-1}(\s),
\end{equation}
which will prove the claim.

To compute $K(\s)(h)$, it suffices to consider $f \in C^\infty_{\comp}(\M)$ of the form $f(x,h) = u(x)\otimes v(h)$, where $u \in C^\infty_{\comp}(\N)$ and $v \in C^\infty_{\comp}(\fh)$; it suffices to consider such functions because tensor-product functions are dense.
Writing $a=tX_\M+h'$, we find:
\[
\begin{split}
T_{-}^{\fa_{\geq 0}}(\s) f (x,h) &= \int_{\fa_{\geq 0}} e^{(\X-\s)(a)} f(x,h) \dd a \\
				&= \int_0^{+\infty} \int_{\fh} e^{-ts_0 - \s'(h')} u(\phi_t(x)) v\Big( h + h' + \int_0^t w\circ\phi_r(x) \dd r \Big) \dd h' \dd t \\
				&=e^{\s'(h)} \left(\int_0^{+\infty} u(\phi_t(x)) e^{-s_0 t} e^{\s' (\int_0^t w\circ \phi_r(x) \dd r)} \dd t\right)	\left(\int_{\fh} e^{-\s'(h')} v(h')\dd h'\right) \\
				& = -e^{\s'(h)} [P_-^{-1}(\s)u](x) \left(\int_{\fh} e^{-\s'(h')} v(h')\dd h'\right)  \\
				&= - e^{\s'(h)} P_-^{-1}(\s) \left(\int_{\fh} e^{-\s'(h')} v(h')u(\bullet) \dd h'\right) = -e^{\s'(h)}(P_-^{-1}(\s) \pi_*(\s)f)(x).
				\end{split}
				\]
				This proves the claim.
\end{proof}

\begin{lemma}
\label{lemma:froid2}
The family $T_{-}^{\fa_{\geq 0}}(\s)$ extends meromorphically to $\fa^*_{\C}$ as a family of bounded operators
\[
T_{-}^{\fa_{\geq 0}}(\s) : \mc{D}'_{\comp,E_{s,\M}^*}(\M) \to \mc{D}'_{E_{s,\M}^*}(\M).
\]
\end{lemma}

\begin{proof}
Recall that $\dd \pi^\top : E_{s,\N}^* \to E_{s,\M}^*$ is an isomorphism (see \S\ref{sssection:phf}). Standard wavefront-set calculus (see \cite[Lemma 4.3.4]{Lefeuvre-book}) implies that, for $f \in \mc{D}'_{\comp,E_{s,\M}^*}(\M)$, $\pi_*(\s) f \in \mc{D}'_{\comp,E_{s,\N}^*}(\N)$. Therefore,
\[
\s \mapsto u(\s) := -P_-(\s)^{-1} \pi_*(\s) f \in \mc{D}'_{E_{s,\N}^*}(\N)
\]
admits a meromorphic extension to $\fa^*_{\C}$ by Theorem \ref{theorem:rt-anosov2}.

The distribution $T_{-}^{\fa_{\geq 0}}(\s)f = e^{\s'(h)}u(\s)$ on $\M$ is a tensor product distribution; its wavefront set is thus given in the trivialization $T^*\M \simeq T^*\N \oplus \fh^*$ by $\{(\xi,0) ~:~ \xi \in E_{s,\N}^*\}$ (see \cite[Lemma 4.2.7]{Lefeuvre-book}), which coincides with $E_{s,\M}^*$ by \eqref{equation:esm}. Thus $\s \mapsto T_{-}^{\fa_{\geq 0}}(\s)f \in \mc{D}'_{E_{s,\M}^*}(\M)$ depends meromorphically on $\s \in \fa^*_{\C}$.
\end{proof}

\subsubsection{Completion of the proof} Finally, we prove the meromorphic extension of $\s \mapsto T_-^{\scrC}(\s)$. We begin with the following lemma.

\begin{lemma}
\label{lemma:xxx}
Let $\scrC\subset \fa_{>0}$ be a open cone containing $\scrL$. There exists a family of bounded operators
\[
R_{\scrC}(\s) : C^\infty_{\comp}(\M) \to C^\infty(\M), \qquad \mc{D}'_{\comp,E_s^*}(\M) \to \mc{D}'_{E_s^*}(\M),
\]
holomorphic with respect to $\s \in \fa^*_{\C}$, such that $T_-^{\scrC}(\s) = T_{-}^{\fa_{\geq 0}}(\s)-R_{\scrC}(\s)$.
\end{lemma}

In particular, we obtain the formula
\begin{equation}
\label{equation:nice-formula-t}
T_-^{\scrC}(\s)f(x,h) = - e^{\s'(h)} P_-^{-1}(\s) \pi_*(\s)f(x) + \mathrm{Hol}(\s),
\end{equation}
where the remainder $\mathrm{Hol}(\s)$ is holomorphic with respect to $\s \in \fa^*_{\C}$ and takes values in $\mc{D}'_{E_s^*}(\M)$.

\begin{proof}
 Define
\[
R_{\scrC}(\s) := \int_{\fa_{\geq 0} \setminus \scrC} e^{(\X-\s)(a)} \dd a
\]
 Then $T_-^{\scrC}(\s) = T_{-}^{\fa_{\geq 0}}(\s)-R_{\scrC}(\s)$, and the claim follows immediately from Proposition \ref{prop:Laplace-outside-L}.
\end{proof}

\begin{proof}[Proof of Theorem \ref{theorem:product-resolvents}]
 The meromorphic extension of
\[
T_-^\scrC(\mathbf{s}) : \mc{D}'_{\comp,E_{s,\M}^*}(\M,\E) \to \mc{D}'_{E_{s,\M}^*}(\M, \mc{E})
\]
to $\fa^*_{\C}$ follows immediately from Lemmas \ref{lemma:froid2} and \ref{lemma:xxx}.

We now prove the claim on the wavefront set. We first deal with $T_{-}^{\fa_{\geq 0}}(\s)$.
 Combining the expression \eqref{equation:formula-sympa} for $T_{-}^{\fa_{\geq 0}}(\s)$ with \eqref{equation:unif} (the bound on the wavefront set of the Schwartz kernel of $P_-(\s)^{-1}$) and the standard composition rules for Schwartz-kernel wavefront sets, we obtain
\[
\begin{split}
\WF'(T_{-}^{\fa_{\geq 0}}(\s)) &\subset \{((m,\xi);e^h(m,\xi)) ~:~ h \in \fh, (m,\xi) \in T^*\M \setminus \{0\}\} \\
& \qquad \cup\{ ((m,\xi); e^{a}(m,\xi)) ~:~ a\in \fa_{\geq 0}, m \in \M, \xi\in E_{s,\M}^* \oplus E_{u,\M}^*\} \\
							&\qquad \cup E^*_{s,\M}|_{\mc{T}_-} \times E^*_{u,\M}|_{\mc{T}_+}. 
							\end{split}
\]
To prove the bound on the wavefront set of the Schwartz kernel of $T_-^{\scrC}(\s)$, 
 we use the identity $T_-^{\scrC}(\s) = T_{-}^{\fa_{\geq 0}}(\s)-R_{\scrC}(\s)$ (Lemma \ref{lemma:xxx}).
 Combining the previous bound on $\WF'(T_{-}^{\fa_{\geq 0}}(\s))$ with Proposition \ref{prop:Laplace-outside-L}, item (ii), we find that
\begin{equation}
\label{equation:lardons}
\begin{split}
\WF'(T^{\scrC}_-(\s)) &\subset \{((m,\xi);e^a(m,\xi)) ~:~ a \in \fh \cup \partial \scrC, (m,\xi) \in T^*\M \setminus \{0\}\} \\
& \qquad \cup\{ ((m,\xi); e^{a}(m,\xi)) ~:~ a\in \fa_{\geq 0}, m \in \M, \xi\in E_{s,\M}^* \oplus E_{u,\M}^*\} \\
							&\qquad \cup E^*_{s,\M}|_{\mc{T}_-} \times E^*_{u,\M}|_{\mc{T}_+}. 
							\end{split}
\end{equation}
However, the Schwartz kernel of $T_-^{\scrC}(\s)$ is supported in the closure of $\{(m, e^a m) ~:~ m \in \M, a \in \scrC\}$.
Consequently, the support condition removes from \eqref{equation:lardons} all parameters $a$ such that $a \in \fa_{\geq 0} \setminus \overline{\scrC}$. This proves the claim.
\end{proof}

\subsection{Dynamical determinants and zeta functions}

\label{section:general-zeta-function-result}

We now introduce and discuss the notion of dynamical determinants for Axiom A actions of type $(k,1)$.

\subsubsection{Holonomies along closed orbits} \label{sssection:biberon} Recall that $\mathbf{P}$ (respectively, $\mathbf{P}^\sharp$) denotes the set of periodic (respectively, primitive periodic) orbits of $(\phi_t)_{t \in \R}$. Given $\gamma \in \mathbf{P}$ and $x \in \gamma$, we define the Poincaré return map
\[
P_\gamma \in \mathrm{GL}(E_{s,\N}(x)\oplus E_{u,\N}(x)), \qquad P_\gamma := \dd\phi_{\ell_\gamma},
\]
where $\ell_\gamma$ denotes the period of $\gamma$. Note that the conjugacy class of $P_\gamma$ is independent of the choice of $x \in \gamma$. We also introduce
\[
E_\gamma : \E_{\N}(x) \to \E_{\N}(x),
\]
the propagator of $\X_{\E_\N}$ along $\gamma$ over time $\ell_\gamma$ (see \eqref{equation:propagator-def}). Again, the conjugacy class of $E_\gamma$ does not depend on the choice of $x \in \gamma$.

The maps $P_\gamma$ and $E_\gamma$ can be interpreted intrinsically in terms of the Axiom A action of type $(k,1)$. Recall that $\pi : \M \to \mc{N}$ is the footpoint projection, $\dd \pi : E_{s/u,\M} \to E_{s/u,\mc{N}}$ is an isomorphism, and $\M \simeq \mc{N} \times \fh$. For $\gamma \in \mathbf{P}$, $x \in \gamma$, choose any point $z \in \pi^{-1}\{x\}$ in the preimage. Then $\tau_{e^{\lambda(\gamma)}}z = z$ by construction, where $\tau : A\times\M \to \M$ is the action. In addition,
\[
\dd(\tau_{e^{\lambda(\gamma)}})_z \in \mathrm{GL}(E_{s,\M}(z) \oplus E_{u,\M}(z))
\]
is conjugate to $P_\gamma$ via $\dd\pi$, and the holonomy of $\E \to \M$ (with respect to the generator $\X_{\mc{E}}$, as introduced in \S\ref{sssection:admissible}) along the loop $[0,1] \ni s \mapsto \tau_{e^{s\lambda(\gamma)}}z$ equals $E_\gamma$ by construction. Note that all these quantities are independent of the choice of decomposition $\fa = \R X_\M \oplus \fh$.

\subsubsection{Dynamical determinant and zeta function}

We define the following function, called the \emph{dynamical determinant}:
\begin{equation}\label{eq:def-general-zeta}
\zeta_\E(\s) = \exp\left( - \sum_{\gamma \in\mathbf{P}} e^{- \s(\lambda(\gamma))} \frac{ \ell_\gamma^\sharp \Tr(E_\gamma)}{\ell_\gamma |\det(1-P_\gamma)|} \right), \qquad \s \in \fa^*_{\C}.
\end{equation}
It follows from \S\ref{sssection:biberon} that \eqref{eq:def-general-zeta} is intrinsic to the $A$-action on $\M$, i.e., independent of the decomposition $\fa = \R X_\M \oplus \fh$.

We obtain the following result:
\begin{theorem}[Holomorphic extension of dynamical determinants to $\fa^*_{\C}$] \label{theorem:general-zeta}
Assume that $\hyperlink{AA1}{\rm(A1)-(A3)}$ hold. Then:
\begin{enumerate}[label=\emph{(\roman*)}]
\item The expression \eqref{eq:def-general-zeta} converges and is holomorphic for $|\Re \s| \gg 0$.
\item The function $\zeta_\E$ admits a holomorphic continuation to $\fa^\ast_\C$ and its zero set coincides with $\sigma_{\mathrm{RS},+}(\E)$.

\end{enumerate}
\end{theorem}

We now record a consequence of the preceding result.
The dynamical Ruelle zeta function associated with the Axiom A action of type $(k,1)$ is defined by the following product over all primitive closed orbits: 
\begin{equation}
\label{equation:zeta-anosov}
\zeta(\mathbf{s}) := \prod_{\gamma \in \mathbf{P}^\sharp} \left(1-e^{-\mathbf{s}(\lambda(\gamma))}\right)^{-1}, \qquad \mathbf{s} \in \fa^*_{\C}.
\end{equation}

The product converges for $|\Re(\s)| \gg 0$ (recall \eqref{equation:bigres}). An immediate corollary of Theorem \ref{theorem:general-zeta} is the following:

\begin{corollary}[Meromorphic extension of the Ruelle zeta function]
\label{corollary:zeta-ruelle}
Assume that $\hyperlink{AA1}{\rm(A1)-(A3)}$ hold. 
Then the Ruelle zeta function \eqref{equation:zeta-anosov} admits a meromorphic extension to $\fa^*_\C$.
\end{corollary}

When $k=0$ (rank-one case), one recovers the usual dynamical zeta function; its meromorphic extension was first established in \cite{GLP-13} and reproved using microlocal methods in \cite{Dyatlov-Zworski-16} when $E_s$ is orientable (see \cite{Dyatlov-Guillarmou-16,Dyatlov-Guillarmou-18} for the Axiom A case), while the general case was treated in \cite{BornsWeil-Shen-2020}. 

\subsubsection{Proofs}

\begin{proof}[Proof of Theorem \ref{theorem:general-zeta}] Taking the logarithmic derivative $\ell_{\E}(\s)$ of \eqref{eq:def-general-zeta}, we observe that the resulting infinite series has the form considered in \cite[Theorem 4]{Dyatlov-Guillarmou-16}, with the corresponding lifted operator $\mathbf{X} := \mathbf{X}_{\mc{N}} - \s'(w)$. Then, for fixed $\s' \in \fh^*_{\C}$, \cite[Theorem 4]{Dyatlov-Guillarmou-16} ensures that $s_0 \mapsto \zeta_\E(s_0,\s')$ converges for $|\Re \s| \gg 0$ and extends analytically in $s_0$ to an entire function, so that 
\[
\lim_{z\to 0} z^{-m}\zeta_\E(s_0 + z,\s') \in \C^\ast
\]
if and only if the spectral projector of $z\mapsto P_+(z,\s')$ at $z=s_0$ has rank $m$. In addition $\fa^*_{\C} \ni \s \mapsto \ell_{\E}(\s)$ admits a meromorphic extension to $\fa^*_{\C}$; this follows from adapting the arguments of \cite[Theorem 4]{Dyatlov-Guillarmou-16}, interpreting $\ell_{\E}(\s)$ as the flat trace $\Tr(P_+(\s)^{-1})$, and then using the meromorphic extension of the Schwartz kernel $P_+(\s)^{-1}$ on relevant anisotropic spaces of distributions established in Theorem \ref{theorem:rt-anosov2}. Next, we use Lemma \ref{lemma:extension} to deduce that $\zeta_\E$ is jointly holomorphic in all its variables.
\end{proof}

\begin{proof}[Proof of Corollary \ref{corollary:zeta-ruelle}]
Let $\fo(E_s)$ be the $\Z_2$ orientation bundle of $E_{s,\N}$. According to \cite[Lemma 2.1]{BornsWeil-Shen-2020}, the dynamical determinant
\begin{equation}
\label{equation:dynamical-determinant}
\zeta_m(\mathbf{s}) := \exp\left(-\sum_{\gamma \in \mathbf{P}} e^{- \s(\lambda(\gamma))} \dfrac{\ell_\gamma^\sharp \Tr(\Lambda^m P_\gamma)\mathrm{sign}\det({P_\gamma}_{|E_s})}{\ell_\gamma |\det(1-P_\gamma)|} \right),
\end{equation}
corresponds to the dynamical determinant introduced in \eqref{eq:def-general-zeta} with $\E = \Lambda^m(E_{s,\M}^*\oplus E_{u,\M}^*)^{\fo(E_s)} $, the bundle of forms twisted by $\fo(E_s)$\footnote{i.e., $E^\fo = \{ (u,\epsilon)~:~u\in E,\ \epsilon\in\fo\}/\{(u,\epsilon)\sim (-u,-\epsilon)\}$.}. By Theorem \ref{theorem:general-zeta}, it admits a holomorphic extension to $\fa^*_\C$. Following \cite[Lemma 2.2]{BornsWeil-Shen-2020}, one finds, on the domain of convergence, that
\begin{equation}
\label{equation:factorization}
\begin{split}
 \prod_{\gamma\in\mathbf P^\sharp} (1-e^{-\mathbf{s}(\lambda(\gamma))})^{-1}  = \prod_{m=0}^{d_{\mc N}-1} \zeta_m(\mathbf{s})^{(-1)^{m+1+\dim E_s}}.
\end{split}
\end{equation}
This proves the meromorphic extension of the Ruelle zeta function.
\end{proof}

\section{Leading resonant hypersurface}

\label{section:leading-resonance}

\subsection{Thermodynamic formalism} We recall standard facts from the thermodynamic formalism of uniformly hyperbolic flows; see \cite{Ruelle-78,Parry-Pollicott-90,Waddington-96} for further details. Throughout this section, we work under the assumptions $\hyperlink{AA1}{\rm(A1)-(A4)}$. 

\subsubsection{Pressure and equilibrium state}
Let $V \in C^\alpha(\mc{N})$ be a Hölder-continuous potential. The pressure of $V$ is defined as
\begin{equation}
\label{equation:pressure}
\mathrm{Pr}(V) = \sup_{\mu \in \mathrm{M}_\phi} h_\mu(\phi_1) + \int_{\mathscr{K}} V \dd \mu,
\end{equation}
where $\mathrm{M}_{\phi}$ denotes the set of flow-invariant probability measures supported on $\mathscr{K}$ introduced in \S\ref{sssection:convexity-limit-cone}, and $h_\mu(\phi_1)$ is the entropy of the measure $\mu$ with respect to the diffeomorphism $\phi_1$. It coincides with
\begin{equation}
\label{equation:definition-pressure}
\mathrm{Pr}(V) = \lim_{T \to +\infty} T^{-1} \log\left( \sum_{\gamma \in \mathbf{P}, \ell_\gamma \leq T} \exp\left(\int_\gamma V\right)\right),
\end{equation}
when $\mathrm{Pr}(V) \geq 0$, where we recall that $\mathbf{P}$ is the set of periodic orbits of the flow. The map $C^\alpha(\mc{N}) \ni V \mapsto \mathrm{Pr}(V)$ is real-analytic.

There exists a unique invariant probability measure $\mu_V$ maximizing \eqref{equation:pressure}, called the \emph{equilibrium state} associated with the potential $V$. One has $\mu_V = \mu_{V'}$ if and only if $V$ and $V'$ are cohomologous up to a constant, that is, there exist $u \in C^\eps(\mathscr{K})$ and $c \in \R$ such that $V = V' + X_{\mc{N}}u + c$. If $V,V' \in C^\alpha(\mc N)$ are not cohomologous (up to a constant), then $\mu_V$ and $\mu_V'$ are mutually singular.

\subsubsection{Derivatives of the pressure} \label{sssection:derivative} For a smooth family of potentials $\eps \mapsto V_\eps \in C^\alpha(\mc{N})$ such that $V := V_0$, the first derivative of the pressure is given by
\begin{equation}
\label{equation:derivative1}
\partial_{\eps} \mathrm{Pr}(V_\eps)|_{\eps = 0} = \int_{\mathscr{K}} \partial_\eps V_\eps|_{\eps=0} \dd \mu_{V}.
\end{equation}

Let $\scrV$ be a small relatively compact open neighborhood of $\scrK$. Given $f_1,f_2 \in C^\alpha(\mathscr{V})$, we let $f_i^{\vee} := f_i - \int_{\mathscr{K}} f_i \dd \mu_{V}$. The \emph{covariance} of $(f_1,f_2)$ with respect to $\mu_{V}$ is defined as
\begin{equation}
\label{equation:covariance}
\mathrm{Cov}_{\mu_{V}}(f_1,f_2) := \lim_{T \to +\infty} \dfrac{1}{T} \int_{\mathscr{K}} \left(\int_0^T f_1^{\vee}(\phi_t x) \dd t \int_0^T f_2^{\vee}(\phi_t x) \dd t\right) \dd \mu_{V}(x),
\end{equation}
and the \emph{variance} is defined by $\mathrm{Var}_{\mu_{V}}(f) := \mathrm{Cov}_{\mu_{V}}(f,f)$. Observe that $\mathrm{Var}_{\mu_{V}}(f) \geq 0$ with equality if and only if $f$ is cohomologous to a constant. The second derivative of the pressure is given by
\begin{equation}
\label{equation:derivative22}
\partial_{\eps_1}\partial_{\eps_2} \mathrm{Pr}(V_{\eps_1,\eps_2})|_{\eps_1=\eps_2= 0} = \mathrm{Cov}_{\mu_{V}}(\partial_{\eps_1} V_{\eps_1,\eps_2}, \partial_{\eps_2} V_{\eps_1,\eps_2})|_{\eps_1=\eps_2=0}+\int_{\mathscr K}\left.\partial_{\eps_1}\partial_{\eps_2}V_{\eps_1,\eps_2}\right|_{\eps_1=\eps_2=0}\dd\mu_V.
\end{equation}
In particular,
\begin{equation}
\label{equation:derivative2}
\partial_{\eps}^2 \mathrm{Pr}(V_\eps)|_{\eps = 0} =\mathrm{Var}_{\mu_{V}}(\partial_\eps V_\eps|_{\eps=0})+\int_{\mathscr K}\left.\partial_\eps^2V_\eps\right|_{\eps=0}\dd\mu_V.
\end{equation}
We will need the following lemma due to Ledrappier \cite{Ledrappier-95}. 

\begin{lemma}
\label{lemma:ledrappier}
Let $V \in C^\eps(\mathscr{K})$ be a Hölder-continuous potential. If $\mathrm{Pr}(-V) = 0$ and $\int_\gamma V > 0$ for each periodic orbit $\gamma \in \mathbf{P}$, then the following statements hold:
\begin{enumerate}[label=\emph{(\roman*)}]
\item\label{it:led1} $\lim_{T \to +\infty} T^{-1} \log \left(\sharp\{\gamma \in \mathbf{P} ~:~ \int_\gamma V \leq T\} \right) = 1$;
\item\label{it:led2} There exists a constant $C > 0$ such that $\int_\gamma V \geq C \ell_\gamma$.
\end{enumerate}
\end{lemma}

\subsubsection{Leading resonance on functions and top forms} \label{sssection:functions} We refer to the article \cite{Humbert-24} for the following facts\footnote{Technically, \cite{Humbert-24} only considers Anosov flows, but the arguments apply \emph{verbatim} to Axiom A flows as well.}. For simplicity, we also assume that $\mc{N}$ is orientable; in this case, $\Lambda^{d_{\mc{N}}} T^*\mc{N}$ is trivial and naturally identified with the density bundle $\Omega^1 \mc{N}$. Recall that $(\phi_t)_{t \in \R}$ is assumed to be transitive (Assumption $\hyperlink{AA4}{\rm(A4)}$).

Let $V :=u+iv \in C^\infty(\mc{N})$ be a smooth potential. Let $\chi \in C^\infty_{\comp}(\mathscr{V})$ such that $\chi \equiv 1$ on a neighborhood of $\mathscr{K}$ ($\chi \equiv 1$ in the Anosov case). Let 
\[
R^{(m)}_\pm(s) := \chi(\mp\mc{L}_{X_{\mc{N}}}+V-s)^{-1}\chi
\]
denote the resolvent acting on $m$-forms, and let $R^{(m),0}_{\pm}(s)$ denote the resolvent on $m$-forms restricted to sections of $\Sigma^m := \Lambda^m(E_s^* \oplus E_u^*)$; that is, forms in the kernel of the contraction $\iota_{X_{\mc{N}}}$. The resolvent $s \mapsto R^{(m)}_{\pm}(s)$ admits a meromorphic extension to $\C$; see \cite{Faure-Sjostrand-11, Dyatlov-Zworski-16,Dyatlov-Guillarmou-16}; its poles are called the resonances of the flow $\phi_t$ on the corresponding bundles.

The resonances of $R^{(0)}_+(s)$ are contained in the half-space $\{\Re(s) \leq \mathrm{Pr}(J_u +u)\}$, where $J_u$ is the unstable Jacobian
\begin{equation}
\label{equation:unstable-jacobian}
J_u := -\partial_t \det(d\phi_t|_{E_u})|_{t=0}.
\end{equation}
Under the assumption that $(\phi_t)_{t \in \R}$ is transitive (Assumption $\hyperlink{AA4}{\rm(A4)}$), if $V = u$ is real-valued, then $s= \mathrm{Pr}(J_u+u)$ is a pole and the associated space of resonant states has dimension $1$; it is spanned by $m_s^{(0)} \in \mc{D}'(\mc{N})$, which satisfies
\[
\WF(m_s^{(0)}) \subset E_u^*,\quad \supp(m_s^{(0)}) \subset \mc{T}_+, \quad (-X_{\mc{N}}+u-\mathrm{Pr}(J_u+u)) m_s^{(0)} = 0.
\]
In addition, $m_s^{(0)}$ is a distribution of order $0$ (that is, it can be paired against continuous $d_{\mc N}$-forms).

Similarly, the resonances of $R^{(d_{\mc{N}})}_-(s)$ are contained in $\{\Re(s) \leq \mathrm{Pr}(J_u+u)\}$. Moreover, $s=\mathrm{Pr}(J_u+u)$ is a pole with a $1$-dimensional space of co-resonant states spanned by $m_u^{(d_{\mc{N}})} \in \mc{D}'(\mc{N},\Lambda^{d_{\mc{N}}} T^*\mc{N})$, a distribution of order $0$, such that 
\[
\WF(m_u^{(d_{\mc{N}})}) \subset E_s^*, \quad \supp(m_u^{(d_{\mc{N}})}) \subset \mc{T}_-, \quad (+\mc{L}_{X_{\mc{N}}}+u-\mathrm{Pr}(J_u+u)) m_u^{(d_{\mc{N}})} = 0.
\]
The product $m_s^{(0)} \times m_u^{(d_{\mc{N}})} \in \mc{D}'(\mc{N}, \Lambda^{d_{\mc{N}}} T^*\mc{N})$ is well-defined in the sense of distributions by the wavefront set calculus (see \cite[Lemma 4.3.1]{Lefeuvre-book} for instance) and is supported in $\mathscr{K} =\mc{T}_- \cap \mc{T}_+$. The resulting equilibrium state associated with the potential $J_u+u$ is
\[
\mu_{J_u+u} = m_s^{(0)} \times m_u^{(d_{\mc{N}})}.
\]
If the flow $(\phi_t)_{t \in \R}$ is topologically mixing, then it is weakly mixing (see \cite[Corollary 7.3.8]{Fisher-Hasselblatt-19}). In turn, this implies that $\mathrm{Pr}(J_u+u)$ is the only resonance for $R^{(0)}_+(s)$ and $R^{(d_{\mc{N}})}_-(s)$ with real part equal to $\mathrm{Pr}(J_u+u)$ (see \cite[Theorem 2]{Humbert-24}).

If $V = u + iv$ is complex-valued, then there is a resonance $s_0$ for $R^{(0)}_+(s)$ and $R^{(d_{\mc{N}})}_-(s)$ with $\Re(s_0) = \mathrm{Pr}(J_u+u)$ if and only if the following arithmeticity condition holds \cite[Corollary 3]{Humbert-24}:
\begin{equation}
\label{equation:arithmeticity}
\forall \gamma \in \mathbf{P}, \qquad \int_\gamma (-v + \Im(s_0)) \in 2\pi \Z.
\end{equation}

\subsubsection{Leading resonance on $d_s$- and $d_u$-forms} \label{sssection:ds-forms}

Similar results hold for $d_s$- and $d_u$-forms, where $d_s = \dim E_s$ and $d_u = \dim E_u$. The resonances of $R^{(d_s),0}_+(s)$ are contained in the half-space $\{\Re(s) \leq \mathrm{Pr}(u)\}$. If the flow $(\phi_t)_{t \in \R}$ is transitive (Assumption $\hyperlink{AA4}{\rm(A4)}$) and $V = u$ is real-valued, then $s= \mathrm{Pr}(u)$ is a pole. The associated space of resonant states is $1$-dimensional and spanned by $m_s^{(d_s)} \in \mc{D}'(\mc{N},\Sigma^{d_s})$; moreover, $m_s^{(d_s)}$ satisfies
\[
\WF(m_s^{(d_s)}) \subset E_u^*,\quad \supp(m_s^{(d_s)}) \subset \mc{T}_+, \quad (-\mc{L}_{X_{\N}}+u-\mathrm{Pr}(u)) m_s^{(d_s)} = 0.
\]
The resonances of $R^{(d_u),0}_-(s)$ are contained in $\{\Re(s) \leq \mathrm{Pr}(u)\}$. Moreover, $s=\mathrm{Pr}(u)$ is a pole with a $1$-dimensional space of co-resonant states spanned by $m_u^{(d_u)} \in \mc{D}'(\mc{N},\Sigma^{d_u})$, which satisfies
\begin{equation}
\label{equation:mudu}
\WF(m_u^{(d_u)}) \subset E_s^*, \quad \supp(m_u^{(d_u)}) \subset \mc{T}_-, \quad (+\mc{L}_{X_{\mc{N}}}+u-\mathrm{Pr}(u)) m_u^{(d_u)} = 0.
\end{equation}
Recall that $\alpha_{\mc{N}}$ is the Anosov $1$-form such that $\alpha_{\mc{N}}(X_{\mc{N}})=1$ and $\alpha_{\mc{N}}(E_s \oplus E_u) = 0$. The wedge product
\begin{equation}
\label{equation:muu}
\mu_u := \alpha_{\mc{N}} \wedge m_u^{(d_u)} \wedge m_s^{(d_s)} \in \mc{D}'(\mc{N}, \Lambda^{d_\mc{N}} T^*\mc{N})
\end{equation}
is well-defined, supported on $\mathscr{K}$, and defines a measure that is the equilibrium state of the potential $u$. If the flow is topologically mixing, then $\mathrm{Pr}(u)$ is the only resonance for $R^{(d_s),0}_+(s)$ and $R^{(d_u),0}_-(s)$ with real part equal to $\mathrm{Pr}(u)$.

If $V = u + iv$ is complex-valued, then there are no resonances $s_0$ for $R^{(d_s),0}_+(s)$ and $R^{(d_u),0}_-(s)$ with $\Re(s_0) = \mathrm{Pr}(u)$ unless the arithmeticity condition \eqref{equation:arithmeticity} holds.

\subsection{Leading resonant hypersurface}

\label{ssection:proof-leading}

We now introduce and study the leading resonant hypersurface. Recall that in contrast to \S\ref{section:results-cocycles}, we have additionally assumed that the flow $(\phi_t)_{t \in \R}$ is topologically transitive on the trapped set (Assumption $\hyperlink{AA4}{\rm(A4)}$) in \S\ref{section:leading-resonance}. This hypersurface is also called the \emph{critical hypersurface} in the literature; see \cite{Sambarino-24} for instance.

\subsubsection{Statement} For $\mathbf{s} \in \fa^*_{\C}$, we write $\mathbf{s} = \Re(\mathbf{s}) + i \Im(\mathbf{s})$. If $\mathbf{C} \subset \fa^*$ is a (local) analytic manifold, we denote by $\mathbf{C}_{\C}\subset \fa^*_{\C}$ its local holomorphic extension to the complexified Lie algebra. The dual Lie algebra $\fa^*$ splits as
\[
\fa^* = \R \alpha_\M \oplus \fh^*,
\]
where $\alpha_\M$ is the $1$-form such that $\alpha_\M(X_\M)=1$ and $\alpha_\M(\fh)=0$, while $\fh^*(X_\M)=0$. We let $\partial_\infty \fa^*$ denote the sphere at infinity of $\fa^*$, that is $\partial_\infty \fa^* = \fa^* \setminus \{0\} / \sim$ where $\varphi \sim \lambda \varphi$ for $\varphi \neq 0$ and $\lambda > 0$. A generic point in $\partial_\infty \fa^*$ is denoted by $[\varphi]$. Given a set $E \subset \fa^*\setminus \{0\}$, we let $[E] := \{ [\varphi] \in \partial_\infty \fa^* ~:~ \varphi \in E\}$ be its projection onto the sphere at infinity. Finally, recall that $\delta : \fa^* \to \R \cup \{\pm \infty\}$ was introduced in \eqref{equation:delta} and that $\ell_\gamma = \alpha_{\M}(\lambda(\gamma))$ is the period of the corresponding closed orbit for $(\phi_t)_{t \in \R}$ on $\scrK \subset \mc{N}$ (see Proposition \ref{proposition:proper-action-phi}, item (iii)).

\begin{theorem}[Leading resonant hypersurface]
\label{theorem:leading}
Assume that $\hyperlink{AA1}{\rm(A1)-(A4)}$ hold. Suppose that $\mc{E} = \Lambda^m(E_{s,\M}^* \oplus E_{u,\M}^*)$ with $m=0$ or $m=d_s$ and write $\sigma^{(m)}_{\mathrm{RS}}$ for the corresponding resonance spectrum.
\begin{enumerate}[label=\emph{(\roman*)}]
\item\label{it:leading-critical-hypersurface} \emph{\textbf{Critical hypersurface.}} There exists an analytic hypersurface $\mathbf{C}^{(m)} \subset\mathfrak{a}^*$, called the \emph{leading resonant hypersurface} (on $m$-forms), dividing $\fa^*$ into two open connected components $\mathbf{C}^{(m),\pm}$ such that
\begin{enumerate}[label=\emph{(\alph*)}]
\item $\mathbf{C}^{(m)} \subset \sigma_{\mathrm{RS}}^{(m)}$ and for all $\varphi \in \mathbf{C}^{(m)}$, the space of resonant states $\mathrm{Res}(\varphi)$ associated with $\varphi$ (see \eqref{equation:ruelle-taylor-resonant}) is $1$-dimensional;

\item $\Re(\sigma_{\mathrm{RS}}^{(m)}) \subset \overline{\mathbf{C}^{(m),-}}$
;
\item $\mathbf{C}^{(m),+}$ is convex;
\end{enumerate}

\item\label{it:leading-strict-convexity} \emph{\textbf{Strict convexity of $\mathbf{C}^{(m)}$.}} For each $d\in\{0,\ldots,k\}$, the following statements are equivalent:
\begin{enumerate}[label=\emph{(\alph*)}]
\item $\mathbf{C}^{(m)}$ has second-order contact at a point with a $d$-dimensional affine plane;
\item $\mathbf{C}^{(m)}$ contains a $d$-dimensional affine plane;
\item the cocycle has rank $\leq k-d$.
\end{enumerate}
In particular, $\mathbf{C}^{(m),+}$ is strictly convex if and only if the cocycle has full rank.

\item\label{it:leading-concavity} \emph{\textbf{Concavity of $\mathbf{C}^{(m)}_{\C}$.}} If the cocycle is non-arithmetic, $\mathbf{s} \in \sigma_{\mathrm{RS}}^{(m)}$ with $\Re(\mathbf{s}) \in \mathbf{C}^{(m)}$ if and only if $\mathbf{s} = \Re(\mathbf{s})$ is real.

\item\label{it:leading-intermediate-hypersurfaces} \emph{\textbf{Intermediate critical hypersurfaces.}} For every $0\leq q\leq d_{\mc N}-1$ with $q\neq d_s$, $\Re(\sigma_{\mathrm{RS}}^{(q)}) \subset \mathbf{C}^{(d_s),-}$.

\item\label{it:leading-zeta-function} \emph{\textbf{Ruelle zeta function.}} The Ruelle zeta function \eqref{equation:zeta-anosov} is holomorphic in $\{\mathbf{s} \in \fa^*_{\C} ~:~ \Re(\mathbf{s}) \in \mathbf{C}^{(d_s),+}\}$ and has a singularity of order $1$ along $\mathbf{C}^{(d_s)}$.

\item\label{it:characterization_C} \emph{\textbf{Intrinsic characterization of $\mathbf{C}^{(d_s)}$.}} The leading resonant hypersurface $\mathbf{C}^{(d_s)}$ coincides with $\{\delta(\varphi)=1\} \cap \mathrm{int}(\scrL^*)$. For all $\varphi \in \mathbf{C}^{(d_s)}$, there exists a constant $C > 0$ such that
\begin{equation}
\label{equation:turbulences0}
\varphi(\lambda(\gamma)) > C \ell_\gamma,\qquad \gamma\in\mathbf P.
\end{equation}
In addition, the map
\[
\mathbf{C}^{(d_s)} \to [\mathbf{C}^{(d_s)}], \qquad \varphi \mapsto [\varphi]
\]
is an analytic diffeomorphism and $[\mathbf{C}^{(d_s)}] = \mathrm{int}[\mathscr{L}^* \setminus \{0\}] \subset \partial_\infty \fa^*$.

\end{enumerate}
\end{theorem}

It follows from item~\emph{\ref{it:leading-critical-hypersurface}} that for all $\varphi \in \mathbf{C}^{(m)}$, there exists an open neighborhood $U \subset \fa^*_{\C}$ containing $\varphi$ such that
\[
\sigma_{\mathrm{RS}}^{(m)} \cap U = \mathbf{C}^{(m)}_{\C} \cap U,
\]
and the corresponding resonant spaces are $1$-dimensional. (Strictly speaking, it is actually a consequence of the fact that the corresponding space of generalized resonant states for $P_\pm(\s)^{-1}$ is one-dimensional.) The meaning of item \emph{\ref{it:leading-zeta-function}} is the following: for any $\mathbf{s}_\star \in \mathbf{C}^{(d_s)}$, there exist a neighborhood $U \subset \fa^*_{\C}$ of $\mathbf{s}_\star$, a holomorphic function $g_{\mathbf{s}_\star}$ on $U$ satisfying
\[
\mathbf{C}^{(d_s)}_{\C}\cap U=\{g_{\mathbf{s}_\star}=0\},
\qquad
\dd g_{\mathbf{s}_\star}\neq 0 \quad\text{on }\mathbf{C}^{(d_s)}_{\C}\cap U,
\]
and a holomorphic non-vanishing function $\widetilde{\zeta}$ on $U$ such that $\zeta(\mathbf{s})=g_{\mathbf{s}_\star}(\mathbf{s})^{-1}\widetilde{\zeta}(\mathbf{s})$. In particular, after shrinking $U$, there exists a constant $C>0$ such that, for all $\mathbf{s}\in U\setminus\mathbf{C}^{(d_s)}_{\C}$,
\[
|\zeta(\mathbf{s})| \geq C d(\mathbf{s},\mathbf{C}_\C^{(d_s)})^{-1},
\]
where $d$ denotes the distance induced by an arbitrary metric on $\fa^*_{\C}$. 

In the rank-one case, the first pole (or \emph{leading resonance}) of the zeta function is the topological entropy $h_{\mathrm{top}}(\phi_1)$ of the flow.

\subsubsection{Proof}

Note that, because $E_{u,\N}$ admits a Hölder-continuous extension outside $\mc{T}_{+,\N}$, the unstable Jacobian $J_u$ is defined on all of $\mc{N}$; only its value on the trapped set $\mathscr{K}$ will be relevant. Recall that $\mathbf s = (s_0,\mathbf{s}')\in \CC \alpha_{\mathcal M}\oplus \fh^*_{\C}$. We begin by giving an explicit description of the leading resonant hypersurface:

\begin{proposition}[Leading resonant hypersurface]
\label{proposition:leading2}
For $m=0$ or $m=d_s$, the leading resonant hypersurface $\mathbf{C}^{(m)}$ is given by
\[
\mathbf{C}^{(m)} = \{ (\mathrm{Pr}(\eps J_u+\mathbf{s}'(w)), \mathbf{s}') ~:~ \mathbf{s}' \in \mathfrak{h}^*\},
\]
where $\eps=1$ if $m=0$, and $\eps=0$ if $m=d_s$.
\end{proposition}

When $J_u$ is constant, the critical hypersurface for functions is a translate of the critical hypersurface on $d_s$-forms, that is $\mathbf{C}^{(0)} = \mathbf{C}^{(d_s)} + J_u \alpha_{\M}$, where $\alpha_{\M}(X_{\M})=1$, $\alpha_{\M}(\mathfrak{h})=0$. This is what happens in algebraic settings such as Anosov representations. In addition, if $X_\M \in \scrL$, then $\mathrm{Pr}(\s'w) > 0$ for all $\s' \in \fh^*$, see \eqref{equation:pression-positive}.

\begin{proof}
Fix $\mathbf{s}' \in \mathfrak{h}^*$. By \S\ref{sssection:ds-forms}, the first resonance of $s_0 \mapsto P^{(d_s)}_+(s_0,\mathbf{s}')^{-1}$ is $s_0 = \mathrm{Pr}(\mathbf{s}'(w))$. It is simple; that is, the space of resonant states is $1$-dimensional. A point $(s_0,\mathbf{s}')$ is a pole of $P^{(d_s)}_+(s_0,\mathbf{s}')^{-1}$ if and only if $(s_0,\mathbf{s}') \in \sigma_{\mathrm{RS}}$ by Theorem~\ref{theorem:rt-anosov2}, item (i).

The hypersurface 
\[
\mathbf{C}^{(d_s)} := \{(\mathrm{Pr}(\mathbf{s}'(w)),\mathbf{s}') ~:~ \mathbf{s}' \in \mathfrak{h}^*\}
\]
divides $\fa^*$ into two connected components
\[
\mathbf{C}^{(d_s),\pm} := \{ (s_0,\mathbf{s}') \in \fa^* ~:~ \pm\bigl(s_0-\mathrm{Pr}(\mathbf{s}'(w))\bigr)>0\}.
\]
In addition, $\mathbf{s} \mapsto P^{(d_s)}_+(\mathbf{s})^{-1}$ is holomorphic on $\{\mathbf{s} \in \fa^*_{\C} ~:~ \Re(\mathbf{s}) \in \mathbf{C}^{(d_s),+}\}$ and the set of poles of $(P^{(d_s)}_+)^{-1}$, that is, the resonance spectrum, is contained in $\overline{\mathbf{C}^{(d_s),-}}$. Finally, the convexity of $\mathbf{C}^{(d_s),+}$ follows immediately from \eqref{equation:derivative2}.
\end{proof}

We also prove a transversality property:

\begin{lemma}
\label{lemma:transverse-intersection}
Let $\varphi = (\Pr(\mathbf s'(w)),\mathbf s') \in \mathbf{C}^{(d_s)}$ for some $\mathbf s'\in\mathfrak h^*$. Then $\R \varphi$ is transverse to $T_\varphi \mathbf{C}^{(d_s)}$.
\end{lemma}

\begin{proof}
By \eqref{equation:derivative1}, the tangent space to $\mathbf{C}^{(d_s)}$ at $\varphi$ is given by
\begin{equation}
\label{equation:espace-tangent}
T_\varphi \mathbf{C}^{(d_s)} = \left\{ \left( \int_{\mathscr{K}} h(w) \dd \mu_{\mathbf{s}'(w)}, h \right) ~:~ h \in \mathfrak{h}^*\right\}.
\end{equation}
To show that $\varphi \notin T_\varphi \mathbf{C}^{(d_s)}$, it thus suffices to show that
\[
\mathrm{Pr}(\mathbf{s}'(w)) \neq \int_{\mathscr{K}} \mathbf{s}'(w) \dd \mu_{\mathbf{s}'(w)},
\]
but this is immediate as $\mathrm{Pr}(\mathbf{s}'(w)) = h_{\mu_{\mathbf{s}'(w)}} + \int_{\mathscr{K}} \mathbf{s}'(w) \dd \mu_{\mathbf{s}'(w)}$ and $h_{\mu_{\mathbf{s}'(w)}} > 0$ (recall that $\scrK$ is not reduced to a finite union of periodic orbits and $(\phi_t)_{t \in \R}$ has positive topological entropy on $\scrK$ by assumption). The normalized transverse generator is
\[
u_\varphi:=\frac{X_\M-\int_{\mathscr K}w\,\dd\mu_{\mathbf s'(w)}}{h_{\mu_{\mathbf s'(w)}}(\phi_1)}.
\]
It satisfies $\varphi(u_\varphi)=1$, annihilates $T_\varphi\mathbf C^{(d_s)}$, and belongs to $\scrL$ by Lemma~\ref{lemma:formula-for-L}.
\end{proof}

\begin{lemma}
\label{lemma:rouge}
$X_\M \in \scrL$ if and only if there exists an invariant measure $\mu_0 \in \mathrm{M}_\phi$ such that $\int_\scrK w \dd\mu_0 = 0$.
\end{lemma}

\begin{proof}
Immediate by Lemma \ref{lemma:formula-for-L}, item (ii).
\end{proof}

Finally, we prove Theorem~\ref{theorem:leading}. We treat the case $m=d_s$; the case $m=0$ is similar. 

\begin{proof}[Proof of Theorem~\ref{theorem:leading}]
\emph{\ref{it:leading-critical-hypersurface}} This follows from Proposition \ref{proposition:leading2}. \\

\emph{\ref{it:leading-strict-convexity}} Suppose that $\mathbf{C}^{(d_s)}$ has second-order contact at a point $\mathbf{s}_\star = (s_{0\star}, \mathbf{s}_\star') \in \mathbf{C}^{(d_s)}$ with a $d$-dimensional plane. This means that there is a linearly independent family of $d$ covectors $\theta_1,..., \theta_d \in \fh^*$ such that
\[
\partial^2_{\eps} \mathrm{Pr}((\mathbf{s}'_\star+ \eps \theta_i)(w))|_{\eps = 0} = 0 = \mathrm{Var}_{\mu_{\mathbf{s}'_\star(w)}}(\theta_i(w)).
\]
Hence $\theta_i(w)$ is cohomologous to a constant $c_i \in \R$. By \eqref{equation:lyapunov-projection}, for every $\gamma\in \mathbf{P}$
\[
(c_i \alpha_{\mathcal M} + \theta_i)(\lambda(\gamma))= 0,
\]
and consequently $\mathscr{L} \subset\cap_{i=1}^d \ker (c_i\alpha_{\mathcal M} + \theta_i)$ which is of dimension $k+1-d$, that is the rank of the cocycle is $\leq k-d$. Hence (a) implies (c), and the converse is also straightforward by the same argument.

Let us show that (a) implies (b). For $t_1,...,t_d \in \R$, one has
\[
\mathrm{Pr}((\mathbf{s}'_\star + \sum_i t_i \theta_i)(w)) = \mathrm{Pr}(\mathbf{s}'_\star(w) + \sum_i t_i c_i) = \mathrm{Pr}(\mathbf{s}'_\star(w)) + \sum_i t_i c_i.
\]
This shows that
\[
\begin{split}
\left\{ \left( \mathrm{Pr}((\mathbf{s}'_\star + \sum_i t_i \theta_i)(w)), \mathbf{s}'_\star + \sum_i t_i \theta_i\right) ~:~ (t_1,\ldots,t_d)\in\R^d\right\} \subset \mathbf{C}^{(d_s)}
\end{split}
\]
is a $d$-dimensional affine plane contained in $\mathbf{C}^{(d_s)}$. Finally, the implication (b) $\implies$ (a) is immediate. \\

\emph{\ref{it:leading-concavity}} Suppose that $\mathbf{s} \in \sigma_{\mathrm{RS}}^{(d_s)}$ with $\Re(\mathbf{s}) \in \mathbf{C}^{(d_s)}$. By \S\ref{sssection:ds-forms}, this implies that for all periodic orbits $\gamma \in \mathbf{P}$:
\[
\int_\gamma (-\Im(\mathbf{s}')(w)+\Im(s_0)) \in 2\pi \Z,
\]
that is $\theta(\lambda(\gamma)) \in 2\pi \Z$, where $\theta = \Im(s_0)\alpha_{\M}+\Im(\mathbf{s}')$. By the non-arithmeticity assumption on the Lyapunov spectrum (see Definition~\ref{definition:non-arithmetic}), this forces $\theta =0$, that is $\Im(\mathbf{s}) = 0$. \\

\emph{\ref{it:leading-intermediate-hypersurfaces}} The claim on the intermediate critical hypersurfaces follows from the corresponding claim in \cite[Theorem 1]{Humbert-24}, applied at fixed $\mathbf{s}' \in \mathfrak{h}^*_{\C}$. \\

\emph{\ref{it:leading-zeta-function}} By item~\ref{it:leading-intermediate-hypersurfaces} and Theorem \ref{theorem:general-zeta}, the dynamical determinants in the factorization formula \eqref{equation:factorization} are all holomorphic and non-vanishing in $\{\Re(\s) \in \mathbf{C}^{(d_s),+}\}$, and so is $\zeta$. (Note that the twist by $\mathfrak o(E_s)$ does not affect the fact that the corresponding intermediate critical hypersurfaces for $\mathfrak o(E_s)$ twisted forms of intermediate degree are contained in $\mathbf{C}^{(d_s),-}$.)

That $\zeta$ is singular to order $1$ is a consequence of the following observation. For a fixed $\mathbf s_\star= (s_{0\star},\mathbf s_\star') \in \mathbf{C}^{(d_s)}$, the resolvent $s \mapsto P^{(d_s)}_+(s+s_{0\star},\mathbf{s_\star}')^{-1}$ has a pole of order $1$ at $s=0$ by \S\ref{sssection:ds-forms}. The structure of $\mathbf{s} \mapsto P^{(d_s)}_+(\mathbf{s})^{-1}$ is thus given by \eqref{equation:cas-cool2}, that is
\[
P^{(d_s)}_+(\mathbf{s})^{-1} = H(\mathbf{s}) - \dfrac{\Pi_{\mathbf{s}'}}{s_0-\sigma(\mathbf{s}')},
\]
where $\mathbf{s} \mapsto H(\mathbf{s})$ is holomorphic, $\Pi_{\mathbf{s}'}$ is the spectral projector onto the resonant state associated with the resonance $(\mathrm{Pr}(\mathbf{s}'(w)),\mathbf{s}')$ and $U'_\star\subset\mathfrak h^*_\C$ is a neighborhood of $\mathbf{s}'_\star$ on which $\sigma:U'_\star\to\C$ is holomorphic. On the real subspace $\mathfrak h^*$, we have $\sigma(\mathbf s') = \mathrm{Pr}(\mathbf{s}'(w))$; uniqueness of holomorphic extensions therefore gives
\[
P^{(d_s)}_+(\mathbf{s})^{-1} = H(\mathbf{s}) - \dfrac{\Pi_{\mathbf{s}'}}{s_0-\mathrm{Pr}(\mathbf{s}'(w))}.
\]
Here, $\mathbf{s}' \mapsto \mathrm{Pr}(\mathbf{s}'(w))$ denotes the local holomorphic germ of the pressure near $\mathbf{s}'_\star$. In turn, using \eqref{equation:factorization}, we find that all the factors $\mathbf{s}\mapsto \zeta_m(\mathbf{s})$ are holomorphic and non-vanishing ($\zeta_m(\mathbf{s})=0$ if and only if $\mathbf{s}$ is a resonance of $P^{(m)}_+$) for $\mathbf{s}$ close to $\mathbf{s}_\star$, except for $m=d_s$, for which we have $\zeta_{d_s}(\mathbf{s}) = (s_0-\mathrm{Pr}(\mathbf{s}'(w))) \widetilde{\zeta}(\mathbf{s})$ for some holomorphic non-vanishing function $\widetilde{\zeta}$. This implies that
\[
\zeta(\mathbf{s}) = (s_0-\mathrm{Pr}(\mathbf{s}'(w)))^{-1} \widetilde{\zeta}(\mathbf{s}),
\]
for some other holomorphic and nonvanishing function $\mathbf{s} \mapsto \widetilde{\zeta}(\mathbf{s})$. Since $(\mathrm{Pr}(\mathbf{s}'(w)),\mathbf{s}') \in \mathbf{C}^{(d_s)}$, we have $|s_0-\mathrm{Pr}(\mathbf{s}'(w))| \leq C d(\mathbf{s},\mathbf{C}^{(d_s)})$, and thus
\[
|\zeta(\mathbf{s})| \geq C d(\mathbf{s},\mathbf{C}_\C^{(d_s)})^{-1}.
\]
This proves the claim.\\

\emph{\ref{it:characterization_C}} Suppose that $\varphi = (\mathrm{Pr}(\mathbf{s}'(w)),\mathbf{s}') \in \mathbf{C}^{(d_s)}$. Observe that for any periodic orbit $\gamma \in \mathbf{P}$:
\begin{equation}
\label{equation:positive-pr}
\varphi(\lambda(\gamma)) = \varphi\left(\ell_\gamma X_{\M} - \int_{\gamma} w\right) = \mathrm{Pr}(\mathbf{s}'(w))\ell_\gamma -\int_\gamma \mathbf{s}'(w) > 0.
\end{equation}
Indeed, 
\[
\mathrm{Pr}(\mathbf{s}'(w)) = \sup_{\mu \in \mathrm{M}_{\phi}} h_\mu (\phi_1) + \int_{\mathscr{K}} \mathbf{s}'(w) \dd \mu,
\]
and the supremum is achieved for a unique measure, the equilibrium state $\mu_{\mathbf{s}'(w)}$. Since $h_\mu(\phi_1) \geq 0$ and $\mu_{\mathbf{s}'(w)} \neq \delta_\gamma$ (the Dirac mass carried by the periodic orbit) because $\mu_{\mathbf{s}'(w)}$ has full support on $\mathscr{K}$, we obtain \eqref{equation:positive-pr}. In particular, this implies that $\varphi \in \mathscr{L}^*$ as $\mathscr{L}$ is the cone generated by all the vectors $\lambda(\gamma)$.

Note that
\[
\mathrm{Pr}(-\varphi(X_{\M}-w)) = \mathrm{Pr}(\mathbf{s}'(w)-\mathrm{Pr}(\mathbf{s}'(w))) = 0.
\]
Applying Lemma~\ref{lemma:ledrappier}, we find that $\delta(\varphi)=1$; moreover, there exists a constant $C > 0$ such that for all $\gamma \in \mathbf{P}$,
\begin{equation}
\label{equation:turbulences}
\varphi(\lambda(\gamma)) \geq C \ell_\gamma.
\end{equation}
This shows \eqref{equation:turbulences0}. In addition, if $X_\M \in \scrL$, it follows from Lemma \ref{lemma:rouge} and \eqref{equation:positive-pr} (applied with a sequence of periodic orbits converging in the weak-$*$ topology to the measure $\mu_0$ provided by Lemma \ref{lemma:rouge}) that
\begin{equation}
\label{equation:pression-positive}
\mathrm{Pr}(\s'w) \geq C,
\end{equation}
where $C$ is the same constant as in \eqref{equation:turbulences}. The estimate \eqref{equation:turbulences} also implies that $\varphi \in \mathrm{int}(\mathscr{L}^*)$ as $|\lambda(\gamma)| \leq C' \ell_\gamma$ for some constant $C' > 0$. This shows that $\mathbf{C}^{(d_s)} \subset \{\delta(\varphi)=1\} \cap \mathrm{int}(\mathscr{L}^*)$.

Conversely, suppose that $\delta(\varphi)=1$ and $\varphi \in \mathrm{int}(\scrL^*)$. By definition of $\delta(\varphi)$, $\zeta(s\varphi)$ is thus given by the convergent product \eqref{equation:zeta-anosov} for $\Re(s) > 1$. For $s > 1$ real and $t > 0$, observe that $\zeta(s\varphi + t \alpha_{\M}) \leq \zeta(s\varphi)$ and $\zeta(s\varphi + t\alpha_{\M})$ is also given by the convergent product \eqref{equation:zeta-anosov}. If $\varphi \in \mathbf{C}^{(d_s),-}$, then, by taking $s>1$ sufficiently close to $1$ so that $s\varphi \in \mathbf{C}^{(d_s),-}$, we see that there must exist a value of $t >0$ such that $s\varphi + t\alpha_\M \in \mathbf{C}^{(d_s)}$. But then $\zeta(s\varphi + t\alpha_\M)=\infty$, which contradicts that $\zeta(s\varphi + t\alpha_\M) < \infty$ is given by the convergent expression \eqref{equation:zeta-anosov}. As a consequence, $\varphi \in \overline{\mathbf{C}^{(d_s),+}}$. If $\varphi \in \mathbf{C}^{(d_s),+}$, then using the definition of $\delta$, and the fact that $\varphi(\lambda(\gamma)) > 0$ by assumption, one easily obtains that $s \mapsto \zeta(s \varphi)$ is given for $s > 1$ by \eqref{equation:zeta-anosov} and blows up as $s \to 1^+$. This, however, contradicts the holomorphy of $\zeta$ in $\mathbf{C}^{(d_s),+}$. Hence $\varphi \in \mathbf{C}^{(d_s)}$. This completes the proof that $\mathbf{C}^{(d_s)}= \{\delta(\varphi)=1\} \cap \mathrm{int}(\scrL^*)$.

Note that $0 \notin \mathbf{C}^{(d_s)}$ as $\mathrm{Pr}(0) = h_{\mathrm{top}}(\phi_1) > 0$, so $[\mathbf{C}^{(d_s)}]$ is well-defined. We have already established that $[\mathbf{C}^{(d_s)}] \subset \mathrm{int}[\mathscr{L}^*\setminus\{0\}]$. Let us show the converse inclusion and that $\mathbf{C}^{(d_s)} \ni \varphi \mapsto [\varphi]$ is a diffeomorphism onto its image. That $\mathbf{C}^{(d_s)} \ni \varphi \mapsto [\varphi]$ is a local diffeomorphism is a direct consequence of Lemma~\ref{lemma:transverse-intersection}. It is also injective (hence a global diffeomorphism onto its image) because if $\varphi \in \mathbf{C}^{(d_s)} = \{\delta=1\}$, then $\delta(\lambda\varphi)=\lambda^{-1}\delta(\varphi)=\lambda^{-1}\neq 1$ for $\lambda \neq 1$, so $\lambda\varphi \notin \mathbf{C}^{(d_s)}$. Finally, it remains to prove that $\mathrm{int}[\mathscr{L}^*\setminus\{0\}] \subset [\mathbf{C}^{(d_s)}]$. If $\varphi \in \mathrm{int}(\mathscr{L}^*\setminus \{0\})$, then $\varphi(\lambda(\gamma)) \geq \eps \ell_\gamma$ for some $\eps > 0$. Since $w$ is bounded, there is also a constant $C_\varphi>0$ such that $\varphi(\lambda(\gamma))\leq C_\varphi\ell_\gamma$ for every $\gamma\in\mathbf P$. In turn, $\delta(\varphi) < +\infty$ is well-defined (using the specification property as in the proof of Lemma~\ref{lemma:ledrappier}) with $0<h_{\mathrm{top}}(\phi_1)/C_\varphi\leq\delta(\varphi)\leq h_{\mathrm{top}}(\phi_1)/\eps<\infty$. As a consequence, there exists a $\lambda > 0$ such that $\delta(\lambda\varphi) = \lambda^{-1}\delta(\varphi)=1$, that is $\lambda \varphi \in \mathbf{C}^{(d_s)}$. This finishes the proof.
\end{proof}

\subsection{Measures of maximal entropy}\label{ssection:mme} For every $\varphi \in \mathbf{C}^{(d_s)}$, there exists a natural measure $\mu_\varphi$ supported on $\mathscr{J}$, invariant under the $A$-action, which plays the role of a measure of maximal entropy.

\subsubsection{Definition of the measures}

 As $\varphi$ is a resonance with $1$-dimensional resonant space (Theorem~\ref{theorem:leading}, item~\ref{it:leading-critical-hypersurface}), there exists $m^s_\varphi \in \mc{D}'(\M, \Lambda^{d_s}(E_{s,\M}^* \oplus E_{u,\M}^*))$, a (nonzero) resonant state, such that
\begin{equation}
\label{equation:msphi}
(-\mathbf{X}-\varphi) m^s_\varphi = 0, \qquad \supp(m^s_\varphi) \subset \mc{T}_+, \qquad \WF(m^s_\varphi) \subset E_u^*.
\end{equation}

When $E_{s,\M}^*$ and $E_{u,\M}^*$ are smooth (i.e., coincide on the incoming/outgoing tails $\mc{T}_\pm$ with smooth subbundles of $T^*\M$ defined in a neighborhood of $\mc{T}_\pm$), a quick computation reveals that $m^s_\varphi$ has no component along $E_{s,\M}^*$:

\begin{lemma} \label{lemma:reduction-es}
If $E_{s,\M}^*$ and $E_{u,\M}^*$ are smooth, then $m^s_\varphi \in \mc{D}'(\M, \Lambda^{d_s} E_{u,\M}^*)$.
\end{lemma}

\begin{proof}
We write
\begin{equation}
\label{equation:decomposition-algebre-exterieure}
\Lambda^{d_s}(E_{s,\M}^*\oplus E_{u,\M}^*) = \oplus_{k=0}^{d_s} \Lambda^k E_{s,\M}^* \otimes \Lambda^{d_s-k} E_{u,\M}^*,
\end{equation}
where each bundle in the sum is invariant under the $A$-action. Hence, the resonance spectrum for the bundle $\Lambda^{d_s}(E_{s,\M}^*\oplus E_{u,\M}^*)$ is the union of all the resonance spectra for $k = 0,...,d_s$. Because the forward action of the flow on $E_{s,\M}^*$ is uniformly contracting, the resonances of $\Lambda^k E_{s,\M}^* \otimes \Lambda^{d_s-k} E_{u,\M}^*$ for $k \neq 0$ must be contained in $\{\Re(\s) \in \mathbf{C}^{(d_s),-}\}$. Hence $m^s_\varphi \in \mc{D}'(\M, \Lambda^{d_s} E_{u,\M}^*)$.
\end{proof}

By duality (see Remark \ref{remark:duality}), there exists a nonzero co-resonant state $m^u_\varphi \in \mc{D}'(\M, \Lambda^{d_u}E_{s,\M}^*)$ such that
\[
(+\mathbf{X}-\varphi) m^u_\varphi = 0, \qquad \supp(m^u_\varphi) \subset \mc{T}_-, \qquad \WF(m^u_\varphi) \subset E_s^*.
\]
Because they correspond to a leading resonance, these resonant states are distributions of order $0$ (see, e.g., \cite{Humbert-24}). Also note that $m^{s,u}_\varphi$ are only well-defined up to rescaling by an arbitrary positive constant. The following facts will not be needed but are simply stated for comparison. More generally, it can be shown that Lemma \ref{lemma:reduction-es} holds regardless of the smoothness of $E_s^*$ and $E_u^*$; further assuming that these bundles are orientable, it can also be shown that $m^{s,u}_\varphi$ are measures (i.e., nonnegative distributions of order $0$), see \cite{Humbert-24}. The measures $m^{s,u}_\varphi$ are called Burger--Roblin measures in the context of Anosov representations in \cite{ELO23}.

The wedge products
\[
\begin{split}
&\mu'_\varphi := m^u_\varphi \wedge m^s_\varphi \in\mc{D}'(\M, \Lambda^{d_\N-1} T^*\M), \\&\mu_\varphi := \vol_{\fa}\wedge \mu'_\varphi \in \mc{D}'(\M, \Lambda^{d_\M} T^*\M),
\end{split}
\]
where $\vol_{\fa}$ was introduced in \S\ref{sssection:distributions}, are well-defined by the wavefront-set characterization above and the product rule (see \cite[Lemma 4.3.1]{Lefeuvre-book}).

The following properties hold:

\begin{lemma}$\mathbf{X} \mu'_\varphi = \X\mu_\varphi= 0$, $\supp(\mu'_\varphi), \supp(\mu_\varphi) \subset \mathscr{J}$, $\WF(\mu'_\varphi), \WF(\mu_\varphi) \subset E_s^*\oplus E_u^*$.
\end{lemma}

\begin{proof}
The support of $\mu_\varphi$ and $\mu'_\varphi$ is contained in the intersection of the supports of $m^{s/u}_{\varphi}$. Since $\mc{T}_+\cap\mc{T}_-=\scrJ$, the claim is immediate. The claim on wavefront sets follows from standard results on the product of distributions (see \cite[Lemma 4.3.1]{Lefeuvre-book} for instance). Finally, we prove the invariance under the action:
\[
\X \mu'_\varphi = \X m^u_\varphi \wedge m^s_\varphi + m^u_\varphi \wedge \X m^s_\varphi = \varphi \otimes m^u_\varphi \wedge m^s_\varphi - \varphi \otimes m^u_\varphi \wedge m^s_\varphi = 0,
\]
and using \eqref{equation:dodo-dodo}:
\[
\X\mu_\varphi = \X \vol_{\fa} \wedge \mu'_\varphi + \vol_{\fa} \wedge \X \mu'_\varphi = 0.
\]
\end{proof}

We let
\[
|\mu_\varphi| \in \mc{D}'(\M, \Omega^1 \M)
\]
be the corresponding distributional section with values in the density bundle. Note that it is supported on $\scrJ$. It follows from Theorem \ref{theorem:measures-max-entropy}, item \emph{\ref{it:existence_BMM}} below, that $|\mu_\varphi|$ is a measure; in particular, it is independent of the extension of $\vol_{\fa}$ outside $\mathscr{J}$ (it only depends on the $0$-jet of $\vol_{\fa}$ on $\mathscr{J}$).

Let $\fh_\varphi := \ker \varphi \subset \fa$. As $\ker \varphi \cap \scrL = \{0\}$, it follows from Proposition \ref{proposition:proper-action-phi}, item (i), that the action of $H_\varphi := \exp(\fh_\varphi)$ on $\M$ is free and proper. The flow generated by $u_\varphi \in \scrL$ on $\M$ descends to a hyperbolic flow $(\phi_t^\varphi)_{t \in \R}$ on $\N_\varphi := \M/H_\varphi$ with trapped set $\scrK_\varphi := \scrJ/H_\varphi$ and periods $\{\varphi(\lambda(\gamma)) ~:~ \gamma \in \mathbf{P}\}$ (Proposition \ref{proposition:proper-action-phi}, item (iii)). Let $\pi_\varphi:\M\to\N_\varphi$ denote the quotient map and also write $\pi_\varphi$ for its restriction to $\scrJ$. This flow satisfies Assumptions $\hyperlink{AA2}{\rm(A2)-(A4)}$ (it is hyperbolic with a single basic set). It is conjugate, up to a time reparametrization, to the flow $(\phi_t)_{t \in \R}$ on $\scrK \subset \N$.

\subsubsection{Statement}

In the following statement, $\sigma : \N \to \M$ and $\sigma_{\varphi} : \N_\varphi \to \M$ are two smooth sections trivializing the bundle as in \S\ref{ssection:definition-free-abelian-cocycles}. The following description holds:

\begin{theorem}[Measures of maximal entropy] \label{theorem:measures-max-entropy} Assume that $\hyperlink{AA1}{\rm(A1)-(A4)}$ hold. Then:

\begin{enumerate}[label=\emph{(\roman*)}]

\item \label{it:existence_BMM} \emph{\textbf{Bowen-Margulis-Sullivan measures.}} For all $f \in C^0_{\comp}(\scrJ)$, the measure $|\mu_\varphi|$ can be written
\begin{equation}
\label{equation:ventilo}
(|\mu_\varphi|,f) = c \int_{\scrK} \int_{\fh} f(\tau_h \sigma(x)) ~\dd h \dd \omega_\varphi(x) = c' \int_{\scrK_\varphi} \int_{\fh_\varphi} f(\tau_h \sigma_\varphi(x)) ~\dd h \dd \nu_\varphi(x),
\end{equation}
for some constants $c,c' > 0$, where:
\begin{itemize}
\item $\nu_\varphi$ is a probability measure supported on $\scrK_\varphi$ and is the measure of maximal entropy of the flow $(\phi_t^\varphi)_{t \in \R}$;
\item $\omega_\varphi$ is a probability measure supported on $\scrK = \scrJ/H$ and is the equilibrium state of the potential $\z'w$ for the flow $(\phi_t)_{t \in \R}$, corresponding to the decomposition $\varphi = z_0\alpha_{\M} + \z'$ (with $z_0=\Pr(\z'w)$, the pressure of the potential $\z'w$).
\end{itemize}

\item \label{it:horocyclic-invariance} \emph{\textbf{Horocyclic invariance of Burger-Roblin measures.}} Suppose that $E_{u,\M} \to \M$ is a smooth vector bundle. Then for all $Z_u \in C^\infty(\M,E_{u,\M})$, $\mc{L}_{Z_u} m^s_\varphi = 0$.

\item\label{it:analytic_BMM} \emph{\textbf{Analyticity of Bowen-Margulis-Sullivan measures.}} If we normalize $\mu_\varphi$ in such a way that $c'=1$ in \eqref{equation:ventilo}, then the map
\[
\mathbf{C}^{(d_s)} \ni \varphi \mapsto \mu_\varphi \in \mc{D}'(\M,\Lambda^{d_\M}T^*\M)
\]
is analytic. If the cocycle has full rank, then $\mu_\varphi$ and $\mu_{\varphi'}$ are mutually singular whenever $\varphi \neq \varphi'$.

\end{enumerate}

\end{theorem}

Note that \eqref{equation:ventilo} is independent of the choice of trivializing sections $\sigma$ and $\sigma_\varphi$. In the following, we will simply write $\mu_\varphi$ in place of $|\mu_\varphi|$. It is likely that the horocyclic invariance (Theorem \ref{theorem:measures-max-entropy}, item \emph{\ref{it:horocyclic-invariance}}) could be established without the additional, highly restrictive assumption that $E_{u,\M}$ is smooth. This assumption, however, simplifies the proof. It poses no issue for applications to Anosov representations because the stable and unstable bundles are smooth.

\subsubsection{Proof} \label{sssection:proof-max-entropy} 

\begin{proof}[Proof of Theorem \ref{theorem:measures-max-entropy}]

\emph{\ref{it:existence_BMM}} Up to arguing on a double cover on which $E_{s/u,\N}^* \to \N$ are orientable, \eqref{equation:ventilo} boils down to
\[
\mu_\varphi= c \cdot \iota_{u_\varphi} \vol_\fa \wedge \pi_\varphi^* \nu_\varphi = c' \cdot \iota_{X_\M} \vol_{\fa} \wedge \pi^* \omega_\varphi.
\]
Let $\pi : \M \to \N$ be the footpoint projection. The Burger-Roblin states satisfy $(-\X-\varphi)m^s_\varphi = 0$ and $(+\X-\varphi)m^u_\varphi = 0$ (with conditions on their support and wavefront set). We write $\varphi = z_0 \alpha_{\M} + \z'$ (with $z_0 = \Pr(\z'w)$). Consequently, they can be written as $m^s_\varphi(x,h) = e^{-\z'(h)} \widetilde{m}_\varphi^s(x)$ and $m^u_\varphi(x,h)=e^{\z'(h)} \widetilde{m}_\varphi^u(x)$, where $\widetilde{m}_\varphi^s \in \mc{D}'(\N, \Lambda^{(d_s)}(E_{s,\N}^*\oplus E_{u,\N}^*))$ satisfies
\begin{equation}
\label{equation:covid1}
(-\mc{L}_{X_{\N}}-z_0+\z' w)\widetilde{m}_\varphi^s = 0, \quad \WF(\widetilde{m}_\varphi^s) \subset E_u^*, \quad \mathrm{supp}(\widetilde{m}_\varphi^s) \subset \mc{T}_+,
\end{equation}
and $\widetilde{m}_\varphi^u \in \mc{D}'(\N, \Lambda^{(d_u)}(E_{s,\N}^*\oplus E_{u,\N}^*))$ satisfies
\begin{equation}
\label{equation:covid2}
(+\mc{L}_{X_{\N}}-z_0+\z' w)\widetilde{m}_\varphi^u = 0, \quad \WF(\widetilde{m}_\varphi^u) \subset E_s^*, \quad \mathrm{supp}(\widetilde{m}_\varphi^u) \subset \mc{T}_-.
\end{equation}
It follows from \S\ref{sssection:ds-forms} that $\omega_\varphi := \alpha_{\N} \wedge \widetilde{m}_\varphi^u \wedge \widetilde{m}_\varphi^s$ is the equilibrium state of the potential $\z'w$ for the flow $(\phi_t)_{t \in \R}$. It is normalized (i.e., a probability measure) with a suitable normalization of $\widetilde{m}_\varphi^{s,u}$. To conclude, it suffices to observe that
\[
\mu_\varphi = \vol_{\fa} \wedge \mu'_\varphi = \iota_{X_\M} \vol_{\fa} \wedge \alpha_{\M} \wedge m^u_\varphi \wedge m^s_\varphi = \iota_{X_\M} \vol_{\fa} \wedge \pi^*\omega_{\varphi}.
\]

We now prove the second assertion in \ref{it:existence_BMM}, concerning the measure of maximal entropy. Let $\pi_\varphi : \mathscr{J} \to \mathscr{K}_\varphi := \mathscr{J}/H_\varphi$ be the projection. On $\scrK_\varphi$, the periods of the hyperbolic flow $(\phi_t^\varphi)_{t \in \R}$ generated by $u_\varphi$ are given by $\{\varphi(\lambda(\gamma)) ~:~\gamma \in \mathbf{P}\}$ and their exponential growth rate is $\delta(\varphi) = 1 = h_{\mathrm{top}}(\phi_1^\varphi)$. Let $X_\varphi$ be the generator of $(\phi_t^\varphi)_{t \in \R}$ and $\alpha_\varphi$ be the Anosov $1$-form associated with the flow (it is only well-defined over $\scrK_\varphi$). 

By construction, $\mu_\varphi = \vol_\fa\wedge\mu_\varphi'$, where $\mu_\varphi' = m^u_\varphi\wedge m^s_\varphi$. Using \eqref{equation:msphi} (and the analogous equation for $m^u_\varphi$), we find that $\X \mu_\varphi' = 0$ and $\mu'_\varphi \in \mc{D}'(\M,\Lambda^{d_\N-1}(E_{s,\M}^*\oplus E_{u,\M}^*))$. Thus $\mu'_\varphi$ is $H_\varphi$-invariant and horizontal, hence $H_\varphi$-basic. Consequently, $\mu_\varphi' = \pi_\varphi^*\nu'_\varphi$ for some distributional $(d_\N-1)$-form $\nu'_\varphi$ supported on $\mathscr{K}_\varphi$. We now claim that $\nu_\varphi:=\alpha_\varphi\wedge\nu'_\varphi$ is (up to rescaling) the measure of maximal entropy for the flow generated by $u_\varphi$ on $\mathscr{K}_\varphi$. 

 Since $(-\X-\varphi)m^s_\varphi = 0$, we find that for all $h \in \ker \varphi$, $-\X(h) m^s_\varphi = 0$ and $(-\X(u_\varphi)-1)m^s_\varphi = 0$. Lemma~\ref{lemma:reduction-es} also gives $\iota_{\X(h)}m^s_\varphi=0$ for $h\in\ker\varphi$; hence $m^s_\varphi$ is $H_\varphi$-basic and descends to a form $\tilde{m}^s_\varphi$ on $\N_\varphi := \M/H_\varphi$ such that $(-\mc{L}_{X_\varphi}-1)\tilde{m}^s_\varphi =(-\mc{L}_{X_\varphi}-h_{\mathrm{top}}(\phi_1^\varphi))\tilde{m}^s_\varphi = 0$, where $X_\varphi$ is the generator of $(\phi_t^\varphi)_{t \in \R}$. That is, $\tilde{m}^s_\varphi$ is the resonant state associated with the measure of maximal entropy. The same argument applies to the co-resonant state $\tilde{m}^u_\varphi$, and $\nu_\varphi=\alpha_\varphi\wedge\nu'_\varphi = \alpha_\varphi \wedge \tilde{m}^u_\varphi \wedge \tilde{m}^s_\varphi$ is the measure of maximal entropy. \\

\emph{\ref{it:horocyclic-invariance}} Let
\[
d_u : C^\infty(\M,\Lambda^{d_s} E_{u,\M}^*) \to C^\infty(\M,\Lambda^{d_s} E_{u,\M}^* \otimes E_{s,\M}^*)
\]
be the component of the exterior derivative mapping $\Lambda^{d_s} E_{u,\M}^*$ to the corresponding summand of the decomposition of $\Lambda^{d_s} E_{u,\M}^* \otimes T^*\M$ into $A$-invariant summands, similarly to \eqref{equation:decomposition-algebre-exterieure} (recall that $E_{s,\M}^*(E_{s,\M}\oplus E_{0,\M})=0$). Then $\X d_u = d_u \X$ by the flow-invariance of the bundles. As a consequence, $d_u m^s_\varphi \in \mc{D}'(\M, \Lambda^{d_s} E_{u,\M}^* \otimes E_{s,\M}^*)$ is a resonant state for the resonance $\varphi$. However, by Theorem \ref{theorem:leading}, item~\ref{it:leading-intermediate-hypersurfaces}, the spectrum on this bundle is contained in $\mathbf{C}^{(d_s),-}$, so $d_u m^s_\varphi=0$. That is, for all $Z_u \in C^\infty(\M, E_{u,\M})$, $\iota_{Z_u} d_u m^s_\varphi = \iota_{Z_u} d m^s_\varphi = 0$. By Cartan's formula, we find:
\[
\mc{L}_{Z_u} m^s_\varphi = \iota_{Z_u} d m^s_\varphi + d \iota_{Z_u} m^s_\varphi = 0,
\] 
using Lemma \ref{lemma:reduction-es}. \\

\emph{\ref{it:analytic_BMM}} This is an immediate consequence of \S\ref{sssection:ds-forms}. The mutual singularity of the equilibrium states follows because the full-rank assumption implies that $\mathbf{s}'_1(w)$ is not cohomologous to $c + \mathbf{s}'_2(w)$ for any constant $c \in \R$.
\end{proof}

\section{Local mixing}

\label{section:decay}

Throughout this section, we work under Assumptions $\hyperlink{AA1}{\rm(A1)-(A4)}$. Fix $\varphi\in\mathbf C^{(d_s)}$.

\subsection{Statement}

We are interested in the ergodic properties of the measures of maximal entropy defined above. We normalize the measure $\mu_\varphi$ such that $c=1$ in \eqref{equation:ventilo}. The probability measure $\nu_\varphi$ (resp. $\omega_{\varphi}$) on $\scrK_\varphi$ (resp. $\scrK$) is ergodic because the corresponding flow $(\phi_t^\varphi)_{t \in \R}$ (resp. $(\phi_t)_{t \in \R}$) is topologically transitive (Assumption \hyperlink{AA4}{\rm(A4)}); see, for instance, \cite[Corollary 7.3.8]{Fisher-Hasselblatt-19}. We let $(e^{tu_\varphi})_{t \in \R}$ be the flow generated by $u_\varphi$ on $\M$.

It is known that if $\hyperlink{AA1}{\rm(A1)-(A4)}$ hold and the cocycle has full rank and is non-arithmetic, then there exists a constant $\kappa > 0$ such that, as $t\to+\infty$, for all $f,g \in C^0_{\comp}(\mathscr{J})$
\begin{equation}
\label{equation:leading-term-asymptotic}
\begin{split}
\int_\mathscr{J} f\circ e^{tu_\varphi} \cdot g~ \dd \mu_\varphi &= \kappa \cdot t^{-k/2} \left(\int_\mathscr{J} f ~\dd \mu_\varphi \int_\mathscr{J} g ~\dd \mu_\varphi + o(1)\right).
\end{split}
\end{equation}
See \cite{Sambarino-15} or \cite[Appendix B]{Sambarino-24}. The constant $\kappa$ is explicit and given by the determinant of a covariance matrix (see \eqref{equation:hessian-phase}). Under the stronger assumption that the Lyapunov spectrum is Diophantine (Definition \ref{definition:diophantine-lyapunov}), one obtains a sharp result for the decay of correlations/local mixing:

\begin{theorem}[Sharp local mixing/decay of correlations]
\label{theorem:mixing-sharp}
Assume that $\hyperlink{AA1}{\rm(A1)-(A4)}$ hold and that the cocycle has full rank and is Diophantine. Then for all $f,g \in C^\infty_{\comp}(\M)$ and $t \geq 1$,
\[
\begin{split}
\int_\mathscr{J} f\circ e^{tu_\varphi} \cdot g~ \dd \mu_\varphi &= \kappa \cdot t^{-k/2}\left(\int_\mathscr{J} f ~\dd \mu_\varphi \int_\mathscr{J} g ~\dd \mu_\varphi + a(t,f,g)\right),
\end{split}
\]
where $a : [1,+\infty) \times C^\infty_{\comp}(\M) \times C^\infty_{\comp}(\M)$ is smooth in all variables, linear in the second and third variables, and
\[
a(t,f,g) \sim \sum_{j \geq 1} t^{-j} C_j(f,g),
\]
where each $C_j : C^\infty_{\comp}(\M) \times C^\infty_{\comp}(\M) \to \C$ is a continuous bilinear form. The asymptotic expansion means that, for every $N\in\Z_{\geq 0}$, for every compact subset $K \subset \M$, there exist $C := C(N,K) > 0$ and $\ell := \ell(N,K) \in\Z_{\geq 0}$ such that for all $f,g \in C^\ell_{\comp}(\M)$ with compact support in $K$:
\[
\left|a(t,f,g)- \sum_{j=1}^{N} t^{-j} C_j(f,g)\right| \leq C t^{-(N+1)}\|f\|_{C^{\ell}(\M)}\|g\|_{C^{\ell}(\M)}, \qquad \forall t \geq 1.
\]
\end{theorem}

We believe that the Diophantine condition in Theorem \ref{theorem:mixing-sharp} is not needed, and that the non-arithmetic condition should suffice.

\subsection{Proof} 

\label{ssection:sharp-decay}

\begin{proof}[Proof of Theorem \ref{theorem:mixing-sharp}]
For $f, g \in C^\infty_{\comp}(\M)$ and $(x,h) \in \N \times \fh$, Fourier inversion gives
\[
f(x,h) = \dfrac{1}{(2\pi)^k} \int_{\fh^*} e^{i \s' h} \widehat{f}_{\s'}(x) \dd \s', \qquad \widehat{f}_{\s'}(x) := \int_{\fh} e^{-i\s' h} f(x,h) \dd h.
\]
Here, $\dd h$ is the Lebesgue measure on $\fh$ determined by $(\iota_{X_\M}\vol_\fa)|_\fh$, and $\dd\s'$ is the dual Lebesgue measure on $\fh^*$. Let $a \in \fa$, $a = tX_{\M} + v$, where $t \in \R, v \in \fh$.
Letting $X_{\s'} := X_{\N} + i\s' w$ for $\s' \in \fh^*$, we find:
\[
\begin{split}
e^{a}f(x,h) & = \dfrac{1}{(2\pi)^k} \int_{\s' \in \fh^*} e^{i \s' h} e^{i\s'v}e^{i \s' \int_0^t w(\phi_s x)\dd s} \widehat{f}_{\s'}(\phi_t x) \dd \s' \\
& = \dfrac{1}{(2\pi)^k} \int_{\s' \in \fh^*} e^{i \s' h} e^{i\s'v} \left(e^{tX_{\s'}} \widehat{f}_{\s'}\right)(x) \dd \s'.
\end{split}
\]
By Theorem \ref{theorem:measures-max-entropy}, item \emph{\ref{it:existence_BMM}}, the measure $\mu_\varphi$ is given by the product $\omega_{\varphi} \otimes \dd h$ in the $(x,h)$ coordinates. By Parseval's identity, we obtain:
\begin{equation}
\label{equation:fourier-etx}
\langle e^{a}f,g\rangle_{L^2(\scrJ,\mu_\varphi)} = \dfrac{1}{(2\pi)^k} \int_{\s' \in \fh^*} e^{i\s' v} \langle e^{tX_{\s'}} \widehat{f}_{\s'}, \widehat{g}_{\s'}\rangle_{L^2(\N,\omega_\varphi)} \dd \s'.
\end{equation}
This corresponds to the correlations in Theorem \ref{theorem:mixing-sharp} with the difference that this is now sesquilinear in $g$. However, up to changing $g$ by $\bar g$, it suffices to prove an expansion for $\langle e^{tu_\varphi}f,g\rangle_{L^2(\scrJ,\mu_\varphi)}$. 

Let $\chi \in C^\infty_{\comp}(\fh^*)$ be a smooth nonnegative bump function equal to $1$ on the ball of radius $\eps_0 > 0$ centered at $0 \in \fh^*$, where $\eps_0 > 0$ will be chosen small enough, and equal to $0$ outside the ball of radius $2\eps_0$. We decompose $u_\varphi = q_\varphi X_\M + v_\varphi$ with $q_\varphi = \alpha_\M(u_\varphi) > 0$, $v_\varphi \in \fh$. Given $f,g \in C^\infty_{\comp}(\mathscr{J})$, applying \eqref{equation:fourier-etx}, we find:
\[
\begin{split}
\langle e^{tu_\varphi}f,g\rangle_{L^2(\scrJ,\mu_\varphi)} & = \dfrac{1}{(2\pi)^k} \underbrace{\int_{\s' \in \fh^*} e^{i\s'(tu_\varphi)} \langle e^{tq_\varphi X_{\s'}} \widehat{f}_{\s'}, \widehat{g}_{\s'}\rangle_{L^2(\N,\omega_\varphi)} \chi(\s') \dd \s'}_{=: \text{(I)}} \\
&+ \dfrac{1}{(2\pi)^k} \underbrace{\int_{\s' \in \fh^*} e^{i\s'(tu_\varphi)} \langle e^{t q_\varphi X_{\s'}} \widehat{f}_{\s'}, \widehat{g}_{\s'}\rangle_{L^2(\N,\omega_\varphi)} (1-\chi(\s'))\dd \s'}_{=: \text{(II)}}.
\end{split}
\]
We deal with the two terms separately. \\

\emph{Term $\mathrm{(II)}$.} We claim that, for every $N\in\mathbb N$, there exist $C > 0$ and $\ell,N'\in\mathbb N_0$ (depending on $N$) such that, for all $|\s'| \geq \eps_0$ and $t \geq 0$,
\begin{equation}
\label{equation:rapid-decay}
|\langle e^{tX_{\s'}} \widehat{f}_{\s'}, \widehat{g}_{\s'}\rangle_{L^2(\N,\omega_{\varphi})}| \leq C \langle t \rangle^{-N} \langle \s' \rangle^{N'} \|\widehat{f}_{\s'}\|_{C^\ell(\N)} \|\widehat{g}_{\s'}\|_{C^\ell(\N)}.
\end{equation}
Here, $\langle y\rangle:=(1+|y|^2)^{1/2}$.
Assume for now that \eqref{equation:rapid-decay} holds. We then find that if
\begin{equation}
\label{equation:decay-fourier}
\|\widehat{f}_{\s'}\|_{C^\ell(\N)} \leq C \langle \s' \rangle^{-(N'+k+1)/2}
\end{equation}
for another constant $C > 0$, with the same estimate for $\widehat{g}_{\s'}$, then
\[
|\langle e^{tX_{\s'}} \widehat{f}_{\s'}, \widehat{g}_{\s'}\rangle_{L^2(\N,\omega_{\varphi})}| \leq C \langle t \rangle^{-N} \langle \s' \rangle^{-(k+1)}.
\] 
The estimate \eqref{equation:decay-fourier} holds for all $f,g \in C^\infty_{\comp}(\mathscr{J})$ and, more generally, for a broader class of functions. As a consequence, integrating over $\mathfrak{h}^*$, we find that
\[
\mathrm{(II)} = \int_{\s' \in \fh^*} e^{i\s'(tu_\varphi)} \langle e^{tq_\varphi X_{\s'}} \widehat{f}_{\s'}, \widehat{g}_{\s'}\rangle_{L^2(\N,\omega_\varphi)} (1-\chi(\s'))\dd \s' = \mc{O}(\langle t \rangle^{-N}).
\]

To prove \eqref{equation:rapid-decay}, we claim that it suffices to adapt the arguments in \cite{Dolgopyat-98-2}.
Choose $\gamma_1,\ldots,\gamma_M\in\mathbf P$ such that $\{\lambda(\gamma_1),\ldots,\lambda(\gamma_M)\}$ is Diophantine. In the notation of \cite{Dolgopyat-98-2}, we replace $\tau$ by a vector-valued function $\boldsymbol{\tau} : \Sigma \to \fa=\R^{k+1}$ and $b$ by a vector-valued spectral parameter $\boldsymbol{b} \in \R^{k+1}$, where $\Sigma$ is the symbolic shift associated with a Markov partition. In our case, $\boldsymbol{\tau} = (\tau,-W)$, where $\tau$ is the return time to the Markov partition, and $W : \Sigma \to \fh$ is the function obtained from $w$ by integration along flow lines between two consecutive return times to the Markov partition; the spectral parameter is $\boldsymbol{b} = (b, \s')$, where $b$ is the Fourier parameter corresponding to the time direction (as in \cite{Dolgopyat-98-2}) and $\s'$ is the Fourier parameter arising from the Fourier decomposition above.

The estimate in \cite[Item (v), Theorem 1]{Dolgopyat-98-2} holds under the Diophantine condition and implies \eqref{equation:rapid-decay} by repeating the argument in \cite[Section 10]{Dolgopyat-98-2}. Indeed, arguing as in \cite[Sections 8 and 13]{Dolgopyat-98-2}, if \cite[Item (v), Theorem 1]{Dolgopyat-98-2} does not hold, then one can construct a sequence of approximate eigenfunctions for a family of spectral parameters $|\boldsymbol{b}| \to +\infty$ with approximately constant modulus. As in \cite[Section 13]{Dolgopyat-98-2}, this results in the estimate
\[
|e^{i n(\beta,\boldsymbol{b}) \boldsymbol{b} \cdot \boldsymbol{\ell}_i}-1| \leq |\boldsymbol{b}|^{-\alpha},
\]
for all $1 \leq i \leq M$, where $n(\beta,\boldsymbol{b}) := [\beta \log|\boldsymbol{b}|]$, and $\boldsymbol{\ell}_i$ is now equal to the integral of $\boldsymbol{\tau}$ along the orbit of the periodic point $\omega_i \in \Sigma$, which corresponds to a point $x_i$ on the periodic orbit $\gamma_i$. The exponent $\alpha > 0$ may be chosen freely. After possibly changing the exponent $\alpha > 0$, the previous estimate may be rewritten as
\[
|e^{i \boldsymbol{b}'\cdot \boldsymbol{\ell}_i}-1| \leq |\boldsymbol{b}'|^{-\alpha'},
\]
Here, $\boldsymbol b':=n(\beta,\boldsymbol b)\boldsymbol b$; for a diverging family $|\boldsymbol{b}'| \to +\infty$, the periodic point $\omega_i$ of the shift corresponds to the periodic orbit $\gamma_i \subset \mathscr{K}$ and 
\[
\boldsymbol{\ell}_i = \ell_{\gamma_i} X_{\M} - \int_0^{\ell_{\gamma_i}} w(\phi_s x_i) \dd s = \lambda(\gamma_i). 
\]
Choose the approximate-eigenfunction accuracy so that $\alpha'>\nu$, where $\nu$ is the exponent of the Diophantine witness. This contradicts the Diophantine assumption and proves \eqref{equation:rapid-decay}. We also refer the reader to \cite[Section 2.2]{Field-Melbourne-Torok-07} where a similar argument is given. \\

\emph{Term $\mathrm{(I)}$.} Let $\s' \mapsto \sigma(\s') \in \C$ be the holomorphic function defined for $\s' \in V$, a neighborhood of $0$ in $\fh^*_{\C}$, such that $\mathbf{C}^{(d_s)}_{\C}$ coincides near $\varphi$ with $\varphi + \{(\sigma(\s'),\s') ~:~ \s' \in V\}$. This parametrization is well-defined for $\s'$ near $0$ because the leading resonance is simple (Theorem \ref{theorem:leading}, item (i)). In the first step above, we then choose $\eps_0 > 0$ such that the ball of radius $2\eps_0$, centered at $0$ is contained in $V \cap \fh^*$.

For $\s$ close to $0$, we can write:
\[
P_+(\s+\varphi)^{-1} = P_+^{\mathrm{hol}}(\s)-\dfrac{\Pi^+_{\s'}}{s_0-\sigma(\s')},
\]
where $\s \mapsto P_+^{\mathrm{hol}}(\s)$ is holomorphic as a bounded operator on compactly supported anisotropic distributions in $\mc{D}'_{E_u^*}$ (Theorem \ref{theorem:rt-anosov2}, item (i)), and
\[
\Pi^+_{\s'} =-\dfrac{1}{2i\pi} \int_\gamma P_+(s_0'\alpha_{\M}+\s'+\varphi)^{-1} \dd s_0'
\]
denotes the corresponding spectral projector onto the resonance $\varphi + \sigma(\s')\alpha_{\M} + \s'$ of $P_+$, with $\gamma$ a small counterclockwise-oriented contour around $0$. Similarly, as the poles of $P_-$ are equal to those of $P_+$ (see Remark \ref{remark:duality}), we can write
\begin{equation}
\label{equation:decalage}
P_-(\s+\varphi)^{-1}=P_-^{\mathrm{hol}}(\s)-\dfrac{\Pi^-_{\s'}}{s_0-\sigma(\s')}.
\end{equation}

Writing $\varphi=z_0\alpha_{\M}+\z'$, one has $\sigma(\s')=\Pr((\z'+\s')w)-\Pr(\z'w)$, where $\Pr(\bullet)$ stands for the analytic continuation of the pressure functional to complex-valued potentials. By \eqref{equation:derivative1} and \eqref{equation:derivative2}, $\fh^* \ni \s' \mapsto \sigma(i\s')$ admits the following expansion near $0$:
\begin{equation}
\label{equation:asymptotic-exp-sigma}
\sigma(i\s') = 0 + i \int_{\scrK} \s'w \dd \omega_\varphi - \dfrac{1}{2} \mathrm{Var}_{\omega_{\varphi}}(\s' w) + \mc{O}(|\s'|^3),
\end{equation}
and $\mathrm{Var}_{\omega_{\varphi}}(\s' w) \geq c|\s'|^2 > 0$ for $\s'\neq0$ as the cocycle has full rank (Theorem \ref{theorem:leading}, item (ii)). In particular, for $\s'$ near $0$, $\Re(\sigma(i\s')) \leq -c|\s'|^2$. 

We now claim that for every $N\in\mathbb N$, there exists $\ell\in\mathbb{Z}_{\geq 0}$ such that for all $\s' \in \fh^*, |\s'| \leq \eps_0$, for all $f,g \in C^\ell(\N)$, one has
\begin{equation}
\label{equation:mamy}
 \langle e^{tX_{\s'}} f, g\rangle_{L^2(\N,\omega_\varphi)} = e^{\sigma(i\s')t}\big\langle\!\langle  \Pi^-_{i\s'}(f m_u^{(d_u)}),g m_s^{(d_s)} \big\rangle\!\rangle_{\N}  + \mc{O}(\langle t \rangle^{-N} \|f\|_{C^\ell} \|g\|_{C^\ell}),
\end{equation}
where the remainder is uniform with respect to $\s'$, and, given $\beta_1 \in \mc{D}'_{E_u^*}(\N, \Lambda^k (E^*_{s,\N} \oplus E^*_{u,\N}))$ and $\beta_2 \in \mc{D}'_{E_s^*}(\N,\Lambda^{d_{\N}-k-1} (E^*_{s,\N} \oplus E^*_{u,\N}))$, we introduce, similarly to \eqref{equation:pairing}:
\begin{equation}
\label{equation:echographie}
\big\langle\!\langle \beta_1, \beta_2 \big\rangle\!\rangle_{\N} := \int_{\N} \alpha_{\N} \wedge \beta_1 \wedge \overline{\beta_2}.
\end{equation}
Note that
\[
\begin{split}
\big\langle\!\langle \Pi^-_{\mathbf{0}}(f m_u^{(d_u)}),g m_s^{(d_s)} \big\rangle\!\rangle_{\N} & = \big\langle\!\langle f m_u^{(d_u)}, m_s^{(d_s)} \big\rangle\!\rangle_{\N} \big\langle\!\langle  m_u^{(d_u)}, gm_s^{(d_s)} \big\rangle\!\rangle_{\N} \\
&= \int_\N f \dd \omega_\varphi \int_\N \overline{g} \dd\omega_\varphi
\end{split}
\]
by \eqref{equation:muu}. We now justify \eqref{equation:mamy}. First, it follows from an adaptation of \cite{Dolgopyat-98-2} that
\begin{equation}
\label{equation:dolgo-bound}
 \langle e^{tX_{\s'}} f,g\rangle_{L^2(\N,\omega_\varphi)} = e^{\theta(\s')t}\left(A_{\s'}(f,g) + \mc{O}(\langle t \rangle^{-N} \|f\|_{C^\ell} \|g\|_{C^\ell})\right),
 \end{equation}
 where the remainder is uniform with respect to $|\s'| \leq \eps_0$ and $\theta(\s') = \Pr((\z'+i\s')w)-\Pr(\z'w)=\sigma(i\s')$, $\theta(\mathbf{0})=0$ and $A_{\mathbf{0}}(f,g)=\int_{\N} f  \dd\omega_\varphi \int_{\N} \overline{g} \dd\omega_\varphi$. This is merely a version of the arguments in \cite{Dolgopyat-98-2} involving an extra small parameter $\s'$; see also \cite{Petkov-Stoyanov-16} where a similar analysis for transfer operators with two parameters is carried out. To identify $A_{\s'}(f,g)$ with $\big\langle\!\langle  \Pi^-_{i\s'}(f m_u^{(d_u)}),g m_s^{(d_s)} \big\rangle\!\rangle_{\N}$, one can argue as follows. Writing
 \[
 \langle e^{tX_{\s'}} f,g\rangle_{L^2(\N,\omega_\varphi)}  = \big\langle\!\langle e^{tP_-(\z+i\s')}(f m_u^{(d_u)}),g m_s^{(d_s)} \big\rangle\!\rangle_{\N},
 \]
 we obtain that \eqref{equation:dolgo-bound} provides a high-frequency estimate (as $s_0 \in \R, s_0 \to \pm \infty$) on the resolvent
 \[
 (P_-(\z+i\s')-is_0)^{-1} = P_-(\z+i(s_0,\s'))^{-1} = -\int_0^{+\infty} e^{tP_-(\z+i(s_0,\s'))} \dd t
 \]
 as a bounded map from $C^\ell(\N)\cdot m_u^{(d_u)}$ to the dual of $C^\ell(\N)\cdot m_s^{(d_s)}$. In turn, using the expansion \eqref{equation:decalage} and a standard contour deformation argument (see, for instance, the proof of \cite[Theorem 4.13]{Cekic-Lefeuvre-Munoz-26}), one obtains \eqref{equation:mamy}. In particular, this shows that $A_{\s'}(f,g)=\big\langle\!\langle  \Pi^-_{i\s'}(f m_u^{(d_u)}),g m_s^{(d_s)} \big\rangle\!\rangle_{\N}$.

In \eqref{equation:mamy}, after integration with respect to $\s'$ (over the support of $\chi$), the last term contributes an $\mc{O}(\langle t\rangle^{-N})$ term. Hence,
\begin{equation}
\label{equation:done}
\text{(I)} = \int_{\fh^*} e^{(i\s'+\sigma(i\s')\alpha_{\M})(tu_\varphi)}\big\langle\!\langle  \Pi^-_{i\s'}(\widehat{f_{\s'}} m_u^{(d_u)}),\widehat{g_{\s'}} m_s^{(d_s)} \big\rangle\!\rangle_{\N}  \chi(\s')\dd\s'+ \mc{O}(\langle t \rangle^{-N}).
\end{equation}
To apply the stationary phase lemma (see, for instance, \cite[Theorem 2.3]{Melin-Sjostrand-75}), we introduce the phase
\[
\Phi(\s') := (i\s'+\sigma(i\s')\alpha_{\M})(u_\varphi)
\]
It satisfies $\Phi(0)=0$ and $d\Phi(0)=0$ because $u_\varphi$ is in the kernel of the linear forms in $T_\varphi \mathbf{C}^{(d_s)}_{\C}$. In addition, for $\s'$ near $0$, $\Re(\Phi(\s')) \leq -c|\s'|^2$. The Hessian is given by
\[
d^2\Phi(0)(\eta,\eta)=-\mathrm{Var}_{\omega_{\varphi}}(\eta w) \alpha_{\M}(u_\varphi) < 0, \qquad \eta\in\fh^*\setminus\{0\}.
\]
A direct computation shows that
\[
\begin{split}
\big\langle\!\langle \Pi^-_{\mathbf{0}}(\widehat f_{0} m_u^{(d_u)}),\widehat g_{0} m_s^{(d_s)} \big\rangle\!\rangle_{\N} & = \int_{\scrK} \widehat f_0 \dd \omega_{\varphi} \int_{\scrK} \widehat g_0 \dd \omega_{\varphi} = \int_{\scrJ} f \dd \mu_\varphi \int_{\scrJ} g \dd\mu_\varphi.
\end{split}
\]
Moreover,
\begin{equation}
\label{equation:hessian-phase}
|\det d^2\Phi(0)|^{1/2} = \alpha_{\M}(u_\varphi)^{k/2} (\det H)^{1/2},
\end{equation}
where $H$ is the symmetric covariance matrix defined by $H(\eta,\eta) := \mathrm{Var}_{\omega_{\varphi}}(\eta w)$ for $\eta \in \fh^*$, and $\det H$ is relative to the dual measure $\dd\s'$.

The cutoff amplitude in \eqref{equation:done} is smooth with respect to $\s'$ with derivatives controlled by finitely many seminorms of $f$ and $g$. Applying the stationary phase lemma, we find
\[
\begin{split}
&\langle e^{tu_\varphi}f,g\rangle_{L^2(\scrJ,\mu_\varphi)}\\
& = \dfrac{t^{-k/2}}{(2\pi)^{k/2}|\det d^2\Phi(0)|^{1/2}}\left(\big\langle\!\langle \Pi^-_{\mathbf{0}}(\widehat f_{0} m_u^{(d_u)}),\widehat g_{0} m_s^{(d_s)} \big\rangle\!\rangle_{\N} + \sum_{j=1}^{N-1} t^{-j} C_j(f,g) + \mc{O}(t^{-N})\right) \\
& = \dfrac{t^{-k/2}}{(2\pi)^{k/2}|\det d^2\Phi(0)|^{1/2}}\left(\int_\mathscr{J} f ~\dd \mu_{\varphi} \int_\mathscr{J} g ~\dd \mu_{\varphi} + \sum_{j=1}^{N-1} t^{-j} C_j(f,g) + \mc{O}(t^{-N})\right).
\end{split}
\]
This concludes the proof.
\end{proof}

\section{Structure theory and dynamics for Anosov subgroups}

\label{section:structure}

In this section, we recall the necessary material from the theory of semisimple Lie groups (see \S\ref{ssection:structure-theory}) and present the locally homogeneous dynamics we will study (see \S\ref{ssec:locally-homogeneous-dynamics}). Discrete subgroups are discussed in \S\ref{ssection:discrete-subgroups}.

Our setting is similar to that of \cite{Gueritaud-Guichard-Kassel-Wienhard-17}. We consider a noncompact connected semisimple linear real algebraic group $\mathbf{G}$, and more specifically its real points $G =\mathbf{G}(\R)$. Here, connected means Zariski-connected: the complex points $\mathbf{G}(\C)$ form a connected complex Lie group in the usual topological sense. In particular, our results apply to the following classical groups: $\mathrm{SL}_n$, $\mathrm{Sp}_{2n}$, $\mathrm{SO}_{p,q}$, $\mathrm{SU}_{p,q}$, $\mathrm{Sp}_{p,q}$. The arguments would be unchanged if $G$ consisted of a finite number of components of $\mathbf{G}(\R)$ instead (which is the assumption made in \cite{Gueritaud-Guichard-Kassel-Wienhard-17}).

\subsection{Structure theory of real reductive groups} \label{ssection:structure-theory}

In this subsection, we review the structure theory of real reductive groups. It will be used later to study the dynamics related to Anosov subgroups. Throughout, we formulate almost all results for a real reductive group $G$ in the sense of Knapp \cite[\S VII.2]{Knapp-96}.
This includes the groups of real points considered in this paper. The greater generality also allows us to apply the same structure theory to the Langlands $M$-factors of parabolic subgroups: these are real reductive in Knapp's sense, but may be disconnected and need not be semisimple, even when the ambient group is connected and semisimple \cite[\S VII.7, Proposition 7.82]{Knapp-96}.
Unless otherwise stated, all results can be found in, or directly derived from \cite[\S VI \& \S VII]{Knapp-96}.

\subsubsection{Notation} The Lie algebra $\fg$ of $G$ is endowed with a \emph{Cartan involution} $\theta$, fixing a subalgebra $\fk$ and acting as $-1$ on a complement $\fs$ such that
	\[
	 [\fs,\fs]\subset \fk, \qquad [\fs,\fk]\subset\fs.
	 \]
	 Correspondingly, there is a global Cartan involution $\Theta:G\to G$ whose fixed-point subgroup $K$ is a maximal compact subgroup of $G$ with Lie algebra $\fk$. The Lie algebra $\fg$ is endowed with a $\theta$- and $\mathrm{Ad}(G)$-invariant non-degenerate bilinear form $\beta$ so that
	\begin{equation}
	\label{equation:produit-scalaire}
	\langle \bullet, \bullet \rangle := - \beta(\bullet, \theta \bullet)
	\end{equation}
	is a scalar product. For semisimple Lie groups $G$, $\beta$ can be chosen to be the Killing form. We use it to define orthogonal complements in all our arguments. The $\operatorname{ad}$-action of elements of $\fs$ is self-adjoint, so all elements of $\fs$ are semisimple.
	
	We choose a maximal subalgebra $\fa\subset \fs$, which is necessarily abelian and consists only of semisimple elements. Correspondingly, we obtain a set of (restricted) roots $\Delta$, which form a (possibly non-reduced) \emph{root system} in $\fa^\ast$. Corresponding to the roots, we have a root space decomposition
	\[
	\fg = \fm \oplus \fa \oplus \bigoplus_{\alpha\in \Delta} \fg_\alpha, 
	\]
	where $\fm = \fz_\fk(\fa)$ and $\fm\oplus \fa = \fz_{\fg}(\fa)$. If the group $\mathbf{G}$ is \emph{split}, as in the $\mathrm{SL}_n$ case, the $\fg_\alpha$ are all lines, but this need not be the case in general.

\subsubsection{Weyl chambers and Weyl groups}\label{sssec:weyl-chambers}

For each $H\in\fa$, denote by $\Delta^{+,-,0}(H)$ the sets of roots that are positive/negative/zero on $H$. The equivalence relation $H\sim H'$ given by $\Delta^\pm(H)=\Delta^\pm(H')$ divides $\fa$ into \emph{generalized Weyl chambers} (or GWC), generically denoted $\scrC$, and we write $\Delta^{+,-,0}(H)= \Delta^{+,-,0}_\scrC$ whenever $H\in\scrC$. We denote the set of all generalized Weyl chambers by $\mathrm{GWC}$.

The open generalized Weyl chambers are the usual \emph{Weyl chambers}, generically denoted $\fapp$, with closure $\fa_+ := \overline{\fapp}$. No root can vanish on an open Weyl chamber, so corresponding to some choice of $\fapp$, we get positive roots $\Delta^+=\Delta^+_\fapp$. We can also choose a simple system $\Sigma\subset \Delta^+$, so that $\Delta^+ \subset \mathrm{span}_{\mathbb{N}}(\Sigma)$. If $\scrC'\subset \overline{\scrC}$, we write $\scrC'\leq \scrC$. For $\scrC_1\leq \fapp$ and $\scrC_2\leq\fapp$, $\mathrm{conv}(\scrC_1,\scrC_2)$ is the strict (i.e., with coefficients in $(0,1)$) convex envelope of $\scrC_1$ and $\scrC_2$. It is the smallest chamber $\scrC$ with $\scrC_1,\scrC_2\leq \scrC$.

 We set $A:=\exp(\fa)$ and define the \emph{Weyl group}
 \[
 W(G,A) \coloneqq N_K(\fa)/Z_K(\fa).
 \]
 Its elements are denoted $w=[\dot{w}]$, where $\dot{w}$ is a choice of representative in $N_K(\fa)$. It is a finite group acting on $\fa$ by $w\cdot H = \mathrm{Ad}(\dot{w})H$ through Euclidean isometries, and $W(G,A)$ is generated by reflections across the hyperplanes $\ker\alpha, \alpha \in \roots$. The action of $W(G,A)$ on $\fa$ also induces an action on the roots in $\Delta$ and hence on the faces $\scrC\in \mathrm{GWC}$. Given an open chamber $\fapp$, $w\mapsto w\cdot\fapp$ is a bijection between $W(G,A)$ and the set of open Weyl chambers. Additionally, for $\scrC\in\mathrm{GWC}$, there exists $w\in W(G,A)$ such that $w\cdot\scrC \subset \fap$.

Since $W(G,A)$ acts simply and transitively on open Weyl chambers, for each $\fapp$, there is a unique element $w_\fapp\in W(G,A)$ with $w_\fapp\cdot\fapp = -\fapp$, called the \emph{longest Weyl group element}. The \emph{opposite involution} is
\begin{equation}\label{equation:opposite-involution}
	\iota_\fapp \coloneqq -w_\fapp,
\end{equation}
which maps $\iota_\fapp : \fa \to \fa$ and fixes $\fapp$ setwise (since $\iota_\fapp(\fapp) = -w_\fapp(\fapp) = \fapp$).

A generalized Weyl chamber $\scrC\leq\fapp$ is called \emph{$\iota_\fapp$-symmetric} if $\iota_\fapp(\scrC)=\scrC$. For a general $\scrC$,
\[
	\scrC^\circ \coloneqq \mathrm{conv}(\scrC,\iota_\fapp(\scrC))
\]
is the smallest $\iota_\fapp$-symmetric Weyl chamber containing $\scrC$.

\subsubsection{Parabolic subgroups}\label{sssec:levi} Associated with each $\scrC\in \mathrm{GWC}$, we have the following subalgebras and subgroups of $G$:
\[
\begin{split}
	&\fa_\scrC = \mathrm{span}\ \scrC = \bigcap_{\alpha\in\Delta^0_\scrC} \ker \alpha, \\
	&\fn_\scrC^{+,-,0}= \bigoplus_{\alpha \in \Delta^{+,-,0}_\scrC} \fg_\alpha, \\
	&\fm_\scrC = \fm \oplus \fa_\scrC^\perp \oplus \fn_\scrC^0,
	\end{split}
\]
where $\fa_\scrC^\perp$ is the orthogonal complement of $\fa_\scrC$ in $\fa$ (computed with respect to \eqref{equation:produit-scalaire}), $\fn^+_\scrC$ and $\fn^-_\scrC$ are nilpotent subalgebras, and $\fm_\scrC$ is a reductive subalgebra. The subspace
\[
	\fp_\scrC := \fm_\scrC\oplus\fa_\scrC\oplus \fn_\scrC^+
\]
is a self-normalizing subalgebra, called a \emph{parabolic subalgebra}. Parabolic subgroups can then be defined as
\[
	P_\scrC = N_G(\fp_\scrC).
\]
As $\fp_\scrC$ is self-normalizing, $P_{\scrC}$ is also self-normalizing. The set of all subgroups of $G$ that are conjugate to $P_\scrC$ is parametrized by the (partial) flag manifold $G/P_\scrC$. We also have for $[\dot{w}]\in W(G,A)$,
 \begin{equation}\label{eq:w_conj_parabolics}
	\dot{w} P_\scrC \dot{w}^{-1} = P_{w\cdot\scrC}.
 \end{equation}
 We may thus parametrize the conjugacy classes of parabolic subgroups by faces of a chosen Weyl chamber $\fapp$. A more standard convention (see e.g. \cite{Gueritaud-Guichard-Kassel-Wienhard-17}) is to parametrize the parabolics by subsets $\vartheta=\Sigma\setminus\Delta_\scrC^0$. As the Weyl chambers play a central role in our arguments, the parametrization of parabolics by Weyl chambers $\scrC$ is, however, more natural in our setting.

We also need some additional subgroups of $G$:
\[
	A_\scrC :=\exp(\fa_\scrC),\qquad  A_\scrC^\perp := \exp(\fa_\scrC^\perp), \qquad L_\scrC = Z_G(\fa_\scrC).
\]
The groups $A_\scrC$ and $A_\scrC^\perp$ are Abelian. The group $L_\scrC$ is a \emph{Levi subgroup} of $P_\scrC$ and satisfies
\[
	 L_\scrC = P_\scrC \cap P_{-\scrC}, \qquad  L_\scrC = M_\scrC A_\scrC,
\]
where $M_\scrC$ is a reductive subgroup of $G$ with compact center, and Lie algebra $\fm_\scrC$. The subspace $\fa_\scrC^\perp \subset \fm_\scrC$ is a Cartan subspace for $M_\scrC$ as $\fa$ is one for $G$. In addition,
\[
	K_\scrC = K\cap M_\scrC
\]
is a maximal compact subgroup of $M_\scrC$. It can equally be expressed as $K_\scrC = Z_K(\fa_\scrC) = Z_K(H)$ for some $H\in \scrC$.

Finally, the \emph{unipotent radical} of $P_\scrC$ is given by
\[
	N_\scrC^\pm = \exp(\fn^\pm_\scrC).
\]
These unipotent subgroups are also characterized by their behavior under conjugation by $A$. That is, if $Q \subset N^-_{\scrC}$ is a compact subset and $(H_\ell)_{\ell \geq 0} \subset \fap$ is a sequence with $\alpha(H_\ell)\to\infty $ for all $\alpha\in \roots_\scrC^+$, then
\begin{equation}\label{eq:unipotent_contraction}
e^{H_\ell} Q e^{-H_\ell} \to_{\ell \to +\infty} \{e_G\}.
\end{equation}

\subsubsection{Langlands decomposition}

The \emph{Langlands decomposition} of a parabolic subgroup $P_\scrC$ gives rise to the following diffeomorphisms:
\begin{equation}
\label{eq:langlands}
\begin{split}
	M_\scrC \times A_\scrC \times N_\scrC^+ \to P_\scrC,\quad (m,a,n)\mapsto man,\\
	N_\scrC^+ \times A_\scrC \times M_\scrC	\to P_\scrC,\quad (n,a,m)\mapsto nam.
\end{split}
\end{equation}
As $M_\scrC$ is itself a reductive group, all the structure theory of $G$ applies to $M_\scrC$ as well. The structure theory of $M_\scrC$ can be described consistently with that of the ambient group $G$. Because this compatibility will be crucial in some of our arguments, we briefly describe these relations.

The maximal compact subgroup of $M_\scrC$ has already been identified above as $K_\scrC$, and its Cartan subgroup is given by $A_\scrC^\perp \coloneqq \exp(\fa_\scrC^\perp)$. The root system of $M_\scrC$ is given by $\Delta^0_\scrC$ and accordingly the Weyl group in $M_\scrC$ is given by 
\begin{equation}\label{eq:weylgrp_of_MscrC}
	W(M_\scrC, A_\scrC^\perp) = \langle \Delta^0_\scrC \rangle.
\end{equation}
Moreover, the action of any $w\in \langle\roots^0_\scrC \rangle$ on $\roots\subset \fa^*$ leaves the splitting $\roots = \roots^+_\scrC \cup \roots^0_\scrC\cup \roots^-_\scrC$ invariant.
As a consequence, we can write the orthogonal projection $\pi_\scrC$ from $\fa$ to $\fa_\scrC$ as
\begin{equation}\label{eq:stab-Weyl-and-projector}
\pi_\scrC= \frac{1}{|\langle \Delta^0_\scrC\rangle|}\sum_{w\in \langle \Delta^0_\scrC\rangle} w.
\end{equation}
In particular, $\pi_\scrC(wH)= \pi_\scrC(H)$ for any $w\in \langle\roots_\scrC^0\rangle$. This implies that
\begin{equation}\label{eq:pi_scrC_fapp}
 \text{if } \fapp\geq \scrC \text{, then }\pi_\scrC(\fapp) = \scrC.
\end{equation}
Additionally, the following holds:
\begin{lemma}\label{lem:adjacent_open_weyl_chamber}
For any two open Weyl chambers $\fapp \geq\scrC$ and $\fapp'\geq \scrC$, there exists a unique $w\in\langle \Delta^0_\scrC \rangle$ so that $w\fapp = \fapp'$.
\end{lemma}
Finally, a choice $\fapp\geq\scrC$ of a positive Weyl chamber for $G$ determines a positive Weyl chamber
\begin{equation}\label{eq:fapp_M_scrC}
\fa_{M_\scrC,++} := \{H\in\fa_{\scrC}^{\perp} : \alpha(H) > 0 \ \text{for all} \ \alpha\in \roots^0_\scrC\cap\roots_\fapp^+\}
\end{equation}
of $M_\scrC$.

\subsubsection{Group decompositions}		
	We recall the standard decompositions of a reductive group $G$:
	
	\begin{description}
	\item[\textbf{$KAK$-decomposition}] Once the open Weyl chamber $\fapp$ is fixed, each $g\in G$ can be written in the form $k_1 e^H k_2^{-1}$ with $k_1,k_2\in K$, $H\in \fap$. The element $H \in \fap$ is unique, and if $H\in\scrC$, $k_1$ and $k_2$ are uniquely defined modulo right multiplication by the \emph{same} element of $K_\scrC$. The \emph{Cartan projection} of $g$ is
	\[
	H=\mu_\fapp(g).
	\]
	The Cartan projection is a proper map, and corresponds to the decomposition of $G/K$ into concentric spheres around the base point $K$. More precisely, we have the following classical lemma (see \cite[Lemma 4.6]{Benoist-97}):
	\begin{lemma}
	\label{lemma:Benoist-compact-Cartan}
	Let $Q \subset G$ be compact. Then there exists a compact set $Q' \subset \fa$ such that for all $g \in G$:
	\[
	\mu_{\fapp}(QgQ) \subset \mu_\fapp(g) + Q'.
	\]
	\end{lemma}
	The Cartan projection satisfies some equivariance properties under the action of elements $w=[\dot{w}]$ of the Weyl group:
	\[
	\mu_{w\cdot\fapp}(g)=w\cdot\mu_{\fapp}(g) = \mathrm{Ad}(\dot{w})\mu_{\fapp}(g), \quad \mu_{-\fapp}(g^{-1}) = - \mu_\fapp(g).
	\]
	For $\scrC\leq \fapp$, we denote by $\mu_\scrC(g)$ the orthogonal projection of $\mu_\fapp(g)$ onto $\fa_\scrC$; it does not depend on the choice of the open Weyl chamber $\fapp\geq\scrC$ (combine Lemma~\ref{lem:adjacent_open_weyl_chamber} with \eqref{eq:stab-Weyl-and-projector} and the equivariance properties above).  \\

	\item[\textbf{Horospherical decomposition}] This decomposition may be viewed as a generalization of the Iwasawa decomposition for a generalized Weyl chamber $\scrC$.  Given $\scrC$, we have the identity
	\[
	G = K M_\scrC \times A_\scrC \times N_\scrC^+. 
	\]
	Notice that if $\scrC$ is open, $M_\scrC< K$, and we recover the usual Iwasawa decomposition $G=KAN$. Unlike the Iwasawa decomposition, the horospherical decomposition for general $\scrC$ is not a diffeomorphism. However, if we write the decomposition in the form
	\begin{equation}\label{eq:horospherical-decomposition}
	   g = k m_\scrC e^{H_\scrC(g)} n_\scrC,
    \end{equation}
	then $H_{\scrC}(g)$ is well-defined and unique.
	The map $g\mapsto H_\scrC(g)$ satisfies several equivariance properties. If $\scrC\leq \fapp$, then
	\[
	H_\scrC(g) = \pi_\scrC H_\fapp(g). 
	\]
	For $g\in G$, $p\in P_\scrC$, 
	\begin{equation}
		\label{eq:H_scrC_formula}
	H_\scrC(gp) = H_\scrC(g) + H_\scrC(p),
   \end{equation}
	and in particular for $m\in M_\scrC$, 
	\[
	H_\scrC(gm) = H_\scrC(g). 
	\]
	
	\item[\textbf{Bruhat decomposition}] The Bruhat decomposition describes the decomposition of a reductive group $G$ into double cosets of parabolic subgroups $P_\scrC$. For $\scrC = \fapp$, the Bruhat decomposition is parametrized by the Weyl group $W(G,A)$. For general $\scrC$ we need to introduce an equivalence relation:
	\begin{equation}\label{eq:equivalence-relation-coarse-Bruhat}
		W(G,A)/\sim\quad  = \quad \langle\roots^0_\scrC\rangle\backslash W(G,A) / \langle\roots^0_\scrC\rangle.
	\end{equation}
	The Bruhat decomposition is (see \cite[Proposition 3.16 (i)\&(iv)]{Borel-Tits-72}):
	\begin{equation}\label{eq:generalized-Bruhat}
		G= \bigsqcup_{[\dot{w}]\in W(G,A)/\sim } P_{-\scrC} \dot{w} P_\scrC.
	\end{equation}
	The only open double coset in this decomposition is $P_{-\scrC}eP_\scrC$, and we call it the \emph{open generalized Bruhat cell}. Furthermore, the map
	\[
		N_{-\scrC}^+\times P_\scrC \ni (n,p)\mapsto np\in P_{-\scrC}P_\scrC
	\]
	is an analytic diffeomorphism onto the open generalized Bruhat cell. \\
	
	\item[\textbf{Jordan decomposition}] We assume that $G$ is a connected algebraic group; it then admits a Jordan decomposition. That is, every element $g\in G$ can be decomposed into a commuting product $g= g_{\mathrm{ell}}\, g_{\mathrm{hyp}}\, g_{\mathrm{nil}}$, such that there exists $\gamma\in G$ so that
	\[
	\gamma g_{\mathrm{ell}} \gamma^{-1} \in K,\quad \gamma g_{\mathrm{hyp}} \gamma^{-1}\in \exp(\fap),\quad \gamma g_{\mathrm{nil}} \gamma^{-1} \in N^+_\fapp.
	\]
	We write $\gamma g_{\mathrm{hyp}} \gamma^{-1} = \exp(\lambda_\fapp(g))$. The map $\lambda_\fapp:G\to\fap,\ g\mapsto\lambda_\fapp(g)$ is well-defined and is called the \emph{Jordan projection}. In addition,
	\[
	\lambda_\fapp(g) = \lim_{n\to +\infty} \frac{1}{n}\mu_{\fapp}(g^n).
	\]
	We set
	\begin{equation}\label{eq:jordan_scrC}
	\lambda_\scrC(g) = \pi_\scrC\bigl(\lambda_\fapp(g)\bigr).
	\end{equation}
	As with the Cartan projection $\mu_\scrC$, this is independent of the choice of $\fapp\geq\scrC$.
	
	\end{description}

\subsection{Higher-rank locally homogeneous dynamics}\label{ssec:locally-homogeneous-dynamics}

Our interest lies in the right action of $A_\scrC$ on the quotient $G/M_\scrC$ given by 
\[
g M_\scrC . e^H = g e^H M_\scrC. 
\]
More precisely, we are interested in locally homogeneous dynamics modeled on the $A_\scrC$-action on $G/M_\scrC$. That is, we want to consider dynamical systems of the form
\[
\Gamma \backslash \Omega \curvearrowleft A_\scrC, 
\]
where $\Gamma < G$ is some discrete group and $\Omega \subset G/M_\scrC$ is a $\Gamma$- and $A_\scrC$-invariant open set. We choose an open Weyl chamber $\fapp\geq \scrC$. Unless explicitly stated, all concepts introduced do not depend on this choice.

\subsubsection{Flag manifolds and transverse flags}

A key role in describing the structure of $G/M_\scrC$ is played by the \emph{flag manifolds} $G/P_\scrC$, along with the notion of \emph{transverse (partial) flags}.

We first note that there is a unique open dense orbit for the action of $G$ on $G/P_\scrC \times G/P_{-\scrC}$. It is the orbit of $(P_\scrC,P_{-\scrC})$. The stabilizer of this pair is $P_\scrC\cap P_{-\scrC} = M_\scrC A_\scrC$. If $(\xi_+,\xi_-)\in G/P_\scrC \times G/P_{-\scrC}$ lies in this orbit, we write $\xi_+ \pitchfork \xi_-$ and say that they are \emph{transverse}. The orbit is denoted
\[
 G/P_\scrC \overset{\pitchfork}{\times} G/P_{-\scrC}.
\]
The set of $\xi^\mp\in G/P_{\mp\scrC}$ that are transverse to $eP_{\pm \scrC}$ is exactly $N^{\pm}_\scrC P_{\mp\scrC}$. More precisely, the map
\begin{equation}
\label{equation:diffeo-la}
 N^{\pm}_\scrC\ni n \mapsto nP_{\mp\scrC} \in \{\xi^{\mp}\in G/P_{\mp\scrC} ~:~ \xi^{\mp}\pitchfork eP_{\pm\scrC}\}
\end{equation}
is a diffeomorphism. More generally, $gP_\scrC \pitchfork h P_{-\scrC}$ if and only if
\begin{equation}\label{eq_transversality_cond}
h^{-1}g\in N_\scrC^- P_\scrC \quad\text{or equivalently}\quad g^{-1}h \in N_{\scrC}^+ P_{-\scrC}.
\end{equation}

We will sometimes need to strengthen the notion of transversality. We call two families of points $\xi_\alpha = g_\alpha P_\scrC \in G/P_\scrC$ and $\eta_\beta = h_\beta P_{-\scrC}\in G/P_{-\scrC}$ \emph{uniformly transverse} if the representatives $g_\alpha$ and $h_\beta$ can be chosen in $K$ and there exists a bounded set $B\subset N_\scrC^-$ such that for all $\alpha,\beta$,
\begin{equation}\label{eq:uniform_transverse}
h_\beta^{-1}g_\alpha\in B P_\scrC.
\end{equation}

Transversality behaves well under the projection maps $\mathrm{pr}_{\pm\scrC'}: G/P_{\pm\scrC} \to G/P_{\pm\scrC'}$ for $\scrC'\leq \scrC$, i.e.
\begin{equation}\label{eq:limit_pt_projection}
  \xi^+\pitchfork\xi^- \Longrightarrow \mathrm{pr}_{+\scrC'}(\xi^+) \pitchfork \mathrm{pr}_{-\scrC'}(\xi^-).
\end{equation}
 Furthermore, if $\scrC$ is $\iota_\fapp$-invariant, then by \eqref{eq:w_conj_parabolics}, $P_\scrC$ and $P_{-\scrC}$ are conjugate via $\dot{w}_\fapp$ and we get a natural equivariant diffeomorphism
\begin{equation}
\label{equation:ilesttard}
\mathcal I_\scrC: G/P_\scrC \overset{\simeq}{\longrightarrow} G/P_{-\scrC}, \qquad gP_\scrC \mapsto g \dot{w}_\fapp P_{-\scrC}
\end{equation}
and we have $\mathcal{I}_{-\scrC} = \mathcal{I}_{\scrC}^{-1}$, using that $w_\fapp$ is an involution. In this context, for $\xi^+,\xi^-\in G/P_\scrC$, the notation $\xi^+\pitchfork\xi^-$ means $\xi^+\pitchfork\mathcal I_\scrC(\xi^-)$.

\subsubsection{Hopf coordinates}
The description of transverse pairs
\[
G/P_\scrC \overset{\pitchfork}{\times} G/P_{-\scrC} \cong G/(M_\scrC A_\scrC)
\]
leads to the so-called \emph{generalized Hopf coordinates}, which give a decomposition of $G/M_\scrC$ into its $A_\scrC$-orbits. These are given by the following diffeomorphism:
\begin{equation}\label{eq:Hopf}
\Psi_\scrC: G/M_\scrC \to (G/P_\scrC \overset{\pitchfork}{\times} G/P_{-\scrC}) \times \fa_\scrC \qquad \Psi_\scrC(gM_\scrC) := (gP_\scrC, gP_{-\scrC}, H_\scrC(g)).
\end{equation}
Since $M_{-\scrC}=M_\scrC$ and $\fa_{-\scrC}=\fa_\scrC$, applying the same construction to $-\scrC$ gives the opposite Hopf coordinates
\[
\Psi_{-\scrC}(gM_\scrC)=\bigl(gP_{-\scrC},gP_\scrC,H_{-\scrC}(g)\bigr).
\]

\subsubsection{Tangent space of the phase space}

We next describe the tangent space $T(G/M_\scrC)$. As the homogeneous space $G/M_\scrC$ is reductive, the tangent bundle decomposes as an associated vector bundle
\begin{equation}\label{eq:hyperbolic_splitting_first}
T(G/M_\scrC) = G \times_{M_\scrC} (\fa_\scrC \oplus \fn^+_\scrC \oplus \fn^-_\scrC). 
\end{equation}
We denote the elements of such associated bundles by equivalence classes $[g,X]$ with $g\in G$ and $X\in \fa_\scrC \oplus \fn^+_\scrC \oplus \fn^-_\scrC$, where $[gm,\mathrm{Ad}(m^{-1})X]=[g,X]$ for all $m\in M_\scrC$.

This leads to the corresponding $G$-invariant decomposition of $T(G/M_\scrC)$:
\begin{equation}\label{eq:stable-unstable}
E_0:= G\times_{M_\scrC} \fa_{\scrC},\qquad E_{s}:= G \times_{M_\scrC} \fn^+_\scrC,\qquad E_{u}:= G \times_{M_\scrC} \fn^-_\scrC.
\end{equation}
These vector bundles are integrable and are tangent, respectively, to the orbit foliation, the stable foliation, and the unstable foliation. The bundles $E_{s/u}$ correspond to the stable/unstable bundles of the $A_\scrC$-action and descend to stable/unstable bundles of the Axiom A flows. However, this statement requires some additional care. As there need not exist $\mathrm{Ad}(M_\scrC)$-invariant norms on $\fn_\scrC^\pm$, there are no canonical metrics on these bundles and the formulation and proofs of hyperbolicity are somewhat involved (see \S\ref{sec:divergence-and-hyperbolicity}). We can likewise define the following integrable weak bundles:
\[
E_{ws}:= G \times_{M_\scrC} (\fa_{\scrC}\oplus \fn^+_\scrC),\qquad E_{wu}:= G \times_{M_\scrC} (\fa_{\scrC}\oplus \fn^-_\scrC).
\]

The leaves of the foliation $W^s$ tangent to $E_s$ have a simple expression in Hopf coordinates:
\begin{equation}\label{eq:formula-strong-stable-leaf}
W^s(gM_\scrC) \simeq \{ (\xi_+, \xi_-, H)\ |\ \xi_+ = gP_\scrC,\ H = H_\scrC(g) \}. 
\end{equation}
Indeed, the space on the right is naturally identified with $gN^+_\scrC M_\scrC$, and its tangent space is $E_s$. Likewise, the leaves of $W^{wu}$ tangent to $E_{wu}$ are given by
\begin{equation}\label{eq:formula-weak-unstable-leaf}
W^{wu}(gM_\scrC) \simeq \{ (\xi_+, \xi_-, H)\ |\ \xi_- = gP_{-\scrC} \}. 
\end{equation}

\subsubsection{Symmetries of the phase space}
Finally, we record a symmetry of the phase space. For $w=[\dot{w}]\in W(G,A)$, the map
	\[
	g M_\scrC \mapsto g \dot{w}^{-1} M_{w\cdot\scrC}
	\]
	defines a diffeomorphism between $G/M_\scrC$ and $G/M_{w\cdot\scrC}$. However, $M_\scrC$ only depends on $\fa_\scrC$; hence, whenever $w\cdot\fa_\scrC = \fa_\scrC$, this map defines a symmetry of the phase space.
	
	The most notable case occurs when $\scrC\leq\fapp$ and $\iota_\fapp \scrC = \scrC$. The corresponding symmetry, $S_\fapp \colon gM_\scrC \mapsto g\dot{w}_\fapp M_\scrC$, is called the \emph{antipodal symmetry}. It behaves as follows in Hopf coordinates:
\begin{lemma}\label{lemma:def-time-symmetry}
 Let $\scrC$ be an $\iota_\fapp$-invariant generalized Weyl chamber. If
 \[
 \Psi_{-\scrC}(gM_\scrC)=(\xi^-,\xi^+,H)\in G/P_{-\scrC}\times G/P_\scrC\times\fa_\scrC,
 \]
 then
 \[
 \Psi_\scrC(S_\fapp(gM_\scrC))=\bigl(\mathcal I_{-\scrC}(\xi^-),\mathcal I_\scrC(\xi^+),w_\fapp(H)\bigr)\in G/P_\scrC\times G/P_{-\scrC}\times\fa_\scrC.
 \]
 Equivalently,
 \[
 \Psi_\scrC\circ S_\fapp
 =
 \bigl(\mathcal I_{-\scrC}\times\mathcal I_\scrC\times w_\fapp\bigr)\circ\Psi_{-\scrC}.
 \]
\end{lemma}
\begin{proof}
Since $\iota_\fapp(\scrC)=\scrC$ and $\iota_\fapp=-w_\fapp$, we have $w_\fapp\cdot\scrC=-\scrC$. Write the horospherical decomposition associated with $-\scrC$ as
\[
g=km e^{H_{-\scrC}(g)}n^-,
\qquad
k\in K,\quad m\in M_\scrC,\quad n^-\in N_\scrC^-.
\]
Set
\[
m':=\dot w_\fapp^{-1}m\dot w_\fapp\in M_\scrC,
\qquad
n^+:=\dot w_\fapp^{-1}n^-\dot w_\fapp\in N_\scrC^+.
\]
Since $w_\fapp^{-1}=w_\fapp$ in the Weyl group, we obtain
\begin{align*}
 g\dot w_\fapp
 &=km e^{H_{-\scrC}(g)}\dot w_\fapp n^+\\
 &=k\dot w_\fapp m'e^{w_\fapp(H_{-\scrC}(g))}n^+.
\end{align*}
This is a horospherical decomposition associated with $\scrC$, and hence
\[
H_\scrC(g\dot w_\fapp)=w_\fapp\bigl(H_{-\scrC}(g)\bigr).
\]
Moreover,
\[
g\dot w_\fapp P_\scrC=\mathcal I_{-\scrC}(gP_{-\scrC}),
\qquad
g\dot w_\fapp P_{-\scrC}=\mathcal I_\scrC(gP_\scrC).
\]
Combining these identities proves the claimed intertwining formula.
\end{proof}

\subsection{Discrete subgroups} \label{ssection:discrete-subgroups}

Let $\Gamma < G$ be a discrete torsion-free subgroup. We introduce several notions that will be used throughout the remainder of the paper.

\subsubsection{Limit cone} We begin with the notion of limit cone, which is a fundamental object in the study of discrete subgroups of higher-rank groups and was introduced by Benoist in \cite{Benoist-97}:

\begin{definition}[Jordan and Cartan limit cones]\label{definition:limit-cone-anosov}

For every generalized Weyl chamber $\scrC\leq\fapp$, we define the
\emph{Jordan limit cone} and the \emph{Cartan limit cone} of $\Gamma$ by
\[
\scrL_{\Gamma,\mathrm J}^{\scrC}
:=
\overline{\{t\lambda_\scrC(\gamma)~:~t>0,\ \gamma\in\Gamma\}}
\subset\overline{\scrC}
\]
and
\begin{equation}\label{eq:asymptotic-cone-limitcone}
\scrL_{\Gamma,\mathrm C}^{\scrC}
:=
\mathrm{AC}\bigl(\mu_\scrC(\Gamma)\bigr)
\subset\overline{\scrC}.
\end{equation}
Here, for a subset $S$ of a finite-dimensional real vector space,
\[
\mathrm{AC}(S)
:=
\left\{
\lim_{\ell\to\infty}t_\ell s_\ell
~:~
t_\ell\to0^+,\ s_\ell\in S
\right\}.
\]
In particular, since
$\lambda_\scrC=\pi_\scrC\circ\lambda_\fapp$,
\[
\scrL_{\Gamma,\mathrm J}^{\scrC}
=
\overline{\pi_\scrC\bigl(\scrL_{\Gamma,\mathrm J}^{\fapp}\bigr)}.
\]
\end{definition}

The two cones satisfy the following unconditional comparison relations:
for every generalized Weyl chamber $\scrC\leq\fapp$,
\[
\scrL_{\Gamma,\mathrm J}^{\scrC}
\subset
\scrL_{\Gamma,\mathrm C}^{\scrC},
\qquad
\scrL_{\Gamma,\mathrm C}^{\scrC}
=
\pi_\scrC\bigl(\scrL_{\Gamma,\mathrm C}^{\fapp}\bigr).
\]
Indeed, the inclusion follows from
$\lambda_\scrC(\gamma)=\lim_{n\to\infty}n^{-1}\mu_\scrC(\gamma^n)$,
whereas the projection identity is a special property of Cartan limit cones;
see \cite[\S2]{Dey-Oh-25}.

\begin{theorem}[Benoist, Dey--Oh]\label{theorem:benoist}

If the Zariski closure of $\Gamma$ is reductive, then, for every
generalized Weyl chamber $\scrC\leq\fapp$,
\[
\scrL_{\Gamma,\mathrm J}^{\scrC}
=
\scrL_{\Gamma,\mathrm C}^{\scrC}.
\]
\end{theorem}

This follows from
\cite[Theorem~2.3 and the identity following (2.10)]{Dey-Oh-25}.
If $\Gamma$ is Zariski dense in $G$, then these common cones are
convex and have nonempty interior in $\fa_\scrC$; see
\cite[Theorem~1.2(a)]{Benoist-97}, together with the projection identity above.

\subsubsection{Dual cones}

For $\bullet\in\{\mathrm J,\mathrm C\}$, we define
\[
\bigl(\scrL_{\Gamma,\bullet}^{\scrC}\bigr)^*
:=
\left\{
\alpha\in\fa_\scrC^*
~:~
\alpha(v)\geq0\quad\text{for every }v\in\scrL_{\Gamma,\bullet}^{\scrC}
\right\},
\]
and denote its interior by
$\bigl(\scrL_{\Gamma,\bullet}^{\scrC}\bigr)^{*,\circ}$. The unconditional
inclusion between the limit cones gives the reversed inclusion
\[
\bigl(\scrL_{\Gamma,\mathrm C}^{\scrC}\bigr)^*
\subset
\bigl(\scrL_{\Gamma,\mathrm J}^{\scrC}\bigr)^*.
\]
Recall that the ambient group $G$ considered in this section is
semisimple. Hence $\overline{\scrC}$ is pointed. Since both limit cones are
contained in $\overline{\scrC}$, each of their dual cones contains
$\scrC^*$ and therefore has nonempty interior.

Moreover, for $\bullet\in\{\mathrm J,\mathrm C\}$ and every
$\alpha\in\bigl(\scrL_{\Gamma,\bullet}^{\scrC}\bigr)^{*,\circ}$,
compactness of the spherical section of the corresponding limit cone gives
an $\varepsilon>0$ such that
\begin{equation}\label{eq:interior-dual-limit-cone}
\alpha(v)\geq\varepsilon\lVert v\rVert,
\qquad
\forall v\in\scrL_{\Gamma,\bullet}^{\scrC}.
\end{equation}

\subsubsection{$\scrC$-divergent sequences. Convergence at infinity} To proceed with the definition of the limit set, we first introduce the notion of $\scrC$-divergent sequences. A sequence $(H_\ell)_{\ell\geq 0}\subset\fa$ is called \emph{$\scrC$-divergent} if $\alpha(H_\ell)\to+\infty$ for all $\alpha\in \roots_{\scrC}^+$; we then write $H_\ell \to_{\scrC} \infty$. A sequence $(\gamma_\ell)_{\ell\geq 0}\subset G$ is called \emph{$\scrC$-divergent} if $\mu_\fapp(\gamma_\ell)\to_{\scrC}\infty$, and we write $\gamma_\ell \to_{\scrC}\infty$; by Lemma~\ref{lem:adjacent_open_weyl_chamber} and \eqref{eq:stab-Weyl-and-projector} this does not depend on the choice of $\fapp\geq\scrC$.

\begin{definition}[Convergence at infinity]
\label{definition:conv-infinity}
 Let $\scrC$ be a generalized Weyl chamber, let $\fapp\geq \scrC$ be an open Weyl chamber, and let $(\gamma_\ell)_{\ell \geq 0}\subset G$ be a $\scrC$-divergent sequence. We call it \emph{convergent at infinity} if there is a Cartan decomposition $\gamma_\ell = k_{1,\ell}\exp(H_\ell)k_{2,\ell}^{-1}$ with $H_\ell\in \fap$ such that $k_{1,\ell}, k_{2,\ell} \in K$ both converge; we write $k_{1/2,\infty}:=\lim_{\ell\to\infty} k_{1/2,\ell}$. In this case, we set
\[
\gamma^+_\scrC :=k_{1,\infty}P_\scrC,\hspace{20pt}  \gamma^-_{\scrC}:=k_{2,\infty}P_{-\scrC}.
\]
They are called the \emph{limit points} of the sequence.
\end{definition}
Every $\scrC$-divergent sequence has a subsequence that is convergent at infinity. Although $k_{1,\infty}$ and $k_{2,\infty}$ are unique only up to right multiplication by an element of $K_\scrC$, the limit points are unique because $K_\scrC < M_\scrC< P_{\pm\scrC}$.
A straightforward computation involving the Bruhat decomposition leads to the following important dynamical interpretation of the limit points as attracting limit points for the action of the sequence $(\gamma_\ell)_{\ell \geq 0}$ on the flag manifolds $G/P_{\pm\scrC}$:

\begin{proposition}\label{prop:proximal}
Let $(\gamma_\ell)_{\ell \geq 0}$ be a $\scrC$-divergent sequence converging at infinity. The following holds:
\begin{align*}
gP_\scrC \pitchfork \gamma^-_{\scrC}		&\Longrightarrow \gamma_\ell gP_\scrC \to_{\ell \to \infty} \gamma^+_\scrC, \\
gP_{-\scrC} \pitchfork \gamma^+_{\scrC} 	&\Longrightarrow \gamma_\ell^{-1} gP_{-\scrC} \to_{\ell \to \infty} \gamma^-_{\scrC}, \\
\end{align*}
Additionally, the convergence is locally uniform in $g$.
\end{proposition}

As a consequence we obtain
\begin{lemma}\label{lem:limit_points_leftright_mult}
	Let $(\gamma_\ell)_{\ell \geq 0}$ be a $\scrC$-divergent sequence converging at infinity and $x,y\in G$. Then $(x\gamma_\ell y)_{\ell \geq 0}$ is $\scrC$-divergent and convergent at infinity, with 
	\begin{equation}\label{eq:limit_left_right_mult}
	(x\gamma_\ell y)_\scrC^+ = x\gamma_\scrC^+~\text{ and }~(x\gamma_\ell y)_\scrC^- = y^{-1}\gamma_\scrC^-
	\end{equation}
\end{lemma}
\begin{proof}
$\scrC$-divergence follows from Lemma~\ref{lemma:Benoist-compact-Cartan}. Let $D:=\{gP_\scrC \in G/P_\scrC ~:~ gP_\scrC\pitchfork \gamma_\scrC^-\}$. For all $h P_\scrC\in y^{-1}D\subset G/P_\scrC$, Proposition~\ref{prop:proximal} implies $(x\gamma_\ell y)hP_\scrC \to x\gamma^+_\scrC$. Since $D\subset G/P_\scrC$ is open and dense, there cannot be a second limit point. Indeed, if there were one, we could pass to the corresponding subsequence, and Proposition~\ref{prop:proximal} would imply that this subsequence also has an open dense attracting set. Hence $(x\gamma_\ell y)_{\ell \geq 0}$ is converging at infinity and we get \eqref{eq:limit_left_right_mult}. The proof for the backward limit point is analogous.
\end{proof}
Furthermore, the following lemma, whose proof is a straightforward computation, records some useful identities for these limit points.

\begin{lemma}\label{lem:forward_backward_same}
Let $(\gamma_\ell)_{\ell \geq 0}$ be $\scrC$-divergent and convergent at infinity and define $\eta_\ell := \gamma_\ell^{-1}$. Let $\fapp\geq \scrC$ be an open Weyl chamber. Then:
\begin{enumerate}[label=\emph{(\roman*)}]
 \item $\gamma_\scrC^- = \eta^+_{-\scrC}$
 \item \label{it:limit_-scrC} If $(\gamma_\ell)_{\ell \geq 0}$ is additionally $\scrC^\circ$-divergent then:
\begin{equation*}
\gamma_{-\scrC}^+ = k_{1,\infty}\dot{w}_\fapp P_{-\scrC},\hspace{20pt}  \gamma_{-\scrC}^- = k_{2,\infty}\dot{w}_\fapp P_{\scrC}.
\end{equation*}
It follows that $\gamma_{\scrC}^\pm \not\pitchfork \gamma_{-\scrC}^\pm$.
\item If $\scrC$ is $\iota_\fapp$-invariant, then:
 \[
  \eta^+_\scrC = \gamma_{-\scrC}^{-} = \mathcal I_{-\scrC} \gamma_\scrC^-,\qquad \eta_\scrC^- = \gamma_{-\scrC}^+ = \mathcal I_\scrC \gamma^+_\scrC.
 \]
\end{enumerate}
\end{lemma}

\begin{remark}\label{rem:limit_point_projection}
As the limit points of a $\scrC$-divergent sequence are constructed using a Cartan decomposition adapted to $\fapp\geq\scrC$, for a smaller generalized Weyl chamber $\scrC'\leq\scrC$ the sequence is also $\scrC'$-divergent and its limit points are obtained by projecting via $\mathrm{pr}_{\pm\scrC'} : G/P_{\pm\scrC} \to G/P_{\pm\scrC'}$:
\begin{equation}\label{eq:limit_point_projection}
\gamma_{\scrC'}^\pm = \mathrm{pr}_{\pm\scrC'}(\gamma^\pm_\scrC).
\end{equation}
\end{remark}

\begin{example}\label{example:limit-point-for-powers}
Let $g \in G$ be such that the sequence $\gamma_\ell := g^\ell$ is $\scrC$-divergent; this holds, for instance, for every element of infinite order in a $\scrC$-Anosov subgroup (see Definition~\ref{def:divergent-regular-transverse-anosov} below). Then $(\gamma_\ell)_{\ell \geq 0}$ converges at infinity in the sense of Definition~\ref{definition:conv-infinity}, the limit points $\gamma^\pm_\scrC$ are fixed by $g$, and by Proposition~\ref{prop:proximal}, they are the attracting and repelling fixed points of $g$ acting on $G/P_{\pm\scrC}$; see \cite[Lemma 5.11]{Gueritaud-Guichard-Kassel-Wienhard-17}.
\end{example}

\subsubsection{Limit set} \label{sssection:limit-set} Next, we recall the notion of a limit set introduced by Quint \cite{Quint-2001, Quint-2002}. The limit set of a $\scrC$-divergent subgroup $\Gamma < G$ is the set of all limit points of sequences in $\Gamma$ that are convergent at infinity (here and below, $\gamma_\ell\to\infty$ means that the sequence eventually leaves every compact subset of $G$):

\begin{definition}[Limit set]
Let $\Gamma < G$ be a $\scrC$-divergent group. Then
\[
\begin{split}
\Lambda_\scrC & := \{\xi \in G/P_{ \scrC} ~:~ \exists (\gamma_\ell)_{\ell \geq 0} \subset \Gamma, \gamma_\ell \to \infty, \text{convergent at }\infty, (\gamma_\ell)_\scrC^+ = \xi \} \\
& \subset G/P_{\scrC}.
\end{split}
\]
This set is the \emph{$\scrC$-limit set} of $\Gamma$.
\end{definition}

Since, for an arbitrary sequence, $(\gamma_\ell)_\scrC^-=(\gamma_\ell^{-1})_{-\scrC}^+$, and since $\Gamma$ is a group, we could equally have defined $\Lambda_{-\scrC}$ using the backward limit points $(\gamma_\ell)_\scrC^-$. Furthermore, if $\scrC$ is $\iota_\fapp$-invariant, then there is a natural identification $G/P_\scrC \simeq G/P_{-\scrC}$ (see \eqref{equation:ilesttard}) and correspondingly, Lemma~\ref{lem:forward_backward_same} implies
\[
\mathcal{I}_\scrC(\Lambda_\scrC) = \Lambda_{-\scrC}.
\]

The set $\Lambda_{\fapp}$ corresponds to the set $\Lambda_\Gamma$ defined in \cite[\S3.6]{Benoist-97}. According to Observation~3 following that definition, $\Lambda_\scrC = \mathrm{pr}_\scrC\Lambda_\Gamma$ (see \eqref{eq:limit_point_projection} for the definition of $\mathrm{pr}_\scrC$). The same reference also shows that \cite[\S3.6]{Benoist-97} the sets $\Lambda_\scrC$ are closed and $\Gamma$-invariant subsets of $G/P_{\scrC}$.

Similarly, we introduce
\[
\begin{split}
\Lambda^{(2)}_{\scrC} & := \left\{ (\xi^+,\xi^-) \in G/P_\scrC \overset{\pitchfork}{\times} G/P_{-\scrC} ~:~ \exists (\gamma_\ell)_{\ell \geq 0} \subset \Gamma, \gamma_\ell \to \infty,(\gamma_\ell)_\scrC^\pm = \xi^\pm \right\} \\
& \subset G/P_\scrC \overset{\pitchfork}{\times} G/P_{-\scrC}.
\end{split}
\]
We emphasize that $\Lambda^{(2)}_{\scrC}$ only contains transverse pairs by definition; there may be limit points $((\gamma_\ell)_{\scrC}^+,(\gamma_\ell)_\scrC^-)$ (for sequences $(\gamma_\ell)_{\ell \geq 0}$) which are not transverse, but these are not recorded by $\Lambda^{(2)}_{\scrC}$. The first observation is that this set is also closed:

\begin{lemma}
\label{lemma:closed}
$\Lambda^{(2)}_{\scrC}\subset \Lambda_{\scrC} \overset{\pitchfork}{\times} \Lambda_{-\scrC}$ is closed and invariant under the diagonal action of $\Gamma$.
\end{lemma}

For Zariski dense subgroups the inclusion is in fact an equality (see \cite[Lemma, item (iv), p. 16]{Benoist-97} for a proof):
\begin{lemma}
\label{lemma:equivalent-lambda}
Suppose $\Gamma < G$ is Zariski dense. Then $\Lambda_{\scrC} \overset{\pitchfork}{\times} \Lambda_{-\scrC} = \Lambda^{(2)}_{\scrC}$.
\end{lemma}

Finally, we can define several properties of discrete subgroups:
\begin{definition}[Properties of discrete subgroups]\label{def:divergent-regular-transverse-anosov}
	Let $\Gamma < G$ be a finitely generated, discrete, torsion-free subgroup. Then $\Gamma$ is called:
	\begin{enumerate}[label=(\roman*)]
	\item\label{it:divergent} \textbf{$\scrC$-divergent} if, for every sequence $(\gamma_\ell)_{\ell \geq 0}\subset \Gamma$ satisfying $\gamma_\ell\to\infty$,
	\[
	\alpha(\mu_\fapp(\gamma_\ell))\to+\infty, \qquad \forall \alpha\in\Delta^+_\scrC.
	\]
	\item\label{it:regular} \textbf{$\scrC$-regular} if there exists $C>0$ such that
	\[
	\alpha(\mu_\fapp(\gamma))> C^{-1}\|\mu_\fapp(\gamma)\| - C, \qquad \forall \alpha\in\Delta^+_\scrC, \gamma\in \Gamma.
	\]
		\item\label{it:transverse} \textbf{$\scrC$-transverse} if $\Gamma$ is $\scrC$-divergent, $\scrC$ is $\iota_\fapp$-invariant and for all $\xi_1,\xi_2 \in \Lambda_\scrC$ with $\xi_1\neq \xi_2$ we have $\xi_1\pitchfork \mathcal I_\scrC(\xi_2)$. More generally, if $\scrC$ is not $\iota_\fapp$-invariant, we say that $\Gamma$ is $\scrC$-transverse if it is $\scrC^\circ$-transverse.
	\item\label{it:Anosov} \textbf{$\scrC$-Anosov} if for some (equivalently, any) word metric $|\bullet|$ on $\Gamma$, there exists $C>0$ such that:
	\[
	\alpha(\mu_\fapp(\gamma))> C^{-1} |\gamma|-C, \qquad \forall \alpha\in\Delta^+_\scrC, \gamma\in\Gamma.
	\]
\end{enumerate}
\end{definition}
These notions form a clear hierarchy: by definition, $\scrC$-Anosov implies $\scrC$-regular, which implies $\scrC$-divergent. In addition, the following holds (see \cite[Section 5.6]{Gueritaud-Guichard-Kassel-Wienhard-17} or \cite{Kapovich-Leeb-Porti-17, Bochi-Potrie-Sambarino-19} for a proof):
\begin{lemma}
\label{lemma:anosov-transverse}
Every $\scrC$-Anosov subgroup $\Gamma < G$ is $\scrC$-transverse.
\end{lemma}

Furthermore, all these notions behave well under the action of the Weyl group on $\scrC$.
\begin{lemma}
\label{lemma:i-invariant}
If $\Gamma < G$ has any of the properties $\scrC$-divergent, $\scrC$-regular, $\scrC$-transverse, or $\scrC$-Anosov, then
\begin{enumerate}[label=\emph{(\roman*)}]
 \item\label{it:lemma:i-invariant-i} For any $w\in W(G,A)$, $\Gamma$ is $w\cdot\scrC$-divergent/regular/transverse/Anosov.
 \item\label{it:lemma:i-invariant-ii} $\Gamma$ is $-\scrC$-divergent/regular/transverse/Anosov.
 \item\label{it:lemma:i-invariant-iii} $\Gamma$ is $\scrC^\circ$-divergent/regular/transverse/Anosov.
\end{enumerate}
\end{lemma}

Let us collect several further properties of transverse subgroups. First, $\Gamma$ is $\scrC$-regular if and only if
 \begin{equation}\label{eq:scrC_reg_limit_cone}
\scrL_{\Gamma,\mathrm C}^{\fapp}
\cap\ker\alpha = \{0\}, \qquad \forall \alpha\in\Delta^+_\scrC.
 \end{equation}
Indeed, one implication follows by multiplying the regularity estimate
by $t_\ell\to0^+$ along a sequence for which
$t_\ell\mu_\fapp(\gamma_\ell)$ converges. Conversely, failure of the
regularity estimate would, after choosing a root in the finite set
$\Delta_\scrC^+$ and normalizing an unbounded sequence of Cartan projections,
produce a nonzero element of
$\scrL_{\Gamma,\mathrm C}^{\fapp}\cap\ker\alpha$.

In addition, if $\Gamma$ is Zariski dense and $\scrC$ is $\iota_\fapp$-invariant, then Lemma~\ref{lemma:equivalent-lambda} and transversality imply
\[
\Lambda_\scrC^{(2)}
= \Lambda_\scrC \overset{\pitchfork}{\times} \Lambda_{-\scrC}
= (\Lambda_\scrC \times\Lambda_{-\scrC}) \setminus \mathrm{diag},
\text{ with }\mathrm{diag} = \{(\xi,\mathcal I_\scrC(\xi)) ~:~ \xi\in\Lambda_\scrC\}.
\]
Independently, for $\scrC$-transverse groups, the limit set does not depend on the choice of face, i.e., for $\scrC'\leq \scrC$ the projection map
\begin{equation}\label{eq:transverse_grp_limit_projection}
\mathrm{pr}_{\scrC'} : \Lambda_\scrC \to \Lambda_{\scrC'}
\end{equation}
is a homeomorphism.

Another feature of transverse groups is that each $\xi_+ \in \Lambda_\scrC$ fails to be transverse to exactly one point in $\Lambda_{-\scrC}$. This defines a homeomorphism
\begin{equation}\label{eq:limit_homeo}
\kappa\colon \Lambda_\scrC \overset{\simeq}{\longrightarrow} \Lambda_{-\scrC}.
\end{equation}
If $\scrC$ is $\iota_\fapp$-invariant, this is simply the restriction of $\mathcal I_\scrC$. If $\scrC$ is not $\iota_\fapp$-invariant, it is a composition of $\mathcal I_{\scrC^\circ}$ with the homeomorphisms of \eqref{eq:transverse_grp_limit_projection}.

\section{Domains of discontinuity}
\label{section:domain}

In this section, we show that $\scrC$-divergent subgroups admit a natural $A_{\scrC}$-invariant domain of discontinuity in $G/M_{\scrC}$, which we call a \emph{DMS domain}. Under the additional assumption of $\scrC$-transversality—and, more strongly, the $\scrC$-Anosov property—the dynamics of $A_{\scrC}$ become progressively better behaved. In particular, one can ultimately show (Theorem~\ref{thm:DMS-main}) that they satisfy all the assumptions of an Axiom A action of type $(k,1)$ (Definition~\ref{definition:axiom-a-action}).

Such domains were first constructed by Delarue, Monclair, and Sanders in \cite{Delarue-Monclair-Sanders-24} for projective Anosov groups and later extended to general Anosov subgroups in \cite{Delarue-Monclair-Sanders-25}.
 We provide a self-contained proof of the existence and key properties of these DMS domains, together with additional results relevant to our spectral study of higher-rank actions.
 Moreover, our techniques and arguments differ in flavor from those in \cite{Delarue-Monclair-Sanders-24,Delarue-Monclair-Sanders-25}.

\subsection{Definition}
As above, let $\scrC$ be a fixed generalized Weyl chamber, let $\Gamma < G$ be a discrete, $\scrC$-divergent subgroup of $G$, and let $\fapp\geq\scrC$ be an auxiliary open Weyl chamber (all constructions will ultimately be independent of this choice). Recall that the Hopf coordinates (see \eqref{eq:Hopf}) provide a diffeomorphism
\[
G/M_\scrC \simeq G/P_\scrC \overset{\pitchfork}{\times} G/P_{-\scrC} \times \fa_\scrC.
\]
The definition of the domain goes as follows:

\begin{definition}[DMS domain]
\label{definition:dms-domain}
Let $\Gamma < G$ be a discrete and torsion-free $\scrC$-divergent subgroup. The DMS domain associated with $\Gamma$ is defined by
\[
\odms :=\left \{gM_\scrC ~:~ \begin{matrix*}  \forall (\gamma_\ell)_{\ell \geq 0}\subset \Gamma \text{ convergent at infinity,} \\ \text{ either } gP_\scrC\pitchfork \gamma^-_\scrC \text{ or } gP_{-\scrC}\pitchfork \gamma^-_{-\scrC}\end{matrix*} \right \} \subset G/M_\scrC.
\]
\end{definition}

Note that $\gamma^-_\scrC$ and $\gamma^-_{-\scrC}$ are the repelling points of the $\scrC$-divergent sequence $(\gamma_\ell)_{\ell \geq 0}$ (see Proposition~\ref{prop:proximal}) acting respectively on $G/P_{-\scrC}$ and $G/P_{\scrC}$. Under the stronger assumption that $\scrC$ is $\iota_{\fapp}$-invariant, we have $\gamma^-_{\scrC} = \mathcal I_\scrC(\gamma^-_{-\scrC})$ (Lemma \ref{lem:forward_backward_same}, item (iii)), so the DMS domain can be rewritten in a simpler way as
\begin{equation}
 \label{eq:simpler_DMS_def}
\odms := \left\{gM_\scrC ~:~  \forall \xi \in \Lambda_{\scrC}, \text{ either } gP_\scrC\pitchfork \mathcal I_\scrC(\xi) \text{ or } gP_{-\scrC}\pitchfork \xi \right\}.
\end{equation}
We will now establish properties of $\odms$ under progressively stronger assumptions on $\Gamma$.

\subsection{Proper discontinuity}
We begin by showing the most basic features of this domain:
\begin{lemma}\label{lemma:basic-feature-Omega-DMS}
The set $\odms$ is an open left $\Gamma$- and right $A_\scrC$-invariant set. Additionally, if $\scrC'\leq \scrC$, then the natural projection from $G/M_\scrC$ to $G/M_{\scrC'}$ maps $\odms$ into $\Omega^{\scrC'}_\mathrm{DMS}$.
\end{lemma}

\begin{proof}
Right $A_\scrC$-invariance and left $\Gamma$-invariance follow directly from the definition. The behavior under the natural projection maps follows from the compatibility of the projection with transversality and with taking limit points, see \eqref{eq:limit_pt_projection} and \eqref{eq:limit_point_projection}.

Let $(g_\ell M_\scrC)_{\ell \geq 0}$ be a sequence in the complement $(\odms)^c$. This means that for each $\ell \geq 0$, there exists a sequence $(\gamma_{\ell,j})_{j\geq 0}\subset \Gamma$ converging at infinity to some $\gamma^\pm_{\pm\scrC,\ell}$, such that $g_\ell P_\scrC \cancel{\pitchfork} \gamma_{\scrC,\ell}^-$ and $g_\ell P_{-\scrC} \not \pitchfork \gamma_{-\scrC,\ell}^-$. By Lemma~\ref{lemma:i-invariant}, \ref{it:lemma:i-invariant-iii}, we may assume that $(\gamma_{\ell,j})_{j \geq 0}$ is $\scrC^\circ$-divergent and converges at infinity. We write $\gamma_{\ell,j} = k_{1,\ell,j}e^{H_{\ell,j}} k_{2,\ell,j}^{-1}$.

Now, assume that $g_\ell M_\scrC \to gM_\scrC$. Let us fix an auxiliary metric on $G/P_\scrC$.
As $ k_{2,\ell,j}P_{-\scrC} \to \gamma_{\scrC,\ell}^-$, for each $\ell$, there exists $j_\ell$ so that for $j\geq j_\ell$, $g_\ell P_\scrC$ is at distance at most $1/\ell$ from the set of flags not transverse to $k_{2,\ell,j}P_{-\scrC}$.
Recall that, by Lemma~\ref{lem:forward_backward_same}, $k_{2,\ell,j}\dot w_\fapp P_{\scrC}\to \gamma^-_{-\scrC,\ell}$ and, as above, we may assume that $g_\ell P_{-\scrC}$ is at distance at most $1/\ell$ from the set of flags not transverse to $k_{2,\ell,j} w_\fapp P_\scrC$. Increase $j_\ell$, if necessary, so that both distance estimates hold and
\[
\min_{\alpha\in\roots_{\scrC^\circ}^+}\alpha\!\left(\mu_\fapp(\gamma_{\ell,j_\ell})\right)\geq\ell,
\]
and set $\gamma_\ell:=\gamma_{\ell,j_\ell}$.
The diagonal sequence is then $\scrC^\circ$-divergent and hence both $\scrC$- and $-\scrC$-divergent. After passing to a subsequence, we may assume that it converges at infinity.
In this case, since the set of flags not transverse to a given flag $P$ depends continuously in the Hausdorff topology on $P$, we deduce that $gM_\scrC$ is not in $\odms$.
\end{proof}
The following lemma is crucial for the properness of the action and is also extremely helpful at several other places:
\begin{lemma}\label{lemma:translation-length}
Let $(\gamma_\ell)_{\ell \geq 0}\subset G$ be a $\scrC$-divergent sequence, converging at infinity. Let $(g_\ell)_{\ell \geq 0}\subset G$ be a sequence such that $g_\ell P_\scrC \pitchfork \gamma_\scrC^-$ is uniformly transverse. Then, there is a bounded set $\mathcal B\subset \fa_\scrC$ such that for all $\ell \geq 0$:
\[
H_\scrC(\gamma_\ell g_\ell) \in \mu_\scrC(\gamma_\ell) + H_\scrC(g_\ell) + \mc B.
\]
More quantitatively, if for $\ell_0 \geq 0$ large enough, $U_{\ell_0}\subset N_\scrC^-$ is a bounded set such that for $\gamma_\ell = k_{1,\ell} e^{\mu_\fapp(\gamma_\ell)} k_{2,\ell}^{-1}$ and for all $\ell\geq\ell_0$
\[
k_{2,\ell}^{-1}g_\ell P_\scrC = n^-_\ell P_\scrC \in U_{\ell_0} P_\scrC\subset N_\scrC^-P_\scrC,
\]
then for all $\ell\geq \ell_0$:
\[
H_\scrC(\gamma_\ell g_\ell) \in \mu_\scrC(\gamma_\ell) + H_\scrC(g_\ell) + H_\scrC(e^{\mu_\fapp(\gamma_\ell)} U_{\ell_0}e^{-\mu_\fapp(\gamma_\ell)}) + H_\scrC(U_{\ell_0}K).
\]
\end{lemma}
\begin{proof}
We write $\gamma_\ell= k_{1,\ell} e^{\mu_\fapp(\gamma_\ell)} k_{2,\ell}^{-1}$ (Cartan decomposition) and $g_\ell = k_\ell p_\ell = k_\ell m_\ell e^{H_\scrC(g_\ell)}n_\ell$ (horospherical decomposition, see \eqref{eq:horospherical-decomposition}). Not all terms are uniquely defined, but $\mu_\fapp(\gamma_\ell)\in \fap$, $H_\scrC(g_\ell)$, and $H_\scrC(\gamma_\ell g_\ell)\in\fa_\scrC$ are.

By assumption, $k_{2,\ell}\to k_{2,\infty}\in K$. By uniform transversality, the $N_\scrC^-$ coordinates of $k_{2,\infty}^{-1}g_\ell P_\scrC$ stay in a bounded set. Since $k_{2,\ell}^{-1}k_{2,\infty}\to e_G$, smoothness of the big-cell coordinates gives the same conclusion for $k_{2,\ell}^{-1}g_\ell P_\scrC$. Consequently, for $\ell_0$ large enough, there are bounded sets $U_{\ell_0}\subset N_\scrC^-$ such that for $\ell\geq \ell_0$:
\[
k_{2,\ell}^{-1}g_\ell P_\scrC = n^-_\ell P_\scrC \in U_{\ell_0} P_\scrC\subset N_\scrC^-P_\scrC.
\]
This gives a decomposition
\begin{equation}\label{eq:proof_key_lemma}
k_{2,\ell}^{-1}g_\ell =  n^-_\ell p_\ell'\text{ where }p_\ell' \in P_\scrC.
\end{equation}
We then write:
\begin{align*}
\gamma_\ell g_\ell & = k_{1,\ell} e^{\mu_\fapp(\gamma_\ell)} k_{2,\ell}^{-1} g_\ell = k_{1,\ell} e^{\mu_\fapp(\gamma_\ell)} n_\ell^- p_\ell' = k_{1,\ell} e^{\mu_\fapp(\gamma_\ell)} n_\ell^-e^{-\mu_\fapp(\gamma_\ell)} e^{\mu_\fapp(\gamma_\ell)}p_\ell' .
\end{align*}
Notice that $(\mu_\fapp(\gamma_\ell))_{\ell \geq 0}$ is $\scrC$-divergent by assumption and $(n_\ell^-)_{\ell \geq \ell_0}\subset U_{\ell_0}$ is bounded. Hence by \eqref{eq:unipotent_contraction}, we have
\[
\tilde n^-_\ell = e^{\mu_\fapp(\gamma_\ell)} n_\ell^-e^{-\mu_\fapp(\gamma_\ell)} \in e^{\mu_\fapp(\gamma_\ell)} U_{\ell_0}e^{-\mu_\fapp(\gamma_\ell)}  \to e_G.
\]
This yields, using \eqref{eq:H_scrC_formula},
\begin{equation}\label{equation:tard}
\begin{split}
H_\scrC(\gamma_\ell g_\ell) &= H_\scrC(k_{1,\ell} \tilde n^-_\ell e^{\mu_\fapp(\gamma_\ell)}p_\ell' ) \\
&= H_\scrC(\tilde n^-_\ell)+ H_\scrC( e^{\mu_\fapp(\gamma_\ell)})+ H_\scrC(p_\ell')\\
&\in \mu_\scrC(\gamma_\ell) + H_\scrC(p_\ell') + H_{\scrC}(e^{\mu_\fapp(\gamma_\ell)} U_{\ell_0}e^{-\mu_\fapp(\gamma_\ell)}).
\end{split}
\end{equation}
To conclude the proof, it now suffices to show that $H_\scrC(p'_\ell)-H_\scrC(g_\ell) \in H_\scrC(U_{\ell_0} K)$. Write $p_\ell' = m_\ell' e^{H_\scrC(p'_\ell)} n_\ell'$ (Langlands decomposition, see \eqref{eq:langlands}). Using \eqref{eq:H_scrC_formula} we compute
\begin{align*}
H_\scrC(p'_\ell) - H_\scrC(g_\ell) &= H_\scrC(p'_\ell) - H_\scrC(p_\ell)\\
&=H_\scrC(p_\ell'p_\ell^{-1}) \\
&\overset{\eqref{eq:proof_key_lemma}}{=} H_\scrC ( (n_\ell^-)^{-1} k_{2,\ell}^{-1} k_\ell) \in H_\scrC (U_{\ell_0} K).  \\
\end{align*}
Combined with \eqref{equation:tard}, this concludes the proof.
\end{proof}

\begin{proposition}\label{prop:odms_proper}
 Let $\Gamma < G$ be a discrete $\scrC$-divergent subgroup. Let
 $\beta \in
 \bigl(\scrL_{\Gamma,\mathrm C}^{\scrC}\bigr)^{*,\circ}$
 lie in the interior of the
 AC--Cartan dual limit cone, let $\fh:= \ker\beta\subset \fa_\scrC$ be the kernel of $\beta: \fa_\scrC\to \RR$, and let $H:=\exp(\fh)< A_\scrC$ be the corresponding subgroup. Then the left-right action of $\Gamma \times H$ on $\odms$ is free and proper. As a consequence,
\[
\mathcal{M} := \Gamma \backslash \odms, \qquad \mathcal{N} := \Gamma\backslash \odms/ H
\]
are analytic manifolds without boundary and $\pi : \mathcal{M} \to \mathcal{N}$ is a principal $H$-bundle with $H\simeq\R^k$.
\end{proposition}

\begin{proof}
We argue by contradiction. If the action is not proper, there exist sequences $(\gamma_\ell, H_\ell)\in \Gamma\times\fh$, $(\gamma_\ell, H_\ell)\to \infty$, and $g_\ell\in G$, such that
\[
g_\ell M_\scrC \to g M_\scrC,\quad \gamma_\ell g_\ell \exp(H_\ell) M_\scrC \to h M_{\scrC},
\]
for points $g M_\scrC$, $h M_{\scrC}$ in $\odms$. First, let us observe that $(\gamma_\ell)_{\ell \geq 0}$ must diverge. If not, $H_\scrC(\gamma_\ell g_\ell)$ is bounded and $H_\scrC(\gamma_\ell g_\ell  \exp(H_\ell))\to H_\scrC(h)$ is also bounded. However, from $H_\scrC(\gamma_\ell g_\ell  \exp(H_\ell)) = H_\scrC(\gamma_\ell g_\ell)+H_\ell$, one would deduce that $(H_\ell)_{\ell \geq 0}$ is bounded too, which is a contradiction.

Since $(\gamma_\ell)_{\ell \geq 0}$ diverges, we may pass to a subsequence and assume that it converges at infinity. Furthermore, by $\scrC$-divergence of $\Gamma$ and Lemma~\ref{lemma:i-invariant}, we know that $(\gamma_\ell)_{\ell \geq 0}$ is $\scrC$-divergent and $-\scrC$-divergent. We can now use the DMS domain properties. As $gM_\scrC \in \odms$, either $gP_\scrC \pitchfork \gamma_\scrC^-$ or $gP_{-\scrC}\pitchfork \gamma_{-\scrC}^-$. Thus, by Lemma~\ref{lemma:translation-length}, either
\[
H_\scrC(\gamma_\ell g_\ell\exp(H_\ell)) \in \mu_\scrC(\gamma_\ell) + H_\scrC(g_\ell) +  H_\ell + \mc B,
\]
or
\[
H_{-\scrC}(\gamma_\ell g_\ell\exp(H_\ell)) \in \mu_{-\scrC}(\gamma_\ell) + H_{-\scrC}(g_\ell) +H_\ell + \mc B,
\]
for some bounded set $\mc B\subset \fa_{\pm\scrC}$.
On the other hand, we must also have $H_{\pm\scrC}(\gamma_\ell g_\ell \exp(H_\ell)) \to H_{\pm\scrC}(h)$, so passing from $\fa_\scrC$ to $\fa_\scrC/\fh$, we have that
\[
[\mu_{\pm\scrC}(\gamma_\ell)] \in \mc B/\fh \subset \fa_{\pm\scrC} / \fh,
\]
	for some (possibly larger) bounded set $\mc B \subset \fa_{\pm\scrC}$.

		Since
		$\beta\in\bigl(\scrL_{\Gamma,\mathrm C}^{\scrC}\bigr)^{*,\circ}$ and
		$\scrL_{\Gamma,\mathrm C}^{-\scrC}
		=-\scrL_{\Gamma,\mathrm C}^{\scrC}$, the projection
		\[
		\scrL_{\Gamma,\mathrm C}^{\pm\scrC}
		\longrightarrow
		\fa_\scrC/\fh
		\]
		is proper, and remains proper on the closure of a sufficiently
		small open conical thickening $\mc U_\pm$ of
		$\scrL_{\Gamma,\mathrm C}^{\pm\scrC}\setminus\{0\}$. By the definition
		of the AC--Cartan cone,
		$\mu_{\pm\scrC}(\gamma_\ell)$ eventually lies in and diverges inside
		$\mc U_\pm$, contradicting the boundedness in the quotient obtained
		above.

	Thus, no such sequences $(\gamma_\ell,H_\ell,g_\ell)$ exist, and the action is proper.

The freeness follows from the same argument: if $\gamma_0 g\exp(H_0)M_\scrC = gM_\scrC$, then we can take the sequence $(\gamma_0^\ell, \ell H_0)$ and the constant sequence $g_\ell=g$.

That $\mathcal M\to \mathcal N$ is a $H\cong\RR^k$-fiber bundle follows directly from the properness and freeness.
\end{proof}

\begin{corollary}\label{cor:flow_on_N}
 For all $u \in \fa_\scrC\setminus\fh$, the flow $(e^{tu})_{t \in \R}$ on $\odms$ induces a nontrivial, analytic and complete flow $(\phi_t)_{t \in \R}$ on $\N$.
\end{corollary}

This is an immediate consequence. If $u\in \fa_\scrC$, then $(e^{tu})_{t \in \R}$ acts on $\odms$ from the right by $A_\scrC$-invariance, and this action descends to $\mathcal N$ by commutativity of $A_\scrC$.
The fact that $\beta(u)\neq 0$ implies that this flow is nontrivial.

In order to match the time direction of the flow with the chosen generalized positive Weyl chamber $\scrC$ we will always assume that $u\in \fa_\scrC\setminus \fh$ lies on the same side of the hyperplane $\fh$ as
$\scrL_{\Gamma,\mathrm C}^{\scrC}$. The induced flow on $\N$ depends only on the class $u+\fh$. We may therefore choose, and henceforth use, a representative $u\in\scrC$ of this positive class.
\subsection{Trapped set}
Note that we have already shown several properties of $\odms$ but not yet that it is nonempty. This changes if we add the transversality assumption. We introduce the \emph{trapped set} which is defined in the Hopf coordinates as
\begin{equation}
\label{equation:trapped-set-up-hopf}
\begin{split}
\tilde{\mathscr{J}}_{\scrC} & := \{ (\xi^+, \xi^-, a) ~:~ \xi^\pm \in \Lambda_{\pm\scrC}, \xi^+\pitchfork \xi^{-}, a\in \fa_\scrC \}\\
& \subset (G/P_\scrC)\overset{\pitchfork}{\times}G/P_{-\scrC}\times\fa_\scrC \simeq G/M_\scrC.
\end{split}
\end{equation}
We will sometimes drop the index $\scrC$ on the trapped set when the context is clear. By the $\Gamma$-invariance of $\Lambda_{\pm\scrC}$, the trapped set is easily seen to be $\Gamma$-invariant. We also observe that if $\scrC' \leq\scrC$, then by \eqref{eq:limit_pt_projection} and \eqref{eq:limit_point_projection} we get
\[
\pi_{\scrC'}\tilde\scrJ_\scrC \subset \tilde{\scrJ}_{\scrC'}.
\]
The following holds:

\begin{lemma}\label{lem:DMS_transversal_grps}
If $\Gamma < G$ is $\scrC$-transverse, then $\tilde{\mathscr{J}}_{\scrC} \subset \odms$.  Additionally, if $\kappa : \Lambda_{\scrC} \to \Lambda_{-\scrC}$ is the homeomorphism defined in \eqref{eq:limit_homeo}, then:
\[
\odms :=\left \{gM_\scrC ~:~   \forall \xi\in \Lambda_\scrC, \text{ either } gP_\scrC\pitchfork \kappa(\xi) \text{ or } gP_{-\scrC}\pitchfork \xi \right \}.
\]
\end{lemma}
Note that transversality is crucial for the inclusion $\tilde{\mathscr{J}} \subset \odms$. An analogue of the second formula also holds without the assumption of transversality, at least for $\iota_\fapp$-invariant $\scrC$, as established in \eqref{eq:simpler_DMS_def}.

\begin{proof}
We start by proving the formula for $\odms$. Recalling Definition \ref{definition:conv-infinity} and Lemma \ref{lem:forward_backward_same}, \ref{it:limit_-scrC}, we find that $\gamma_\scrC^-\in\Lambda_{-\scrC}\text{ and }\gamma_{-\scrC}^-\in\Lambda_\scrC$ are never transverse, so, because $\Gamma$ is $\scrC$-transverse, we must have $\gamma_{\scrC}^- = \kappa(\gamma_{-\scrC}^-)$. Now, the announced formula is a consequence of the following facts: the complement of $\odms$ is closed, the non-transversality condition is closed, and the set of limit points is dense in $\Lambda_\scrC$.

We now prove the inclusion $\tilde{\mathscr{J}}_{\scrC} \subset \odms$. Let $gM_{\scrC} \in \tilde{\mathscr{J}}_{\scrC}$, and write $\xi^-:=gP_{-\scrC} \in \Lambda_{-\scrC}$. Let $\xi \in \Lambda_{\scrC}$. If $\xi^- \pitchfork \xi$, then $gM_{\scrC} \in \odms$. If $\xi^-$ is not transverse to $\xi$, then $\xi^-=\kappa(\xi)$ by the transversality assumption. However, we then have $gP_{\scrC} \pitchfork \kappa(\xi)$ since $\kappa(\xi)=\xi^-=gP_{-\scrC}$. Thus $gM_{\scrC} \in \odms$.
\end{proof}

Let $\beta\in\bigl(\scrL_{\Gamma,\mathrm C}^{\scrC}\bigr)^{*,\circ}$, and set $\fh=\ker\beta$ and $H=\exp(\fh)< A_\scrC$, as in Proposition~\ref{prop:odms_proper}. We now introduce:
\[
\mathscr{J} := \Gamma\backslash\tilde{\mathscr{J}}, \qquad \mathscr{K} := \mathscr{J}/H.
\]
It follows from \eqref{equation:trapped-set-up-hopf} that, in the Hopf coordinates, 
\begin{equation}
\label{equation:trapped-set-hopf}
\begin{split}
\scrJ & = \Gamma \backslash\{(\xi^+,\xi^-,a) ~:~ \xi^\pm \in \Lambda_{\pm \scrC}, \xi^+ \pitchfork \xi^-, a \in \fa_{\scrC}\}, \\
 \scrK &= \Gamma \backslash\{(\xi^+,\xi^-,s) ~:~ \xi^\pm \in \Lambda_{\pm \scrC}, \xi^+ \pitchfork \xi^-, s \in \fa_{\scrC}/\fh\}.
 \end{split}
\end{equation}
We also introduce the incoming and outgoing tails:
\begin{equation}
\label{equation:tails-anosov}
\begin{split} 
\mc{T}_{-,\M} & := \Gamma \backslash\{(\xi^+,\xi^-,a) ~:~ \xi^+ \in \Lambda_\scrC, \xi^- \in G/P_{-\scrC}, \xi^+ \pitchfork \xi^-, a \in \fa_{\scrC}\}, \\
\mc{T}_{+,\M} & := \Gamma \backslash\{(\xi^+,\xi^-,a) ~:~ \xi^+ \in G/P_\scrC, \xi^- \in \Lambda_{-\scrC}, \xi^+ \pitchfork \xi^-, a \in \fa_{\scrC}\}, \\
\mc{T}_{-,\N} & := \mc{T}_{-,\M}/H, \\
\mc{T}_{+,\N} & := \mc{T}_{+,\M}/H. \\
\end{split}
\end{equation}
These are unions of weak stable and unstable manifolds of points in the trapped set, introduced in \eqref{eq:formula-strong-stable-leaf} and \eqref{eq:formula-weak-unstable-leaf}. We now justify the terminology ``trapped set'' and ``incoming/outgoing tails'' and show that the above objects have the required dynamical properties of the trapped set and incoming/outgoing tails defined for Axiom A actions of type $(k,1)$ in \S\ref{ssection:hyperbolic-base-dynamics}.

We first prove the following lemma which captures the core of the properties of the trapped set and which, in this form, will also be useful in the proof of hyperbolicity (Proposition~\ref{prop:hyperbolicity}).

\begin{lemma}\label{lem:trapped_set_tool}
Let $\Gamma$ be $\scrC$-divergent, and let $\fh\subset\fa_\scrC$ and $u\in\fa_\scrC\setminus\fh$ be as in Proposition~\ref{prop:odms_proper} and Corollary~\ref{cor:flow_on_N}.
 Let $t_\ell\geq 0, \gamma_\ell\in\Gamma$ and $g_\ell M_\scrC H$ be sequences such that 
 \begin{itemize}
\item $t_\ell\to+\infty$,
\item $g_\ell M_\scrC H\to gM_\scrC H\in \odms/H$,
\item $\gamma_\ell g_\ell e^{t_\ell u}M_\scrC H \to g'M_\scrC H\in\odms/H$.
 \end{itemize}
Then $(\gamma_\ell)_{\ell \geq 0}$ converges at infinity and
\[
 gP_\scrC = \gamma_{-\scrC}^- \in \Lambda_{\scrC}, \qquad g'P_{-\scrC} = \gamma_{-\scrC}^+ \in \Lambda_{-\scrC}.
\]
In particular, $gP_\scrC \not\pitchfork \gamma_\scrC^-$.
\end{lemma}

\begin{proof}
We may choose the representatives such that $g_\ell\to g$ and $\gamma_\ell g_\ell e^{t_\ell u}m_\ell h_\ell \to g'$ for some $m_\ell\in M_\scrC$ and $h_\ell\in H$.
Note that
\[
 H_\scrC(\gamma_\ell g_\ell e^{t_\ell u}m_\ell h_\ell) = H_\scrC(\gamma_\ell g_\ell) + t_\ell u + \log(h_\ell)
\]
and, as $\gamma_\ell g_\ell e^{t_\ell u}m_\ell h_\ell \to g'$, this sequence must be bounded in $\fa_\scrC$. Applying $\beta$ to the displayed identity gives
\[
\beta\!\left(H_\scrC(\gamma_\ell g_\ell)\right)=-t_\ell\beta(u)+O(1),
\]
because $\beta(\log h_\ell)=0$. Hence $(\gamma_\ell)$ leaves every compact subset of $\Gamma$. Because $t_\ell\to\infty$, the distance between $t_\ell u$ and the hyperplane $\fh$ diverges and thus $H_\scrC(\gamma_\ell g_\ell)$ must be unbounded in order to compensate for that. Thus $(\gamma_\ell)_{\ell \geq 0}$ is an unbounded sequence. After possibly passing to a subsequence (we will see below that this is not necessary), we can assume that it is $\pm\scrC$-divergent and convergent at infinity.

As $gM_\scrC\in \odms$ we have either $gP_\scrC \pitchfork \gamma_\scrC^-$ or $gP_{-\scrC} \pitchfork \gamma_{-\scrC}^-$. Let us assume $gP_\scrC \pitchfork \gamma_\scrC^-$. Then, for $\ell$ large enough, we have $g_\ell e^{t_\ell u} m_\ell h_\ell P_\scrC \pitchfork \gamma_{\scrC}^-$ uniformly and, by Lemma~\ref{lemma:translation-length}, we get that
\[
 H_\scrC(\gamma_\ell g_\ell e^{t_\ell u}m_\ell h_\ell) = \mu_\scrC(\gamma_\ell) + t_\ell u +\log(h_\ell) +O(1).
\]
As above, this sequence has to be bounded in $\fa_\scrC$.
As in the proof of Proposition~\ref{prop:odms_proper},
$\mu_\scrC(\gamma_\ell)$ diverges in a sufficiently small conical thickening of
$\scrL_{\Gamma,\mathrm C}^{\scrC}$. This thickening may be chosen on the same
side of $\fh$ as $u$, so $\mu_\scrC(\gamma_\ell)$ and $t_\ell u$ cannot
compensate modulo $\fh$.
 So $gP_{\scrC}$ is not transverse to $\gamma_\scrC^-$.

We thus know that $gP_{-\scrC}\pitchfork \gamma^-_{-\scrC}$ and again for all sufficiently large $\ell$ the same transversality holds uniformly, that is, $g_\ell P_{-\scrC} \pitchfork \gamma^-_{-\scrC}$. By Proposition~\ref{prop:proximal} we conclude that:
\[
 g'P_{-\scrC} = \lim_{\ell \to \infty} \gamma_\ell g_\ell P_{-\scrC} = \gamma^+_{-\scrC}
\]
By definition of transversality $g'P_{\scrC}\pitchfork g'P_{-\scrC} = \gamma^+_{-\scrC}$. For $\eta_\ell = \gamma_\ell^{-1}$, by Lemma~\ref{lem:forward_backward_same}, we get $\gamma^+_{-\scrC} =\eta_{\scrC}^-$. We can apply Proposition~\ref{prop:proximal}, which yields:
\[
 gP_\scrC = \lim_{\ell\to\infty} \eta_\ell(g' P_\scrC) = \eta^+_\scrC = \gamma^-_{-\scrC}.
\]
To summarize, $\gamma_{-\scrC}^{\pm}$ are uniquely determined by $gP_\scrC$ and $g'P_{-\scrC}$. Apply the preceding argument to an arbitrary subsequence: it has a further subsequence converging at infinity, and every such further subsequence has these same two limit flags. Consequently, the original sequence converges at infinity.

The last point of the lemma follows from Lemma~\ref{lem:forward_backward_same}, \ref{it:limit_-scrC}.
\end{proof}

We now use Lemma~\ref{lem:trapped_set_tool} to give a dynamical description of $\mc{T}_{\pm,\N}$ and $\scrK$ (this should be compared to Assumption $\hyperlink{AA2}{\rm(A2)}$):

\begin{proposition}\label{prop:dynamics_trapped_set}
 Let $\Gamma$ be $\scrC$-transverse, and let $\fh\subset\fa_\scrC$ and $u\in\fa_\scrC\setminus\fh$ be chosen as in Proposition~\ref{prop:odms_proper} and Corollary~\ref{cor:flow_on_N}, and let $(\phi_t)_{t \in \R}$ be the induced flow on $\N$. Then
 \begin{enumerate}[label=\emph{(\roman*)}]
  \item For any compact set $C \subset \N \setminus \mc{T}_{\mp,\N}$, the map $C\times\mathbb{R}_\pm \ni (x,t)\mapsto \phi_t(x)$ is proper.
  \item If $\mathscr K$ is compact, then there exists an open relatively compact set $\mathscr V$ such that for any $x\in \mc{T}_{\pm,\N}$, there is $T \geq 0$ such that for all $t\geq T$ we have $\phi_{\mp t}(x) \in \mathscr V$.
 \end{enumerate}
\end{proposition}
\begin{proof} 
(i) Let $C \subset \mc{N} \setminus \mc T_{-,\N}$ be a compact subset. We need to show that $C \times \R_+ \ni (x,t) \mapsto \phi_t(x)$ is proper.
If not, there exist sequences $(x_n)\subset C$ and $(t_n)$ with $t_n\to+\infty$ such that $\phi_{t_n} x_n \to y \in \mc{N}$.
By passing to a subsequence we can assume $x_n\to x\in C$. This yields a sequence $g_n M_{\scrC} H\to g M_\scrC H \in \odms/H \subset G/M_\scrC H$
and a sequence $(\gamma_n)\subset\Gamma$ such that
\[
\gamma_n g_n e^{t_n u} M_{\scrC}H \to g' M_{\scrC}H \in \odms/H.
\]
Lemma~\ref{lem:trapped_set_tool} thus implies that $gP_\scrC = \gamma_{-\scrC}^-\in \Lambda_\scrC$.
However, this contradicts $C\cap \mc T_{-,\N}  = \emptyset$. The case $\mc T_{+,\N}$ follows by reversing time in the above argument.\\

(ii) Finally, we need to show that elements in $\mc T_{-,\N}$ are trapped in the future in a neighborhood $\scrV$ of the trapped set $\scrK$. Choose a smooth metric on $\N$ and a relatively compact open neighborhood $W$ of $\scrK$. If $0<\varepsilon<d(\scrK,\N\setminus W)$, then
\[
\scrV:=\{x\in\N:d(x,\scrK)<\varepsilon\}
\]
is an open neighborhood with $\overline{\scrV}\subset W$, and hence is relatively compact. Indeed, let $x:= \Gamma g M_\scrC H \in \mc{T}_{-,\N}$, where the lift $gM_{\scrC}$ is given in the Hopf coordinates by $(gP_{\scrC}, gP_{-\scrC}, H_{\scrC}(g))$ with $gP_{\scrC} \in \Lambda_{\scrC}$. There exists at least one point $\xi^- \in \Lambda_{-\scrC}$ which is transverse to $gP_{\scrC}$. By \eqref{equation:diffeo-la}, we thus have $\xi^- = g n^+ P_{-\scrC}$ for some $n^+ \in N^+_{\scrC}$. Then $g n^+ M_{\scrC} \in \tilde{\scrJ}$. Moreover, $g n^+ e^{tu} M_{\scrC} = g e^{t u} e^{-t u} n^+ e^{tu} M_{\scrC}$. In the quotient space $\Gamma \backslash G/(M_{\scrC}H)$, $g n^+ M_{\scrC}$ projects onto a point $x_0 = \Gamma g n^+ M_{\scrC}H \in \scrK$ and thus $\phi_t(x_0) = \Gamma g n^+ e^{tu} M_{\scrC}H  \in \scrK$, which is compact. Using that $e^{-t u} n^+ e^{tu} = e_G + \mc{O}(e^{-\eps t})$, we find that $d(\Gamma g n^+ e^{tu} M_{\scrC}H, \Gamma g e^{tu}M_{\scrC}H) \to_{t \to +\infty} 0$. That is, the point $x \in \mc{T}_{-,\N}$ satisfies $d(\phi_t (x),\scrK) \to_{t \to +\infty} 0$. This proves the claim.
\end{proof}

\subsection{Hyperbolicity}
\label{sec:divergence-and-hyperbolicity}

In \S\ref{ssec:locally-homogeneous-dynamics}, we introduced $G$-invariant, $A_{\scrC}$-invariant subbundles $E_0,E_s,E_u$ of $T(G/M_\scrC)$, see \eqref{eq:stable-unstable}. They descend naturally to $\M$ and $\N$ and are the natural candidates for stable/unstable bundles. By abuse of notation, we use the same symbols for these descended bundles. However, since $M_\scrC$ is noncompact there is in general no $G$-invariant Riemannian metric on $G/M_\scrC$.
We have to content ourselves with an arbitrary analytic Riemannian metric on $\N$. The proof that the above-mentioned bundles are (un)stable bundles is quite subtle and requires the assumption of $\scrC$-regularity (see Definition~\ref{def:divergent-regular-transverse-anosov}\,(\ref{it:regular})), which we will make from now on.

Recall from \S\ref{ssec:locally-homogeneous-dynamics} the notation $[g,X]$ for tangent vectors to $G/M_\scrC$. Under the projection $G/M_\scrC\supset\odms\to\Gamma\backslash\odms/H=\mc N$, tangent vectors in $T\mc N$ can be represented by equivalence classes $[\Gamma g,X]$. Here, $[\Gamma g,X]=[\Gamma g',X']$ if and only if there exist $\gamma\in\Gamma$, $m\in M_\scrC$, and $h\in H$ such that $\gamma gmh=g'$ and $X=\mathrm{Ad}(mh)X'$.

Let $\tilde\Omega_{\mathrm{DMS}}^\scrC\subset G$ be the preimage of $\odms$ under the natural projection $G\to G/M_\scrC$. We can choose a real-analytic Riemannian metric on $\mc N$. Set
\[
\mathfrak v:=\mathbb Ru\oplus\fn_\scrC^+\oplus\fn_\scrC^-.
\]
Since $\fa_\scrC=\fh\oplus\mathbb Ru$, for every $g\in\tilde\Omega_{\mathrm{DMS}}^\scrC$ the map
\[
\mathfrak v\ni X\longmapsto[\Gamma g,X]\in T_{\Gamma gM_\scrC H}\mc N
\]
is an isomorphism. Moreover, $\mathfrak v$ is preserved by $\mathrm{Ad}(M_\scrC A_\scrC)$. The Riemannian metric on $\mc N$ therefore defines a norm $\|X\|_g:=\|[\Gamma g,X]\|$ on $\mathfrak v$, depending continuously on $g$. Hence, for any compact set $B\subset \tilde\Omega_{\mathrm{DMS}}^\scrC$, there exists a constant $C>0$ such that for $g,g'\in B$,
\[
\frac{1}{C}\leq \frac{\|[\Gamma g,X]\|}{\|[\Gamma g',X]\|}\leq C,
\qquad X\in\mathfrak v\setminus\{0\}.
\]
To simplify notation we take an arbitrary reference point $o\in B$ and define a reference norm on $\mathfrak v$ by setting $\|X\|_{\mathfrak v} := \|[\Gamma o, X]\|$. We then get for $g \in B$:
\begin{equation}\label{eq:riemann_metric_bounds_general}
\frac{1}{C}\leq\frac{\|[\Gamma g,X]\|}{\|X\|_{\mathfrak v}}\leq C,
\qquad X\in\mathfrak v\setminus\{0\}.
\end{equation}

This provides an estimate on the derivative of the flow in Lie-theoretic terms:

\begin{lemma}\label{lemma:local-bound-differential-flow}
Fix an arbitrary Riemannian metric on $\mc N$. Let $B\subset \mathcal N$ be compact, and let $\tilde B\subset\tilde\Omega_{\mathrm{DMS}}^\scrC$ be compact such that $B = \Gamma\tilde B M_\scrC H$.
Then there exists a constant $C>0$ such that, for all $x \in B$ and $t > 0$ with $\phi_t(x)\in B$, there exist $g,g'\in\tilde B$, $h\in H$, and $m\in M_\scrC$
such that $g$ and $g'$ represent $x$ and $\phi_t(x)$, respectively, and the following holds for every $X\in\mathfrak v\setminus\{0\}$:
\begin{equation}\label{eq:bdd_norms}
\frac{1}{C}\leq
\frac{\|(d\phi_t)([\Gamma g,X])\|_{T_{\phi_t(x)}\mathcal N}}
{\|\mathrm{Ad}(hm e^{-tu})X\|_{\mathfrak v}}
\leq C.
\end{equation}
\end{lemma}

\begin{proof}
By assumption, we have for $x = \Gamma g M_\scrC H\in B$
\[
g e^{tu} = \gamma g' h m,
\]
with $g'\in \tilde B$, $\gamma\in\Gamma$, $h\in H$, $m\in M_\scrC$. We write
\begin{align*}
d(e^{t u})[\Gamma g, X] &= \frac{d}{ds}_{|s=0} \Gamma g e^{sX} e^{tu}H M_\scrC \\
						&= \frac{d}{ds}_{|s=0} \Gamma g e^{tu} e^{-tu}e^{sX}e^{tu} H M_\scrC \\
						&= [\Gamma g e^{tu}, \mathrm{Ad}(e^{-tu}) X]\\
						&= [\Gamma g', \mathrm{Ad}(hm e^{-tu})X].
\end{align*}
Since $g$ and $g'$ lie in $\tilde B$, a bounded set, we can use \eqref{eq:riemann_metric_bounds_general} and get the required estimate.
\end{proof}

With Lemma \ref{lemma:local-bound-differential-flow} in hand, we can give a general contraction estimate. (We restrict to the $E_s$ bundle, the statement for $E_u$ is obtained by reversing time.)
\begin{lemma}\label{lemma:hyperbolicity-tool}
Let $\Gamma<G$ be a $\scrC$-divergent group. Assume that for some sequence $g_\ell\in G$, there are $\gamma_\ell\in\Gamma$, $m_\ell\in M_\scrC$, and $a_\ell\in\fa_\scrC$ such that
\[
g_\ell M_\scrC\to g M_\scrC,\qquad \gamma_\ell g_\ell m_\ell e^{a_\ell} M_\scrC \to g' M_\scrC,\qquad \gamma_\ell\to\infty,
\]
with $gM_\scrC, g'M_\scrC\in\odms$. Then, for all sufficiently large $\ell$, $a_\ell\in \scrC\cup-\scrC$.

If $a_\ell\in \scrC$ for all sufficiently large $\ell$, we choose an open Weyl chamber $\fapp\geq \scrC$ and consider the induced open Weyl chamber $\fa_{M_\scrC,++}\subset \fa_\scrC^\perp$ of $M_\scrC$ (see \S\ref{sssec:levi}). If we assume that $g_\ell$ and $g'_\ell = \gamma_\ell g_\ell m_\ell e^{a_\ell}$ are bounded (which can always be achieved by choosing the $m_\ell$ appropriately), then there exists a constant $C>0$ such that
\begin{equation}\label{eq:approximation-of-mu}
\|\mu_\fapp (\gamma_\ell^{-1}) - (a_\ell + \mu_{\fa_{M_\scrC,++}}(m_\ell))\|\leq C,
\end{equation}
and
\begin{equation}
\label{equation:adml}
\| \mathrm{Ad}(m_\ell^{-1} e^{-a_\ell})\|_{\fn^+_\scrC} \leq C \exp\left(-\inf_{\alpha\in\Delta^+_\scrC} \alpha(\mu_\fapp(\gamma_\ell^{-1}))\right) \to_{\ell \to \infty} 0.
\end{equation}
The constant $C>0$ depends only on a bounded set containing the sequences $(g_\ell)_{\ell \geq 0}$ and $(g'_\ell)_{\ell \geq 0}$.
\end{lemma}

\begin{proof}
For any $\beta\in
\bigl(\scrL_{\Gamma,\mathrm C}^{\scrC}\bigr)^{*,\circ}$,
properness of the action of $\Gamma$ on $\odms/\exp(\ker\beta)$ (Proposition \ref{prop:odms_proper}) implies that $a_\ell \mod \ker\beta\to \infty$. We can take $u\in \scrC$ not in $\ker \beta$, so that $a_\ell = t_\ell u \mod \ker\beta$, with $t_\ell\to\pm \infty$. Let us consider the case $t_\ell\to +\infty$ for which we will show that $a_\ell \in \scrC$ (the other case leads to $a_\ell \in -\scrC$).  Lemma~\ref{lem:trapped_set_tool} then applies and gives
\[
 g'P_{-\scrC} = \gamma_{-\scrC}^+,\quad gP_\scrC =\gamma_{-\scrC}^-.
\]
Since we must have $g P_{-\scrC}\pitchfork gP_\scrC$, we deduce that for all sufficiently large $\ell$ and sufficiently small $\varepsilon>0$, $g_\ell \exp(B_\varepsilon(\fn_\scrC^+))P_{-\scrC}\pitchfork \gamma_{-\scrC}^-$ uniformly. By Proposition~\ref{prop:proximal}, for $X_\ell \in B_\varepsilon(\fn_\scrC^+)$,
\[
 \gamma_\ell g_\ell e^{X_\ell}P_{-\scrC} \to\gamma^+_{-\scrC} = g'P_{-\scrC}.
\]
Plugging in $\gamma_\ell g_\ell = g'_\ell m_\ell^{-1} e^{-a_\ell}$, we get
\[
 g_\ell'm_\ell^{-1} e^{-a_\ell} e^{X_\ell} P_{-\scrC} = g_\ell' \exp(\mathrm{Ad}(m_\ell^{-1} e^{-a_\ell}){X_\ell})P_{-\scrC} \to g'P_{-\scrC},
\]
so we conclude
\begin{equation}\label{eq:hyper_first_step}
\|\mathrm{Ad}(m_\ell^{-1} e^{-a_\ell})\|_{\fn^+_\scrC} \to 0.
\end{equation}

Now, we introduce the $K_\scrC \exp(\overline{\fa_{M_\scrC,++}}) K_\scrC$ decomposition of $m_\ell$:
\[
m_\ell = k_{1,\ell} e^{Y_\ell} k_{2,\ell}^{-1},\qquad k_{1,\ell},k_{2,\ell}\in K_\scrC,\ Y_\ell\in \overline{\fa_{M_\scrC,++}}.
\]
Then, using that $K_{\scrC} < M_{\scrC}$ commutes with $A_{\scrC}$, we can write
\[
\mathrm{Ad}(m_\ell^{-1} e^{-a_\ell}) = \mathrm{Ad}(k_{2,\ell} )\mathrm{Ad}(e^{-Y_\ell-a_\ell})\mathrm{Ad}(k_{1,\ell}^{-1}). 
\]
Observe that, by compactness of $K_{\scrC}$, there exists a uniform constant $C>0$ such that
\begin{equation}
 \label{equation:gustos1}
\frac{1}{C}\|\mathrm{Ad}(e^{-Y_\ell-a_\ell})\|_{\fn^+_\scrC} \leq \|\mathrm{Ad}(m_\ell^{-1} e^{-a_\ell})\|_{\fn^+_\scrC} \leq C \|\mathrm{Ad}(e^{-Y_\ell-a_\ell})\|_{\fn^+_\scrC}.
\end{equation}
Now let us consider $\|\mathrm{Ad}(e^{-Y_\ell-a_\ell})\|_{\fn^+_\scrC}$: $\mathrm{Ad}(e^{-Y_\ell-a_\ell})$ preserves the root space decomposition $\fn_\scrC^+ = \bigoplus_{\alpha\in\Delta^+_\scrC}\fg_\alpha$ and acts on $\fg_\alpha$ by multiplication with $e^{-\alpha(Y_\ell + a_\ell)}$; since the root spaces are orthogonal with respect to the scalar product $\langle\cdot,\cdot\rangle$ introduced in \eqref{equation:produit-scalaire}, we have
\begin{equation}
\label{equation:gustos2}
\|\mathrm{Ad}(e^{-Y_\ell-a_\ell})\|_{\fn^+_\scrC}  = e^{-\lambda_\ell},
\end{equation}
where $\lambda_\ell = \min\{ \alpha(a_\ell+Y_\ell)~:~\alpha\in\Delta^+_\scrC\}$.  From \eqref{eq:hyper_first_step}, we deduce that $\lambda_\ell \to+\infty$. On the other hand, for $\alpha \in \Delta^+_{\fapp}\cap \Delta^0_\scrC$, $\alpha(a_\ell + Y_\ell) = \alpha(Y_\ell) \geq 0$ by definition of $\overline{\fa_{M_\scrC,++}}$. Together with the definition of $\lambda_\ell$, this shows that $a_\ell+Y_\ell\in\fa_+=\overline{\fapp}$ for all sufficiently large $\ell$. Since $a_\ell=\pi_\scrC(a_\ell+Y_\ell)$, formula~\eqref{eq:stab-Weyl-and-projector} and the invariance of $\Delta_\scrC^+$ under $\langle\Delta_\scrC^0\rangle$ give, for every $\alpha\in\Delta_\scrC^+$,
\[
\alpha(a_\ell)
=\frac{1}{|\langle\Delta_\scrC^0\rangle|}
\sum_{w\in\langle\Delta_\scrC^0\rangle}
(w^{-1}\alpha)(a_\ell+Y_\ell)
\geq\lambda_\ell>0
\]
for all sufficiently large $\ell$. Hence $a_\ell\in\scrC$. From
\[
\gamma_\ell^{-1} = g_\ell m_\ell e^{a_\ell} (g'_\ell)^{-1}, 
\]
the boundedness of $(g_\ell)_{\ell \geq 0},(g_\ell')_{\ell \geq 0}$ and Lemma \ref{lemma:Benoist-compact-Cartan}, we deduce that there is a $C\geq 0$ depending on the sets bounding $(g_\ell)_{\ell \geq 0},(g_\ell')_{\ell \geq 0}$ such that
\[
\|\mu_\fapp(\gamma_\ell^{-1}) - (a_\ell + Y_\ell) \| \leq C.
\]
This proves \eqref{eq:approximation-of-mu}. The estimate \eqref{equation:adml} then follows from \eqref{eq:approximation-of-mu}, \eqref{equation:gustos1} and \eqref{equation:gustos2}.
\end{proof}

\begin{proposition}\label{prop:hyperbolicity}
Let $\Gamma < G$ be a $\scrC$-divergent group, and let $u$ and $\fh=\ker\beta$ be as in Proposition~\ref{prop:odms_proper} and Corollary~\ref{cor:flow_on_N}.
Fix an arbitrary Riemannian metric on $\mc N$ and a compact set $B\subset\mathcal N$.
Then
\begin{equation}\label{eq:stable_nonquant}
\lim_{\substack{t\to +\infty,\\ x, \phi_t(x)\in B}} \|d_x \phi_t|_{E_s} \| = 0.
\end{equation}
In addition, if $\Gamma$ is $\scrC$-regular (Definition \ref{def:divergent-regular-transverse-anosov}, \ref{it:regular}), let $\fapp\geq\scrC$ be an open Weyl chamber. Then, for any
\begin{equation}\label{eq:condition_lyapunov}
 0\leq \lambda < \inf_{H \in \scrL_{\Gamma,\mathrm C}^{\fapp} \cap (u+\fh+\fa_\scrC^\perp)} \min_{\alpha \in\roots_\scrC^+} \alpha(H),
\end{equation}
there exists $C>0$ such that for all $x\in B$ and $t>0$ such that $\phi_t(x)\in B$, we have:
\begin{equation}\label{eq:stable_expon}
 \|d_x\phi_t|_{E_s}\|\leq Ce^{-\lambda t}.
\end{equation}
\end{proposition}

\begin{remark}
Under the $\scrC$-regularity assumption, the slice
$\scrL_{\Gamma,\mathrm C}^{\fapp}\cap(u+\fh+\fa_\scrC^\perp)$ is compact and disjoint from the root walls
$\bigcup_{\alpha\in\roots_\scrC^+}\ker\alpha$. Hence the infimum in
\eqref{eq:condition_lyapunov} is strictly positive whenever the slice is nonempty. If the slice is empty, the proof shows that the return times $t>0$ with $x,\phi_t(x)\in B$ are bounded.
\end{remark}

\begin{proof}
Let us first choose $\tilde B\subset \tilde\Omega_{\mathrm{DMS}}^\scrC\subset G$ compact such that $B= \Gamma \tilde BM_\scrC H$.
We also pick an arbitrary sequence $x_\ell \in B$ and an arbitrary sequence of times $t_\ell\to +\infty$ such that $\phi_{t_\ell}(x_\ell)\in B$.
We write $x_\ell = \Gamma g_\ell M_\scrC H$ and $\phi_{t_\ell}(x_\ell) = \Gamma g'_\ell M_\scrC H$ with $g_\ell,g'_\ell\in \tilde B$. By definition of the flow $(\phi_t)_{t \in \R}$, we can write
\begin{equation}
 \label{eq:hyper_proof_relations}
\gamma_\ell g_\ell e^{t_\ell u} =  g'_\ell m_\ell h_\ell  \text{ with } \gamma_\ell\in \Gamma,~ m_\ell\in M_\scrC,~ h_\ell\in H.
\end{equation}
By Lemma~\ref{lemma:local-bound-differential-flow} and \eqref{eq:stable-unstable}, we deduce that for some constant $C>0$ depending only on $B$
\[
 \|d_{x_\ell}\phi_{t_\ell}|_{E_s}\| \leq C \|\mathrm{Ad}(h_\ell m_\ell e^{-t_\ell u})  \|_{\fn^+_\scrC}.
\]
After passing to a subsequence, we can assume that $g_\ell \to g \in \tilde B\subset \tilde\Omega_{\mathrm{DMS}}^\scrC$ and $g'_\ell \to g' \in \tilde B$ as well. Arguing as in the proof of Lemma~\ref{lem:trapped_set_tool}, we find that $\gamma_\ell\to\infty$, so that Lemma~\ref{lemma:hyperbolicity-tool} applies. This gives directly \eqref{eq:stable_nonquant}. For the quantitative estimate \eqref{eq:stable_expon} we use Lemma~\ref{lemma:hyperbolicity-tool} again. In the notation of that lemma, the $M_\scrC$-factor is $m_\ell^{-1}$ and $a_\ell=t_\ell u-\log(h_\ell)$. A normalization argument using \eqref{eq:approximation-of-mu}, $\scrC$-regularity, and the coercivity estimate \eqref{eq:interior-dual-limit-cone} for $\beta$ shows that
$t_\ell^{-1}\mu_\fapp(\gamma_\ell^{-1})$ is bounded. Every accumulation point belongs, by the definition of the AC--Cartan cone, to $\scrL_{\Gamma,\mathrm C}^{\fapp}$; the same approximation and the expression for $a_\ell$ place it in $u+\fh+\fa_\scrC^\perp$. Hence
\[
d\Big(t_\ell^{-1} \mu_\fapp(\gamma_\ell^{-1}),\ \scrL_{\Gamma,\mathrm C}^{\fapp}\cap(u+ \fh+\fa_\scrC^\perp) \Big) \to 0.
\]
Together with the $\scrC$-regularity of $\Gamma$ (in the incarnation of \eqref{eq:scrC_reg_limit_cone}) and \eqref{equation:adml}, this implies the required quantitative estimate \eqref{eq:stable_expon}.
\end{proof}

\subsection{Anosov subgroups and Axiom A actions of type $(k,1)$}
Finally, we can establish the result announced in the introduction asserting the existence of a suitable discontinuity domain for the left-right $\Gamma \times A_{\scrC}$-action on $G/M_{\scrC}$ such that the $A_{\scrC}$-action is an Axiom A action of type $(k,1)$ in the sense of Definition \ref{definition:axiom-a-action}:

\begin{theorem}[Properties of the DMS domains]
\label{thm:DMS-main}
Let $\Gamma< G$ be a torsion-free $\scrC$-Anosov subgroup. Then there exists an open set invariant under the left $\Gamma$-action and the right $A_\scrC$-action $\odms\subset G/M_\scrC$, such that $\Gamma$ acts properly discontinuously on $\odms$. The right $A_\scrC$-action on the analytic manifold $\mc M := \Gamma \backslash \odms$ is an Axiom A action of type $(k,1)$ in the sense of Definition~\ref{definition:axiom-a-action}. Furthermore, the hyperplane may be chosen as $\fh:=\ker\beta\subset\fa_\scrC$ for any $\beta$ in the interior of the AC--Cartan dual cone, i.e., for any $\beta\in\bigl(\scrL_{\Gamma,\mathrm C}^{\scrC}\bigr)^{*,\circ}$.
\end{theorem}

\begin{proof}
As $\scrC$-Anosov subgroups are $\scrC$-divergent, the well-definedness and openness of $\odms$ and the left-right $\Gamma \times A_\scrC$-invariance follow from Lemma~\ref{lemma:basic-feature-Omega-DMS}. As Anosov groups are also transverse (Lemma \ref{lemma:anosov-transverse}), the nonemptiness follows from Lemma~\ref{lem:DMS_transversal_grps}. We now need to check that assumptions $\hyperlink{AA1}{\rm(A1)}$, $\hyperlink{AA2}{\rm(A2)}$ and $\hyperlink{AA3}{\rm(A3)}$ are fulfilled.

For any $\beta \in \bigl(\scrL_{\Gamma,\mathrm C}^{\scrC}\bigr)^{*,\circ}$, let $\fh :=\ker\beta\subset\fa_\scrC$ and $H := \exp(\fh)< A_\scrC$. By $\scrC$-divergence of $\Gamma$, $H$ is a codimension-one subgroup that acts freely and properly according to Proposition~\ref{prop:odms_proper}. This shows $\hyperlink{AA1}{\rm(A1)}$.

Let $u \in \fa_{\scrC} \setminus \fh$. Recall the definition of $\tilde \scrJ, \scrJ, \scrK$ in \eqref{equation:trapped-set-up-hopf} and \eqref{equation:trapped-set-hopf}. The definitions show that $\mc T_{\pm,\N}$ are closed and flow-invariant and that $\mc T_{-,\N}\cap\mc T_{+,\N}=\scrK$. The inclusion $\scrL_{\Gamma,\mathrm J}^{\scrC}\subset\scrL_{\Gamma,\mathrm C}^{\scrC}$ and the reversed inclusion of the dual cones imply that $\beta\in\bigl(\scrL_{\Gamma,\mathrm J}^{\scrC}\bigr)^{*,\circ}$. For $\scrC$-Anosov subgroups it is known that for all $\beta$ as chosen above, $\Gamma$ acts co-compactly on $\tilde{\mathscr J}/\exp(\ker\beta)$ (see e.g. \cite[Theorem 3.2.2]{Sambarino-24}). Consequently $\mathscr K$ is compact. Thus Proposition~\ref{prop:dynamics_trapped_set} shows precisely that $\hyperlink{AA2}{\rm(A2)}$ is fulfilled.

Finally, as $\scrC$-Anosov implies that $\Gamma$ is $\scrC$-regular, the assumption on uniform hyperbolicity (Assumption $\hyperlink{AA3}{\rm(A3)}$) follows from Proposition~\ref{prop:hyperbolicity} together with the fact that $\scrK$ is compact for Anosov subgroups.
\end{proof}

\subsubsection{Connection with the general theory of Axiom A actions}\label{sec:connection-axiom-a}

 We now want to further relate algebraic properties of $\Gamma$ to the general dynamical properties of Axiom A actions in \S\ref{sec:introduction-geometric-setting}--\S\ref{section:decay}.

 We have now shown in Theorem \ref{thm:DMS-main} that any $\scrC$-Anosov subgroup gives rise to an Axiom A action of type $(k,1)$. We translate this into the terminology used in \S\ref{sec:introduction-geometric-setting}--\S\ref{section:decay}. We begin with the right $A_\scrC$-action on $\mc M = \Gamma \backslash \odms$. Let $\fh := \ker \beta$ for some $\beta \in \bigl(\scrL_{\Gamma,\mathrm C}^{\scrC}\bigr)^{*,\circ}$ and $u \in \fa_{\scrC} \setminus \fh$. As before, we regard $u$ as a vector field on $\mc M$, using the right $A_\scrC$-action. Let
\[
\mc{N} := \Gamma\backslash \odms / \exp(\fh).
\]
The flow $(e^{tu})_{t \in \R}$ on $\M$ descends to the quotient $\N$ and defines a flow $(\phi_t)_{t \in \R}$. Observe that
\begin{equation}
\label{equation:description-flow}
\phi_t(\Gamma gM_\scrC H) = \Gamma ge^{tu}M_\scrC H, \qquad t \in \R, g \in G.
\end{equation}
In the terminology of \S\ref{sec:introduction-geometric-setting}, we have $X_\M = u$.

The following holds:

\begin{proposition}[Properties of the free Abelian cocycle]
\label{proposition:wagons}
Let $\Gamma < G$ be a non-elementary $\scrC$-Anosov subgroup of $G$ and $\beta\in (\scrL_{\Gamma,C}^\scrC)^{*,\circ}$. Then the $A_\scrC$ action on $\Gamma\backslash\Omega^{\scrC}_{\mathrm{DMS}}$, with $H=\ker\beta< A_\scrC$  satisfies the assumptions $\hyperlink{AA1}{\rm(A1)-(A4)}$ of an Axiom A action of type $(k,1)$. In addition, if $\Gamma$ is Zariski dense, then the free Abelian cocycle has full rank and is non-arithmetic.
\end{proposition}

\begin{proof}
Theorem~\ref{thm:DMS-main} established that $\hyperlink{AA1}{\rm(A1)-(A3)}$ are satisfied. Assumption \hyperlink{AA4}{\rm(A4)} follows from the topological transitivity of the $\Gamma$-action on $\Lambda_{\scrC} \overset{\pitchfork}{\times} \Lambda_{-\scrC}$. This is a standard fact, but we briefly recall the proof. It suffices to show that for open subsets $U_1, U_2 \subset \Lambda_{\scrC}$ and $V_1,V_2 \subset \Lambda_{-\scrC}$ such that $U_i \times V_i \subset \Lambda_{\scrC} \overset{\pitchfork}{\times} \Lambda_{-\scrC}$, there exists $\gamma \in \Gamma$ such that $\gamma(U_1 \times V_1) \cap U_2 \times V_2 \neq \emptyset$. Since $\Gamma$ is $\scrC$-Anosov, it is word-hyperbolic and admits continuous equivariant boundary maps $\xi_\scrC^+:\partial_\infty\Gamma\to\Lambda_{\scrC}$ and $\xi_\scrC^-:\partial_\infty\Gamma\to\Lambda_{-\scrC}$, see \cite{Gueritaud-Guichard-Kassel-Wienhard-17, Kapovich-Leeb-Porti-17, Bochi-Potrie-Sambarino-19}.
Since $\Gamma$ is non-elementary, $\partial_\infty\Gamma$ is perfect. Hence $(\xi_\scrC^+)^{-1}(U_1)\times(\xi_\scrC^-)^{-1}(V_2)$ contains an off-diagonal pair. The attracting and repelling fixed-point pairs of infinite-order elements are dense in $\partial_\infty\Gamma^{(2)}$, see \cite[Proposition~4.2(4)]{Kapovich-Benakli-02}; hence their images are dense in $\Lambda_{\scrC}\overset{\pitchfork}{\times}\Lambda_{-\scrC}$.
Thus there exists $\gamma \in \Gamma$ such that $\gamma^+_{\scrC} \in U_1$ and $\gamma^-_{\scrC} \in V_2$.
Then, applying $\gamma^{-n}(U_1 \times V_1)$ for $n \gg 1$ large enough, we find that $\gamma^{-n}(U_1 \times V_1) \cap U_2 \times V_2 \neq \emptyset$.
Since the flow translates the $\fa_\scrC/\fh$-coordinate with nonzero speed $\beta(u)$, this boundary-pair transitivity proves \hyperlink{AA4}{\rm(A4)}.

We now assume that $\Gamma$ is Zariski dense. By Lemma~\ref{lemma:jordan-lyapunov-spectrum} below, the dynamical limit cone is $\scrL_{\Gamma,\mathrm J}^{\scrC}$. Theorem~\ref{theorem:benoist} shows that this cone has nonempty interior in $\fa_\scrC$, so the cocycle has full rank. Moreover, the additive group generated by $\lambda_\fapp(\Gamma)$ is dense in $\fa$, see \cite{Benoist-00}. Its image under $\pi_\scrC$, which is the group generated by $\lambda_\scrC(\Gamma)$, is therefore dense in $\fa_\scrC$. Lemma~\ref{lemma:jordan-lyapunov-spectrum} identifies this with the group generated by the Lyapunov spectrum, so the cocycle is non-arithmetic.
\end{proof}

We now discuss periodic orbits of the flow $(\phi_t)_{t \in \R}$:

\begin{lemma}[Jordan and Lyapunov spectra agree]\label{lemma:jordan-lyapunov-spectrum}
The Jordan spectrum of the Anosov representation $\{\lambda_\scrC (\gamma) ~:~ \gamma \in \Gamma\setminus\{e\}\} \subset \fa_\scrC$ coincides with the Lyapunov spectrum of the Axiom A action of type $(k,1)$, see \eqref{equation:lyapunov-spectrum}. Consequently, the limit cone of the Axiom A action of type $(k,1)$ (Definition \ref{definition:limit-cone-cocycle}) and the Jordan limit cone $\scrL_{\Gamma,\mathrm J}^{\scrC}$ (Definition~\ref{definition:limit-cone-anosov}) coincide.
\end{lemma}

\begin{proof}
Consider a periodic point $\phi_{T_0}(x_0) = x_0$ with orbit $c$. By \eqref{equation:description-flow}, we can write $x_0 = \Gamma g_0 M_\scrC H $ for some $g_0 \in G$ and there exist $\gamma_0 \in \Gamma, m_0\in M_\scrC, H_0\in \fh$ such that
\begin{equation}\label{eq:lyapunof_jordan_calc}
 g_0e^{T_0u}m_0  e^{H_0} = \gamma_0 g_0.
\end{equation}
We deduce that $\Gamma g_0 \exp(T_0 u + H_0)M_\scrC = \Gamma g_0 M_\scrC$. Following the discussion in \S\ref{sssection:lyapunov}, by \eqref{eq:lyapunov_are_A_periods}, the Lyapunov projection $\lambda(c)$ of the periodic orbit from \S\ref{sec:introduction-geometric-setting} (defined in \eqref{equation:lyapunov-projection}) satisfies
\[
\lambda(c)= T_0 u + H_0.
\]
Let us show its relation to the group-theoretic Jordan projection $\lambda_\scrC$ defined in \eqref{eq:jordan_scrC}. Choose an auxiliary open Weyl chamber $\fapp\geq\scrC$.
We need to use the Jordan projection in $M_\scrC$ and we choose the compatible open Weyl chamber $\fa_{M_\scrC,++}$ for $M_\scrC$ (see \S\ref{sssec:levi}).
Equation~\eqref{eq:lyapunof_jordan_calc} gives
\[
\gamma_0^{-\ell}g_0m_0^\ell e^{\ell\lambda(c)}=g_0.
\]
Apply Lemma~\ref{lemma:hyperbolicity-tool} with $\gamma_\ell=\gamma_0^{-\ell}$, $g_\ell=g_0$, $m_\ell=m_0^\ell$, and $a_\ell=\ell\lambda(c)$. Since $\beta(a_\ell)=\ell T_0\beta(u)>0$, its positive branch yields $\lambda(c)\in\scrC$.
Then, relying on Example~\ref{example:limit-point-for-powers},
\[
g_0 P_{ \scrC} = (\gamma_0)^+_{\scrC},\ g_0 P_{-\scrC} = (\gamma_0)^-_{\scrC},
\]
so that the orbit $c$ is uniquely determined by $\gamma_0$. Using $\lambda_\fapp(\gamma)=\lim_{n\to\infty}\frac1n\mu_\fapp(\gamma^n)$ and \eqref{eq:approximation-of-mu}, we also obtain
\[
\lambda_{\fapp}(\gamma_0) = T_0 u + H_0 + \lambda_{\fa_{M_\scrC,++}}(m_0).
\]
Hence $\lambda(c) \in \{\lambda_\scrC(\gamma), \gamma\in \Gamma\}$ as claimed.

To see the converse inclusion, take an arbitrary $\gamma\in \Gamma\setminus\{e\}$; then by the Jordan decomposition there is $g_0\in G$ such that $g_0 \gamma g_0^{-1} = k_0 e^{H_0} n_0$. For $\gamma\neq e$, the Anosov property gives $H_0\neq0$. Since $H_0=\lambda_\fapp(\gamma)\in\scrL_{\Gamma,\mathrm J}^{\fapp}\subset\scrL_{\Gamma,\mathrm C}^{\fapp}$ and $\Gamma$ is $\scrC$-regular, $\alpha(H_0)>0$ for $\alpha\in\roots^+_\scrC$. Now by commutativity of $k_0, e^{H_0}$ and $n_0$, we conclude that $k_0\in K\cap M_\scrC$ and $n_0\in N_\fapp^+\cap M_\scrC$. To see this, let us write
$n_0 = \exp(X)$ for some $X\in\fn = \oplus_{\alpha\in \roots_\fapp^+} \fg_\alpha$. By commutativity we conclude that $\operatorname{ad}(H_0)(X) = 0$. But as for $\alpha\in \roots_\scrC^+$ we have $\alpha(H_0)>0$, this implies that the component of $X$ in these $\fg_\alpha$ must vanish and we conclude that $X\in \bigoplus_{\alpha\in \roots_\fapp^+\cap\roots_\scrC^0} \fg_\alpha \subset \fm_\scrC$. For the compact factor, commutativity gives $k_0\in Z_K(H_0)$. Since $\alpha(H_0)>0$ for every $\alpha\in\roots_\scrC^+$, we have $Z_K(H_0)< K_\scrC=K\cap M_\scrC$, and hence $k_0\in M_\scrC$.
By Example~\ref{example:limit-point-for-powers}, $g_0^{-1}P_\scrC=\gamma_\scrC^+$ and $g_0^{-1}P_{-\scrC}=\gamma_\scrC^-$; hence Lemma~\ref{lem:DMS_transversal_grps} gives $g_0^{-1}M_\scrC\in\tilde{\mathscr J}_\scrC\subset\odms$. Writing $H_0=\pi_\scrC(H_0)+(H_0-\pi_\scrC(H_0))$, where $H_0-\pi_\scrC(H_0)\in\fa_\scrC^\perp$, and setting $m_0:=k_0e^{H_0-\pi_\scrC(H_0)}n_0\in M_\scrC$, we obtain
\[
\Gamma g_0^{-1}e^{\lambda_\scrC(\gamma)}M_\scrC=\Gamma g_0^{-1}M_\scrC,
\]
so $\lambda_\scrC(\gamma)=\pi_\scrC(H_0)$ is an $A_\scrC$-period of $\Gamma g_0^{-1}M_\scrC$; since $\beta(\lambda_\scrC(\gamma))>0$ and $\beta(u)>0$, \eqref{eq:lyapunov_are_A_periods} identifies it with a Lyapunov projection.
\end{proof}

\begin{remark}
If $\scrC$ is $\iota_\fapp$-invariant, the time reversal symmetry (Lemma \ref{lemma:def-time-symmetry}) maps a periodic orbit with Lyapunov projection $\lambda$ to another one with projection $\iota_\fapp(\lambda)$. If $\scrC$ is not invariant, what survives of the symmetry is that for every orbit with projection $\pi_\scrC \lambda$, there is another one with projection $\pi_\scrC \iota(\lambda)$. 
\end{remark}

Recall from \eqref{equation:delta} that, given $\varphi\in \bigl(\scrL_{\Gamma,\mathrm J}^{\scrC}\bigr)^{*,\circ} \subset \fa_\scrC^\ast$, $\delta(\varphi) \in (0,\infty)$ is defined by
\[
\delta(\varphi) := \lim_{T \to +\infty} \frac{1}{T}\log \sharp \{ [\gamma]\in [\Gamma] ~:~ \varphi(\lambda_\scrC(\gamma))\leq T\},
\]
where $[\Gamma]$ is the set of conjugacy classes of the subgroup $\Gamma < G$. Indeed, $[\Gamma]\setminus\{[e]\}$ is in one-to-one correspondence with the set of periodic orbits $\mathbf{P}$ for the flow $(\phi_t)_{t \in \R}$ (and $\mathbf{P}^\sharp$ identifies with $[\Gamma]^\sharp$). The following statement gives a relation between the asymptotic growth of Cartan and Jordan projections:

\begin{proposition}[Growth of Cartan and Jordan projections]\label{prop:cartan==lyapunov}
Let $\varphi\in \bigl(\scrL_{\Gamma,\mathrm J}^{\scrC}\bigr)^{*,\circ} \subset \fa_\scrC^\ast$. Then
\begin{equation}
\label{equation:delta-cartan}
\delta(\varphi) = \lim_{T\to+\infty} \frac{1}{T}\log \sharp\{ \gamma\in\Gamma ~:~ \varphi(\mu_\scrC(\gamma))\leq T\}.
\end{equation}
\end{proposition}

We refer to \cite[Corollary 5.5.3]{Sambarino-24} for a proof.  In particular, the critical hypersurface defined in the literature as the $1$-level set of the right-hand side of \eqref{equation:delta-cartan} (see \cite[Theorem A]{Sambarino-24} for instance) coincides with the critical hypersurface defined in \S\ref{section:leading-resonance}.

\subsubsection{Consequences}
It follows from Proposition \ref{proposition:wagons} that all the following results established in \S\ref{section:results-cocycles}--\S\ref{section:decay} hold for discrete and Zariski-dense Anosov subgroups. We also emphasize that all these results are intrinsic to the Anosov action on $\M$ and \emph{independent} of $\mc{N}$, provided such a quotient exists, as explained there.
\begin{enumerate}[label=(\roman*)]
\item Existence of a resonance spectrum $\sigma_{\mathrm{RS},+}$ in $(\fa_\scrC^*)_\C$ and resonant states supported in $\mc{T}_{+,\M} \subset \M$ (Theorem \ref{theorem:rt-anosov1});
\item Holomorphic extension to $(\fa_\scrC^*)_\C$ of the dynamical determinant $\zeta_{\E}$ (see \eqref{eq:def-general-zeta}) defined for any admissible bundle $\E \to \M$ (Theorem \ref{theorem:general-zeta}), and meromorphic extension to $(\fa_\scrC^*)_\C$ of the Ruelle zeta function
\[
\zeta(\mathbf{s}) = \prod_{[\gamma]\in[\Gamma]^\sharp} (1-e^{-\mathbf{s}(\lambda_{\scrC}(\gamma))})^{-1},
\]
see Corollary \ref{corollary:zeta-ruelle};
\item Description of the leading resonant hypersurface $\mathbf{C}^{(d_s)} \subset \fa_\scrC^*$ (Theorem \ref{theorem:leading});
\item Description of the measures of maximal entropy $\mu_\varphi$ supported on $\scrJ$ and defined for any $\varphi \in \mathbf{C}^{(d_s)}$ (Theorem \ref{theorem:measures-max-entropy});
\item Local mixing/decay of correlations for the diagonal flow $(e^{tu_\varphi})_{t\in\R}$ with respect to $\mu_\varphi$ (see \eqref{equation:leading-term-asymptotic} and the discussion in \S\ref{section:decay}), and an asymptotic expansion of the correlation function provided the representation is Diophantine (Theorem \ref{theorem:mixing-sharp});
\item Meromorphic extension of the product resolvent (Theorem \ref{theorem:product-resolvents}).
\end{enumerate}

In particular, this establishes Theorems \ref{theorem:intro-rt}, \ref{theorem:intro-leading} and \ref{theorem:intro-zeta-anosov} stated in the introduction, as well as Theorem \ref{theorem:mixing-intro}. The genericity of Diophantine property, which is established under certain dimension assumptions for almost every representation with respect to Lebesgue measure is established in \S\ref{ssection:diophantine-representations}. Theorem \ref{theorem:intro-poincare} on the meromorphic extension of Poincaré series is postponed to \S\ref{section:poincare}.

\begin{proof}[Proof of Theorems \ref{theorem:intro-rt}, \ref{theorem:intro-leading} and \ref{theorem:intro-zeta-anosov}, and \ref{theorem:mixing-intro}]
Theorem \ref{theorem:intro-rt} follows from item (i) above; Theorem \ref{theorem:intro-leading} follows from item (iii); Theorem \ref{theorem:intro-zeta-anosov} follows from item (ii); Theorem \ref{theorem:mixing-intro} follows from items (iv) and (v), combined with the Diophantine property of Anosov representations established in \S\ref{ssection:diophantine-representations}.
\end{proof}

\section{Poincaré series}
\label{section:poincare}

Let $\Gamma < G$ be a discrete subgroup and let $\fapp$ be an open Weyl chamber. The associated multivariate Poincaré series is defined as
\[
P_{\Gamma,\fapp}(x,y ; \mathbf s) := \sum_{\gamma\in\Gamma} e^{-\mathbf s(\mu_\fapp(y^{-1}\gamma x))}, \qquad x, y \in \Gamma\backslash G /K,\ \s \in \fa^\ast_\C.
\] Here and below, $x$ and $y$ also denote arbitrary representatives in $G$; the $K$-bi-invariance of the Cartan projection and reindexing of the sum over $\Gamma$ show that the result is independent of these choices.
This defines a function on the subset of $(\Gamma\backslash G /K)^2\times \fa^\ast_\C$ where the series converges; for $\scrC$-Anosov subgroups this domain will be shown to be nonempty (see Theorem \ref{theorem:Poincare}).
Geometrically, this corresponds to summing over all geodesic segments connecting $x$ to $y$ in the quotient $\Gamma\backslash G/K$ and weighting them by a higher-rank analogue of length (the Cartan projection). Recall that we use the Euclidean structure of $\fa$ to identify $\fa_\scrC^\ast$ with a subspace of $\fa^\ast$. For an appropriate choice of generalized Weyl chamber $\scrC \leq \fapp$, we can restrict $P_{\Gamma,\fapp}(x,y ; \mathbf s)$ to $(\fa_\scrC^*)_\C$ and obtain
\[
 P_{\Gamma,\scrC}(x,y ; \mathbf s):=P_{\Gamma,\fapp}|_{(\fa_\scrC^*)_\C}(x,y ; \mathbf s) = \sum_{\gamma\in\Gamma} e^{-\mathbf s(\mu_\scrC(y^{-1}\gamma x))}, \quad \s \in (\fa_\scrC^\ast)_\C.
\]
Observe that $P_{\Gamma,\scrC}(x,y ; \mathbf s)$ is well-defined and independent of the choice of maximal Weyl chamber $\fapp\geq\scrC$ because $\mu_\scrC$ is independent of this choice.

\medskip

The main result of this section is that $P_{\Gamma,\scrC}(x,y;\bullet)$ admits a meromorphic extension to $(\fa_{\scrC}^*)_{\C}$:

\begin{theorem}\label{theorem:Poincare}
Let $\Gamma < G$ be a torsion-free $\scrC$-Anosov subgroup. Then, for all $x,y \in \Gamma\backslash G /K$, $P_{\Gamma,\scrC}(x,y;\bullet)$ converges absolutely in $\{\s \in (\fa_{\scrC}^*)_{\C} ~:~ \Re(\s) \in \mathbf{C}^{(d_s),+}\}$ and admits a meromorphic continuation to $(\fa_\scrC^*)_{\C}$.
\end{theorem}
\begin{remark}\label{rem:spectral}
It follows also from the proof that the singularities of $P_{\Gamma, \scrC}$ have a spectral interpretation. For example, if $\scrC$ is $\iota_\fapp$-invariant, then the singularities are contained in the resonances $\sigma^{(d_s)}_{\mathrm{RS}}$ of the action (or in a version twisted by an orientation bundle, see below for details).
\end{remark}
In the Borel Anosov case, we obtain a continuation to $\fa^\ast_\C$, while in the weaker $\scrC$-Anosov case we have to restrict to subspaces. Our proof strategy follows that of \cite{Dang-Riviere-24} in the rank-one case.

\begin{remark}\label{rem:sym_scrC}
Assume $\scrC\geq\scrC'$ and $\Gamma$ is $\scrC$-Anosov (thus also $\scrC'$-Anosov). By definition, $P_{\Gamma, \scrC'} =( P_{\Gamma,\scrC})_{|(\fa_{\scrC'}^*)_\C}$, so the statement of Theorem~\ref{theorem:Poincare} for $\scrC'$ follows directly from the statement for the larger $\scrC$ by restricting the meromorphic function to an analytic subvariety not contained in its polar divisor. This restriction is well-defined because $\fa^*_{\scrC'}$ intersects the domain of absolute convergence of $P_{\Gamma,\scrC}$. Given a $\scrC$-Anosov group $\Gamma$, we may therefore assume that $\scrC$ is maximal, and in particular $\iota_\fapp$-invariant.
\end{remark}

\subsection{Basic properties of currents} \label{ssection:currents}
We use the notion of currents to prove the meromorphic continuation (Theorem \ref{theorem:Poincare}). Here, we briefly recall elementary facts about currents and refer the reader to \cite[Chapter 10]{Lefeuvre-book} for further discussion.

On an $n$-dimensional manifold $\mc M$, for $0\leq m\leq n$, one defines the space of \emph{$m$-dimensional currents} to be the topological dual of $C^\infty_{\comp}(\mc M, \Lambda^mT^*\mc M)$, the space of differential forms of degree $m$ with compact support.
When $\mc{M}$ is oriented, $m$-currents can be identified with distributional forms of degree $n-m$ because every continuous linear functional $\omega \mapsto \alpha(\omega)$ can be written uniquely as $\omega \mapsto (\omega, \eta)$ for some $\eta \in \mc{D}'(\mc{M},\Lambda^{n-m}T^*\M)$ (see \S\ref{sssection:distributions} for a brief reminder on distributions), where $(\bullet,\bullet)$ denotes the continuous extension to distributions of the nondegenerate pairing
\begin{equation}
 \label{eq:bilinear_current_pairing}
C^\infty_{\comp}(\mc M, \Lambda^mT^*\mc M) \times C^\infty(\mc M, \Lambda^{n-m}T^*\mc M) \ni (\omega,u) \mapsto \int_{\mc{M}} \omega \wedge u.
\end{equation}
Note that this identification is non-canonical; the ordering of wedge products must be consistent to avoid orientation-related sign errors.

An $m$-dimensional oriented properly embedded submanifold $S\subset \mc M$ defines a current $[S]$ via
\[
(\omega, [S]):=\int_S\omega.
\]
In local coordinates, we can always express such a current as follows. Choose a coordinate chart $\kappa: U\to V\subset \RR_x^{m}\times \RR^{n-m}_y$  such that $\kappa(U\cap S) \subset \RR^{m}\times\{0\}$, $(\partial_{x_1},\ldots \partial_{x_{m}})$ is a positively oriented frame on $S$, and $(\partial_{x_1},\ldots \partial_{x_{m}},\partial_{y_1},\ldots,\partial_{y_{n-m}})$ is positively oriented on $\mc M$. Then, one computes that
\begin{equation}\label{eq:submanifold_current_in_coordinates}
\kappa_*[S] = \delta_{y=0} \cdot \dd y_1\wedge \ldots \wedge \dd y_{n-m} \in \mathcal D'(V, \Lambda^{n-m}T^*V),
\end{equation}
where $\delta_{y=0}$ is the standard Dirac distribution supported on the set $\{y=0\}\subset V$. From this expression, we deduce that $\WF([S]) = N^*S$. This allows us to define pairings of transverse submanifolds:

\begin{lemma}\label{lem:submanifold_pairing}
Let $S_1, S_2\subset\M$ be two oriented submanifolds of respective dimensions $n-m$ and $m$ that intersect transversely. If the intersection is finite and $f$ is a smooth function on $S_2$, the pairing $([S_1], f \cdot[S_2])$ is well-defined and
 \[
  ([S_1], f\cdot [S_2]) = \sum_{x\in S_1 \cap S_2} \sigma(x) f(x).
 \]
Here $\sigma(x)= \pm 1$ is the orientation sign of the frame $(v_1,\ldots, v_{n-m}, w_1,\ldots w_m) \in T_x\mc M$ relative to the orientation of $\mc M$, where $(v_1,\ldots v_{n-m}) \in T_xS_1$ is positively oriented on $S_1$ and $(w_1,\ldots w_m)\in T_xS_2$ is positively oriented on $S_2$.
\end{lemma}

\begin{proof}See \cite[Proposition 10.2.6]{Lefeuvre-book}.\end{proof}
 We will actually be in a slightly more involved geometric situation in which the submanifolds are not orientable. Recall that if $E\to\M$ is a vector bundle and $\fo$ is a $\Z_2$-bundle, we denote by $E^\fo$ the vector bundle over $\M$ whose fiber over $x$ is
\[
\{ (u, \eps) ~:~ u\in E_x,\ \eps\in \fo_x\}/\{(u,\eps)\sim (-u,-\eps)\}. 
\]
If $S\subset \M$ is a non-orientable submanifold of dimension $m$, we denote by $\fo(TS)$ its bundle of orientations. Assume that $\fo(TS)$ is the restriction to $S$ of a $\Z_2$-bundle $\fo$ over $\M$. Then we can regard the total space $\fo$ as a double cover of $\M$, and $\fo(TS)$ as a submanifold of $\fo$ that is orientable and forms a double cover of $S$. Note that $\fo$ has a unique nontrivial deck transformation $\tau$, which reverses the orientation of $\fo(TS)$. It also reverses the orientation of the $\Z_2$-bundle $\fo^\ast\fo$ (the pullback to the total space $\fo$ of the bundle $\fo$). Since $\fo^\ast\fo$ is trivial, it admits a section $\sigma$, and 
\[
\tau^\ast ( [\fo(TS)]\times \sigma ) = (-[\fo(TS)])\times (-\sigma ) = [\fo(TS)]\times \sigma . 
\]
It follows that $[\fo(TS)]\times \sigma$ is a well-defined distributional section of twisted forms $(\Lambda^{n-m}T^*\M)^\fo$ which we denote by $[S]^\fo$. It only depends on choosing an orientation of the double cover $\fo$ and of $\fo(TS)$.
\begin{corollary}\label{cor:submanifold-pairing-non-oriented}
Let $S_1,S_2\subset \M$ be two non-orientable connected submanifolds intersecting transversely at a finite number of points. Let $\mathfrak{o}$ be a $\Z_2$-bundle over $\M$ such that for $j=1,2$, as bundles
\[
\mathfrak{o}_{|S_j} \simeq \mathfrak{o}(TS_j). 
\]
Then, once orientations have been chosen for $\fo$, $\fo(TS_1)$, and $\fo(TS_2)$, $[S_1]^\fo \wedge [S_2]^\fo$ is a well-defined distributional section of $\Lambda^nT^*\M$ and
\[
[S_1]^\fo \wedge [S_2]^\fo = \sum_{z\in S_1\cap S_2} \sigma(z)[\{z\}],
\]
for some signs $\sigma(z)\in\{\pm 1\}$. 
\end{corollary}

\begin{proof}
As above, let $\tau$ denote the non-trivial deck transformation of the cover $\mathfrak{o}\to\M$. Since $[\fo(TS_j)]\times \sigma_j$ is $\tau$-invariant for $j=1,2$, so is the wedge product
\[
([\fo(TS_1)] \times \sigma_1) \wedge ( [\fo(TS_2)]\times \sigma_2) = [\fo(TS_1)]\wedge [\fo(TS_2)] \times (\sigma_1\sigma_2). 
\]
Since $\fo^\ast\fo$ is trivial, there is a global $\eps\in\{\pm 1\}$ such that $\sigma_2 = \eps \sigma_1$. Lemma~\ref{lem:submanifold_pairing} then gives
\[
([\fo(TS_1)] \times \sigma_1) \wedge ( [\fo(TS_2)]\times \sigma_2) = \sum_{z\in \fo(TS_1)\cap \fo(TS_2)} \sigma(z) [\{z\}]. 
\]
Here, $\tau$-invariance means that the equality $\sigma(-z) = \sigma(z)$ shows that this sign depends only on the base point $[z]\in S_1\cap S_2$. 
\end{proof}

We will also need the following result:
\begin{lemma}\label{lem:submanifold_propagation}
 Let $\tau: \RR^{\ell}\curvearrowright \mc M$ be a locally free smooth action on an $n$-dimensional oriented manifold $\M$. Let $X_1,\ldots, X_\ell$ be the vector fields on $\M$ induced by $\tau$ from the standard basis of $\RR^\ell$. Let $S\subset \mc M$ be an oriented properly embedded $m$-dimensional submanifold. Let $\mathcal C\Subset\RR^\ell$ be open and assume that the action map
 \[
 F:\mathcal C\times S\longrightarrow\mc M,\qquad (a,x)\longmapsto\tau(a)x
 \]
is an embedding. Set $R:=F(\mathcal C\times S)$ and define
 \[
  \mathbf a: R\to \RR^\ell,\quad F(a,x)\mapsto a
 \]
which is smooth. Then, for $f$ a smooth function defined in a neighborhood of $\overline{\mathcal C}$ we have
 \[
   \iota_{X_\ell}\ldots \iota_{X_1} \int_{\mc C} f(a)\cdot(\tau(a)_*[S]) \dd a = (f\circ \mathbf a)\cdot[R],
 \]
where $\dd a$ denotes the Lebesgue measure on $\RR^\ell$. We orient $R$ so that, whenever $(v_1,\ldots, v_m) \in T_xS$ is a positively oriented frame, $((\dd\tau(a))_xv_1,\ldots,(\dd\tau(a))_xv_m,X_1,\ldots,X_\ell)\in T_{\tau(a)x}R$ is positively oriented.
\end{lemma}

Similarly, if $S$ is non-orientable and its orientation bundle coincides with a globally defined $\Z_2$-bundle $\fo$, combining this lemma with the discussion before Corollary \ref{cor:submanifold-pairing-non-oriented}, we obtain
 \[
   \iota_{X_\ell}\ldots \iota_{X_1} \int_{\mc C} f(a)\cdot(\tau(a)_*[S]^\fo) \dd a = (f\circ\mathbf a)\cdot[R]^\fo.
 \]

\begin{proof}
By hypothesis, $F$ identifies $\mc C\times S$ diffeomorphically with $R$; in particular, $\mathbf a=\operatorname{pr}_1\circ F^{-1}$ is smooth. It therefore remains to verify the asserted identity of currents, which is a local statement.

 Fix an arbitrary point $\tau(a_0)x_0$ in $R$, with $a_0\in \mc C$ and $x_0\in S$. Then we can choose a local coordinate chart on an open neighborhood $U\subset \M$
 \[
 \kappa: U\to V\subset \RR_x^{m}\times\RR_\tau^\ell\times \RR_y^{q}, \qquad q=n-m-\ell
 \]
 positively oriented on $\mc M$, such that $\kappa(U \cap \tau(a_0) S) = \RR^m\times \{0\}$ and such that
 $\partial_{\tau_j} = X_j$ and $\partial_{x_1}\ldots,\partial_{x_m}$ is positively oriented on $\tau(a_0) S$. We may further assume that $V= B^{\RR^m}_\varepsilon(0) \times B^{\RR^\ell}_\varepsilon(0)
 \times B^{\RR^q}_\varepsilon(0)$. In these coordinates, \eqref{eq:submanifold_current_in_coordinates} gives
 \[
  \kappa_*[\tau(a_0 +a)S] = \delta_{y=0}\delta_{\tau =a} \dd \tau_1\wedge\ldots\wedge \dd \tau_\ell\wedge \dd y_1\wedge\ldots\wedge \dd y_q
 \]
 and hence
 \begin{align*}
 \int_{B_\varepsilon(a_0)}f(a)\kappa_*( \iota_{X_\ell}\ldots \iota_{X_1}(\tau(a))_*[S])\dd a=&\int_{B_\varepsilon(a_0)}f(a)\iota_{\partial_{\tau_\ell}}\ldots \iota_{\partial_{\tau_1}}\kappa_*([\tau(a)S])\dd a\\
  =& f(a_0+\tau)\delta_{y=0}\cdot \dd y_1\wedge\ldots\wedge \dd y_q.
 \end{align*}
The latter expression can be identified (once more using \eqref{eq:submanifold_current_in_coordinates}) with $\kappa_*((f\circ\mathbf a)[R])$.
\end{proof}

\subsection{Intersection of propagated cospheres}

Following \cite{Dang-Riviere-24} in the rank-one case, we aim to lift the problem from $G/K$ to the phase space $G/M_\scrC$. We first need to understand the possible intersections of submanifolds of the form $(xK)M_\scrC$ and $(yKe^H)M_\scrC$ in $G/M_\scrC$ for $x,y\in G$ and $H\in\scrC\subset \fa_\scrC$. The Borel Anosov case is comparatively simple, so we explain it first as a guide to the more challenging arguments below.

Let $x,y\in G$ and assume that $xK\cap yKe^H$ is nonempty for some $H\in \fapp$; then there are $k_1,k_2\in K$ such that
\[
 xk_2 = yk_1e^{H} \Rightarrow y^{-1}x = k_1e^{H}k_2^{-1}.
\]
The right-hand side is precisely the KAK decomposition of $y^{-1}x$, which (by regularity of $H\in\fapp$) is unique up to right multiplication of $k_1$ and $k_2$ by elements of $M_\fapp$. Thus, given $x,y\in G$ in regular position, there is precisely one translation element $H\in \fapp$ for which the intersection is nonempty, and the unique intersection point $xk_2M_\fapp = yk_1e^HM_\fapp$ can be computed from the KAK decomposition.

Let us now consider the non-Borel case. Let $\scrC$ be a generalized Weyl chamber and $\fapp\geq\scrC$ an open Weyl chamber. Given $x,y\in G$, we define
\begin{equation}\label{eq:def-z_C}
z_\scrC(x,y) := xk_2M_\scrC,\qquad \text{where } y^{-1} x = k_1 e^{\mu_\fapp(y^{-1}x)} k_2^{-1}.
\end{equation}
The inverse decomposition $x^{-1}y=k_2e^{-\mu_\fapp(y^{-1}x)}k_1^{-1}$ is the KAK decomposition relative to $-\fapp\geq-\scrC$. Since $M_{-\scrC}=M_\scrC$, we have
\[
\begin{aligned}
 &z_{-\scrC}(y,x)=yk_1M_\scrC,\quad \gamma z_\scrC(x,y)=z_\scrC(\gamma x,\gamma y),\\
 &\gamma z_{-\scrC}(x,y)=z_{-\scrC}(\gamma x,\gamma y),\quad \gamma\in G.
\end{aligned}
\]
As $\fa_{\scrC}^\perp\subset \fm_\scrC$, we conclude that
\[
 z_\scrC(x,y) = xk_2M_\scrC =z_{-\scrC}(y,x)e^{\mu_\scrC(y^{-1}x)} = yk_1e^{\mu_\scrC(y^{-1}x)}M_\scrC \in (xK)M_\scrC \cap (yKe^{\mu_\scrC(y^{-1}x)})M_\scrC
\]
belongs to $(xK)M_\scrC \cap (yKe^{\overline{\scrC}})M_\scrC$. Although it is not the only intersection point, as Lemma~\ref{lemma:form-of-intersection} shows, it is the only one relevant for the Poincaré series. 

Another subtlety in the general case is that the submanifolds $xKM_\scrC$ are not leaves of a foliation of $G/M_\scrC$. For notational convenience, define
\begin{equation}\label{eq:VxBundle}
\V_x (xkM_\scrC) = T_{xkM_\scrC}(xKM_\scrC) \text{ and }\V_{y,H}(yke^H M_\scrC) = T_{yke^H M_\scrC} (y K e^H M_\scrC). 
\end{equation}
We caution that $\V_x$ is only a subbundle of $T(G/M_\scrC)_{|xKM_\scrC}$ and cannot be described as an associated bundle except in very special cases. 

We now prove the following proposition.
\begin{proposition}\label{prop:intersection-non-borel}
Let $x,y\in G$ satisfy $\alpha(\mu_\fapp(y^{-1}x))>0$ for all $\alpha \in \Delta^+_\scrC$.
Then $z_\scrC(x,y)$ does not depend on the choice of the open Weyl chamber $\fapp\geq \scrC$ and is a point at which $(xK)M_\scrC$ and $(yK\exp(\scrC))M_\scrC$ intersect transversely.
\end{proposition}

\begin{proof}
Let $w\cdot\fapp$ be another Weyl chamber satisfying $w\cdot\fapp\geq \scrC$. Recall from Lemma~\ref{lem:adjacent_open_weyl_chamber} that $w$ must lie in $\langle \Delta^0_\scrC\rangle$ and preserve $\scrC$. Since $\langle\Delta^0_\scrC\rangle\simeq W(M_\scrC,A_\scrC^\perp)$, we can represent $w=[\dot{w}]$ with $\dot{w}\in N_{K_\scrC}(\fa_\scrC^\perp)$; the corresponding KAK decomposition is
\[
y^{-1}x = k_1\dot{w}^{-1} e^{w\cdot\mu_\fapp(y^{-1}x)}\dot{w} k_2^{-1}.
\]
The corresponding $z_\scrC(x,y)$ is $xk_2 \dot{w}^{-1} M_\scrC$. However, since $\dot{w}\in M_\scrC$, this is the same point as before; hence $z_\scrC(x,y)$ is defined independently of the choice of the adjacent open Weyl chamber.

To study the intersection of $xK M_\scrC$ and $yKM_\scrC \exp(\scrC)$, recall from \eqref{eq:hyperbolic_splitting_first} that for fixed $g\in G$, we can identify the tangent space
\[
T_{gM_\scrC}G/M_\scrC \cong \fa_\scrC\oplus \fn_\scrC^+\oplus\fn_\scrC^- \text{ via } X\mapsto \left.\frac{d}{dt}\right|_{t=0} g\exp(tX)M_\scrC.
\]

Let us thus fix a KAK decomposition $y^{-1} x = k_1 e^{H} k_2^{-1}$, where $H:=\mu_\fapp(y^{-1}x)$, and set $H_\scrC:=\pi_\scrC(H)$. Then, since $e^{H-H_\scrC}\in A_\scrC^\perp< M_\scrC$,
\[
xk_2M_\scrC=yk_1e^HM_\scrC=yk_1e^{H_\scrC}M_\scrC.
\]
After identifying $T_{xk_2M_\scrC} G/M_\scrC\cong \fa_\scrC\oplus \fn_\scrC^+\oplus\fn_\scrC^-$ using the representative $xk_2=yk_1e^H$, a straightforward computation yields
\begin{equation}\label{eq:TxK}
T_{xk_2 M_\scrC}(xK)M_\scrC \cong \bigoplus_{\alpha\in \roots_\scrC^+}\fk_\alpha=:\mathfrak v, \text{ where } \fk_\alpha = \{X+\theta X : X\in\fg_\alpha\}
\end{equation}
and
\begin{equation}\label{eq:TxKexpA}
T_{xk_2M_\scrC} (yKe^{H_\scrC})M_\scrC \cong \mathrm{Ad}(e^{-H})\mathfrak v.
\end{equation}
Indeed, changing the representative from $yk_1e^{H_\scrC}$ to $xk_2=yk_1e^{H_\scrC}e^{H-H_\scrC}$ changes the fiber coordinates by $\mathrm{Ad}(e^{-(H-H_\scrC)})$, and
\[
\mathrm{Ad}(e^{-(H-H_\scrC)})\mathrm{Ad}(e^{-H_\scrC})=\mathrm{Ad}(e^{-H}).
\]
However, as $\alpha(H)>0$ for all $\alpha\in \roots_\scrC^+$, we directly get that
\begin{equation}\label{eq:transverse_inter_lie_algebra}
\left(\bigoplus_{\alpha\in \roots_\scrC^+}\fk_\alpha\right) \cap \mathrm{Ad}(e^{-H})\left(\bigoplus_{\alpha \in \roots_\scrC^+}\fk_\alpha\right)=\{0\},
\end{equation}
The full tangent space of $(yK\exp(\scrC))M_\scrC$ at this point is $\mathrm{Ad}(e^{-H})\mathfrak v\oplus\fa_\scrC$. Equation~\eqref{eq:transverse_inter_lie_algebra}, the root-space splitting, and the dimension count therefore prove transversality.
\end{proof}

We now describe the remaining intersection points, which do not contribute to the Poincaré sum. To simplify the analysis of the intersection $xKM_\scrC\cap yK\exp(\scrC)M_\scrC$, we can replace $x$ and $y$ respectively by $g x k$ and $g y k'$ for some $g\in G$ and $k,k'\in K$. We can thus assume that $x=e$ and $y = e^{-H}$ for some $H=\mu_\fapp(y^{-1} x)\in \fap$. Let us write $H_\scrC:=\pi_\scrC H\in \scrC$ for the projection to $\fa_\scrC$. After these modifications, $z_\scrC(x,y) = eM_\scrC$.

 Recall also that for any $H\in\fa$ with $\alpha(H)>0$ for all $\alpha\in \Delta_\scrC^+$, we have $Z_K(H) < K_\scrC = K \cap M_{\scrC}$.

\begin{lemma}\label{lemma:form-of-intersection}
Let $\scrC$ be a generalized Weyl chamber and let $H\in \fa$ satisfy $\alpha(H)>0$ for all $\alpha\in \Delta_\scrC^+$. Then
\[
\bigcup_{[w]\in W(G,A)/\langle \Delta_\scrC^0\rangle} \left\{ k \dot{w} M_\scrC : k\in Z_K(H)\right\} = KM_\scrC \cap (e^{-H} KA_\scrC)M_\scrC,
\]
where, for each $[w]\in W(G,A)/ \langle \Delta^0_\scrC\rangle$, we choose a representative $\dot{w}\in N_K(\fa)$. If $A_\scrC$ is replaced by $\exp(\scrC)$, we must impose the additional condition
\begin{equation}\label{eq:lem_full_intersection}
\pi_\scrC (w^{-1} H )\in \scrC.
\end{equation}
\end{lemma}

\begin{proof}
Note that the left-hand side is indeed well-defined and does not depend on the representative $\dot w$. First, because the Weyl group is $W(G,A) = N_K(\fa)/Z_K(\fa)$, right multiplication of $\dot w$ by an element of $Z_K(\fa) = M_\fapp$ has no effect, since $M_\fapp$ is contained in $M_\scrC$.
Second, we may alter $w$ by an element of $\langle \roots_\scrC^0\rangle$, but this is precisely the Weyl group $W(M_\scrC,A_\scrC^\perp)$
(see \eqref{eq:weylgrp_of_MscrC}) and this is also absorbed in $M_\scrC$.

We first prove the inclusion ``$\subseteq$''. That $k\dot wM_\scrC\in KM_\scrC$ is clear by definition. For membership in the other submanifold, use $k\in Z_K(H)$ and $\fa_\scrC^\perp\subset\fm_\scrC$ to compute
\[
 k\dot wM_\scrC = e^{-H} k \dot w (\dot w^{-1} e^{H} \dot w) M_\scrC = e^{-H} k \dot w e^{\pi_\scrC(w^{-1} H)}M_\scrC \in (e^{-H} k A_\scrC)M_\scrC.
\]
If we assume \eqref{eq:lem_full_intersection}, this point even lies in $(e^{-H}k\exp(\scrC))M_\scrC$.

The inclusion ``$\supseteq$'' is more demanding. An arbitrary point of intersection $kM_\scrC$ must satisfy
\[
k = e^{-H} k' m e^{H_\scrC}
\]
for some $k'\in K$, $m\in M_\scrC$, and $H_\scrC\in \fa_\scrC$. First note that the point of intersection does not change if we multiply $k$ or $k'$ on the right by an element of $K_\scrC$. Multiplying both $k$ and $k'$ on the left by an element of $Z_K(H)$ changes the point of intersection, but as our aim is to prove that there is $\dot w\in N_K(\fa)$ such that $k\dot w \in Z_K(H)$, we may do this as needed.

Start with the Cartan decomposition of $M_\scrC$. We write $m= k_\scrC e^s$, with $k_\scrC \in K_\scrC$ and $s\in\fs\cap\fm_\scrC$. Replacing $k'$ by $k'k_\scrC$, we may assume $m\in \exp\fs$ and find
\begin{equation}\label{eq:a-conjugate-s+H}
k e^{\mathrm{Ad}(k^{-1})H} = e^H k = k' e^{s+H_\scrC}.
\end{equation}
By uniqueness of the Cartan decomposition in $G$, this implies that $k=k'$ and $\mathrm{Ad}(k^{-1}) H = s+H_\scrC$.
Next, we claim that, by multiplying $k$ on the left by an element of $Z_K(H)$, we can ensure that $\mathrm{Ad}(k)\fa_\scrC \subset \fa$. This can be seen as follows:

By definition, $\fa$ commutes with $H$, so $\fa \subset \fs\cap Z_\fg(H)$ is a maximal abelian subalgebra. But since $\fa_\scrC$ commutes with $s+H_\scrC$, $\mathrm{Ad}(k)\fa_\scrC$ must also commute with $H$. We thus have another abelian subalgebra $\mathrm{Ad}(k)\fa_\scrC\subset \fs\cap Z_\fg(H)$. According to \cite[Proposition 7.29]{Knapp-96} (with reductive group $Z_G(H)$ instead of $G$), there exists $k_0 \in Z_K(H) $ such that
\[
\mathrm{Ad}(k_0)\mathrm{Ad}(k)\fa_\scrC \subset \fa
\]
and we have proved the claim.

We use the following lemma.
\begin{lemma}
Let $\scrC$ be a generalized Weyl chamber. If $\fa_\scrC$ is contained in another maximal abelian subspace $\fa'$, there exists $k_\scrC \in K_\scrC$ such that $\mathrm{Ad}(k_\scrC)\fa' = \fa$.
\end{lemma}

\begin{proof}
Since all maximal abelian subspaces of $\fs$ are conjugate under $K$, we find some $k\in K$ such that
\[
k \fa' k^{-1} = \fa.
\]
Then $k \fa_\scrC k^{-1} \subset \fa$. Let $H_0\in \fa_\scrC$ be regular. By \cite[Lemma 7.38]{Knapp-96}, there exists $k_0\in N_K(\fa)$ so that
\[
\mathrm{Ad}(k_0) H_0 = \mathrm{Ad}(k) H_0.
\]
Replacing $k$ by $k_0^{-1}k$, we may thus assume that $k \in Z_K(H_0)=K_\scrC$.
\end{proof}

Applying this lemma to $\fa_\scrC \subset \mathrm{Ad}(k^{-1})\fa$, we find $k_\scrC \in K_\scrC$ such that
\[
\mathrm{Ad}(k_\scrC)\mathrm{Ad}(k^{-1})\fa = \fa.
\]
We deduce that $k_\scrC k^{-1}\in N_K(\fa)$. Thus, after appropriate left and right multiplication, we have shown that $k\in N_K(\fa)$ as required.
\end{proof}

\subsection{Orientability}

In order to express the Poincar\'e series as a pairing of currents, we need to discuss the orientation of $G/M_\scrC$ and those of its submanifolds $(xK)M_\scrC$.
We start by observing that $G/M_\scrC$ is itself orientable. Indeed, recall the $\mathrm{Ad}(M_\scrC)$-invariant decomposition \eqref{eq:hyperbolic_splitting_first}. Naturally, $\mathrm{Ad}(M_\scrC)$ acts trivially on the first summand $\fa_\scrC$. Now, the Cartan involution $\theta$ of $\fg$ swaps $\fn_\scrC^+$ and $\fn_\scrC^-$, and conjugates the $\mathrm{Ad}(M_\scrC)$-action, so that $\mathrm{Ad}(M_\scrC)$ acts on $\fn_\scrC^+ \oplus \fn_\scrC^-$ with positive determinant.

Next, we observe that there is a diffeomorphism
\[
G/P_\scrC\simeq K/K_\scrC \owns k K_\scrC \mapsto xkM_\scrC \in (xK) M_\scrC\subset G/M_\scrC.
\]
The partial flag manifold is not always orientable (see \cite{Patrao-et-al-2012} for a rather complete discussion of this subject). However, recall that the unstable distribution $E_u$ is integrable, and its leaves are subsets of $G/P_\scrC$; this suggests that the nonorientability of $xKM_\scrC$ is related to the orientation of $E_u$.
\begin{lemma}
If $G/P_\scrC$ is non-orientable, then $E_u$ is not orientable, and 
\[
\fo(E_u)_{|(xK)M_\scrC} \simeq \fo(T(xK)M_\scrC).
\]
Likewise, 
\[
\fo(E_{ws})_{|(xKA_\scrC) M_\scrC} \simeq \fo(T (xKA_\scrC)M_\scrC),
\]
Since $G/M_\scrC$ is orientable, $\fo(E_{ws}) \simeq \fo(E_u)$. 
\end{lemma}

\begin{proof}
Recall from \eqref{eq:TxK} that $T_{xk M_\scrC} (xKM_\scrC)$ is transverse to $E_{ws}$, so that the projection map 
\[
\pi : E_u \oplus E_{ws} \to E_u
\]
restricts to an isomorphism $\pi|_{T(xKM_\scrC)}:T(xKM_\scrC)\xrightarrow{\sim}(E_u)|_{xKM_\scrC}$ and therefore induces an isomorphism of orientation bundles. Since $G/P_\scrC$ is not orientable, neither is $xKM_\scrC$, so that, as a covering space, $\fo(TxKM_\scrC)$ must be connected. Since $G/M_\scrC$ is itself connected, this implies that $\fo(E_u)$ is connected, and thus $E_u$ is non-orientable. The case of $xKA_\scrC M_\scrC$ is identical.
\end{proof} Fix also an orientation of $\fa_\scrC$, use it to orient the $A_\scrC$-action directions, and require the ordered action basis introduced below to be positive for this orientation.

\begin{lemma}\label{lem:orientations:compatible}
We choose orientations separately in two cases. 
\begin{enumerate}[label=\emph{(\roman*)}]
	\item If $G/P_\scrC$ is orientable, we choose an orientation $\eps$ of $G/M_\scrC$, and an orientation $\delta$ of $KM_\scrC$. Then we obtain a natural orientation on each $xKM_\scrC$ and each $xKA_\scrC M_\scrC$ by pushforward.
	\item If $G/P_\scrC$ is non-orientable, we choose an orientation $\eps$ of $\fo(E_u)$, and an orientation $\delta$ of $\fo(T KM_\scrC)$. Then we obtain a natural orientation on each $\fo(TxKM_\scrC)$ and each $\fo(TxKA_\scrC M_\scrC)$ by pushforward. 
\end{enumerate}
By choosing $\eps$ and $\delta$ appropriately, we can ensure that whenever $x,y\in G$ satisfy $\alpha(\mu_\fapp(y^{-1}x))>0$ for each $\alpha\in\Delta^+_\scrC$, and $z_\scrC(x,y)\in xKM_\scrC \cap y K \exp(\scrC) M_\scrC$ is the point of transverse intersection introduced in \eqref{eq:def-z_C}, the intersection number introduced in Lemma~\ref{lem:submanifold_pairing} or Corollary~\ref{cor:submanifold-pairing-non-oriented} satisfies
\[
\sigma(z_\scrC(x,y)) = 1.
\]
\end{lemma}

\begin{proof}
Replacing $\eps$ by $-\eps$ changes $\sigma(z_\scrC(x,y))$ to $-\sigma(z_\scrC(x,y))$, so we only need to show that $\sigma(z_\scrC(x,y))$ does not depend on $x,y$. The first observation is that,
since we defined our orientation by pushforward, and $G$ preserves the orientation of $G/M_\scrC$, for $\gamma\in G$,
\[
\sigma(z_\scrC(\gamma x, \gamma y)) = \sigma(z_\scrC(x,y)). 
\]
Also, recall that $z_\scrC(xk, yk') = z_\scrC(x,y)$ for $k,k'\in K$. Using the $KAK$ decomposition in \eqref{eq:def-z_C}, we conclude that
\[
\sigma(z_\scrC(x,y)) = \sigma(z_\scrC( e^{\mu_{\fapp}(y^{-1}x)}, e ) ). 
\]
The second observation is that because $z_\scrC(x,y)$ is a transverse point of intersection and because $xKM_\scrC$ and $yK\exp(\scrC) M_\scrC$ vary in smooth families with smoothly varying orientations, the number $\sigma(z_\scrC(x,y))$ must be locally constant. Since the domain
\[
\{ a\in \fap ~:~ \alpha(a)>0,\ \alpha\in\Delta^+_\scrC\}
\]
is connected, $\sigma(z_\scrC(x,y))$ is thus constant, independent of $x,y$. 
\end{proof}

\subsection{Anosov subgroups and asymptotics of intersections}

In the previous subsection, we have studied the intersection points on $G/M_\scrC$. We now prove the results needed to work on the biquotient $\Gamma\backslash G/M_\scrC$, or, more precisely, on the smooth subset $\Gamma\backslash \odms$ of this biquotient.

For fixed $x,y\in G$, studying the intersection of $\Gamma (xK)M_\scrC$ and $\Gamma (yK\exp(\scrC)) M_\scrC$  inside $\Gamma \backslash G/M_\scrC$ is equivalent to studying the intersection points of $(xK)M_\scrC$ and $(\gamma yK\exp(\scrC)) M_\scrC$  inside $G/M_\scrC$ for all $\gamma\in \Gamma$. We first consider the regular intersection points $z_\scrC(x,\gamma y)$.

\begin{proposition}\label{prop:late-intersections-are-in-DMS}
Let $\Gamma$ be a $\scrC$-transverse subgroup and let $\fapp\geq \scrC$ be an open Weyl chamber. Let $x,y\in G$ be fixed. There exists a precompact set $B\subset \odms$ such that, for all but finitely many $\gamma\in\Gamma$, the following statements hold:
\begin{enumerate}[label=\emph{(\roman*)}]
 \item\label{it:intersect_dms_1} $\alpha(\mu_\fapp(y^{-1}\gamma^{-1} x))>0$ for all $\alpha\in\roots^+_\scrC$.
 \item $z_\scrC( x, \gamma y)\in B$, with accumulation points lying in the $\Gamma$-invariant preimage in $\odms$ of $\mc{T}_{+,\M}\subset\M$.
\end{enumerate}
\end{proposition}

\begin{proof}
To prove this proposition, we will first introduce notation for Cartan decompositions. We start by fixing a sequence of pairwise distinct elements $\gamma_\ell \in\Gamma$, and points $x,y\in G$. We write the Cartan decompositions
\begin{equation}\label{eq:def-k_ell-k'_ell-...}
y^{-1} \gamma_\ell^{-1} x = k_\ell e^{H_\ell} k_\ell'^{-1},\quad \gamma_\ell^{-1} = \tilde{k}_\ell e^{\tilde{H}_\ell} \tilde{k}_\ell'^{-1}.
\end{equation}
After passing to a subsequence, we may assume that the resulting sequences in $K$ converge. Since $\scrC$-transversality implies $\scrC$-divergence, $\alpha(\tilde{H}_\ell)\to +\infty$ for every $\alpha\in\roots^+_\scrC$. Thus, by Lemma~\ref{lemma:Benoist-compact-Cartan}, $\alpha(\tilde{H}_\ell)>0$ and $\alpha(H_\ell)>0$ for all sufficiently large $\ell$. This proves \ref{it:intersect_dms_1}.

The intersection point $z_\scrC(x,\gamma_\ell y)$ is $xk'_\ell M_\scrC$, corresponding to Hopf coordinate boundary points
\[
\xi_\ell^\pm = x k'_\ell P_{\pm \scrC}. 
\]
Using the convergence of $k_\ell'$ we can introduce
$\xi^\pm =\lim_{\ell\to\infty}\xi^\pm_\ell = xk'_\infty P_{\pm\scrC}$ and by definition we have \begin{equation}\label{eq:xipm}
 \xi^+ \pitchfork \xi^-.
\end{equation}
Now, we claim that $\xi^-\in\Lambda_{-\scrC}$. To see this, we invert \eqref{eq:def-k_ell-k'_ell-...}, $x^{-1} \gamma_\ell y = k'_\ell e^{-H_\ell} k_\ell^{-1}$, and deduce by applying twice Lemma~\ref{lem:limit_points_leftright_mult} that 
\begin{equation}
	\label{eq:xi-inlimitset}
\gamma_{-\scrC}^+ =  (\gamma_\ell y)_{-\scrC}^+ = (x x^{-1}\gamma_\ell y)_{-\scrC}^+ =x k'_\infty P_{-\scrC} = \xi^-,
\end{equation}
which proves the claim. 

Since $z_\scrC(x,\gamma_\ell y)\in xK M_\scrC$ and the latter is compact in $G/M_\scrC$, the sequence can leave every precompact subset of $\odms$ only if, for $\xi^\pm := \lim_{\ell\to \infty} \xi_\ell^\pm$, there is an $\eta\in \Lambda_\scrC$ such that $\xi^-\not\pitchfork \eta$ and $\xi^+\not\pitchfork \mathcal{I}_{\scrC}(\eta)$. These are precisely the points removed in the definition of $\odms$. By $\scrC$-transversality of $\Gamma$ (recall also that, by Remark~\ref{rem:sym_scrC}, we can work with $\iota_\fapp$-symmetric $\scrC$), $\xi^-\not\pitchfork \eta$ implies
$\xi^- = \mathcal{I}_{\scrC}(\eta)$. Then $\xi^+\not\pitchfork \mathcal{I}_{\scrC}(\eta)$ would imply $\xi^+\not\pitchfork \xi^-$, contradicting \eqref{eq:xipm}.

\end{proof}

Let
\[
q:\odms\longrightarrow\Gamma\backslash\odms
\]
be the quotient map.

\begin{lemma}\label{lemma:DMS-ambient-separation}
Let $\scrC$ be a generalized Weyl chamber, and let $\Gamma<G$ be a discrete, torsion-free, $\scrC$-divergent subgroup. If $K\subset\Gamma\backslash\odms$ is compact, then
\[
\overline{q^{-1}(K)}^{\,G/M_{\scrC}}=q^{-1}(K)\subset\odms.
\]
\end{lemma}

\begin{proof}
By Proposition~\ref{prop:odms_proper}, there is a compact set $\widetilde K\subset\odms$ such that
\[
q^{-1}(K)=\Gamma\widetilde K.
\]
Assume that $z_\ell=\gamma_\ell p_\ell\in\Gamma\widetilde K$ converges to some $z\in G/M_{\scrC}\setminus q^{-1}(K)$, where $\gamma_\ell\in\Gamma$ and $p_\ell\in\widetilde K$. After passing to a subsequence, we may assume that $p_\ell\to p\in\widetilde K$. Necessarily $\gamma_\ell\to\infty$; after passing to a further subsequence, we may assume that $(\gamma_\ell)$ converges at infinity. By Lemma~\ref{lemma:i-invariant}, it is both $\scrC$- and $-\scrC$-divergent.

Write $p=gM_{\scrC}$ and choose representatives $p_\ell=g_\ell M_{\scrC}$ such that $g_\ell\to g$. Since $p\in\odms$, either
\[
gP_{\scrC}\pitchfork\gamma_{\scrC}^-
\qquad\text{or}\qquad
gP_{-\scrC}\pitchfork\gamma_{-\scrC}^-.
\]
In the first case, $g_\ell P_{\scrC}\pitchfork\gamma_{\scrC}^-$ is uniformly transverse for all sufficiently large $\ell$. Lemma~\ref{lemma:translation-length} gives a bounded set $\mc B\subset\fa_{\scrC}$ such that
\[
H_{\scrC}(\gamma_\ell g_\ell)
\in
\mu_{\scrC}(\gamma_\ell)+H_{\scrC}(g_\ell)+\mc B.
\]
The sequence $H_{\scrC}(g_\ell)$ is bounded, whereas $\|\mu_{\scrC}(\gamma_\ell)\|\to\infty$. It follows that $H_{\scrC}(\gamma_\ell g_\ell)$ is unbounded, contradicting the convergence of $z_\ell=\gamma_\ell g_\ell M_{\scrC}$ in $G/M_{\scrC}$. The second case follows analogously.
\end{proof}

\begin{lemma}\label{lemma:bad-intersection-are-away}
Assume that $\scrC$ is an $\iota_\fapp$-symmetric face, that $x,y\in G$ are fixed and that $B\subset \odms$ is a $\Gamma$-invariant set such that $\Gamma \backslash B\subset \Gamma\backslash \odms$ is precompact. Then there are only finitely many $\gamma\in \Gamma$ such that the nonregular intersection points of $xKM_\scrC$ and $\gamma y K \exp(\scrC) M_\scrC$ lie in $B$ (see Lemma~\ref{lemma:form-of-intersection}). In other words,
\[
\begin{aligned}
\#\Bigl\{\gamma\in\Gamma:\;&
\bigl(B\cap xKM_\scrC\cap\gamma yK\exp(\scrC)M_\scrC\bigr) \setminus\{z_\scrC(x,\gamma y)\}\neq\varnothing
\Bigr\}<\infty.
\end{aligned}
\]
\end{lemma}

\begin{proof}
Assume that there is a diverging sequence $\gamma_\ell \in \Gamma$ and, for each $\ell$, a point $z_\ell\in B$ distinct from $z_\scrC(x, \gamma_\ell y)$ and lying in the intersection $xKM_\scrC\cap \gamma_\ell yK\exp(\scrC) M_\scrC $.
Set $Q:=\overline{\Gamma\backslash B}\subset\Gamma\backslash\odms$, which is compact by assumption.
According to Lemma \ref{lemma:form-of-intersection}, $z_\ell = xk'_\ell k_{\scrC,\ell} w_\ell M_\scrC$, where $k'_\ell$ is defined as in \eqref{eq:def-k_ell-k'_ell-...} and $k_{\scrC,\ell}\in K_\scrC$ and $w_\ell \in N_K(\fa)$ satisfy $[w_\ell]\neq [e]$ in $ W(G,A)/\langle \Delta^0_\scrC\rangle$.
Additionally, \eqref{eq:lem_full_intersection} gives
\begin{equation}\label{eq:positivity-condition}
\pi_\scrC(  w_\ell ^{-1} \mu_\fapp(y^{-1}\gamma_\ell^{-1} x)) \in \scrC.
\end{equation}
After possibly passing to a subsequence, we can ensure that
\[
k'_\infty = \lim_{\ell\to\infty}k'_\ell,\quad k_{\scrC,\infty} = \lim_{\ell\to\infty} k_{\scrC, \ell}, \quad w_\infty = \lim_{\ell\to\infty} w_\ell
\]
exist. By construction,
\[
z_\ell\longrightarrow z_\infty:=xk'_\infty k_{\scrC,\infty}w_\infty M_\scrC.
\]
Since $B\subset q^{-1}(Q)$, Lemma~\ref{lemma:DMS-ambient-separation} gives
\[
z_\infty\in\overline{B}^{\,G/M_\scrC}\subset q^{-1}(Q)\subset\odms.
\]
Moreover,
\[
xk'_\infty k_{\scrC, \infty}  w_\infty P_\scrC \not\pitchfork xk'_\infty P_{-\scrC} =\xi^-
\]
Furthermore, recall from \eqref{eq:xi-inlimitset} that $\xi^-\in\Lambda_{-\scrC}$.
On the other hand,
\[
x k'_\infty k_{\scrC,\infty} w_\infty P_{-\scrC} \pitchfork xk'_\infty w_\fapp P_\scrC = \mathcal I_{-\scrC}(\xi^-)
\]
if and only if (using $\iota_\fapp$-invariance and $k_{\scrC,\infty}\in K_\scrC < P_{-\scrC}$)
\[
w_\infty w_\fapp P_\scrC \pitchfork P_{-\scrC}.
\]
This is equivalent to $w_\infty w_\fapp \in \langle \Delta^0_\scrC \rangle$, i.e., $w_\infty = w_\fapp w'$ with $w'\in \langle \Delta^0_\scrC\rangle$ since $w_\fapp \in N_{W(G,A)}(\langle\Delta^0_\scrC\rangle)$.
However, this contradicts \eqref{eq:positivity-condition}. Indeed, for $H\in\fa$,
\[
\pi_\scrC( (w')^{-1} w_\fapp H) = -\iota_\fapp(\pi_\scrC H).
\]
To summarize, we have found $\xi^-\in\Lambda_{-\scrC}$ such that
\[
 \lim_{\ell\to\infty}z_\ell P_\scrC \not\pitchfork \xi^-\text{ and }
 \lim_{\ell\to\infty}z_\ell P_{-\scrC} \not\pitchfork \mathcal I_{-\scrC}(\xi^-)
\]
which contradicts $z_\infty\in\odms$ by Definition~\ref{definition:dms-domain}.
\end{proof}
We summarize the results of this section in the following technical but useful lemma on the existence of suitable cutoff domains, which will play a crucial role in the proof of the meromorphic continuation of the Poincaré series.

\begin{lemma}\label{lem:cutoff_poinc}
 Let $\Gamma$ be a $\scrC$-transverse subgroup and let $x,y\in G$. Then there are $\Gamma$-invariant open subsets $U\subset W\subset\odms$ and $H_0\in \scrC$ such that:
 \begin{enumerate}[label=\emph{(\roman*)}]
  \item $\Gamma\backslash U\Subset \Gamma\backslash W\Subset\Gamma\backslash \odms$ are precompact subsets.
  \item\label{it:2lem_cutoff} If $\mu_\scrC(y^{-1}\gamma^{-1} x) \in H_0+\scrC$ then $z_\scrC(x, \gamma y) \in U$ and $z_{-\scrC}(\gamma y,x) \in U$.
  \item If $\mu_\scrC(y^{-1}\gamma^{-1} x) \in H_0+\scrC$ then $xKM_\scrC\cap \gamma y K\exp(\scrC)M_\scrC\cap W = \{z_\scrC(x,\gamma y)\}$.
 \end{enumerate}
\end{lemma}
\begin{proof}
 First, Lemma~\ref{lemma:i-invariant} shows that $\Gamma$ is also $-\scrC$-transverse, while Definition~\ref{definition:dms-domain} gives $\Omega_{\mathrm{DMS}}^{-\scrC}=\odms$. Apply Proposition~\ref{prop:late-intersections-are-in-DMS} to $(x,y)$ for $\scrC$ and to $(y,x)$ for $-\scrC$, using the auxiliary chamber $-\fapp$ in the latter case. We obtain precompact open sets $B_+,B_-\subset\odms$ such that, for all but finitely many $\gamma\in\Gamma$, $z_\scrC(x,\gamma y)\in B_+$ and $z_{-\scrC}(y,\gamma^{-1}x)\in B_-$. Set $B:=B_+\cup B_-$ and $U:=\bigcup_{\delta\in\Gamma}\delta B$. By equivariance, $z_{-\scrC}(\gamma y,x)=\gamma z_{-\scrC}(y,\gamma^{-1}x)\in\gamma B_-\subset U$, while $z_\scrC(x,\gamma y)\in B_+\subset U$. Next, we take a slightly larger open $\Gamma$-invariant subset $W\subset \odms$ such that $\Gamma\backslash W\subset \Gamma \backslash \odms$ is still precompact. More precisely, choose an open set $V\subset\Gamma\backslash\odms$ such that
\[
\overline{q(U)}\subset V\Subset\Gamma\backslash\odms
\]
and set $W:=q^{-1}(V)$. Then, by Lemma~\ref{lemma:bad-intersection-are-away}, we can ensure that for all but finitely many $\gamma\in\Gamma$:
 \[
 xKM_\scrC\cap \gamma y K\exp(\scrC)M_\scrC\cap W = \{z_\scrC(x,\gamma y)\}.
 \]
Next, we simply take $H_0\in\scrC$ large enough that, for each of the finitely many exceptional elements $\gamma$ above, $\mu_\scrC(y^{-1}\gamma^{-1} x) \not \in H_0 + \scrC$.
\end{proof}

\subsection{Proof of meromorphic continuation}

For $\mc B \subset \fa_\scrC$, define a truncated version of the Poincaré series ($x,y\in \Gamma\backslash G/ K$, $\s\in(\fa_\scrC^*)_\C$):
\[
  P_{\Gamma,\mc B}(x,y,\s) := \sum_{\gamma\in \Gamma,~ \mu_\scrC(y^{-1}\gamma x)\in\mc B} e^{-\s(\mu_\scrC(y^{-1}\gamma x))}.
\]
If $\mc B$ is bounded, this is a finite sum, since $\Gamma$ is discrete.

Following \S\ref{section:laplace_transform} (see \eqref{eq:Laplace-transform-dynamics}), we introduce
\[
T^{\mc B}(\s):= \int_{\mc B} e^{-\s(a)}(\tau(a))_*  \dd a = \int_{\mc B} e^{-(\X+\s)(a)} \dd a.
\]
For bounded $\mc B$, this is a well-defined operator on currents and it was shown in \S\ref{section:laplace_transform} that it can be meromorphically extended to $(\fa^*_{\scrC})_{\C}$ for certain noncompact cones. In what follows, we will set ($\fo$ standing here for the bundle of orientations on $E_u$)
\[
[xK]^{(\fo)} = \begin{cases} [xK] & \text{ if $G/P_\scrC$ is orientable}\\ [xK]^\fo & \text{if $G/P_\scrC$ is not orientable} \end{cases}.
\] Fix a positively oriented, Lebesgue-volume-one basis $(H_1,\ldots,H_{k+1})$ of $\fa_\scrC$, use this ordering to orient the action directions, and denote the induced action vector fields by the same symbols.
The following holds:

\begin{lemma}
 Let $\mc B\subset\scrC$ be a bounded open set and let $x,y\in G$, and assume that $\mu_\scrC(y^{-1}x)\notin\partial\mc B$. Then the pairing of the currents $[xK]^{(\fo)}$ and $T^{\mc B}(\s)[yK]^{(\fo)}$ is well-defined after restriction to a small neighborhood $U$ of $z_\scrC(x,y)$, and we have for $\s\in(\fa_\scrC^\ast)_\C$
 \[
  \big([xK]^{(\fo)}, \iota_{H_{k+1}}\ldots\iota_{H_1}T^{\mc B}(\s)[yK]^{(\fo)}\big)_{U\subset G/M_\scrC} =   \mathbf{1}_{\mc B}(\mu_\scrC(y^{-1}x)) e^{-\s(\mu_\scrC(y^{-1}x))}.
 \]
\end{lemma}
\begin{proof}
 According to Lemma~\ref{lem:submanifold_propagation}, the current $\iota_{H_{k+1}}\ldots\iota_{H_1}T^{\mc B}(\s)[yK]^{(\fo)}$ is the current $[yK\exp(\mc B)]^{(\fo)}$ multiplied by the smooth function $e^{-\s(a)}$. Indeed, by the Hopf-coordinate diffeomorphism \eqref{eq:Hopf}, the map
 \[
 (K/K_\scrC)\times\mc B\longrightarrow G/M_\scrC,\qquad (kK_\scrC,a)\longmapsto yk e^aM_\scrC
 \]
 is an embedding. Thus Lemma~\ref{lem:submanifold_propagation} applies directly to $\mc B$. Proposition~\ref{prop:intersection-non-borel} ensures that the pairing is well-defined, and Lemma~\ref{lem:submanifold_pairing}, Corollary~\ref{cor:submanifold-pairing-non-oriented}, and Lemma~\ref{lem:orientations:compatible} give the formula.
\end{proof}

We deduce that the truncated Poincaré series can be expressed as a pairing of currents:

\begin{proposition}\label{prop:formula-lifting}
 Let $\mc B\subset\scrC$ be bounded, let $x,y\in G$ and $\s\in(\fa_\scrC^*)_\C$, and let $W\subset G/M_\scrC$ be a $\Gamma$-invariant open set. Assume that $z_\scrC(x, \gamma y)\in W$ for every $\gamma$ satisfying $\mu_\scrC(y^{-1}\gamma^{-1} x)\in \mc B$ and that, inside $W$, every point of $xKM_\scrC\cap\Gamma yKM_\scrC\exp(\mc B)$ is one of these distinguished points. Then the truncated Poincar\'e series is given by the following pairing of currents on $\Gamma\backslash G/M_\scrC$:
\[
 P_{\Gamma,\mc B}(x,y,\s) = \big([\Gamma xK]^{(\fo)}, \iota_{H_{k+1}}\ldots \iota_{H_1} T^{\mc B}(\s)[\Gamma y K]^{(\fo)}\big)_{\Gamma\backslash W}
\].
\end{proposition}
\begin{proof}
Using a standard unfolding argument, we obtain
\begin{align*}
 \big([\Gamma xK]^{(\fo)}, \iota_{H_{k+1}}\ldots \iota_{H_1} T^{\mc B}(\s)[\Gamma y K]^{(\fo)}\big)_{\Gamma\backslash W} =& \sum_{\gamma\in \Gamma}\big([xK]^{(\fo)}, \iota_{H_{k+1}}\ldots \iota_{H_1} T^{\mc B}(\s)[ \gamma y K]^{(\fo)}\big)_{W}\\
 =&\sum_{\substack{\gamma\in \Gamma\\ \mu_\scrC(y^{-1}\gamma^{-1} x)\in\mc B}} e^{-\s(\mu_\scrC(y^{-1}\gamma^{-1} x))}.
\end{align*}
This proves the claim.
\end{proof}

We can now prove Theorem \ref{theorem:Poincare}.
\begin{proof}[Proof of Theorem \ref{theorem:Poincare}]
Let us take $U, W\subset \odms$ and $H_0\in\scrC$ from Lemma~\ref{lem:cutoff_poinc}. Since $\scrC$-Anosov implies $\scrC$-divergence, Lemma~\ref{lemma:Benoist-compact-Cartan} shows that all but finitely many $\gamma\in \Gamma$ satisfy
$\mu_\scrC(y^{-1}\gamma^{-1} x)\in H_0+\scrC$.

Let $\scrC_r :=\{H\in \scrC, \|H\|< r\}$. Then $\mathcal B = H_0 + \scrC_r$ and $W$ satisfy the assumptions of Proposition~\ref{prop:formula-lifting}. For $|\Re(\s)|\gg0$, choose radii tending to infinity outside the countable set $\{\|\mu_\scrC(y^{-1}\gamma^{-1}x)-H_0\|:\gamma\in\Gamma\}$; along these radii, the Anosov assumption on $\Gamma$ permits us to take the limit. We obtain
\[
P_{\Gamma,\scrC}(x,y,\s)  = ([\Gamma x K]^{(\fo)}, \iota_{H_{k+1}}\dots\iota_{H_1} T^{H_0+\scrC}(\s)[\Gamma y K]^{(\fo)})_{\Gamma\backslash W} + \text{ holomorphic}.
\]
The holomorphic remainder is a finite sum corresponding to the finitely many exceptional elements $\gamma$, which thus has a holomorphic continuation to $(\fa_\scrC^\ast)_\C$.

Let $\chi\in C^\infty_{\comp}(\Gamma\backslash W)$ be a cutoff function equal to $1$ in $\Gamma\backslash U$. Then
\[
\begin{split}
([\Gamma x K]^{(\fo)}, \iota_{H_{k+1}}\dots\iota_{H_1} & T^{H_0+\scrC}(\s)[\Gamma y K]^{(\fo)})_{\Gamma\backslash W} \\
&  = (\chi[\Gamma x K]^{(\fo)}, \iota_{H_{k+1}}\dots\iota_{H_1} T^{H_0+\scrC}(\s)\chi[\Gamma y K]^{(\fo)})_{\M}.
\end{split}
\]
In order to insert $\chi$ in front of the current $[\Gamma yK]$, observe that every $\gamma$ occurring in the series satisfies
$z_\scrC(x,\gamma y)=z_{-\scrC}(\gamma y,x)e^{\mu_\scrC(y^{-1}\gamma^{-1}x)}$. Thus, $z_{-\scrC}(\gamma y,x)$ is the corresponding unpropagated input point, and it belongs to $U$ by Lemma~\ref{lem:cutoff_poinc}, item~\ref{it:2lem_cutoff}.

The wavefront sets of $[\Gamma xK]^{(\fo)}$ and $[\Gamma yK]^{(\fo)}$ are contained in the conormals $\V_x^\perp$ and $\V_y^\perp$, respectively (recall the definition of the $xK$-dependent vertical bundle $\V_x$ in \eqref{eq:VxBundle} as the tangent bundle to submanifolds $xKM_\scrC$). Recalling \eqref{eq:transverse_inter_lie_algebra}, we find that at each contributing distinguished intersection point $p=z_\scrC(x,\gamma y)$, if $\widehat H:=\mu_\fapp(y^{-1}\gamma^{-1}x)$ and $H:=\pi_\scrC(\widehat H)\in\scrC$, then
\[
(\V_{y,H})_p\oplus(\V_x)_p=(E_s)_p\oplus(\V_x)_p,
\]
and
\[
T_p(G/M_\scrC)=(\V_x)_p\oplus(E_s)_p\oplus(E_0)_p.
\]
Since $E^\ast_s= (E_0 \oplus E_s)^\perp$, it must be transverse to $\V_x^\perp$. By symmetry, we also have $E^\ast_u \cap  \V_x^\perp = \{0\}$. We also obtain that $(\V_{y,H} + E_0)^\perp$ is transverse to $ \V_x^\perp$. Because the action is abelian, $T^{H_0+\scrC}(\s)=e^{-(\X+\s)(H_0)}T^\scrC(\s)$, and contractions by the invariant fields $H_j$ commute with the transform. Moreover, $\iota_{H_{k+1}}\cdots\iota_{H_1}[\Gamma yK]^{(\fo)}$ is a $d_s$-form in the transverse bundle $\Lambda^{d_s}(E_s^*\oplus E_u^*)$, with the orientation twist retained when necessary. Thus the finite translation changes no poles and the relevant coefficient degree is $d_s$.

The action of $T^{H_0+\scrC}$ on $[\Gamma y K]^{(\fo)}$ is defined through the natural lift of the action to the bundle of (twisted) forms, and this is certainly an admissible lift in the sense of \S~\ref{sssection:admissible}. We deduce from Theorem~\ref{theorem:product-resolvents} and \eqref{equation:wftc} that $\iota_{H_{k+1}}\cdots\iota_{H_1}T^{H_0+\scrC}(\s)\chi[\Gamma y K]^{(\fo)}$ has a meromorphic continuation as a form-valued distribution, with wavefront set contained in
\[
\begin{split}
\bigcup_{H\in H_0+\partial \scrC} \V_{y,H}^\perp \ \cup\  \bigcup_{H\in H_0+\scrC} (\V_{y,H}+E_0)^\perp \ \cup\ E^\ast_u.
\end{split}
\]
The first summand here corresponds to contributions from parameters $a$ on the boundary of $H_0+\scrC$. For almost every value of $H_0$, the corresponding intersection is empty, so that the wavefront set condition is satisfied, and
\[
\chi [\Gamma x K]^{(\fo)} \wedge \iota_{H_{k+1}}\dots \iota_{H_1} T^{H_0+\scrC}(\s)\chi[\Gamma y K]^{(\fo)}
\]
is a well-defined compactly supported top-form-valued distribution, depending meromorphically on $\s$.
\end{proof}

\newpage

\appendix

\section{Triviality of $\R^k$-principal bundles}

\label{appendix:trivial-topology}

\begin{lemma}
\label{lemma:trivial}
Every $\R^k$-principal bundle over a paracompact manifold $N$ admits a section (i.e. is trivial). This holds in $C^\infty$ \emph{and} analytic settings. 
\end{lemma}

\begin{proof}
One can use the theory of classifying spaces, or directly construct a global section for any such bundle. Let $\mathcal{U}=\{U\}$ be a locally finite open cover of $N$, sufficiently fine so that we have a system of local smooth sections $\{\sigma_U\}$ (that can be smoothly extended to slightly larger open sets $\{\tilde{U}\}$). Enumerate the cover as $U_0, \dots, U_\ell,\dots$, and set $V_\ell = U_0 \cup \dots \cup U_\ell$. Then we construct a sequence $\sigma_\ell$ of sections of the bundle over $V_\ell$. First, set $\sigma_0 = \sigma_{U_0}$. Next, observe that on $U_0\cap U_1$,
\[
\sigma_{U_1}= \sigma_{U_0} + f_{1},
\]
for some smooth $\R^k$-valued function $f_{1}$, which we can extend continuously to $U_1$ as a function $\tilde{f}_{1}$, using some cutoff function supported in $\tilde{U}_0\cap\tilde{U}_1$. Then set 
\[
\sigma_1 = \begin{cases} \sigma_0 & \text{ on $U_0$,} \\  \sigma_{U_1} - \tilde{f}_{1} & \text{ on $U_1$,}\end{cases}.
\]
We proceed by induction, replacing at each step $\ell\geq 2$, $U_0$ by $V_{\ell-1}$, and $U_1$ by $U_\ell$. Since $N$ is paracompact, for any $x$ in $N$, $\sigma(x):=\sigma_\ell(x)$ is defined and constant for all $\ell>\ell(x)$ for some $\ell(x)<\infty$. This provides a smooth global section $\sigma$ as announced.

The argument above relies on the fineness of the sheaf of $C^\infty$ functions, which implies trivial cohomology. For the sheaf of real analytic functions the same argument cannot apply. However we can use a series of results to reach the same conclusion. 
\begin{enumerate}
	\item Each real analytic paracompact manifold $E$ admits an embedding into a complex manifold $E^\C$, where it is the fixed point set of an anti-holomorphic involution (a theorem of Whitney).
	\item Each neighbourhood of $E$ in $E^\C$ contains a neighbourhood of $E$ that is Stein (a theorem of Grauert). 
	\item The $p$-cohomology of the sheaf of holomorphic sections on Stein manifolds vanishes for $p>0$ (Cartan's Theorem B). 
\end{enumerate}
In the argument above, if the bundle were real analytic, we could pick our local sections $\sigma_j$ to be real analytic, and consider a small complex neighbourhood $E^\C$ of $E$, covered by open sets $U_j^\C$, so that $U_j^\C \cap E = U_j$, $j\geq 0$, and so that each $\sigma_j$ extends holomorphically to $U_j^\C$. Then Cartan's theorem B ensures that the 1-cocycle $\sigma_i-\sigma_j$ on $U_i^\C \cap U_j^\C$ is a coboundary and we conclude that there exists a real analytic global section of the bundle. 
\end{proof}

\section{Multivariable Fredholm theory}

Let $E_1$ and $E_2$ be Banach spaces, and let $U \subset \C^n$. We recall the following standard terminology:

\begin{definition}[Holomorphic and meromorphic families of operators]
A family of bounded operators $U \ni \mathbf{s} \mapsto Q(\mathbf{s}) \in \mc{L}(E_1,E_2)$ is \emph{holomorphic} if, for every $u \in E_1$ and $v \in E_2^*$, the function $U \ni \mathbf{s} \mapsto (Q(\mathbf{s})u, v) \in \C$ is holomorphic. A family $\mathbf{s} \mapsto P(\mathbf{s}) \in \mc{L}(E_1,E_2)$ is \emph{meromorphic} if, locally, it can be written as $P(\mathbf{s}) = Q(\mathbf{s})/r(\mathbf{s})$ for a holomorphic family $\mathbf{s} \mapsto Q(\mathbf{s}) \in \mc{L}(E_1,E_2)$ and a holomorphic function $r$ that is not identically zero. Finally, we call $\mathbf{s} \mapsto P(\mathbf{s})$ \emph{finitely meromorphic} if it is the sum of a meromorphic family of finite-rank operators and a holomorphic family of operators.
\end{definition}

We prove the following version of the analytic Fredholm theorem in several variables:

\begin{theorem}[The analytic Fredholm theorem]
\label{theorem:fredholm-several}
Let $U \subset \C^n$ be open, connected, and let $U \ni \mathbf{s} \mapsto P(\mathbf{s}) \in \mc{L}(E_1,E_2)$ be a holomorphic family of Fredholm operators between Banach spaces $E_1$ and $E_2$. Suppose that there exists $\mathbf{s}_0 \in U$ such that $P(\mathbf{s}_0) : E_1 \to E_2$ is invertible. Then $U \ni \mathbf{s} \mapsto P(\mathbf{s})^{-1} \in \mc{L}(E_2,E_1)$ is a finitely meromorphic family of operators. The subset
\[
S := \{\mathbf{s} \in U ~:~ P(\mathbf{s}) : E_1 \to E_2 \text{ is not invertible} \}
\]
is a complex variety of codimension $1$; i.e., it is locally the zero set of a holomorphic function, and $\mathbf{s} \mapsto P(\mathbf{s})^{-1} \in \mc{L}(E_2,E_1)$ is holomorphic on $U \setminus S$.
\end{theorem}

Notice that the invertibility assumption at $\mathbf{s}_0$ implies that the family of Fredholm operators $\mathbf{s} \mapsto P(\mathbf{s})$ has index $0$.

\begin{proof}
If $E_1 = E_2 \simeq \C^N$, the statement is immediate because
\[
P(\mathbf{s})^{-1} = \dfrac{1}{\det P(\mathbf{s})} \operatorname{adj}(P(\mathbf{s})),
\]
where $\operatorname{adj}$ stands for the adjugate matrix. The maps $\mathbf{s} \mapsto \operatorname{adj}(P(\mathbf{s}))$ and $\mathbf{s} \mapsto \det(P(\mathbf{s}))$ are holomorphic, and $S$ is the zero set of $\det(P(\bullet))$. (The function $\det(P(\bullet))$ is not identically zero because $P(\mathbf{s}_0)$ is invertible.)

We now prove the general claim by reducing it to the finite-dimensional case. Let $\mathbf{s}_1 \in U$. Since $P(\mathbf{s}_1)$ is Fredholm (of index $0$), we can find $F_1\subset E_1, G_2 \subset E_2$ such that
\[
F_1 \oplus \ker P(\mathbf{s}_1) = E_1, \qquad  \ran P(\mathbf{s}_1)\oplus G_2 = E_2.
\]
Because $P(\mathbf{s}_1)$ has index $0$, $\dim G_2 = \dim \ker P(\mathbf{s}_1)$. We may therefore represent $P(\mathbf{s}_1)$ in $2 \times 2$ block-matrix form as
\[
P(\mathbf{s}_1) = \begin{pmatrix} A(\mathbf{s}_1) & 0 \\ 0 & 0 \end{pmatrix},
\]
where $A(\mathbf{s}_1)$ denotes the induced (and invertible) operator $A(\mathbf{s}_1) : F_1 \to \ran(P(\mathbf{s}_1))$. In the same basis, we can write
\[
P(\mathbf{s}) = \begin{pmatrix} A(\mathbf{s}) & B(\mathbf{s}) \\ C(\mathbf{s}) & D(\mathbf{s}) \end{pmatrix},
\]
where $B,C,D$ are finite-rank operators, and $D$ is a linear map between finite-dimensional spaces of the same dimension. For $\mathbf{s}$ near $\mathbf{s}_1$, $A(\mathbf{s})$ is invertible, and its inverse is given by a convergent power series and is therefore holomorphic. Moreover, $P$ is invertible if and only if the linear map between finite-dimensional spaces $D- CA^{-1}B: \ker(P(\mathbf{s}_1)) \to G_2$ is invertible; in that case, $P^{-1}$ is given by
\begin{equation}
\label{equation:2x2}
 \begin{pmatrix} A^{-1} + A^{-1}B(D- CA^{-1}B)^{-1}C A^{-1} & -A^{-1}B(D- CA^{-1}B)^{-1} \\ -(D- CA^{-1}B)^{-1}CA^{-1} & (D- CA^{-1}B)^{-1} \end{pmatrix}.
\end{equation}
This is verified by direct multiplication. Notice that $(D- CA^{-1}B)(\mathbf{s})$ is invertible if and only if $f(\mathbf{s}) \neq 0$, where $f(\mathbf{s}) := \det\!\bigl((D-CA^{-1}B)(\mathbf{s})\bigr)$; the function $f$ is holomorphic in $\mathbf{s}$.

Choose a locally finite cover $(U_k)_{k \geq 0}$ of $U$ and, on each $U_k$, a holomorphic function $f_k$ (defined as above) such that $f_k(\mathbf{s}) \neq 0$ if and only if $P(\mathbf{s})$ is invertible. By assumption, there exists $\mathbf{s}_0 \in U$ such that $P(\mathbf{s}_0)$ is invertible; this implies that the functions $f_k$ are all non-zero. On each $U_k$, $\mathbf{s} \mapsto ((D- CA^{-1}B)(\mathbf{s}))^{-1}$ is meromorphic by the finite-dimensional part of the proof. Going back to \eqref{equation:2x2}, this implies that $\mathbf{s} \mapsto P(\mathbf{s})^{-1}$ is meromorphic too. Finally, the only term in \eqref{equation:2x2} that is not of finite rank is the holomorphic family $A^{-1}$; hence, $P(\mathbf s)^{-1}$ is finitely meromorphic.
\end{proof}

We now write $\mathbf{s} = (z,w) \in \C \times \C^{n-1}$ and consider the more specific case where $E_1 = E_2 =: E$ and $P(\mathbf{s}) = A(w)-z$ is a holomorphic family of Fredholm operators, where $w \mapsto A(w) \in \mc{L}(E)$ is a holomorphic family of operators. In this case, for each fixed $w$, assuming that there is a $z_0$ such that $A(w)-z_0$ is invertible, the operator $z \mapsto (A(w)-z)^{-1} \in \mc{L}(E)$ is meromorphic on its domain of definition, as follows from the standard analytic Fredholm theorem; see, for example, \cite[Theorem 21.1.23]{Lefeuvre-book}. The poles of this family of operators are given by the intersection $\sigma(A(w)) := S \cap \C_{w}$, where $\C_w := \{(z,w) ~:~ z \in \C\}$. We let $\rho(A(w))$ be the complement of $\sigma(A(w))$ on the domain of definition of $z \mapsto A(w)-z$.

Let $\mathbf{s}_0 = (z_0,w_0) \in S$. Choose a positively oriented simple closed contour $\gamma\subset\rho(A(w_0))$ enclosing $z_0$ and no other point of $\sigma(A(w_0))$, and define the spectral projector by
\begin{equation}
\label{equation:piw}
\Pi_{w} = -\dfrac{1}{2i\pi} \int_\gamma (A(w)-z)^{-1} \dd z \in \mc{L}(E).
\end{equation}
After shrinking the $w$-neighborhood, $\gamma\subset\rho(A(w))$ for every $w$, so the contour integral defines a holomorphic projection $\Pi_w$. Finite meromorphy implies that $N := \dim(\ran \Pi_{w_0})<\infty$. At $w=w_0$, $\Pi_{w_0}$ is the spectral projector onto the generalized eigenspace $\ker (A(w_0)-z_0)^N$. Notice that $\dim(\ran \Pi_w) = \Tr(\Pi_w) = N$ is constant.

In the following statement, for $n \geq 1$, we let $\mathrm{Sym}(\C^n) := \C^n/\mathfrak{S}_n$, where $\mathfrak{S}_n$ is the permutation group of $n$ elements, and the action of $\sigma \in \mathfrak{S}_n$ on $z=(z_1,...,z_n) \in \C^n$ is given by $\sigma . z := (z_{\sigma(1)}, ..., z_{\sigma(n)})$. Note that an element $[z] \in \mathrm{Sym}(\C^n)$ is simply an unordered set $\{z_1, ..., z_n\}$ with $z_i \in \C$.

\begin{theorem}
\label{theorem:fredholm-several2}
Let $\C^n \ni (z,w) \mapsto A(w) -z \in \mc{L}(E)$ be a holomorphic family of Fredholm operators such that for each $w \in \C^{n-1}$, there exists $z_0 \in \C$ such that $A(w)-z_0$ is invertible. Let $\mathbf s_0 = (z_0,w_0)\in S$. Then there exist open neighborhoods $V \subset \C^n$ of $\mathbf{s}_0$, $W \subset \C^{n-1}$ of $w_0$, and a continuous map
\[
W \to \mathrm{Sym}(\C^N), \qquad w \mapsto \{z_1(w), ..., z_N(w)\}
\]
such that $z_i(w_0)=z_0$ for all $1 \leq i \leq N$ and the following statements hold:

\begin{enumerate}[label=\emph{(\roman*)}]
\item The function $f : V \to \C$ defined by $f(z,w) := (z-z_1(w)) ... (z-z_N(w))$ is holomorphic and $S \cap V = \{f = 0\}$. In addition, for all $w$ close enough to $w_0$, the poles of $z \mapsto (A(w)-z)^{-1}$ are given by $\{z_1(w), ..., z_N(w)\}$ (counted with algebraic multiplicity).
\item There exist holomorphic families $V \ni (z,w) \mapsto H(z,w), R(z,w) \in \mc{L}(E)$, with $\operatorname{rank}R(z,w)\leq N$ such that
\[
(A(w)-z)^{-1} = H(z,w) + \dfrac{R(z,w)}{(z-z_1(w))...(z-z_N(w))}.
\]
On $V\setminus S$, these families are given by the following expressions, whose right-hand sides extend holomorphically to $V$:
\[
\begin{split}
R(z,w)&= \Pi_w (z-z_1(w))...(z-z_N(w))(A(w)-z)^{-1}\Pi_w, \\
H(z,w) & = (\mathbf{1}-\Pi_w) (A(w)-z)^{-1}(\mathbf{1}-\Pi_w).
\end{split}
\]
\end{enumerate}
\end{theorem}

The theorem can be immediately extended to any open subset $U \subset \C^n$ such that for every $w\in \C^{n-1}$, $U_w:=U \cap (\C \times \{w\})$ is connected and nonempty. It can also be shown that $R(z_0,w_0)=-(A(w_0)-z_0)^{N-1}\Pi_{w_0}$. In the particular case where $N=1$, $w \mapsto z(w) := z_1(w)$ is holomorphic; the expression can then be simplified (by Taylor expanding $R$ with respect to $z$) to
\begin{equation}
\label{equation:cas-cool}
(A(w)-z)^{-1} = \widetilde{H}(z,w) - \dfrac{\Pi_w}{z-z(w)},
\end{equation}
for some holomorphic operator-valued map $(z,w) \mapsto \widetilde{H}(z,w)$. More generally, near a point $(z_0,w_0)$ where the spectrum is analytic and can be locally described as a graph $\{(z(w),w) ~:~ w \text{ near } w_0\}$ for some holomorphic function $w \mapsto z(w)$, we find that
\begin{equation}
\label{equation:cas-cool-2}
(A(w)-z)^{-1} = \widetilde{H}(z,w) - \sum_{k=1}^N \dfrac{(A(w)-z(w))^{k-1}\Pi_w}{(z-z(w))^k}.
\end{equation}
This is a holomorphically parameterized version of the standard analytic Fredholm theorem; see, for example, \cite[Exercise 21.1.25]{Lefeuvre-book}.

\begin{proof}
Write
\begin{equation}
\label{equation:decomp-holo}
A(w)-z = (\mathbf{1}-\Pi_w)(A(w)-z)(\mathbf{1}-\Pi_w) + \Pi_w (A(w)-z) \Pi_w.
\end{equation}
After shrinking $V$, the operator $(\mathbf{1}-\Pi_w)(A(w)-z)(\mathbf{1}-\Pi_w)$ is invertible on $\ran(\mathbf{1}-\Pi_w)$ with inverse given by $H(z,w)$.

Define $f(z,w) := \det(\Pi_w(A(w)-z)\Pi_w)$; this is a holomorphic function for $(z,w)$ close enough to $(z_0,w_0)$ as both $(z,w) \mapsto \Pi_w, A(w)-z$ are holomorphic. For $w=w_0$, $z \mapsto f(z,w)$ vanishes to order $N$ at $z=z_0$. By the Weierstrass preparation theorem (see \cite[Corollary 6.1.2]{Hormander-90} for instance), we can find continuous functions $w \mapsto z_i(w)$ for $1 \leq i \leq N$ and $w$ close to $w_0$ (satisfying $z_i(w_0)=z_0$), and a nonvanishing holomorphic function $h$ such that $f(z,w) = h(z,w)(z-z_1(w))...(z-z_N(w))$. Notice that $(z,w) \mapsto (z-z_1(w))...(z-z_N(w))$ is holomorphic in both variables. The subset
\[
S := \{(z,w) ~:~ A(w)-z : E \to E \text{ is not invertible} \}
\]
is thus given for $(z,w)$ close to $(z_0,w_0)$ by $\{f=0\} = \{(z_i(w), w), 1 \leq i \leq N\}$.

The finite-dimensional spaces $F_w := \ran(\Pi_w)$ may all be identified with $F_{w_0}$ because $\Pi_{w_0} : F_{w} \to F_{w_0} \simeq \C^N$ is an isomorphism for $w$ close to $w_0$. Under this identification, we have
\[
\Pi_w (A(w)-z) \Pi_w = \Pi_{w_0}^{-1} (\underbrace{\Pi_{w_0} \Pi_w (A(w)-z) \Pi_w \Pi_{w_0}^{-1}}_{=: Q(w)-z}) \Pi_{w_0},
\]
where $Q(w) \in \mc{L}(F_{w_0}) \simeq \mc{L}(\C^N)$ depends holomorphically on $w$. Observe that
\[
(Q(w)-z)^{-1} = \dfrac{1}{\det(Q(w)-z)} \operatorname{adj}(Q(w)-z) = \dfrac{1}{f(z,w)}\operatorname{adj}(Q(w)-z),
\]
where $\operatorname{adj}$ denotes the adjugate matrix. Notice that $(z,w) \mapsto \operatorname{adj}(Q(w)-z) \in \mc{L}(\C^N)$ depends holomorphically on $(z,w)$. Going back to \eqref{equation:decomp-holo}, we find
\begin{equation}
\label{equation:lala}
(A(w)-z)^{-1} =  H(z,w) + \dfrac{1}{f(z,w)} \Pi_{w_0}^{-1} \operatorname{adj}(Q(w)-z) \Pi_{w_0}.
\end{equation}
Notice that
\[
\begin{split}
R(z,w) & := \Pi_w (z-z_1(w))...(z-z_N(w))(A(w)-z)^{-1}\Pi_w \\
& = \Pi_{w_0}^{-1}(z-z_1(w))...(z-z_N(w))(Q(w)-z)^{-1} \Pi_{w_0} \\
& =  \dfrac{1}{h(z,w)} \Pi_{w_0}^{-1}\operatorname{adj}(Q(w)-z) \Pi_{w_0}.
\end{split}
\]
This expression is holomorphic for $(z,w)$ near $(z_0,w_0)$ because $h(z,w) \neq 0$. Returning to \eqref{equation:lala} yields
\[
(A(w)-z)^{-1} =  H(z,w) + \dfrac{R(z,w)}{(z-z_1(w))...(z-z_N(w))},
\]
which proves the claim.

\end{proof}

\section{A complex extension lemma}

In what follows, let $\mathbf{s} = (z,w) \in \C \times \C^{n-1}$ and let $\C_w := \{(z,w) ~:~z \in \C\}$. We will need the following result.

\begin{lemma}
\label{lemma:extension}
Let $U \subset \C^n$ be an open subset such that for every $w\in \C^{n-1}$, $U_w:=U \cap (\C \times \{w\})$ is connected and nonempty. Let $h : U \to \C$ be a nonvanishing holomorphic function. Suppose that 
\[
(z,w) \mapsto h_1(z,w) := \partial_z h(z,w)/h(z,w)
\]
admits a meromorphic extension to $\C^n$. Assume further that for all $w \in \C^{n-1}$, $\C \ni z \mapsto h_1(z, w)$ is a meromorphic function with simple poles and positive integer residues. Then $h$ admits a holomorphic extension to $\C^n$.
\end{lemma}

\begin{proof}
Since each $U_w$ is nonempty, choose $z_0(w)\in U_w$ for each $w\in\C^{n-1}$, and set
\[
\tilde{h}(z,w):= h(z_0(w),w) \exp\left( \int_{z_0(w)}^z h_1(z',w)dz' \right). 
\] Here the integral is taken along a path from $z_0(w)$ to $z$ in the complement of the poles of $h_1(\bullet,w)$.
Whenever $(z,w)$ is in the non-singular set of $h_1$, this is a well-defined complex number. Indeed, the choice of integration path in $\C_w$ is immaterial because $h_1(\bullet, w)$ has integer residues: taking a different path multiplies the result by $\exp(2\pi i m)=1$ for some $m\in\mathbb Z$. Now consider $(z,w)$ in the singular set of $h_1$. Then near $z$,
\[
h_1(z',w) = \frac{k}{z'-z} + \mathcal{H}(z'),
\]
where $k$ is a positive integer and $\mathcal H$ is holomorphic. We deduce that $\tilde{h}(z',w)= (z'-z)^k C + (z'-z)^{k+1}\tilde{ \mathcal{H}}(z')$ with $\tilde{\mathcal{H}}$ holomorphic, and we can thus extend $\tilde{h}$ to a well-defined function on $\C^n$ that is holomorphic in the $z$-variable.

Let us check that $\tilde{h}$ is actually holomorphic in all variables and extends $h$. For this, we claim that $\tilde{h}$ does not depend on the choice of $z_0(w)$. Indeed, let $\widehat z_0(w)$ be another choice of point in $U_w$ for each $w$. Since $U_w$ is connected, there is a path in $U_w$ joining $z_0(w)$ to $\widehat z_0(w)$. This path lies in the domain of $h$, and integration along it therefore gives
\[
\tilde{h}(\widehat z_0(w),w) = h(\widehat z_0(w),w). 
\]
This shows both that $\tilde{h}$ does not depend on the choice of $z_0(w)$ and that $\tilde{h}\equiv h$ on $U$.

To prove holomorphy away from the polar set, fix $(z,w)$ in its complement and choose $z_0\in U_w$. By openness, $z_0\in U_{w'}$ for all $w'$ near $w$. Choose a compact path from $z_0$ to $z$ that avoids the poles in the fiber over $w$; after shrinking the parameter and endpoint neighborhoods, this path, followed by a short endpoint segment, avoids the polar set for all nearby $(z',w')$. Differentiating under the integral then shows that $\tilde{h}$ is $C^1$ in $(z,w)$ and satisfies the Cauchy--Riemann equations near any point $(z,w)$ not in the singular set of $h_1$. Thus, $\tilde{h}$ is holomorphic on the complement of an analytic set (the singular set of $h_1$). To conclude that it is globally holomorphic, it remains to prove that $\tilde{h}$ is locally bounded near each point of the singular set of $h_1$; the removable singularity theorem in several variables then applies.

Fix $\s_\ast=(z_\ast,w_\ast)$ in the polar set of $h_1$. After shrinking to a polydisc $P=\Delta\times V$ centered at $\s_\ast$, the reduced polar set of $h_1$ is defined by a monic Weierstrass polynomial
\[
q(z,w)=\prod_{a=1}^N g_a(z,w),
\]
where the $g_a$'s are pairwise distinct irreducible Weierstrass polynomials. Denote by $D_a=\{g_a=0\}$ the corresponding irreducible components. We can thus write $h_1(z,w)=d(z,w)/q(z,w)$ where $d$ is holomorphic and does not vanish identically on any of the $D_a$'s. We let $D_a^{\mathrm{reg}}$ denote the open dense subset of smooth points such that the projection $D_a \to V$ is a smooth submersion, and define
\[
D_a^\circ := D_a^{\mathrm{reg}} \setminus \left(\bigcup_{b \neq a} D_b \cup \{\partial_z g_a=0\}\right).
\]
This is a dense open subset of $D_a$ which is also connected as the complement of a (proper) analytic subset of a complex analytic variety does not disconnect. On $D_a^\circ$, we have
\[
h_1 = \dfrac{f_a}{g_a} + H_a,
\]
where $f_a$ and $H_a$ are holomorphic. The residue of $z \mapsto h_1(z,w)$ at the corresponding fiberwise pole $D_a \cap \C_w$ is then
\[
\rho_a := f_a/\partial_z g_a|_{D_a^\circ} = (g_ah_1)/\partial_z g_a|_{D_a^\circ}.
\]
By hypothesis, $\rho_a$ takes values in $\Z_{> 0}$ and $\rho_a$ is holomorphic as $\partial_z g_a \neq 0$ on $D_a^{\circ}$. Consequently, by connectedness of $D_a^\circ$, we deduce that $\rho_a = k_a$ is constant on $D_a^{\circ}$.

We deduce that $f_a - k_a \partial_z g_a$ is holomorphic and vanishes on $D_a^{\circ}$. It thus vanishes on $D_a$ and we may write $f_a-k_a \partial_z g_a = g_a p_a$ on an open neighborhood of $\s_\ast$ for some holomorphic function $p_a$. This yields:
\[
h_1 = f_a/g_a + H_a = k_a \partial_z g_a/g_a + p_a + H_a
\]
Repeating this process for every $a$, we obtain that
\[
h_1 = \partial_z Q / Q + H,
\]
where $Q := \prod_{a=1}^N g_a^{k_a}$ and $H$ are both holomorphic.

Choose a holomorphic function $G$ on $P$ such that
\[
\partial_zG=H.
\]
For instance, one may integrate $H$ in the $z$-variable from a fixed point of $\Delta$. On $P\setminus\{Q=0\}$, we then have
\[
\partial_z\log(\tilde h/(Qe^G)) = \dfrac{\partial_z(\tilde h/(Qe^G))}{\tilde h/(Qe^G)} =h_1-\frac{\partial_zQ}{Q}-\partial_zG=0.
\]
Hence $\partial_z(\tilde h/(Qe^G)) = 0$ away from $\{Q=0\}$, so
\begin{equation}
\label{equation:jarze}
\tilde{h}(z,w) = c(w) Q(z,w) e^{G(z,w)},
\end{equation}
for some holomorphic function $w \mapsto c(w)$, where the equality holds on $\{Q\neq 0\}$. However, letting $(z,w)$ converge to $\{Q=0\}$, we deduce from \eqref{equation:jarze} that $\tilde{h}(z,w)$ remains uniformly bounded. Hence, $\tilde{h}$ is locally bounded near the singular set, and is therefore holomorphic. This concludes the proof.

\end{proof}

\section{Diophantine subsets and representations}

\label{appendix:diophantine}

\subsection{Diophantine subsets in Euclidean space} 

Recall from Definition \ref{definition:diophantine} that a subset $A \subset \R^k$ is \emph{Diophantine} if there exist $C,\nu>0$ such that, for every $\xi\in\R^k$ with $|\xi|\geq C$, there exists $\lambda \in A$ such that
\begin{equation}
\label{equation:diophantine-appendix}
|e^{i\xi\cdot\lambda}-1|\geq |\xi|^{-\nu}.
\end{equation}
More generally, given a fixed exponent $\nu > 0$, we say that $A \subset \R^k$ is $\nu$-Diophantine if there exists $C > 0$ such that \eqref{equation:diophantine-appendix} holds for all $|\xi| \geq C$.

\begin{proposition}
\label{proposition:null-set}
For all $\nu > k$, the set of $\nu$-Diophantine $(k+1)$-tuples has full measure in $(\R^k)^{k+1}$.

In other words,
\[
\{ \boldsymbol{\lambda}=(\lambda_0, ..., \lambda_k) \in (\R^k)^{k+1} ~:~ \exists C > 0, \forall |\xi| \geq C, \exists \ell \in \{0,...,k\}, |e^{i\xi\cdot\lambda_\ell}-1| \geq |\xi|^{-\nu}\}
\]
has full measure in $(\R^k)^{k+1}$.
\end{proposition}
More generally, one could formulate the proposition for $d$-tuples rather than $(k+1)$-tuples. We emphasize that the proof fails for $d \leq k$.

\begin{proof}
It suffices to show that
\[
\{ \boldsymbol{\lambda} \in(\R^k)^{k+1} ~:~  \exists |\xi_n| \to +\infty, \forall \ell = 0,\ldots, k, |e^{i\lambda_\ell \cdot\xi_n}-1| \leq |\xi_n|^{-\nu}\}
\]
has zero Lebesgue measure in $(\R^k)^{k+1}$. In turn, it suffices to show that for all $R \geq 1$,
\begin{equation}
\label{equation:la}
 Z= \{ \boldsymbol{\lambda} \in(R\mathbf{D})^{k+1} ~:~  \exists |\xi_n| \to +\infty, \forall \ell = 0,\ldots, k, |e^{i\lambda_\ell \cdot\xi_n}-1| \leq |\xi_n|^{-\nu}\}
\end{equation}
has zero measure, where $\mathbf{D} := (-2\pi,2\pi)^k \subset \R^k$.

For $j\geq1$, let $\mathbf{A}_j := \{j \leq |\xi| \leq j+1\} \subset \R^k$ denote the annulus and define
\[
Z_j := \{ \boldsymbol{\lambda} \in(R\mathbf{D})^{k+1} ~:~ \exists \xi \in \mathbf{A}_j, \forall \ell = 0,\ldots, k, |e^{i\xi\cdot\lambda_\ell}-1| \leq |\xi|^{-\nu}\}.
\]
By definition, $Z\subset \limsup_{j\to\infty} Z_j$, where the limit superior is taken in the sense of the Borel--Cantelli lemma. By this lemma, it suffices to establish that $\sum_{j \geq 1} |Z_j| < \infty$ to prove that $|Z| = 0$.

Now fix $\xi \in \R^k$, $L > 0$, and define
\[
Z(\xi,L) := \{ \boldsymbol{\lambda} \in (R \mathbf{D})^{k+1}~:~ \forall \ell = 0,\ldots, k, |e^{i\xi\cdot\lambda_\ell}-1| \leq L |\xi|^{-\nu}\}.
\]
A direct computation shows that there exists $C > 0$ such that
\begin{equation}
\label{equation:zjixieps}
|Z(\xi,L)| \leq C R^{k(k+1)}L^{k+1} |\xi|^{-\nu(k+1)},
\end{equation}
for all $|\xi|\geq1$ and $L>0$. Indeed, this can be first established for $\xi=|\xi|\mathbf{e}_m, 1 \leq m \leq k$, and by covering $\R^k$ with cones of the form $|\xi_m| \geq C |\xi|$.

We now cover $\mathbf{A}_j$ by $M_j$ balls of radius $j^{-\nu}$, centered at $\xi_1, ..., \xi_{M_j} \in \mathbf{A}_j$. This can be achieved with
\begin{equation}
\label{equation:mj}
M_j \leq C j^{k-1}/j^{-\nu k} = C j^{k(\nu+1)-1},
\end{equation}
for some uniform constant $C > 0$. We now claim that
\[
Z_j \subset \bigcup_{s=1}^{M_j} Z(\xi_s, L),
\]
for some $L=L(R,k,\nu)>0$ independent of $j$. Indeed, suppose that $\boldsymbol{\lambda}$ lies in $Z_j$; that is, there exists $\xi\in\mathbf{A}_j$ such that, for every $\ell$, $|e^{i\xi\cdot\lambda_\ell}-1| \leq |\xi|^{-\nu}$. One has $\xi \in B(\xi_s,j^{-\nu})$ for some $s \in \{1,...,M_j\}$. Hence, writing $\eta := \xi_s-\xi$, we find:
\[
|e^{i \lambda_\ell \cdot\xi_s}-1| = |e^{i\lambda_\ell \cdot\xi} e^{i\lambda_\ell cdot\eta}-1| = |e^{i \lambda_\ell \cdot\xi}(1+\mc{O}(j^{-\nu})) -1| \leq |e^{i\lambda_\ell \cdot\xi}-1| + \mc{O}(j^{-\nu}) \leq L |\xi_s|^{-\nu},
\]
for some $L=L(R,k,\nu)>0$ independent of $j$. Combining \eqref{equation:zjixieps} and \eqref{equation:mj}, we thus find that for some $C_R>0$,
\[
|Z_j| \leq C_R j^{-\nu(k+1)} j^{k(\nu+1)-1} = C_R j^{k-\nu-1}.
\]
(Note that if the length of the tuple $\boldsymbol{\lambda}$ were smaller than the dimension $k$, then $\nu$ would appear with a positive sign and the proof would break down). By assumption, $\nu>k$, so $k-\nu-1<-1$ and the series $\sum_{j\geq1}|Z_j|$ converges. This concludes the proof.
\end{proof}

\subsection{Diophantine representations}

\label{ssection:diophantine-representations}

Given a word-hyperbolic group $\Gamma$ and a semisimple Lie group $G$ of noncompact type, we show that, under an appropriate condition on $\Gamma$ and $G$, \emph{most} representations $\rho : \Gamma \to G$ are Diophantine in the sense that their Lyapunov spectrum satisfies the Diophantine condition (see Definition \ref{definition:diophantine}). Throughout this subsection, we assume that $\Gamma$ is finitely generated, with $s\geq1$ (nontrivial) generators of infinite order, and finitely presented, with $r\geq0$ relations. That is,
\[
\Gamma = \langle \gamma_1, ..., \gamma_s  ~:~ r_i(\gamma_1, ..., \gamma_s) = 1, 1 \leq i \leq r \rangle,
\]
where each $r_i$ is a relation involving products and inverses. Consequently, $\mathrm{Hom}(\Gamma,G)$ can be identified with
\[
\mathrm{Hom}(\Gamma,G) \simeq \{ (g_1, ..., g_s) \in G^s ~:~ R_i(g_1, ..., g_s)=1, 1 \leq i \leq r\},
\]
where $R_i : G^s \to G$ are nonconstant real-analytic maps corresponding to the relations $r_i$, $1 \leq i \leq r$. The space $\mathrm{Hom}(\Gamma,G)$ is a real-analytic subvariety of $G^s$; it is therefore a stratified space with an open dense maximal stratum $S_0 \subset \mathrm{Hom}(\Gamma,G)$ which is an analytic submanifold of $G^s$ (see \cite[2.10-2.11]{Bierstone-Milman-88}). By definition of the maximal stratum, on $S_0$, the rank of $G^{s} \to G^r, (g_1,...,g_s) \mapsto (R_1(g_1,...,g_s), ..., R_r(g_1,...,g_s))$ is maximal, equal to $0\leq d \leq r \cdot \dim G$. For simplicity, we will further assume that $d = r \cdot \dim G$, which is the case in most examples, and therefore $\dim S_0 =  (s-r)\dim(G)$. The general case can be handled by a slight adaptation of the argument. As $S_0$ is an analytic submanifold, there is a well-defined Lebesgue class; it therefore makes sense to speak of Lebesgue-almost every point of $S_0$.

\begin{proposition}[Lebesgue almost every representation is Diophantine]
\label{proposition:almost-diophantine}
Assume that $\dim S_0 = (s-r) \dim(G)$ and that $s \geq (r\dim(G)+1)(\mathrm{rank}(G)+1)$. Then Lebesgue almost every Borel Anosov representation $\rho : \Gamma \to G$ in $S_0$ is Diophantine.
\end{proposition}

The statement is vacuous if there are no Borel Anosov representations. A particular case where the proposition applies is $\Gamma = \pi_1(\Sigma_g)$ and $G = \mathrm{SL}_n(\R)$, provided $2g \geq n^3$. In this case $\Gamma = \langle a_1, b_1, ..., a_g,b_g  ~:~ \prod_{i=1}^g [a_i,b_i]=1\rangle$, so $s = 2g, r = 1$, $\dim(\mathrm{SL}_n(\R))=n^2-1$ and $\mathrm{rank}(\mathrm{SL}_n(\R))+1=n$. Note that the open subset $\mc{M}$ of $\mathrm{Hom}(\pi_1(\Sigma_g),\mathrm{SL}_n(\R))$ corresponding to the Hitchin component (the connected component of Fuchsian representations) is known to be an analytic manifold of dimension $(2g-1)(n^2-1)$, see \cite{Hitchin-92} (we do not mod out by the action of inner automorphisms of $\mathrm{SL}_n(\R)$ here). The map $R : G^{s} \to G, R(a_1,b_1,\ldots,a_g,b_g) = \prod_{i=1}^g [a_i,b_i]$ is analytic, and has maximal rank equal to $\dim G$ on an open subset of $\mc{M}$ which corresponds to the maximal stratum $S_0$. Indeed, if the rank were strictly smaller, then $\mc{M}$ would be a higher-dimensional analytic submanifold near the regular points of $R$; this would then contradict the fact that it is analytic and of dimension $(2g-1)(n^2-1)$. Consequently, the condition $\dim S_0 = (s-r)\dim(G)$ is satisfied in this case.

\begin{proof}
Let $n:=\mathrm{rank}(G)+1$, and denote by $\mathbf g:=(g_1,\ldots,g_s)\in S_0$ a generic point. Consider the maps $\phi_k : S_0 \to G^n$ defined for $0 \leq k \leq r\dim(G)$ by $\phi_k(\mathbf{g}) = (g_{kn+1},g_{kn+2},...,g_{(k+1)n})$. We claim that there exists an index $k \in \{0, ..., r\dim(G)\}$ such that the differential of $\phi_k$ is surjective on an open subset of full measure of $S_0$. This will imply the claimed Diophantine property.

For $0 \leq k \leq r\dim(G)$, $\phi_k : S_0 \to G^n$ is an analytic map; therefore, its differential has maximal rank $d_k \leq n\dim(G)$ on an open subset of full measure of $S_0$. Assume for a contradiction that $d_k < n\dim(G)$ for all $0 \leq k \leq r\dim(G)$. Let $\mathbf{g}_k$ be a regular point of $\phi_k$, that is, a point where the differential has maximal rank. Then there exists an open subset $U_k \subset S_0$ around $\mathbf{g}_0$ such that $d\phi_k$ has maximal rank and $V_k := \phi_k(U_k)$ is an analytic submanifold of dimension strictly less than $n\dim(G)$. We may also further assume that $U := \cap_{0\leq k \leq r\dim(G)}U_k \neq \emptyset$ is nonempty by choosing the $U_k$'s as nested subsets. We then set
\[
X_k := \underbrace{G^n \times ... \times G^n}_{k \text{ times}} \times V_k \times \underbrace{G^n \times ... \times G^n}_{r\dim(G)-k \text{ times}} \times G^{s-(r\dim(G)+1)n} \supset S_0 \cap U_k
\]
and consider the nested intersection over $k$,  $X := \cap_{0\leq k \leq r\dim(G)} X_k$. By construction, $S_0 \cap U \subset X$. However, the manifolds $X_k$ are all transverse and of dimension $< s \dim(G)$. Therefore, $\dim X \leq s\dim(G) - (r\dim(G)+1) = (s-r)\dim(G)-1$, which contradicts $\dim S_0 = (s-r)\dim(G)$. Hence, one of the maps $\phi_k$ has a surjective differential on an open subset of full measure of $S_0$. Up to reordering the generating set, we can assume that this property holds for $k=0$.

Now consider the map
\begin{equation}
\label{equation:psi-beau}
\Psi_{\mathrm J} : S_0 \to \fa^{n}, \qquad \Psi_{\mathrm J}(\mathbf{g}) := (\lambda_\fapp(g_1),\ldots,\lambda_\fapp(g_n)).
\end{equation}
It is analytic as the representation is assumed to be Borel Anosov and all the elements $g_i$ have distinct eigenvalues (the Jordan projection is analytic on elements with distinct eigenvalues). Since Diophantine subsets of $\fa^n \simeq (\R^{n-1})^n$ have full measure (see Proposition \ref{proposition:null-set}), it suffices to show that the differential of $\Psi_{\mathrm J}$ is surjective on an open subset of full measure of $S_0$. However, this is immediate as $\phi_0$ satisfies this property (by construction) and the Jordan projection $\lambda_{\fapp}$ does as well on the Borel-Anosov locus. This concludes the proof.

\end{proof}

\bibliographystyle{alpha}
\bibliography{biblio}

\begin{thebibliography}{PSMdSS12}

\bibitem[Alb99]{Albuquerque-99}
Paul Albuquerque.
\newblock {Patterson--Sullivan} theory in higher rank symmetric spaces.
\newblock {\em Geom. Funct. Anal.}, 9(1):1--28, 1999.

\bibitem[AM25]{Anantharaman-Monk-25}
Nalini Anantharaman and Laura Monk.
\newblock Friedman--{Ramanujan} functions in random hyperbolic geometry and
  application to spectral gaps {II}, 2025.

\bibitem[Ana00]{Anantharaman-00}
Nalini Anantharaman.
\newblock Precise counting results for closed orbits of {Anosov} flows.
\newblock {\em Ann. Sci. {\'E}c. Norm. Sup{\'e}r. (4)}, 33(1):33--56, 2000.

\bibitem[Ben97]{Benoist-97}
Y.~Benoist.
\newblock Propriet\'es asymptotiques des groups lin\'eaires.
\newblock {\em Geom. Funct. Anal.}, 7(1):1--47, 1997.

\bibitem[Ben00]{Benoist-00}
Yves Benoist.
\newblock Asymptotic properties of linear groups. {II}.
\newblock In {\em Analysis on homogeneous spaces and representation theory of
  Lie groups. Based on activities of the RIMS Project Research '97,
  Okayama-Kyoto, Japan, during July and August 1997}, pages 33--48. Tokyo:
  Kinokuniya Company Ltd., 2000.

\bibitem[BGHW25]{Guedes-Bonthonneau-Guillarmou-Hilgert-Weich-20}
Yannick~Guedes Bonthonneau, Colin Guillarmou, Joachim Hilgert, and Tobias
  Weich.
\newblock Ruelle-taylor resonances of {Anosov} actions.
\newblock {\em J. Eur. Math. Soc. (JEMS)}, 27(8):3085--3147, 2025.

\bibitem[BL98]{Babillot-Ledrappier-98}
Martine Babillot and Fran{\c{c}}ois Ledrappier.
\newblock Lalley's theorem on periodic orbits of hyperbolic flows.
\newblock {\em Ergodic Theory Dyn. Syst.}, 18(1):17--39, 1998.

\bibitem[BLW26]{arxiv_v1}
Yannick~Guedes Bonthonneau, Thibault Lefeuvre, and Tobias Weich.
\newblock The spectrum of anosov representations.
\newblock {\em arXiv preprint arXiv:2603.24519 (v1)}, 2026.

\bibitem[BM88]{Bierstone-Milman-88}
Edward Bierstone and Pierre~D. Milman.
\newblock Semianalytic and subanalytic sets.
\newblock {\em Publ. Math., Inst. Hautes {\'E}tud. Sci.}, 67:5--42, 1988.

\bibitem[Bor07]{Borthwick-07}
David Borthwick.
\newblock {\em Spectral theory of infinite-area hyperbolic surfaces}, volume
  256 of {\em Prog. Math.}
\newblock Basel: Birkh{\"a}user, 2007.

\bibitem[Bow74]{Bowen-74}
Rufus Bowen.
\newblock Maximizing entropy for a hyperbolic flow.
\newblock {\em Math. Systems Theory}, 7:300--303, 1974.

\bibitem[BPS19]{Bochi-Potrie-Sambarino-19}
Jairo Bochi, Rafael Potrie, and Andr{\'e}s Sambarino.
\newblock Anosov representations and dominated splittings.
\newblock {\em J. Eur. Math. Soc. (JEMS)}, 21(11):3343--3414, 2019.

\bibitem[BT72]{Borel-Tits-72}
Armand Borel and Jacques Tits.
\newblock Complements {\`a} l'article: ''{Groupes} reductifs''.
\newblock {\em Publ. Math., Inst. Hautes {\'E}tud. Sci.}, 41:253--276, 1972.

\bibitem[BT07]{Baladi-Tsuji-07}
Viviane Baladi and Masato Tsujii.
\newblock Anisotropic {H{\"o}lder} and {Sobolev} spaces for hyperbolic
  diffeomorphisms.
\newblock {\em Ann. Inst. Fourier}, 57(1):127--154, 2007.

\bibitem[Bur90]{Burger-90}
Marc Burger.
\newblock Horocycle flow on geometrically finite surfaces.
\newblock {\em Duke Math. J.}, 61(3):779--803, 1990.

\bibitem[Bur93]{Burger-93}
M.~Burger.
\newblock Intersection, the {Manhattan} curve, and {Patterson--Sullivan} theory
  in rank 2.
\newblock {\em Int. Math. Res. Not.}, 1993(7):217--225, 1993.

\bibitem[BW22]{Bonthonneau-Weich-22}
Yannick~Guedes Bonthonneau and Tobias Weich.
\newblock Ruelle-pollicott resonances for manifolds with hyperbolic cusps.
\newblock {\em J. Eur. Math. Soc. (JEMS)}, 24(3):851--923, 2022.

\bibitem[BWS21]{BornsWeil-Shen-2020}
Yonah Borns-Weil and Shu Shen.
\newblock Dynamical zeta functions in the nonorientable case.
\newblock {\em Nonlinearity}, 34(10):7322--7334, 2021.

\bibitem[CDDP22]{Cekic-Delarue-Dyatlov-Paternain-22}
Mihajlo Ceki{\'c}, Benjamin Delarue, Semyon Dyatlov, and Gabriel~P. Paternain.
\newblock The {Ruelle} zeta function at zero for nearly hyperbolic 3-manifolds.
\newblock {\em Invent. Math.}, 229(1):303--394, 2022.

\bibitem[CE71]{Conley-Easton-71}
C.~Conley and R.~Easton.
\newblock Isolated invariant sets and isolating blocks.
\newblock {\em Trans. Am. Math. Soc.}, 158:35--61, 1971.

\bibitem[CG97]{Choi-Goldman-97}
Suhyoung Choi and William~M. Goldman.
\newblock The classification of real projective structures on compact surfaces.
\newblock {\em Bull. Am. Math. Soc., New Ser.}, 34(2):161--171, 1997.

\bibitem[CGL24]{Cekic-Guillarmou-Lefeuvre-24}
Mihajlo Ceki{\'c}, Colin Guillarmou, and Thibault Lefeuvre.
\newblock Local lens rigidity for manifolds of {Anosov} type.
\newblock {\em Anal. PDE}, 17(8):2737--2795, 2024.

\bibitem[CK26]{Choi-Kim-26}
Inhyeok Choi and Dongryul~M. Kim.
\newblock Classification of horospherical invariant measures in higher rank:
  {The Full Story}, 2026.
\newblock Version 2.

\bibitem[CL24]{Cekic-Lefeuvre-24}
Mihajlo {Ceki{\'c}} and Thibault {Lefeuvre}.
\newblock {Semiclassical analysis on principal bundles}.
\newblock {\em arXiv e-prints}, page arXiv:2405.14846, May 2024.

\bibitem[CLMT26]{Cekic-Lefeuvre-Munoz-26}
Mihajlo Cekić, Thibault Lefeuvre, and Sebastián Muñoz-Thon.
\newblock Decay of correlations on abelian covers of isometric extensions of
  volume-preserving anosov flows, 2026.

\bibitem[CS23]{Chow-Sarkar-23}
Michael Chow and Pratyush Sarkar.
\newblock Local mixing of one-parameter diagonal flows on {Anosov} homogeneous
  spaces.
\newblock {\em Int. Math. Res. Not.}, 2023(18):15834--15895, 2023.

\bibitem[CS24]{Chow-Sarkar-24}
Michael Chow and Pratyush Sarkar.
\newblock Exponential mixing and essential spectral gaps for {Anosov}
  subgroups.
\newblock Preprint, {arXiv}:2408.11274 [math.{DS}] (2024), 2024.

\bibitem[Dan24]{Dang-24}
Nguyen-Thi Dang.
\newblock Topological mixing of positive diagonal flows.
\newblock {\em Isr. J. Math.}, 260(1):1--71, 2024.

\bibitem[DFG15]{Dyatlov-Faure-Guillarmou-15}
Semyon Dyatlov, Fr{\'e}d{\'e}ric Faure, and Colin Guillarmou.
\newblock Power spectrum of the geodesic flow on hyperbolic manifolds.
\newblock {\em Anal. PDE}, 8(4):923--1000, 2015.

\bibitem[DG16]{Dyatlov-Guillarmou-16}
Semyon Dyatlov and Colin Guillarmou.
\newblock Pollicott-ruelle resonances for open systems.
\newblock {\em Ann. Henri Poincar{\'e}}, 17(11):3089--3146, 2016.

\bibitem[DG18]{Dyatlov-Guillarmou-18}
Semyon Dyatlov and Colin Guillarmou.
\newblock Afterword: {Dynamical} zeta functions for axiom {A} flows.
\newblock {\em Bull. Am. Math. Soc., New Ser.}, 55(3):337--342, 2018.

\bibitem[DGM25]{Delarue-Guillarmou-Monclair-2025}
Benjamin Delarue, Colin Guillarmou, and Daniel Monclair.
\newblock Spectra of lorentzian quasi-fuchsian manifolds, 2025.

\bibitem[DGRS20]{Dang-Guillarmou-Riviere-Shen-20}
Nguyen~Viet Dang, Colin Guillarmou, Gabriel Rivi{\`e}re, and Shu Shen.
\newblock The {Fried} conjecture in small dimensions.
\newblock {\em Invent. Math.}, 220(2):525--579, 2020.

\bibitem[DKLZ26]{DKLZ26}
Nguyen-Bac Dang, Michael Kapovich, Mikhail Lyubich, and Shengyuan Zhao.
\newblock Equidistribution of currents under anosov group actions.
\newblock {\em arXiv preprint arXiv:2607.22920}, 2026.

\bibitem[DKV79]{Duistermaat-Kolk-Varadarajan-79}
J.~J. Duistermaat, J.~A.~C. Kolk, and V.~S. Varadarajan.
\newblock Spectra of compact locally symmetric manifolds of negative curvature.
\newblock {\em Invent. Math.}, 52:27--93, 1979.

\bibitem[DMS25a]{Delarue-Monclair-Sanders-24}
Benjamin Delarue, Daniel Monclair, and Andrew Sanders.
\newblock Locally homogeneous {Axiom} {A} flows {I}: projective {Anosov}
  subgroups and exponential mixing.
\newblock {\em Geom. Funct. Anal.}, 35:673--735, 2025.

\bibitem[DMS25b]{Delarue-Monclair-Sanders-25}
Benjamin Delarue, Daniel Monclair, and Andrew Sanders.
\newblock Locally homogeneous {Axiom} {A} flows {II}: geometric structures for
  {Anosov} subgroups.
\newblock Preprint, {arXiv}:2502.20195 [math.{GT}] (2025), 2025.

\bibitem[DO25]{Dey-Oh-25}
Subhadip Dey and Hee Oh.
\newblock Deformations of {Anosov} subgroups: limit cones and growth
  indicators.
\newblock {\em J. Lond. Math. Soc. (2)}, 112(3):e70280, 2025.

\bibitem[Dol98]{Dolgopyat-98-2}
Dmitry Dolgopyat.
\newblock Prevalence of rapid mixing in hyperbolic flows.
\newblock {\em Ergodic Theory Dyn. Syst.}, 18(5):1097--1114, 1998.

\bibitem[DR24]{Dang-Riviere-24}
Nguyen~Viet Dang and Gabriel Rivi{\`e}re.
\newblock Poincar{\'e} series and linking of {Legendrian} knots.
\newblock {\em Duke Math. J.}, 173(1):1--74, 2024.

\bibitem[DZ16]{Dyatlov-Zworski-16}
Semyon Dyatlov and Maciej Zworski.
\newblock Dynamical zeta functions for {Anosov} flows via microlocal analysis.
\newblock {\em Ann. Sci. {\'E}c. Norm. Sup{\'e}r. (4)}, 49(3):543--577, 2016.

\bibitem[DZ17]{Dyatlov-Zworski-17}
Semyon Dyatlov and Maciej Zworski.
\newblock Ruelle zeta function at zero for surfaces.
\newblock {\em Invent. Math.}, 210(1):211--229, 2017.

\bibitem[ELMV11]{Einsiedler-Lindenstrauss-Michel-Venkatesh-11}
Manfred Einsiedler, Elon Lindenstrauss, Philippe Michel, and Akshay Venkatesh.
\newblock Distribution of periodic torus orbits and {Duke}'s theorem for cubic
  fields.
\newblock {\em Ann. of Math. (2)}, 173(2):815--885, 2011.

\bibitem[ELO23]{ELO23}
Samuel Edwards, Minju Lee, and Hee Oh.
\newblock Anosov groups: local mixing, counting and equidistribution.
\newblock {\em Geometry \& Topology}, 27(2):513--573, 2023.

\bibitem[FH19]{Fisher-Hasselblatt-19}
Todd Fisher and Boris Hasselblatt.
\newblock {\em Hyperbolic flows}.
\newblock Zur. Lect. Adv. Math. Berlin: European Mathematical Society (EMS),
  2019.

\bibitem[FMT07]{Field-Melbourne-Torok-07}
Michael Field, Ian Melbourne, and Andrei T{\"o}r{\"o}k.
\newblock Stability of mixing and rapid mixing for hyperbolic flows.
\newblock {\em Ann. Math. (2)}, 166(1):269--291, 2007.

\bibitem[FRS08]{Faure-Roy-Sjostrand-08}
Fr{\'e}d{\'e}ric Faure, Nicolas Roy, and Johannes Sj{\"o}strand.
\newblock Semi-classical approach for {Anosov} diffeomorphisms and {Ruelle}
  resonances.
\newblock {\em Open Math. J.}, 1:35--81, 2008.

\bibitem[FS11]{Faure-Sjostrand-11}
Fr{\'e}d{\'e}ric Faure and Johannes Sj{\"o}strand.
\newblock Upper bound on the density of {Ruelle} resonances for {Anosov} flows.
\newblock {\em Commun. Math. Phys.}, 308(2):325--364, 2011.

\bibitem[FT13]{Faure-Tsuji-13}
Fr{\'e}d{\'e}ric Faure and Masato Tsujii.
\newblock Band structure of the {Ruelle} spectrum of contact {Anosov} flows.
\newblock {\em C. R., Math., Acad. Sci. Paris}, 351(9-10):385--391, 2013.

\bibitem[Gan68]{Gangolli-68}
Ramesh Gangolli.
\newblock Asymptotic behaviour of spectra of compact quotients of certain
  symmetric spaces.
\newblock {\em Acta Math.}, 121:151--192, 1968.

\bibitem[GBGW24]{Guedes-Bonthonneau-Guillarmou-Weich-24}
Yannick Guedes~Bonthonneau, Colin Guillarmou, and Tobias Weich.
\newblock {SRB} measures for {Anosov} actions.
\newblock {\em J. Differ. Geom.}, 128(3):959--1026, 2024.

\bibitem[Gel63]{Gelfand-62}
I.~M. Gel'fand.
\newblock Automorphic functions and the theory of representations.
\newblock In {\em Proceedings of the International Congress of Mathematicians,
  Stockholm 1962}, pages 74--85, Djursholm, 1963. Institut Mittag-Leffler.

\bibitem[GGKW17]{Gueritaud-Guichard-Kassel-Wienhard-17}
Fran{\c{c}}ois Gu{\'e}ritaud, Olivier Guichard, Fanny Kassel, and Anna
  Wienhard.
\newblock Anosov representations and proper actions.
\newblock {\em Geom. Topol.}, 21(1):485--584, 2017.

\bibitem[GHW18]{Guillarmou-Hilgert-Weich-18}
Colin Guillarmou, Joachim Hilgert, and Tobias Weich.
\newblock Classical and quantum resonances for hyperbolic surfaces.
\newblock {\em Math. Ann.}, 370(3-4):1231--1275, 2018.

\bibitem[GHW21]{Guillarmou-Hilgert-Weich-21}
Colin Guillarmou, Joachim Hilgert, and Tobias Weich.
\newblock High frequency limits for invariant {Ruelle} densities.
\newblock {\em Ann. Henri Lebesgue}, 4:81--119, 2021.

\bibitem[GL06]{Gouezel-Liverani-06}
S{\'e}bastien Gou{\"e}zel and Carlangelo Liverani.
\newblock Banach spaces adapted to {Anosov} systems.
\newblock {\em Ergodic Theory Dyn. Syst.}, 26(1):189--217, 2006.

\bibitem[GLP13a]{Giulietti-Liverani-Pollicott-13}
Paolo Giulietti, Carlangelo Liverani, and Mark Pollicott.
\newblock Anosov flows and dynamical zeta functions.
\newblock {\em Ann. Math. (2)}, 178(2):687--773, 2013.

\bibitem[GLP13b]{GLP-13}
Paolo Giulietti, Carlangelo Liverani, and Mark Pollicott.
\newblock Anosov flows and dynamical zeta functions.
\newblock {\em Ann. Math. (2)}, 178(2):687--773, 2013.

\bibitem[GM12]{Guillarmou-Mazzeo-12}
C.~Guillarmou and R.~Mazzeo.
\newblock Resolvent of the {Laplacian} on geometrically finite hyperbolic
  manifolds.
\newblock {\em Invent. Math.}, 187:99--144, 2012.

\bibitem[Gol88]{Goldman-88}
William~M. Goldman.
\newblock Topological components of spaces of representations.
\newblock {\em Invent. Math.}, 93(3):557--607, 1988.

\bibitem[Gra58]{Grauert-58}
Hans Grauert.
\newblock On {Levi}'s problem and the imbedding of real-analytic manifolds.
\newblock {\em Ann. Math. (2)}, 68:460--472, 1958.

\bibitem[GW12]{Guichard-Wienhard-12}
Olivier Guichard and Anna Wienhard.
\newblock Anosov representations: domains of discontinuity and applications.
\newblock {\em Invent. Math.}, 190(2):357--438, 2012.

\bibitem[GZ97]{Guillope-Zworski-97}
Laurent Guillop{\'e} and Maciej Zworski.
\newblock Scattering asymptotics for {Riemann} surfaces.
\newblock {\em Ann. of Math. (2)}, 145(3):597--660, 1997.

\bibitem[Hit92]{Hitchin-92}
Nigel~J. Hitchin.
\newblock Lie groups and {Teichm{\"u}ller} space.
\newblock {\em Topology}, 31(3):449--473, 1992.

\bibitem[HM23]{Hide-Magee-23}
Will Hide and Michael Magee.
\newblock Near optimal spectral gaps for hyperbolic surfaces.
\newblock {\em Ann. of Math. (2)}, 198(2):791--824, 2023.

\bibitem[HMT25]{Hide-Macera-Thomas-25}
Will Hide, Davide Macera, and Joe Thomas.
\newblock Spectral gap with polynomial rate for {Weil--Petersson} random
  surfaces, 2025.

\bibitem[H{\"o}r90]{Hormander-90}
Lars H{\"o}rmander.
\newblock {\em An introduction to complex analysis in several variables.},
  volume~7 of {\em North-Holland Math. Libr.}
\newblock Amsterdam etc.: North-Holland, 3rd revised ed. edition, 1990.

\bibitem[Hub56]{Huber-56}
H.~Huber.
\newblock {\"U}ber eine neue {Klasse} automorpher {Funktionen} und ein
  {Gitterpunktproblem} in der hyperbolischen {Ebene}. {I}.
\newblock {\em Comment. Math. Helv.}, 30:20--62, 1956.

\bibitem[Hub59]{Huber-59}
H.~Huber.
\newblock Zur analytischen {Theorie} hyperbolischer {Raumformen} und
  {Bewegungsgruppen}.
\newblock {\em Math. Ann.}, 138:1--26, 1959.

\bibitem[Hum25a]{Humbert-24}
Tristan Humbert.
\newblock First {Ruelle} resonance for an {Anosov} flow with smooth potential.
\newblock {\em Ergodic Theory Dyn. Syst.}, 45(8):2439--2475, 2025.

\bibitem[Hum25b]{Humbert-25}
Tristan Humbert.
\newblock Measure of maximal entropy for minimal {Anosov} actions.
\newblock Preprint, {arXiv}:2503.03457 [math.{DS}] (2025), 2025.

\bibitem[HWW23]{Hilgert-Weich-Wolf-23}
Joachim Hilgert, Tobias Weich, and Lasse~L. Wolf.
\newblock Higher rank quantum-classical correspondence.
\newblock {\em Anal. PDE}, 16(10):2241--2265, 2023.

\bibitem[J{\'e}z21]{Jezequel-21}
Malo J{\'e}z{\'e}quel.
\newblock Global trace formula for ultra-differentiable {Anosov} flows.
\newblock {\em Commun. Math. Phys.}, 385(3):1771--1834, 2021.

\bibitem[JT25]{Jin-Tao-25}
Long Jin and Zhongkai Tao.
\newblock Counting {Pollicott}-{Ruelle} resonances for axiom a flows.
\newblock {\em Commun. Math. Phys.}, 406(2):43, 2025.
\newblock Id/No 26.

\bibitem[KB02]{Kapovich-Benakli-02}
Ilya Kapovich and Nadia Benakli.
\newblock Boundaries of hyperbolic groups.
\newblock In {\em Combinatorial and geometric group theory}, volume 296 of {\em
  Contemporary Mathematics}, pages 39--93. American Mathematical Society,
  Providence, RI, 2002.

\bibitem[Kim25]{Kim-25}
Dongryul~M. Kim.
\newblock Conformal measure rigidity and ergodicity of horospherical
  foliations, 2025.
\newblock Version 3.

\bibitem[KLP17]{Kapovich-Leeb-Porti-17}
Michael Kapovich, Bernhard Leeb, and Joan Porti.
\newblock Anosov subgroups: dynamical and geometric characterizations.
\newblock {\em Eur. J. Math.}, 3(4):808--898, 2017.

\bibitem[KLP18]{Kapovich-Leeb-Porti-18}
Michael Kapovich, Bernhard Leeb, and Joan Porti.
\newblock A {Morse} lemma for quasigeodesics in symmetric spaces and
  {Euclidean} buildings.
\newblock {\em Geom. Topol.}, 22(7):3827--3923, 2018.

\bibitem[Kna96]{Knapp-96}
Anthony~W. Knapp.
\newblock {\em Lie groups beyond an introduction}, volume 140 of {\em Prog.
  Math.}
\newblock Boston: Birkh{\"a}user, 1996.

\bibitem[KOP26]{KOP26}
Dongryul Kim, Hee Oh, and Wenyu Pan.
\newblock Higher-rank local mixing for cusped positive representations.
\newblock {\em in preparation}, 2026.

\bibitem[KS86]{Katsuda-Sunada-86}
Atsushi Katsuda and Toshikazu Sunada.
\newblock Homology of closed geodesics in certain {Riemannian} manifolds.
\newblock {\em Proc. Am. Math. Soc.}, 96:657--660, 1986.

\bibitem[KS94]{Katok-Spatzier-94}
Anatole Katok and Ralph~J. Spatzier.
\newblock First cohomology of {Anosov} actions of higher rank abelian groups
  and applications to rigidity.
\newblock {\em Publications Math\'ematiques de l'IH\'ES}, 79:131--156, 1994.

\bibitem[Lab06]{Labourie-06}
Fran{\c{c}}ois Labourie.
\newblock Anosov flows, surface groups and curves in projective space.
\newblock {\em Invent. Math.}, 165(1):51--114, 2006.

\bibitem[Lal87]{Lalley-87}
S.~P. Lalley.
\newblock Distribution of periodic orbits of symbolic and axiom {A} flows.
\newblock {\em Adv. Appl. Math.}, 8:154--193, 1987.

\bibitem[Lal89a]{Lalley-89b}
Steven~P. Lalley.
\newblock Closed geodesics in homology classes on surfaces of variable negative
  curvature.
\newblock {\em Duke Math. J.}, 58(3):795--821, 1989.

\bibitem[Lal89b]{Lalley-89a}
Steven~P. Lalley.
\newblock Renewal theorems in symbolic dynamics, with applications to geodesic
  flows, noneuclidean tessellations and their fractal limits.
\newblock {\em Acta Math.}, 163(1-2):1--55, 1989.

\bibitem[Led95]{Ledrappier-95}
Fran{\c{c}}ois Ledrappier.
\newblock Structure at the border of manifolds with negative curvature.
\newblock In {\em S\'eminaire de th\'eorie spectrale et g\'eom\'etrie. Ann\'ee
  1994-1995}, pages 97--122. St. Martin d'H{\`e}res: Univ. de Grenoble I,
  Institut Fourier, 1995.

\bibitem[Lef25]{Lefeuvre-book}
Thibault Lefeuvre.
\newblock {\em Microlocal analysis in hyperbolic dynamics and geometry. {With}
  a contributed chapter by {Yann} {Chaubet}}, volume~32 of {\em Cours Sp{\'e}c.
  (Paris)}.
\newblock Paris: Soci{\'e}t{\'e} Math{\'e}matique de France (SMF), 2025.

\bibitem[Lin04]{Link-04}
Gabriele Link.
\newblock Hausdorff dimension of limit sets of discrete subgroups of higher
  rank {Lie} groups.
\newblock {\em Geom. Funct. Anal.}, 14(2):400--432, 2004.

\bibitem[Lin06]{Link-06}
Gabriele Link.
\newblock Ergodicity of generalised {Patterson--Sullivan} measures in higher
  rank symmetric spaces.
\newblock {\em Math. Z.}, 254(3):611--625, 2006.

\bibitem[Liv04]{Liverani-04}
Carlangelo Liverani.
\newblock On contact {Anosov} flows.
\newblock {\em Ann. Math. (2)}, 159(3):1275--1312, 2004.

\bibitem[LO23]{Lee-Oh-20}
Minju Lee and Hee Oh.
\newblock {Invariant Measures for Horospherical Actions and Anosov Groups}.
\newblock {\em International Mathematics Research Notices},
  2023(19):16226--16295, 10 2023.

\bibitem[LW24]{Lipnowski-Wright-24}
Michael Lipnowski and Alex Wright.
\newblock Towards optimal spectral gaps in large genus.
\newblock {\em Ann. Probab.}, 52(2):545--575, 2024.

\bibitem[LWW26]{Lutsko-Weich-Wolf-26}
Christopher Lutsko, Tobias Weich, and Lasse~L. Wolf.
\newblock Polyhedral bounds on the joint spectrum and temperedness of locally
  symmetric spaces.
\newblock {\em Duke Math. J.}, 2026.
\newblock To appear.

\bibitem[Mar69]{Margulis-69}
Grigory~A. Margulis.
\newblock Applications of ergodic theory to the investigation of manifolds of
  negative curvature.
\newblock {\em Funct. Anal. Appl.}, 3:335--336, 1969.

\bibitem[MM87]{Mazzeo-Melrose-87}
Rafe~R. Mazzeo and Richard~B. Melrose.
\newblock Meromorphic extension of the resolvent on complete spaces with
  asymptotically constant negative curvature.
\newblock {\em J. Funct. Anal.}, 75(2):260--310, 1987.

\bibitem[MS75]{Melin-Sjostrand-75}
Anders Melin and Johannes Sj{\"o}strand.
\newblock Fourier integral operators with complex-valued phase functions.
\newblock Fourier {Integr}. {Oper}. part. differ. {Equat}., {Colloq}. int.
  {Nice} 1974, {Lect}. {Notes} {Math}. 459, 120-223 (1975)., 1975.

\bibitem[Oh26]{Oh-ICM-26}
Hee Oh.
\newblock Dynamics and rigidity through the lens of circles.
\newblock In {\em International Congress of Mathematicians 2026}, pages
  240--271. SIAM, 2026.

\bibitem[OP19]{Oh-Pan-19}
Hee Oh and Wenyu Pan.
\newblock Local mixing and invariant measures for horospherical subgroups on
  abelian covers.
\newblock {\em Int. Math. Res. Not.}, 2019(19):6036--6088, 2019.

\bibitem[Pat76]{Patterson-76}
Samuel~J. Patterson.
\newblock The limit set of a {Fuchsian} group.
\newblock {\em Acta Math.}, 136:241--273, 1976.

\bibitem[Pol85]{Pollicott-85}
Mark Pollicott.
\newblock On the rate of mixing of {A}xiom {A} flows.
\newblock {\em Inventiones Mathematicae}, 81(3):413--426, 1985.

\bibitem[Pol91]{Pollicott-91}
Mark Pollicott.
\newblock Homology and closed geodesics in a compact negatively curved surface.
\newblock {\em Am. J. Math.}, 113(3):379--385, 1991.

\bibitem[PP90]{Parry-Pollicott-90}
William Parry and Mark Pollicott.
\newblock {\em Zeta functions and the periodic orbit structure of hyperbolic
  dynamics}, volume 187-188 of {\em Ast{\'e}risque}.
\newblock Paris: Soci{\'e}t{\'e} Math{\'e}matique de France, 1990.

\bibitem[PS87]{Phillips-Sarnak-87}
Ralph Phillips and Peter Sarnak.
\newblock Geodesics in homology classes.
\newblock {\em Duke Math. J.}, 55:287--297, 1987.

\bibitem[PS01]{Pollicott-Sharp-01}
Mark Pollicott and Richard Sharp.
\newblock Asymptotic expansions for closed orbits in homology classes.
\newblock {\em Geom. Dedicata}, 87(1-3):123--160, 2001.

\bibitem[PS17]{Potrie-Sambarino-17}
Rafael Potrie and Andr{\'e}s Sambarino.
\newblock Eigenvalues and entropy of a {Hitchin} representation.
\newblock {\em Invent. Math.}, 209(3):885--925, 2017.

\bibitem[PS24]{Pollicott-Sharp-24}
Mark Pollicott and Richard Sharp.
\newblock Zeta functions in higher {Teichm{\"u}ller} theory.
\newblock {\em Math. Z.}, 306(3):20, 2024.
\newblock Id/No 37.

\bibitem[PSMdSS12]{Patrao-et-al-2012}
Mauro Patr{\~a}o, Luiz A.~B. San~Martin, La{\'e}rcio~J. dos Santos, and Lucas
  Seco.
\newblock Orientability of vector bundles over real flag manifolds.
\newblock {\em Topology Appl.}, 159(10-11):2774--2786, 2012.

\bibitem[Qui01]{Quint-2001}
Jean-Fran{\c{c}}ois Quint.
\newblock Exponential divergence of discrete groups in higher rank.
\newblock {\em C. R. Acad. Sci., Paris, S{\'e}r. I, Math.}, 333(2):87--90,
  2001.

\bibitem[Qui02]{Quint-2002}
J.-F. Quint.
\newblock Patterson-{Sullivan} measures in higher rank.
\newblock {\em Geom. Funct. Anal.}, 12(4):776--809, 2002.

\bibitem[Rob03]{Roblin-03}
Thomas Roblin.
\newblock {\em Ergodicit{\'e} et {\'e}quidistribution en courbure
  n{\'e}gative}.
\newblock Number~95 in M{\'e}m. Soc. Math. Fr. (N.S.). Soci{\'e}t{\'e}
  Math{\'e}matique de France, Paris, 2003.

\bibitem[Rue78]{Ruelle-78}
David Ruelle.
\newblock {\em Thermodynamic formalism. {The} mathematical structures of
  classical equilibrium. {Statistical} mechanics. {With} a foreword by
  {Giovanni} {Gallavotti}}, volume~5 of {\em Encycl. Math. Appl.}
\newblock Cambridge University Press, Cambridge, 1978.

\bibitem[Rue86]{Ruelle-86}
David Ruelle.
\newblock Resonances of chaotic dynamical systems.
\newblock {\em Physical Review Letters}, 56(5):405--407, 1986.

\bibitem[Rue87]{Ruelle-87}
David Ruelle.
\newblock Resonances for {A}xiom {A} flows.
\newblock {\em Journal of Differential Geometry}, 25(1):99--116, 1987.

\bibitem[Sam14]{Sambarino-14}
A.~Sambarino.
\newblock Hyperconvex representations and exponential growth.
\newblock {\em Ergodic Theory Dyn. Syst.}, 34(3):986--1010, 2014.

\bibitem[Sam15]{Sambarino-15}
Andr{\'e}s Sambarino.
\newblock The orbital counting problem for hyperconvex representations.
\newblock {\em Ann. Inst. Fourier}, 65(4):1755--1797, 2015.

\bibitem[Sam24]{Sambarino-24}
Andr{\'e}s Sambarino.
\newblock A report on an ergodic dichotomy.
\newblock {\em Ergodic Theory Dyn. Syst.}, 44(1):236--289, 2024.

\bibitem[Sel56]{Selberg-56}
Atle Selberg.
\newblock Harmonic analysis and discontinuous groups in weakly symmetric
  {R}iemannian spaces with applications to {D}irichlet series.
\newblock {\em J. Indian Math. Soc. (N.S.)}, 20:47--87, 1956.

\bibitem[Sha92]{Sharp-92}
Richard Sharp.
\newblock Prime orbit theorems with multi-dimensional constraints for {Axiom}
  {A} flows.
\newblock {\em Monatsh. Math.}, 114(3-4):261--304, 1992.

\bibitem[Sha93]{Sharp-93}
Richard Sharp.
\newblock Closed orbits in homology classes for {Anosov} flows.
\newblock {\em Ergodic Theory Dyn. Syst.}, 13(2):387--408, 1993.

\bibitem[SP16]{Petkov-Stoyanov-16}
Luchezar Stoyanov and Vesselin Petkov.
\newblock Ruelle transfer operators with two complex parameters and
  applications.
\newblock {\em Discrete Contin. Dyn. Syst.}, 36(11):6413--6451, 2016.

\bibitem[Sul79]{Sullivan-79}
Dennis Sullivan.
\newblock The density at infinity of a discrete group of hyperbolic motions.
\newblock {\em Publ. Math. Inst. Hautes {\'E}tudes Sci.}, 50:171--202, 1979.

\bibitem[Thi09]{Thirion-09}
Xavier Thirion.
\newblock Mixing properties of {Weyl} chamber flows of ping-pong groups.
\newblock {\em Bull. Soc. Math. Fr.}, 137(3):387--421, 2009.

\bibitem[Wad96]{Waddington-96}
Simon Waddington.
\newblock Large deviation asymptotics for {Anosov} flows.
\newblock {\em Ann. Inst. Henri Poincar{\'e}, Anal. Non Lin{\'e}aire},
  13(4):445--484, 1996.

\bibitem[WW23]{Weich-Wolf-23}
Tobias Weich and Lasse~L. Wolf.
\newblock Absence of principal eigenvalues for higher rank locally symmetric
  spaces.
\newblock {\em Commun. Math. Phys.}, 403(3):1275--1295, 2023.

\bibitem[WX22]{Wu-Xue-22}
Yunhui Wu and Yuhao Xue.
\newblock Random hyperbolic surfaces of large genus have first eigenvalues
  greater than $\frac{3}{16}-\varepsilon$.
\newblock {\em Geom. Funct. Anal.}, 32(2):340--410, 2022.

\end{thebibliography}

\end{document}